\documentclass[aos,preprint]{imsart}
\usepackage{amsfonts,amsmath,amssymb,amsthm,mathtools,bbm,bm}
\usepackage[notrig]{physics}
\usepackage{commath,relsize}
\usepackage{graphicx,caption,subcaption,hyperref,enumitem,multirow} 
\usepackage[ruled,vlined,linesnumbered]{algorithm2e}
\usepackage[nameinlink,capitalize,noabbrev]{cleveref}
\usepackage{crossreftools}
\usepackage[authoryear]{natbib}
\usepackage[title,toc]{appendix}
\usepackage[stickstoo,vvarbb]{newtxmath} 
\usepackage[dvipsnames,x11names]{xcolor}
\usepackage{tikz}
\usepackage{multibib} 
\newcites{S}{References}
\hypersetup{colorlinks=True, urlcolor=MidnightBlue, citecolor=Blue, linkcolor=RedViolet}

\numberwithin{equation}{section}
\newenvironment{equations}{\equation\aligned}{\endaligned\endequation}
\newtheorem{theorem}{Theorem}[section]
\newtheorem{lemma}[theorem]{Lemma}
\newtheorem{proposition}[theorem]{Proposition}
\newtheorem{corollary}[theorem]{Corollary}
\theoremstyle{definition}
\newtheorem{definition}{Definition}
\newtheorem{assumption}{Assumption}
\newcommand{\assumnum}[1]{\textsf{\Alph{#1}}}

\newtheorem{examplex}{Example}
\newenvironment{example}
  {\pushQED{\qed}\examplex}
  {\popQED\endexamplex}
\newtheorem{remarkx}{Remark}
\newenvironment{remark}
  {\pushQED{\qed}\remarkx}
  {\popQED\endremarkx}

\crefname{assumption}{Assumption}{Assumptions}
\crefname{examplex}{Example}{Examples}
\crefname{theorem}{Theorem}{Theorems}
\crefname{section}{Section}{Sections}
\crefname{enumi}{Assumption}{Assumptions}
\crefname{equation}{Inequality}{Inequalities}
\Crefname{equation}{Equation}{Equations}
\def\vQ{\textit{Q}}
\def\hvQ{\widehat{\textit{Q}}}
\def\net{\textup{\textsf{net}}}
\def\icubes{\mathlarger{\mathlarger{\boxplus}}}
\def\mid{\mathrm{mid}}
\def\relu{\mathrm{ReLU}}
\def\SS{\texttt{\textup{SS}}}

\makeatletter
\newcommand*{\rom}[1]{\expandafter\@slowromancap\romannumeral #1@}
\makeatother

\newcommand{\floor}[1]{\left\lfloor {#1} \right\rfloor}
\newcommand{\ceil}[1]{\left\lceil {#1} \right\rceil}
\newcommand{\Norm}[1]{\left\| {#1} \right\|}
\newcommand{\nnorm}[2][0]{\norm[{#1}]{{#2}}_{n,2}}
\newcommand{\fnorm}[2][0]{\norm[{#1}]{{#2}}_{\textup{F}}}
\newcommand{\opnorm}[2][0]{\norm[{#1}]{{#2}}_{\textup{op}}}

\DeclareMathOperator*{\argmin}{\arg\!\min}
\DeclareMathOperator*{\argmax}{\arg\!\max}
\def\iidsim{\stackrel{\textup{iid}}{\sim}}
\def\indsim{\stackrel{\textup{ind}}{\sim}}
\def\zero{\mathbf 0}
\def\one{\mathbf 1}
\def\ind{\mathbbm{1}}
\def\kl{\textup{KL}}
\def\supp{\textup{supp}}

\def\Ber{\texttt{\textup{Bernoulli}}}

\def\Cat{\texttt{\textup{Cat}}}

\def\Unif{\texttt{\textup{Unif}}}
\def\Pois{\texttt{\textup{Poisson}}}

\def\Dir{\texttt{\textup{Dir}}}
\def\Geo{\texttt{\textup{Geometric}}}

\def\N{\texttt{\textup{N}}}

\def\d{\textup{d}}
\def\e{\textup{e}}

\def\P{\textsf{\textup{P}}}
\def\R{\mathbb{R}}

\makeatletter
\def\greekvectors#1{%
 \@for\next:=#1\do{%
    \def\X##1;{\expandafter\def\csname b##1\endcsname{\bm{\csname##1\endcsname}}}
    \expandafter\X\next;}
 \@for\next:=#1\do{%
    \def\X##1;{\expandafter\def\csname h##1\endcsname{\widehat{\csname##1\endcsname}}}
    \expandafter\X\next;}
 \@for\next:=#1\do{%
    \def\X##1;{\expandafter\def\csname t##1\endcsname{\widetilde{\csname##1\endcsname}}}
    \expandafter\X\next;}
 \@for\next:=#1\do{%
    \def\X##1;{\expandafter\def\csname ba##1\endcsname{\bar{\csname##1\endcsname}}}
    \expandafter\X\next;}
 \@for\next:=#1\do{%
    \def\X##1;{\expandafter\def\csname c##1\endcsname{\check{\csname##1\endcsname}}}
    \expandafter\X\next;}
 \@for\next:=#1\do{%
    \def\X##1;{\expandafter\def\csname u##1\endcsname{\underline{\csname##1\endcsname}}}
    \expandafter\X\next;}
 \@for\next:=#1\do{%
    \def\X##1;{\expandafter\def\csname hb##1\endcsname{\widehat{\bm{\csname##1\endcsname}}}}
    \expandafter\X\next;}
}
\greekvectors{alpha,beta,gamma,delta,epsilon,zeta,theta,kappa,lambda,mu,nu,xi,pi,rho,tau,sigma,phi,chi,psi,omega,varepsilon,varpi,varphi,vartheta,  Delta,Gamma,Theta,Lambda,Xi,Pi,Sigma,Phi,Psi,Omega}
    
\@tfor\next:=abfghijlmnopqrstuwxyzABCDGHIJKLMOQSTUVWXYZ\do{%
    \def\command@factory#1{\expandafter\def\csname #1\endcsname{\mathbf{#1}} }
    \expandafter\command@factory\next
}
\@tfor\next:=abcdeghijklnopqrstuvwxyzABCDEFGHIJKLMNOPQRSTUVWXYZ\do{%
    \def\command@factory#1{\expandafter\def\csname b#1\endcsname{\mathbbm{#1}} }
    \expandafter\command@factory\next
}
\@tfor\next:=abcdeghijklnpqrstuvwxyzABCDEFGHIJKLMNOPQRSTUVWXYZ\do{%
    \def\command@factory#1{\expandafter\def\csname t#1\endcsname{\texttt{\textup{#1}}} }
    \expandafter\command@factory\next
}
\@tfor\next:=ABCDEFGHIJKLMNOPQRSTUVWXYZ\do{%
    \def\command@factory#1{\expandafter\def\csname c#1\endcsname{\mathcal{#1}} }
    \expandafter\command@factory\next
}
\@tfor\next:=abdefghijklmnopqrstuvwxyzABCDEFGHIJKLMNOPQRSTUVWXYZ\do{%
    \def\command@factory#1{\expandafter\def\csname s#1\endcsname{\textsf{\textup{#1}}}}
    \expandafter\command@factory\next
}
\@tfor\next:=abcdefghjklmnopqrstuvwxyzABCDEFGHIJKLMNOPQRSTUVWXYZ\do{
    \def\command@factory#1{\expandafter\def\csname f#1\endcsname{\mathfrak{#1}} }
    \expandafter\command@factory\next
}
\@tfor\next:=abcdefghjklmnopqrstuvwxyzABCDEFGHIJKLMNOPQRSTUVWXYZ\do{
    \def\command@factory#1{\expandafter\def\csname sc#1\endcsname{\mathscr{#1}} }
    \expandafter\command@factory\next
}
\@tfor\next:=abcdeghijklmnopqstuvwxyzABCDEFGHIJKLMNOPQRSTUVWXYZ\do{%
    \def\command@factory#1{\expandafter\def\csname ba#1\endcsname{\bar{#1}} }
    \expandafter\command@factory\next
}
\@tfor\next:=xyzABCDEFGHIJKLMNOPQRSTUVWXYZ\do{%
    \def\command@factory#1{\expandafter\def\csname h#1\endcsname{\widehat{#1}} }
    \expandafter\command@factory\next
}
\@tfor\next:=abcdefghijklmnopqrstuvwxyzABCDEFGHIJKLMNOPQRSTUVWXYZ\do{%
    \def\command@factory#1{\expandafter\def\csname ti#1\endcsname{\widetilde{#1}} }
    \expandafter\command@factory\next
}
\@tfor\next:=abcdefghjklmnopqrstuvwxyzABCDEFGHIJKLMNOPQRSTUVWXYZ\do{
    \def\command@factory#1{\expandafter\def\csname u#1\endcsname{\underline{#1}} }
    \expandafter\command@factory\next
}

\begin{document}

\begin{frontmatter}

\title{Adaptive variational Bayes: Optimality, computation and applications}
\runtitle{Adaptive variational Bayes}

\begin{aug}
\author[A]{\fnms{Ilsang} \snm{Ohn}\ead[label=e1,mark]{ilsang.ohn@inha.ac.kr}}
\and
\author[B]{\fnms{Lizhen} \snm{Lin}\ead[label=e2,mark]{lizhen01@umd.edu}}
\address[A]{Department of  Statistics, Inha University \\ 
\printead{e1} }%
\address[B]{Department of Mathematics, The University of Maryland \\ 
\printead{e2} }%
\end{aug}
\runauthor{I. Ohn and L. Lin}

\begin{abstract}
In this paper, we explore adaptive inference based on variational Bayes. Although several studies have been conducted to analyze the contraction properties of variational posteriors, there is still a lack of a general and computationally tractable variational Bayes method that performs adaptive inference. To fill this gap, we propose a novel \textit{adaptive variational Bayes} framework, which can operate on a collection of models. The proposed framework first computes a variational posterior over each individual model separately and then combines them with certain weights to produce a variational posterior over the entire model. It turns out that this combined variational posterior is the closest member to the posterior over the entire model in a predefined family of approximating distributions. We show that the adaptive variational Bayes attains optimal contraction rates adaptively under very general conditions. We also provide a methodology to maintain the tractability and adaptive optimality of the adaptive variational Bayes even in the presence of an enormous number of individual models, such as sparse models. We apply the general results to several examples, including deep learning and sparse factor models, and derive new and adaptive inference results. In addition, we characterize an implicit regularization effect of variational Bayes and show that the adaptive variational posterior can utilize this. 
\end{abstract}



\end{frontmatter}

\section{Introduction}

The bias-variance trade-off, a fundamental principle of statistical inference, suggests that a statistical model needs to be appropriately chosen in order to gain useful information from data efficiently. For example, in nonparametric regression or density estimation, the complexity of a model should be selected or specified according to the smoothness of the true regression function or density to optimize the prediction or estimation risk. Another example is high-dimensional linear regression, where it is almost necessary to impose certain sparsity on the regression coefficients to reduce statistical variability in high-dimensional estimation. It will be easier to determine a suitable model when we have some knowledge of the underlying data-generating process. However, this knowledge is rarely available in practice. Therefore, accurate inference requires the statistical method to be \textit{adaptive}, i.e., to be able to learn from data the relevant features of the data-generating process so that appropriate model determinations can be made.

From this adaptive inference viewpoint, Bayesian inference is appealing. Given a collection of models, a Bayesian procedure can automatically learn the appropriate model via a hierarchical prior design, which first assigns a prior over models followed by a prior on the parameter given the selected model. A number of studies \citep{belitser2003adaptive,lember2007universal,ghosal2008nonparametric,arbel2013bayesian,gao2020general,han2021oracle} have shown that with a carefully designed hierarchical prior, one can conduct optimal inference adaptively using the (original) posterior distribution. However, posterior inference based on sampling the posterior distribution over different models is usually computationally demanding and often very inefficient, especially when moves between models are poorly proposed in a Markov chain Monte Carlo (MCMC) algorithm.

High-dimensional and/or huge data sets have been more frequent in practical applications, with which computing posterior distributions exactly by MCMC algorithms is time-consuming and even intractable. An attractive and popular alternative is \textit{variational Bayes} that offers computationally fast approximations, called \textit{variational posteriors}, of posterior distributions using optimization algorithms. Along with the increasing popularity of variational Bayes, theoretical analysis of variational posteriors has been conducted recently.  \citet{zhang2020convergence} provided mild and general sufficient conditions for establishing contraction rates of variational posteriors. \citet{pati2018statistical} and  \citet{yang2020alpha} developed variational Bayes theoretic frameworks that can deal with latent variable models.  \cite{alquier2020concentration} investigated the contraction properties of variational fractional posteriors where the likelihood is replaced by its fractional version with a positive exponent less than 1. There are several studies that derived contraction rates of variational posteriors for a specific statistical model, for example, mixture models \citep{cherief2018consistency}, sparse (Gaussian) linear regression  \citep{ray2021variational, yang2020variational}, sparse logistic linear regression \citep{ray2020spike} and sparse factor models \citep{ning2021spike}.

Nonetheless, the adaptivity of variational posteriors remains an important and largely open problem. Although the existing theoretical frameworks have produced adaptive variational Bayes methods for some specific models, they require specially designed prior distributions and variational families to achieve adaptivity, which restricts the theory's usefulness. Another problem concerns computation. Construction of a computationally tractable variational Bayes method over a collection of models is not an easy task in general. Exceptionally, \citet{zhang2020convergence} proposed a potentially adaptive variational Bayes procedure with an additional model selection stage, which we call \textit{model selection variational Bayes}. They showed that after selecting the best model among multiple models and conducting variational inference over the selected model, the variational posterior arising from this process could be optimal under mild assumptions. Although not clearly stated, this procedure appears capable of achieving adaptive optimality in many problems. In \citet{cherief2019consistency}, a similar model selection approach was applied to variational fractional posteriors and its theoretical properties were studied. 

In this paper, we propose a new variational Bayes method, called \textit{adaptive variational Bayes}. Instead of selecting the best model and using the variational posterior over the selected model for subsequent inference tasks, the proposed framework aggregates multiple variational posteriors, each of which is obtained for each individual model separately, with certain weights to produce a variational posterior over the entire model, referred to as the \textit{adaptive variational posterior}. This aggregated variational posterior turns out to be a closer approximation to the original posterior than the variational posterior over the selected model. Theoretically, the adaptive variational posterior can attain optimal contraction rates adaptively for a wide variety of statistical problems under mild conditions on priors and variational families. To the best of our knowledge, our framework is the first general recipe for establishing the adaptive contraction of variational posteriors.

We summarized our contributions as follows.

\begin{enumerate}
    \item \textbf{Computational tractability.} As we have mentioned previously, Bayesian inference via a posterior distribution over multiple models has been shown to be an adaptively optimal procedure in a wide range of statistical applications. Therefore naturally, it is desirable to approximate the posterior distribution over the entire model, which is not easy to compute and often intractable. We demonstrate that the adaptive variational posterior, which is an aggregate of individual variational posteriors, is the closest member to the posterior in a predefined family of approximating distributions. This implies that the adaptive variational Bayes can inherit the computational tractability of variational Bayes methods used for obtaining the individual variational posteriors as long as the model complexity is not too large.
    
    \item \textbf{Adaptive contraction rate and model selection consistency.} We formulate mild conditions under which the adaptive variational posterior can attain optimal contraction rates adaptively. Our theoretical conditions are slightly simpler than those of \cite{zhang2020convergence} in the sense that the ``prior mass condition'' can be ``hidden'', and that a condition associated with a stronger divergence than the Kullback-Leibler divergence can be relaxed. We want to clarify that these technical simplifications are not related to any aspect of the proposed method, but are made by rearranging the proof of \cite{zhang2020convergence}. We also provide some easily verifiable sufficient conditions for our theoretical assumptions. Moreover, we show that the adaptive variational Bayes method does not severely overestimate and underestimate the ``best'' model that leads to an optimal contraction rate. We apply our general theory to deep neural network models and derive adaptive optimal contraction rates in a number of applications. 
    
    \item \textbf{Extension to combinatorial model spaces.} Although completely parallelizable, the computation of the adaptive variational posterior becomes intractable when the number of models is extremely large because of the need to obtain a variational posterior for every individual model. This is the case for statistical models involving a ``combinatorial'' model structure such as high-dimensional sparse linear regression, where the individual models may be divided by a sparse pattern of the regression coefficients. This is clearly different from a ``nested'' model space such as a mixture model, in which the individual models can be ordered by their complexity, such as the number of mixture components. We show that, however, the proposed approach can be applied to combinatorial model spaces by utilizing tailored priors and variational families. Our method is particularly useful for model spaces with both combinatorial and nested structures, such as sparse factor models and high-dimensional nonparametric regression, and we study these examples.
    
    \item \textbf{Regularization via variational approximation.} In fact, under the theoretical conditions we formulate in our main theory, of which a key part is related to the choice of a prior, the original posterior distribution can contract adaptively also. In this regard, our theory does not reveal a theoretical merit of the adaptive variational posterior over the original posterior. However, we find that there is a situation where the adaptive variational posterior is guaranteed to behave well while the original posterior is not. This phenomenon, which we refer to as \textit{implicit variational Bayes regularization}, comes from the choice of variational families.
    
    \item \textbf{Theoretical justification of the use of quasi-likelihood.} We consider the use of quasi-likelihoods in the proposed adaptive variational Bayes framework. We formulate conditions on quasi-likelihoods under which adaptive contraction rates are achieved by variational quasi-posteriors. We apply the general result to stochastic block models and nonparametric regression with sub-Gaussian errors. 
\end{enumerate}

The rest of the paper is organized as follows. We first introduce some notation in the rest of this section.  In \cref{sec:method}, we develop a general framework for adaptive variational Bayes inference. Concretely, we provide a design of the prior distribution and variational family, and describe a general and simple scheme for computation of the proposed variational posterior. In \cref{sec:theory}, we study the contraction properties of the proposed variational posterior, including oracle rates, adaptivity and model selection properties.  In \cref{sec:dnn}, we apply our general results to variational deep learning. In \cref{sec:sparse}, we propose an approach by which the adaptive variational Bayes is applicable even when the number of individual models is quite large. In \cref{sec:regularization}, we study the implicit variational Bayes regularization and derive adaptive contraction rates using this regularization effect.  In \cref{sec:quasi}, we study the theoretical properties of variational quasi-posteriors where the usual likelihood function is replaced with an alternative quasi-likelihood.

\subsection{Notation}

Let $\R$, $\R_+$, $\R_{\ge0}$, $\bZ$, $\bN_0$ and $\bN$  be the sets of real numbers, positive numbers, nonnegative numbers, integers, nonnegative integers and natural numbers, respectively. We denote by $\ind(\cdot)$ the indicator function.  For two integers $z_1,z_2\in\bZ$ with $z_1\le z_2$, we let $[z_1:z_2]:=\cbr{z\in\bZ:z_1\le z\le z_2}$ and if $z_2\ge2$, we use the shorthand $[z_2]:=[1:z_2]$ and $\bN_{\ge z_2}:=\bN\setminus[z_2-1]=\cbr{z\in\bN:z\ge z_2 }$. For $d\in\bN$, let $\zero_d$ and $\one_d$ denote the $d$-dimensional vectors of 0's and of 1's, respectively and let $\I_d$ denote the $d\times d$-dimensional identity matrix. For a $d$-dimensional vector $\x:=(x_j)_{j\in[d]}\in\R^d$, we denote $|\x|_q:=\del[0]{\sum_{j=1}^d|x_j|^q}^{1/q}$ for  $q\ge1$, which are the usual Euclidean $q$-norm, and denote $|\x|_0:=\sum_{j=1}^d\ind(x_j\neq0)$ and $|\x|_\infty:=\max_{j\in[d]}|x_j|$.  Let $\Delta_d:=\cbr[0]{\balpha\in[0,1]^d:|\balpha|_1=1}$. We denote by $\bS_{++}^{d}$ the set of $d\times d$ symmetric positive definite matrices. For a real number $x\in\R$, we define $\floor{x}:=\max\{z\in\bZ:z\le x\}$,  and $\ceil{x}:=\min\{z\in\bZ:z\ge x\}$. For two real numbers $a,b\in\R$, we write $a\vee b:=\max\{a,b\}$ and $a\wedge b:=\min\{a,b\}$. Moreover, for two real vectors  $\a:=(a_j)_{j\in[d]}\in\R^d$ and $\b:=(b_j)_{j\in[d]}\in\R^d$ with the same dimension, we write $\a\vee\b:=(a_j\vee b_j)_{j\in[d]}$ and $\a\wedge\b:=(a_j\wedge b_j)_{j\in[d]}$. For an arbitrary set $B$, we denote by $B^\complement$ its complement and by $|B|$ its cardinality. Let $\bP(B)$ be the powerset of $B$, i.e., $\bP(B):=\{B':B'\subset B\}$. For two positive sequences $(a_n)_{n\in \mathbb{N}}$ and $(b_n)_{n\in \mathbb{N}}$, we write $a_n\lesssim b_n$ or  $b_n\gtrsim a_n$ or $a_n=\scO(b_n)$, if there exists a positive constant $C>0$ such that $a_n\le Cb_n$ for any $n\in \mathbb{N}$. Moreover, we write $a_n\asymp b_n$ if both $a_n\lesssim b_n$  and  $a_n\gtrsim b_n$ hold. We write $a_n=\sco(b_n)$ if $\lim_{n\to\infty}a_n/b_n=0.$  Absolute constants denoted by e.g., $\fc_1,\fc_2,\dots$ may vary from place to place.

For a measurable space $(\bX,\cX)$, we denote by $\cP(\bX)$ the set of all probability measures supported on $\bX$. For a probability measure $\P\in\cP(\bX)$ and a function $g$ on $\bX$, we write $\P g:=\int_{\bX} g\d\P$, i,e., $\P g$ denotes the expectation of $g$ with respect to the measure $\P$. For two probability distributions $\P_1\in\cP(\bX)$ and $\P_2\in\cP(\bX)$, we denote by $\kl(\P_1, \P_2)$ the Kullback-Leibler (KL) divergence from $\P_2$ to $\P_1$, which is defined by $ \kl(\P_1,\P_2):=\int \log(\frac{\d\P_1}{\d\P_2})\d\P_1$ if $\P_1\ll\P_2$ and $ \kl(\P_1,\P_2):=\infty$ otherwise. For simplicity, we slightly abuse a notation to denote $\kl(\balpha_1,\balpha_2):=\kl(\Cat(\balpha_1), \Cat(\balpha_2))$ for any $\balpha_1,\balpha_2\in\Delta_m$ and any $m\in\bN$, where $\Cat(\balpha)$ stands for the categorical distribution with probability vector $\balpha.$ For $\rho\in(0,1)\cup(1,\infty)$, the $\rho$-R\'enyi divergence from $\P_2$ to $\P_1$ is defined as $\sD_\rho(\P_1,\P_2):=\frac{1}{\rho-1}\log\del[1]{\int \del{\frac{\d\P_1}{\d\P_2}}^{\rho-1}\d\P_1}$ if $\P_1\ll\P_2$ and $ \sD_\rho(\P_1,\P_2):=\infty$ otherwise. For $x\in\bX$, let $\delta(\cdot;x)$ be a Dirac-delta measure at $x$ such that $\delta(B;x)=1$ if $x\in B$ and $\delta(B;x)=0$ otherwise.

\section{Adaptive variational Bayes}
\label{sec:method}

In this section, we develop a novel variational Bayes framework called adaptive variational Bayes. We first describe a statistical setup, then introduce an estimation procedure in the proposed framework as well as a general and simple scheme for its computation.

\subsection{Statistical experiment and models}
\label{sec:method:exprmt}

We describe our setup for a \textit{statistical experiment}, which, in this paper, is defined as a pair of a sample space and a set of some distributions on the sample space. For each sample size $n\in\bN$,  suppose that we observe a $\bY_n$-valued sample $\Y^{(n)}$, where $\bY_n$ is a measurable \textit{sample space} equipped with a reference $\sigma$-finite measure $\mu_n$.  We then model the sample as having a distribution $\P_{\blambda}^{(n)}\in\cP(\bY_n)$ determined by a \textit{natural parameter} $\blambda$ in a measurable \textit{natural parameter space} $\Lambda_n$. The natural parameter space can be infinite-dimensional, for 
instance, see \cref{example:mixture} below, where the natural parameter space is given as the space of density functions. We assume that there exists a nonnegative function $\sp_n:\Lambda_n\times \bY_n\mapsto\R_{\ge0}$, called a \textit{likelihood (function)}, such that, for every parameter $\blambda\in\Lambda_n$,  $\int_{\bY_n}\sp_n(\blambda,\y^{(n)})\d\mu_n(\y^{(n)})=1$ and 
    \begin{equation}
        \label{eq:likelihood}
        \P_{\blambda}^{(n)}(\d\y^{(n)})=\sp_n(\blambda, \y^{(n)}) \mu_n(\d\y^{(n)})
    \end{equation}
for any $\y^{(n)}\in\bY_n$. We denote by $\cP(\bY_n;\sp_n,\Lambda_n)$ the set of all distributions of the form \labelcref{eq:likelihood}. 

With a countable set of model indices $\cM_n$ that we call a \textit{model space}, we consider a collection of probability models $\{\cP_{n,m}\}_{m\in\cM_n}$  for estimation based on the sample $\Y^{(n)}$, where each model is of the form
    \begin{equation*}
    \cP_{n,m}:=\cbr{\P_{\sT(\btheta)}^{(n)}\in\cP(\bY_n;\sp_n,\Lambda_n):\btheta\in\Theta_{n,m}}
    \end{equation*}
with a parameter space $\Theta_{n,m}$ and a measurable map $\sT:\cup_{m\in\cM_n}\Theta_{n,m}\mapsto \Lambda_n$ called a \textit{natural parameterization map}. For simplicity, we let $\Theta_{n,\cM_n}:=\cup_{m\in\cM_n}\Theta_{n,m}$. The map $\sT$ may depend on the sample size $n$ but we do not specify the subscript $n$ to this for brevity. We refer to ``submodels'' $\{\cP_{n,m}\}_{m\in\cM_n}$ as \textit{individual models} or simply models. 

In order to obtain computational tractability (see \cref{subsec:algorithm}), we assume that the parameter spaces $\{\Theta_{n,m}\}_{m\in\cM_n}$ are disjoint. This disjointness assumption is satisfied when the dimensions of the parameter spaces are different from each other. For example, $\{\R^m\}_{m\in\bN}$ are disjoint.
Furthermore, the natural parametrization map can mitigate any restriction of the disjointness assumption. One may construct a collection of disjoint parameter spaces, which possibly yields non-disjoint natural parameter spaces such that $\cbr{\sT(\btheta):\btheta\in\Theta_{n,m'}}\cap\cbr{\sT(\btheta):\btheta\in\Theta_{n,m}}\neq \emptyset$ for some $m\neq m'$.

\subsubsection{Model space and examples}
\label{subsec:modelspace}

For the rest of the paper, we specify the following two types of model spaces. 

\begin{definition}[Model spaces] $\:$
\begin{enumerate}
    \item We say that a model space $\cM_n$ is \textit{nested} or has a nested structure  if $\cM_n$ is a subset of $\bN^q$ for some $q\in\bN$.
    \item We say that a model space $\cS_n$ is \textit{combinatorial} or has a combinatorial structure if $\cS_n$ is the powerset of some subset $\bar{S}$ of $\bN$, i.e., $\cS_n=\bP(\baS):=\cbr{S:S\subset\baS}$.
\end{enumerate}
\end{definition}

We can impose an (strict partial) order $<$  on a nested model space $\cM_n\subset\bN^q$, which is defined sensibly in the context. For example, we write $m< m'$ for two model indices $m:=(m_j)_{j\in[q]},m':=(m_j')_{j\in[q]}\in\cM_n$, if $m_j< m_j'$ for every $j\in[q]$, or we write  $m< m'$  if   $\prod_{j\in[q]}m_j<\prod_{j\in[q]}m_j'$. We relate this order to the complexities of the individual models measured in an appropriate sense. Then we can say that the model $m'$ is more complex than the model $m$ if $m< m'$ and sort the individual models by their complexities. A collection of models with a nested model space is used to adapt ``smoothness'' for nonparametric regression or density estimation, by choosing an appropriately complex model among the individual models with various complexities. We provide two examples.

\begin{example}[Gaussian mixture model]
\label{example:mixture}
Consider a statistical experiment in which a $d$-dimensional real-valued sample $\Y^{(n)}:=(\Y_1,\dots, \Y_n)\in\bY_n:=(\R^d)^{\otimes n}$ of size $n$ is observed. Let $\cG^d$ denote the set of all probability density functions on $\R^d$. Let $\sp_n:\cG^d\times(\R^d)^{\otimes n}\mapsto\R_{\ge0}$ be a likelihood function such that
    \begin{equation*}
        \sp_n(g, \Y^{(n)})=\prod_{i=1}^ng(\Y_i) \mbox{ for $g\in\cG^d$,}
    \end{equation*}
and let $\P_{g}^{(n)}$ be a distribution of the sample $\Y^{(n)}$ with a probability density function $\sp_n(g,\cdot)$. We aim to estimate the density of the sample and we use Gaussian mixtures for this. Given \textit{the number of components} $m\in\bN$, and a mixture parameter $\btheta:=(\bvarpi, \bvartheta_1,\dots, \bvartheta_m, \bSigma_1,\dots,\bSigma_m)\in\Theta_{n,m}:=\Delta_m\times(\R^d)^{\otimes m}\times(\bS_{++}^d)^{\otimes m}$,  we let $\sT:\cup_{m\in\bN}\Theta_{n,m}\mapsto \cG^d$ be a map defined as
    \begin{equation*}
        \sT(\btheta)=\sum_{k=1}^m\varpi_kg_{\N(\bvartheta_k,\bSigma_k)},
    \end{equation*}
where $g_{\N(\bvartheta,\bSigma)}$ stands for the probability density function of the multivariate Gaussian distribution with mean $\bvartheta$ and covariance matrix $\bSigma$. Usually, we consider a nested model space $\cM_n:=[m_{\max}]$ with a pre-specified upper bound $m_{\max}\in\bN$ of the number of components. Then a Gaussian mixture model refers to a collection of the models  $\{\cP_{n,m}:=\cbr[0]{\P_{\sT(\btheta)}^{(n)}:\btheta\in\Theta_{n,m}}\}_{m\in\cM_n}$ with the disjoint
parameter spaces $\{\Theta_{n,m}\}_{m\in\cM_n}$. 
\end{example}

\begin{example}[Stochastic block model]
\label{example:sbm}
Suppose that we observe a sample $(Y_{i,j})_{(i,j)\in[n]^2:i>j}$ from a graph of $n$ nodes, where $Y_{i,j}\in\{0,1\}$ indicates whether there is a connection between the $i$ and $j$-th nodes. For $m\in[n]$, let 
    \begin{align}
        \cU_m&:=\cbr{\U\in [0,1]^{m\times m}:\U=\U^\top}, \label{eq:def_sm}\\
        \cZ_{n, m}&:=\cbr{\Z=(\z_1,\dots,\z_n)^\top\in\{0,1\}^{n\times m}: |\z_i|_1=1}.  \label{eq:def_znm}
    \end{align}
The stochastic block model assumes that the connectivities between $n$ nodes follow a distribution
    \begin{equation*}
        (Y_{i,j})_{(i,j)\in[n]^2:i>j}\sim \P_{\sT(\U, \Z)}^{(n)}:=\bigotimes_{i=2}^{n}\bigotimes_{j=1}^{i-1}\Ber(\Omega_{i,j})
        \mbox{ with }\Omega_{i,j}:=\z_i^\top\U\z_j
    \end{equation*}
for a \textit{community-wise connectivity probability matrix} $\U\in\cU_m$ and a \textit{community assignment matrix} $\Z:=(\z_1,\dots, \z_n)^\top\in\cZ_{n,m}$ with \textit{the number of communities} $m\in[n]$, where $\sT$ is a map defined by $\sT(\U,\Z)=(\z_i^\top\U\z_j)_{(i,j)\in[n]^2:i>j}$. Note that the parameter spaces $\{\cU_m\times\cZ_{n,m}\}_{m\in[n]}$ are disjoint. Here the model space $[n]$ is nested. For computational easiness, we consider some subset $\cM'_n\subset[n]$ which is also nested.
\end{example}

Unlike a nested model space, a combinatorial model space defined as $\cS_n:=\bP(\baS)$ has the ``exponentially growing'' cardinality $2^{|\baS|}$. A leading example is a sparse model, of which the model space $\cS_n$ is given by $\bP([d_n])$ with $d_n\in\bN$ being the number of parameters and each member of $\cS_n$ is a set of indices of ``active'' parameters. In this paper, we investigate model spaces with both combinatorial and nested structures instead of model spaces that only have a combinatorial structure, because our proposal has some merits for a nested structure. A sparse factor model is an example of the former.

 \begin{example}[Sparse factor model]
\label{example:sparsefactor}
Consider a factor model, where a $d_n$-dimensional  real-valued sample $\Y^{(n)}:=(\Y_1,\dots, \Y_n)\in\bY_n:=(\R^{d_n})^{\otimes n}$ of size $n$ is assumed to follow a distribution $\P_{\bSigma}^{(n)}:=\bigotimes_{i=1}^n\N\del{\zero_{d_n}, \bSigma}$ with a covariance matrix given by
    \begin{equation*}
       \bSigma= \sT(\L):=\L\L^\top+\I_{d_n}\in\bS_{++}^{d_n}
    \end{equation*}
for a \textit{factor loading matrix} $\L\in\R^{d_n\times m}$ and the factor dimensionality $m\in\bN$. To deal with high-dimensional settings where $d_n$ is much larger than $n$, we may assume the sparsity of the loading matrix. For a matrix $\L=(L_{j,k})_{j\in[d_n],k\in[m]}\in\R^{d_n\times m}$, let $\L_{j,:}:=(L_{j,k})_{k\in[m]}$ be the $j$-th row of $\L$ for $j\in[d_n]$ and we denote the ``row'' support of  $\L$ by
    \begin{align*}
        \supp(\L):=\cbr{j\in[d_n]:\abs[0]{\L_{j,:}}_0> 0}.
    \end{align*}
Given an upper bound $m_{\max}\in\bN$ of the factor dimensionality, we consider parameter spaces defined  as
    \begin{equation*}
        \Theta_{n,m,S}:=\cbr{\L\in\R^{d_n\times m}: \supp(\L)=S}.
    \end{equation*}
for each $m\in[m_{\max}]$ and $S\subset[d_n]$, which are disjoint. Note that the model space $[m_{\max}]\times \bP([d_n]) $ has both nested and combinatorial structures and its cardinality $m_{\max}2^{d_n}$ is exponential in the dimension $d_n$ of the sample.
\end{example}

As we will see in \cref{subsec:algorithm}, the extremely large cardinality of a model space makes our proposed adaptive variational Bayes computationally intractable. Therefore we need a modified approach that can address the case of combinatorial structures. We first focus on nested model spaces whose cardinalities are manageable, and describe details of the approach for dealing with combinatorial structures later in \cref{sec:sparse}.

\subsection{Prior and variational posterior}

Following the idea of the use of a hierarchical prior on model and parameter spaces suggested in the previous studies on Bayesian adaptation  \citep{lember2007universal,ghosal2008nonparametric,gao2020general,han2021oracle}, we consider a hierarchical prior distribution $\Pi_n$ on the entire parameter space $\Theta_{n,\cM_n}$, which is of the form
    \begin{equation}
    \label{eq:prior}
        \Pi_n=\sum_{m\in\cM_n}\alpha_{n,m}\Pi_{n,m},
    \end{equation}
where $\balpha_n:=(\alpha_{n,m})_{m\in\cM_n}\in \Delta_{|\cM_n|}$ and $\Pi_{n,m}\in\cP(\Theta_{n,m})$ for each $m\in\cM_n$. As $\Pi_n(\btheta\in\Theta_{n,m})=\sum_{m\in\cM_n}\alpha_{n,m}\Pi_{n,m}(\btheta\in\Theta_{n,m})=\alpha_{n,m}$, the term $\alpha_{n,m}$ represents the prior probability of a model $m\in\cM_n$. The (original) posterior distribution induced from the prior in \labelcref{eq:prior} is given by
    \begin{equation}
        \d\Pi_n(\btheta|\Y^{(n)}):=\frac{\sp_n(\sT(\btheta), \Y^{(n)})\d\Pi_n(\btheta)}{\int \sp_n(\sT(\btheta), \Y^{(n)})\d\Pi_n(\btheta) }.
    \end{equation}
    
Instead of exactly computing the posterior  $\Pi_n(\cdot|\Y^{(n)})$, here we seek the closest distribution, denoted by $\hvQ_n$, to the posterior among a variational family $\cQ_n\subset\cP(\Theta_{n,\cM_n})$ of ``hierarchical'' distributions given by
    \begin{equation}
    \label{eq:vfamily}
        \cQ_n:=\cbr[4]{\sum_{m\in\cM_n}\gamma_{n,m}\vQ_{n,m}:(\gamma_{n,m})_{m\in\cM_n}\in \Delta_{|\cM_n|}, \vQ_{n,m}\in\cQ_{n,m}},
    \end{equation}
where $\cQ_{n,m}\subset\cP(\Theta_{n,m})$ is a  predefined set of some distributions over the individual model $m\in\cM_n$. That is, we consider the variational posterior defined as
    \begin{equations}
    \label{eq:vbayes}
        \hvQ_n\in\argmin_{\vQ\in\cQ_n}\kl\del[1]{\vQ,\Pi_n(\cdot|\Y^{(n)})}.
    \end{equations}
We will show in the next section that the variational posterior $\hvQ_n$ attains optimal contraction rates adaptively under mild conditions. We, therefore, refer to it as the \textit{adaptive variational posterior distribution}.  

We close this subsection by introducing some additional notation and well-known facts. Let $\sp_{n,\Pi}(\Y^{(n)}):=\int\sp_n(\sT(\btheta), \Y^{(n)})\d\Pi(\btheta)$  be a marginal likelihood of a sample $\Y^{(n)}$ with respect to a prior distribution $\Pi$. We define
    \begin{equation}
       \scE_n(\vQ,\Pi,\sp_n):=-\int\log\sp_n(\sT(\btheta), \Y^{(n)})\d\vQ(\btheta)+\kl(\vQ, \Pi)
    \end{equation}
for any distributions $\vQ,\Pi\in\cP(\Theta_{n,\cM_n})$, where we suppress the dependence on the natural parametrization map $\sT$ and the sample $\Y^{(n)}$ for simplicity. Then by simple algebra, it can be shown that the minimization program in \labelcref{eq:vbayes} is equivalent to
    \begin{equation}
        \hvQ_n\in\argmin_{\vQ\in\cQ_n}\scE_n(\vQ,\Pi_n,\sp_n).
    \end{equation}
The negative of the function $\scE_n$ is called the evidence lower bound (ELBO) in the related literature, which is named after the property that the ELBO is a lower bound of the log marginal likelihood, i.e., $\log\sp_{n,\Pi}(\Y^{(n)})\ge-\scE_n(\vQ,\Pi,\sp_n)$.

\subsection{A general scheme for computation}
\label{subsec:algorithm}

The minimization program in \labelcref{eq:vbayes} needs to be done over distributions on the entire parameter space $\Theta_{n,\cM_n}$ which is usually quite ``big'' and even has a varying dimension. Therefore, this seems to be computationally intractable. However, as stated in the following theorem, if the parameter spaces are disjoint as we have assumed, the minimization program can be done by combining the minimization results over the individual models.

\begin{theorem}
\label{thm:compute}
Suppose that the parameter spaces $\{\Theta_{n,m}\}_{m\in\cM_n}$ are disjoint. Then the variational posterior distribution in \labelcref{eq:vbayes} satisfies
    \begin{equation}
        \label{eq:vposterior}
        \hvQ_n=\sum_{m\in\cM_n}\hgamma_{n,m}\hvQ_{n,m},
    \end{equation}
where
    \begin{equation}
    \label{eq:vpost_comp}
        \hvQ_{n,m}\in\argmin_{\vQ\in\cQ_{n,m}}\scE_n\del[1]{\vQ,\Pi_{n,m},\sp_n}
    \end{equation}
and 
    \begin{equation}
     \label{eq:vpost_mprob}
        \hgamma_{n,m}:=\frac{1}{\hat{Z}_{\gamma,n}} \alpha_{n,m}\exp\del[1]{-\scE_n\del[1]{\hvQ_{n,m},\Pi_{n,m},\sp_n}}
    \end{equation}
with $\hat{Z}_{\gamma,n}:=\sum_{m'\in\cM_n}\alpha_{n,m'}\exp\del[1]{-\scE_n(\hvQ_{n,m'},\Pi_{n,m'},\sp_n)}$ being the normalizing constant.
\end{theorem}

Thanks to the above theorem, if a computation algorithm for the variational posterior over each individual model is tractable and the number of individual models is not much large (which is the case for nested model spaces), the adaptive variational posterior distribution is also easily computable. Furthermore, since each individual variational posterior is computed separately, the adaptive variational Bayes procedure is easily parallelizable. We summarize a general procedure for computation in \cref{alg:general}.

\begin{algorithm}[!htp]
  \caption{Adaptive variational Bayes}
  \label{alg:general}
    \KwIn{Sample $\Y^{(n)}$, prior distribution $\Pi_n=\sum_{m\in\cM_n}\alpha_{n,m}\Pi_{n,m}$, variational families $\{\cQ_{n,m}\}_{m\in\cM_n}$.}
    \KwOut{The variational posterior distribution $\hvQ_n$.}
    \For {$m\in\cM_n$} {
       Find  $\hvQ_{n,m}\in\argmin_{\vQ\in\cQ_{n,m}}\scE_n(\vQ,\Pi_{n,m},\sp_n)$.\\
       Compute $\tgamma_{n,m}=\alpha_{n,m}\exp\del[0]{-\scE_n(\hvQ_{n,m},\Pi_{n,m},\sp_n)}$.
    }
    Compute $ \hgamma_{n,m}=\tgamma_{n,m}/(\sum_{m'\in\cM_n}\tgamma_{n,m'})$ for $m\in\cM_n$. \\
    \Return{ $\hvQ_n:=\sum_{m\in\cM_n}\hgamma_{n,m}\hvQ_{n,m}$}
\end{algorithm}

\begin{remark}
The vector of the variational posterior model probabilities $\hbgamma_{n}:=(\hgamma_{n,m})_{m\in\cM_n}$ can be viewed as the output of the softmax function $(z_m)_{m\in\cM_n}\mapsto(\e^{z_m}/\sum_{m'\in \cM_n}\e^{z_{m'}})_{m\in \cM_n}$ with the input $z_m =\log \alpha_{n,m}-\scE_n(\hvQ_{n,m},\Pi_{n,m},\sp_n)$, and its computation can be numerically unstable. The overflow and underflow issues are well known for the softmax function, which arise when $z_m$ is a too-large positive value so that the numerator $\e^{z_m}$ is considered as infinity or every $z_{m'}$ is a too-large negative value so that the denominator $\sum_{m'\in \cM_n}\e^{z_{m'}}$ is considered as zero. A conventional solution is to shift all inputs by their maximum value $z_{\max}:=\max_{m'\in\cM_n} z_{m'}$, that is, to compute $\e^{z_m-z_{\max}}/\sum_{m'\in \cM_n}\e^{z_{m'}-z_{\max}}$ which is the same as the value before the shift. This prevents the overflow issue since $z_m-z_{\max}\le 0$ for every $m$ as well as the underflow issue since there is at least one $m^*$ such that $\e^{z_{m^*}-z_{\max}}=\exp(0)=1$ and so dividing by zero is avoided. In our numerical studies,  we use this technique to stabilize the softmax function.
\end{remark}

\subsection{Practicalities}

In this subsection, we provide a remark on the practical implementation of the adaptive variational Bayes. For a successful application, the proposed framework requires two ingredients. The first is a prior distribution on multiple parameter spaces, which will lead to the theoretically and/or empirically well-behaved posterior, and the second is a pair of a variational family and a computation algorithm for each individual model, with which the variational optimization can be efficiently solved. Then we use \cref{alg:general} to get the variational approximation of the posterior on the entire parameter space. The adaptive variational posterior is expected to work reasonably well given a good prior, as well as is tractable when an efficient computation algorithm is used and the number of individual models is moderate. An user may refer to some previous work to choose both a prior and a computation algorithm. We illustrate this with examples.

\begin{example}[Gaussian mixture model]
Consider the Gaussian mixture model in \cref{example:mixture}.  A Bayesian mixture model with a varying number of components has been widely used in applications, see \citet{miller2018mixture} and references therein. \citet{miller2018mixture} shows, in their numerical studies, that the posterior performs well when they impose the $\Geo(p)$ prior on the number of components $m\in\bN$, i.e., set $\alpha_{n,m}=(1-p)^{m-1}p$ for a given hyperparameter $p\in(0,1)$, and conditional on $m$, a product of Dirichlet, Gaussian and inverse gamma prior on the mixture parameter $\btheta\in\Theta_{n,m}$. Following this, we may consider an exponentially decaying prior such as $\alpha_{n,m}\propto(1-p)^{m-1}\propto \e^{-m\log\del[0]{1-p}}$ for $m\in[m_{\max}]$.  When we consider the mean-field variational family for each the mixture model with fixed $m$, the coordinate ascent algorithm \citep[e.g., Section 10.2 of][]{bishop2006pattern} can be used to obtain the variational posterior for each $m\in[m_{\max}]$. Then aggregating them with \cref{alg:general}, we get the variational posterior that approximates the original posterior over the Gaussian mixture model with a varying number of components.
\end{example}

\begin{example}[Stochastic block model]
\citet{geng2019probabilistic} studied Bayesian estimation of the stochastic block model with the varying number of communities, which we considered in \cref{example:sbm}.  The authors used the $\Pois(1)$ prior on the number of communities $m$ and conditional on $m$, beta prior on each $U_{k,h}$ and Dirichlet-multinomial prior on each $\z_i$ for their theoretical and empirical studies. Following this suggestion, we may consider a truncated Poisson prior on $m\in[m_{\max}]$, and the same conditional prior on the parameter. For a variational optimization algorithm for fixed $m$, we can use the coordinate ascent algorithm for the mean-field variational family, given in \citet{zhang2020theoretical}. By applying the adaptive variational Bayes, we get the variational posterior over the stochastic block model with a varying number of communities.
\end{example}

\subsection{Comparison with model selection variational Bayes} 
\label{sec:method:compare}

The model selection  variational Bayes procedure considered in \cite{zhang2020convergence} proposes to solve the minimization program
    \begin{equation}
    \label{eq:compare:msvb_minimization}
        \min_{m\in\cM_n}\min_{\vQ\in\cQ_{n,m}}
        \cbr{\scE_n(\vQ,\Pi_{n,m},\sp_n)-\log\alpha_{n,m}}.
    \end{equation}
Note that the distribution $\hvQ_{n,\hat{m}_n}$ with $\hat{m}_n\in\argmax_{m\in\cM_n}\hgamma_{n,m}$ is a solution to \labelcref{eq:compare:msvb_minimization} since
    \begin{align*}
       -\log\hgamma_{n,m}&\propto -\log \alpha_{n,m}+ \scE_n(\hvQ_{n,m},\Pi_{n,m},\sp_n)\\
       &=\min_{\vQ_{n,m}\in\cQ_{n,m}} \cbr{\scE_n(\vQ_{n,m},\Pi_{n,m},\sp_n)-\log\alpha_{n,m}}.
    \end{align*}
We call $\hvQ_{n,\hat{m}_n}$ the \textit{model selection variational posterior}. Note that every individual variational posterior needs to be obtained to optimize the selection criterion for the model selection variational Bayes. Therefore, the computational costs of the two variational Bayes approaches are exactly the same. But we note that the adaptive variational posterior is supported on ``the entire parameter space'' $\Theta_{n,\cM_n}$, while the model selection variational posterior is supported on ``the best individual parameter space''.

Since the adaptive variational posterior $\hvQ_n$ is the optimal solution of the variational optimization problem in \labelcref{eq:vbayes}, it is clear that $\hvQ_n$  this is a better approximation to the original posterior $\Pi_n(\cdot|\Y^{(n)})$ than the model selection variational posterior $\hvQ_{n,\hat{m}_n}$. Thus, if we aim to recover the original posterior distribution as precisely as possible, the adaptive variational posterior is always preferable to the model selection one.

\begin{proposition}[Comparison of variational approximation gaps]
\label{prop:inferior_msvb}
Let $\Y^{(n)}$ be a sample generated from the distribution $\P_\star^{(n)}$. Then
    \begin{equation}
       \kl\del[1]{\hvQ_n, \Pi_n(\cdot|\Y^{(n)})}
       \le \kl\del[1]{\hvQ_{n,\hat{m}_n}, \Pi_n(\cdot|\Y^{(n)})}
    \end{equation}
with $\P_\star^{(n)}$-probability 1, where $\hat{m}_n\in\argmax_{m\in\cM_n}\hgamma_{n,m}$.
\end{proposition}

\begin{proof}
The result follows from the fact that $\cup_{m\in\cM_n}\cQ_{n,m}\subset\cQ_n$, which in particular implies that $\hvQ_{n,\hat{m}_n}\in\cQ_n.$
\end{proof}

In \cref{appen:numeric_msvb} in the Supplementary Material, we provide some numerical examples to show the superiority of the adaptive variational Bayes over the model selection variational Bayes.

\section{Concentration properties of adaptive variational posterior}
\label{sec:theory}

In this section, we provide adaptive posterior contraction rates and model selection consistency of the adaptive variational posterior distribution  \labelcref{eq:vposterior} in general situations. Our strategy is to derive a contraction rate $\epsilon_n$ of the original posterior $\Pi_n(\cdot|\Y^{(n)})$ to the true parameter $\blambda^\star$ first and then bound the \textit{variational approximation gap} by an upper bound of an appropriate order as
    \begin{align}
    \label{eq:theory:vgap}
        \P_{\blambda^\star}^{(n)} \sbr{\kl\del[1]{\hvQ_n, \Pi_n(\cdot|\Y^{(n)})}}\lesssim n\epsilon_n^2.
    \end{align}
Then we have the same contraction rate $\epsilon_n$ of the adaptive variational posterior with the help of the classical change-of-measure lemma, which we provide in \cref{lemma:variational_ineq} in the Supplementary Material for the reader's convenience.

\subsection{Prelude: Contraction of individual variational posteriors} 

In this subsection, we first study the concentration property of each individual variational posterior $\hvQ_{n,m}$ supported on the parameter space $\Theta_{n,m}$ for $m\in\cM_n$. The contraction rates of variational posteriors were thoroughly investigated in \cite{zhang2020convergence}, but we provide the result under a more concise and weaker set of assumptions, which we will explain later. Moreover, our description is in a slightly more delicate fashion in the sense that a contraction rate of $\hvQ_{n,m}$ is decomposed into two terms, the approximation and estimation errors of the model $m\in\cM_n$. It turns out that identifying this decomposition is useful to select prior model probabilities $\balpha_n$ suitably, see \cref{lemma:model_prob_choice}.

Let $(\Lambda_n^\star)_{n\in\bN}$ be a sequence of natural parameter spaces such that $\Lambda_n^\star\subset \Lambda_n$ for each $n\in\bN$. Throughout this paper, we assume that the true distribution of the sample $\Y^{(n)}$ belongs to a \textit{true model} $\cP^\star_n:=\cbr[0]{\P_{\blambda^\star}^{(n)}:\blambda^\star\in\Lambda_n^\star}$. Note that the true model does not need to be contained in the model $\cup_{m\in\cM_n}\cP_{n,m}$ used for estimation. Let $\scd_n:\Lambda_n\times \Lambda_n\mapsto \R_{\ge0}$ be a metric that will be used to measure a degree of contraction.

\begin{assumption}[Individual model, prior and variational family] 
\label{assume:individual}
There exist absolute constants  $J_0>0$, $\fc_1>0$ and $\fc_2>0$ such that the following hold for any  $\blambda^\star\in\Lambda_n^\star$,  any $m\in\cM_n$ and any sufficiently large $n\in\bN$.
\begin{enumerate}[label=\assumnum{assumption}\textsf{\arabic*}]
    \item \label{assume:individual:testing} 
    (Testing) There exist a test function $\varphi_{n,m}:\bY_n\mapsto [0,1]$ such that
        \begin{align}
            \max\cbr{\P_{\blambda^\star}^{(n)}[\varphi_{n,m}], \sup_{\btheta\in\Theta_{n,m}:\scd_n(\sT(\btheta),\blambda^\star)\ge J_0\zeta}
            \P_{\sT(\btheta)}^{(n)} [1-\varphi_{n,m}]}
            \le \exp\del[1]{-\fc_1n\zeta^2}
         \end{align}
    for any $\zeta>\zeta_{n,m}\ge n^{-1/2}$.

    \item \label{assume:individual:variational} 
    (Prior and variational family) There exists a distribution $\vQ_{n,m}^*\in\cQ_{n,m}$ such that
        \begin{equation}
           \kl\del[1]{\vQ_{n,m}^*, \Pi_{n,m}}  
            + \vQ_{n,m}^*\sbr{\kl\del[1]{\P_{\blambda^\star}^{(n)}, \P_{\sT(\btheta)}^{(n)}}} 
            \le \fc_2n (\eta_{n,m}+\zeta_{n,m})^2.
    \end{equation}
\end{enumerate}
\end{assumption}

We can interpret $\zeta_{n,m}>0$ and $\eta_{n,m}>0$ are \textit{estimation and approximation errors} of the model  $m\in\cM_n$, respectively. Suppose that the log of the (local) covering number of the parameter space $\Theta_{n,m}$ is proportional to $n\zeta_{n,m}^2$ then a standard argument for testing construction leads to a test function $\varphi_{n,m}$ of which the type-\rom{1} error is bounded by $\e^{\fc_1'n\zeta_{n,m}^2}\e^{-\fc_2'n(J_0\zeta)^2}$ for some $\fc_1',\fc_2'>0$, see for example, Theorem 7.1 of \citet{ghosal2000convergence}. Thus to make this error decrease exponentially, the ``distance'' $J_0\zeta$ between the alternative hypothesis set and the true parameter $\blambda^\star$ should be larger than $\zeta_{n,m}$. Meanwhile, the prior and variational family condition of \cref{assume:individual:variational} involves the true distribution, so a certain term representing an approximation error should be included in the upper bound, which we denote by $\eta_{n,m}$. 

\cref{assume:individual:testing} is a very standard assumption for deriving contraction rates of original posterior distributions \citep{ghosal2000convergence,ghosal2007convergence}. In \cref{appen:theory:testing} in the Supplementary Material, we develop a new sufficient condition of the testing condition, which we will frequently use in our applications. \cref{assume:individual:variational} is made to control the variational approximation gap $\P_{\blambda^\star}^{(n)}\sbr[0]{\kl\del[0]{\hvQ_{n,m},\Pi_{n,m}(\cdot|\Y^{(n)})}}$ to the ``individual'' original posterior $\Pi_{n,m}(\cdot|\Y^{(n)})\propto \sp_n(\sT(\btheta),\Y^{(n)})\d\Pi_{n,m}(\btheta)$. The same condition of \cref{assume:individual:variational} is assumed by \citet{zhang2020convergence}, which is named as (C4) therein.  \citet{alquier2020concentration} showed that only \cref{assume:individual:variational} is enough to derive optimal contraction of variational fractional posteriors, thanks to the nice theoretical property of fractional likelihoods, which was illustrated in a number of papers \citep{walker2001bayesian,zhang2006epsilon,bhattacharya2019bayesian}.

Our technical proofs largely follow several existing works, particularly, \citet{ghosal2000convergence} on posterior contraction, \citet{ghosal2008nonparametric} and \citet{han2021oracle} on Bayesian adaptation and \citet{zhang2020convergence} on the convergence of variational posteriors. But we make some technical simplifications of the theoretical conditions in \citet{zhang2020convergence}. Concretely, we show that the ``prior mass condition'' can be ``hidden'', and that a condition associated with a stronger $\rho$-R\'enyi divergence with $\rho>1$ than the KL divergence can be relaxed. We refer to \cref{appen:theory:simplifications} in the Supplementary Material for a detailed explanation. Although these simplifications may lead to succinct proofs for applications of our theory, we do not insist that they are substantial improvements because we do not find any example that satisfies our conditions while violating those of \cite{zhang2020convergence}.

The next theorem derives contraction rates of the individual variational posteriors. The proof is almost similar to that of our main theorem \cref{thm:conv}, which appears in the next subsection, so we omit it. 

\begin{theorem}
\label{prop:conv_indiv}
Under \cref{assume:individual}, we have
    \begin{equation}
        \sup_{\blambda^\star\in\Lambda_n^\star}\sup_{m\in\cM_n}\P_{\blambda^\star}^{(n)}\sbr{\hvQ_{n,m}\del{\scd_n(\sT(\btheta),\blambda^\star)\ge A_n(\eta_{n,m}+\zeta_{n,m})}}=\sco(1)
    \end{equation} 
for any diverging sequence $(A_n)_{n\in\bN}\to\infty$. 
\end{theorem}

\subsection{Adaptive contraction rates}
\label{sec:theory:rate}

In the previous subsection, we showed that $\hvQ_{n,m}$ contracts at the rate $\eta_{n,m}+\zeta_{n,m}$. Clearly, if the ``best'' model $m^*_n\in\cM_n$ such that $\eta_{n,m^*_n}+\zeta_{n,m^*_n}\asymp\min_{m\in\cM_n}(\eta_{n,m}+\zeta_{n,m})$ is known, one can utilize the best individual variational posterior $\hvQ_{n,m^*_n}$ for inference, which contracts at an \textit{oracle rate} defined as
    \begin{equation}
    \label{eq:oracle_rate}
        \epsilon_n:=\epsilon_n(\cM_n):=\inf_{m\in\cM_n}(\eta_{n,m}+\zeta_{n,m}).
    \end{equation}
But this is, in general, not the case in practice. We will show that our adaptive variational posterior can attain the oracle rate without any information about the best model when we use appropriate prior model probabilities $\balpha_n$ that impose sufficient mass on the best model $m_n^*$ while very small mass on unnecessarily complex models.

\begin{assumption}[Aggregation] 
\label{assume:aggregation} 
There exist absolute constants $H_0>1$ and $\fc_3,\dots,\fc_6>0$  such that the following hold for any sufficiently large $n\in\bN$.

\begin{enumerate}[label=\assumnum{assumption}\textsf{\arabic*}]
    \item \label{assume:aggregation:model_set} 
    (Model space) 
    The cardinality of the model space $\cM_n$ is bounded as
    \begin{align}
    \label{eq:model_cardinality}
        |\cM_n|\le \exp\del[1]{\fc_3n\epsilon_n^2}.
    \end{align}

    \item \label{assume:aggregation:model_prior_pen}
    (Prior model probabilities: regularization)
    The prior model probabilities $\balpha_n:=(\alpha_{n,m})_{m\in\cM_n}\in\Delta_{|\cM_n|}$ satisfies
        \begin{align}
        \label{eq:assume:aggregation:largemodels}
            \sum_{m\in\cM_n:\zeta_{n,m}\ge H\epsilon_n}\alpha_{n,m}
            &\le \exp\del[1]{-\fc_4 n(H\epsilon_n)^2}
        \end{align}
    for any $H>H_0$.
    \item \label{assume:aggregation:model_prior_mass}
    (Prior model probabilities: concentration)
    There exists a model $ m_n^*\in\cM_n$ such that $\eta_{n,m_n^*}+\zeta_{n,m_n^*}\le (1+\fc_5)\epsilon_n$ and 
        \begin{align}
         \label{eq:assume:aggregation:best}
            \alpha_{n,m_n^*} \ge\exp\del[1]{-\fc_6n\epsilon_n^2}.
        \end{align} 
\end{enumerate}
\end{assumption}

\cref{assume:aggregation:model_set} is in general very mild from a practical perspective. The upper bound in \labelcref{eq:model_cardinality} allows any large constant bounds. Moreover, when $n\epsilon_n^2\gtrsim \log n$, which is satisfied in most of applications, \cref{assume:aggregation:model_set} is met by any polynomial bounds of $n$. This condition is also theoretically satisfactory, as we will discuss in \cref{remark:model_construction} at the end of this subsection. 

A similar requirement on prior model probabilities to \cref{assume:aggregation:model_prior_pen,assume:aggregation:model_prior_mass} has been commonly assumed for adaptivity of original posteriors \citep{lember2007universal,ghosal2008nonparametric,arbel2013bayesian,gao2020general,han2021oracle}. In every application considered in the mentioned studies, prior model probabilities can be chosen without knowing any aspects of the true distribution. We can adhere to their recommendations. Moreover, from the decomposition of approximation and estimation errors in \cref{assume:individual}, in \cref{lemma:model_prob_choice} below, we suggest a general method to construct prior model probabilities satisfying both  \labelcref{eq:assume:aggregation:largemodels} and \labelcref{eq:assume:aggregation:best}. This method only requires knowing the estimation errors that are independent of $\Lambda_n^\star$ in most cases. Thus, this is fully adaptive to the true parameter. In every example we shall examine, the choice of the prior model probabilities does not depend on the true parameter.

\begin{lemma}[Adaptive choice of prior model probabilities]
\label{lemma:model_prob_choice}
Suppose that prior model probabilities $\balpha_n:=(\alpha_{n,m})_{m\in\cM_n}\in\Delta_{|\cM_n|}$  are given by 
        \begin{align}
            \alpha_{n,m} =\frac{1}{Z_{\alpha,n}}\e^{-\fa_0n\zeta_{n,m}^2}, \quad m\in \cM_n
        \end{align}
for $\fa_0>0$, where $Z_{\alpha,n}:=\sum_{m\in\cM_n}\e^{-\fa_0n\zeta_{n,m}^2}$ is the normalizing constant. Then under \cref{assume:aggregation:model_set},  \cref{assume:aggregation:model_prior_pen,assume:aggregation:model_prior_mass} holds.
\end{lemma}

Under \cref{assume:individual,assume:aggregation} together, the original posterior $\Pi_n(\cdot|\Y^{(n)})$ contracts at the oracle rate $\epsilon_n$, see \cref{thm:post_conv} in the Supplementary Material. When the variational approximation gap is bounded as in \labelcref{eq:theory:vgap}, the same contraction rate $\epsilon_n$ is attained by the adaptive variational posterior. Because a variational approximation gap is usually proportional to the complexity of a parameter space on which a variational posterior lives, it is seemingly impossible to attain \labelcref{eq:theory:vgap} due to the very large complexity of the entire parameter space $\Theta_{n,\cM_n}$. However, the following theorem allows us to circumvent this issue.

\begin{theorem}[Variational approximation gap]
\label{thm:vgap}
For any $\blambda^\star\in\Lambda_n^\star$, 
    \begin{equations}
    \label{eq:vgap_ineq}
       \P_{\blambda^\star}^{(n)}&\sbr{\kl\del[1]{\hvQ_{n},\Pi_{n}(\cdot|\Y^{(n)})}}\\
        &\le \inf_{\vQ\in\cQ_n} \cbr{\kl(\vQ,\Pi_{n}) 
            + \vQ\sbr{\kl\del[1]{\P_{\blambda^\star}^{(n)}, \P_{\sT(\btheta)}^{(n)}}} }\\
        &\le \inf_{m\in\cM_n}\inf_{\vQ_{m}\in\cQ_{n,m}}\cbr{-\log (\alpha_{n,m}) + \kl\del[1]{\vQ_{m}, \Pi_{n,m}}  
            +  \vQ_m\sbr{\kl\del[1]{\P_{\blambda^\star}^{(n)}, \P_{\sT(\btheta)}^{(n)}}} }.
    \end{equations}
Further suppose that \cref{assume:individual:variational} and \cref{assume:aggregation:model_prior_mass} hold. Then \labelcref{eq:theory:vgap} holds with $\epsilon_n=\epsilon_n(\cM_n)$.
\end{theorem}

We now arrive at our main result on adaptive contraction.

\begin{theorem}[Adaptive contraction rate]
\label{thm:conv}
Under \cref{assume:individual,assume:aggregation}, we have
        \begin{equation}
        \label{eq:conv}
        \sup_{\blambda^\star\in\Lambda_n^\star}\P_{\blambda^\star}^{(n)}\sbr{\hvQ_n\del{\scd_n(\sT(\btheta),\blambda^\star)\ge A_n\epsilon_{n}}}=\sco(1) 
    \end{equation} 
for any diverging sequence $(A_n)_{n\in\bN}\to\infty$. 
\end{theorem}

\begin{remark}
The contraction rate $\epsilon_n$ of the adaptive variational posterior, as we saw in its definition in \labelcref{eq:oracle_rate}, is monotonically decreasing in the model space $\cM_n$ in the sense that $\epsilon_n(\cM_n')\ge \epsilon_n(\cM_n)$ when $\cM_n'\subset \cM_n$. 
If we choose a smaller number of individual models than necessary, a contraction rate of the adaptive variational posterior would be sub-optimal. However, as we will explain in the following remark, it is possible to find a model space that results in an optimal rate.
\end{remark}

\begin{remark}
\label{remark:model_construction}
In order for our contraction rate $\epsilon_n(\cM_n)$ to be identical to the optimal rate of the problem of interest, the model space should contain ``sufficiently many'' models, but ``not too many'' for computational efficiency. For a nested model space in which the individual models can be sorted consistently by their complexities $\{\zeta_{n,m}\}_{m\in\cM_n}$, we can construct such a model space.  To be concrete, suppose that the optimal rate is given as $n^{-\kappa}$ for some $\kappa\in(0,1/2]$. Consider a model space $\cM_n$ such that
    \begin{align}
    \label{eq:modelspace:nested}
        \cbr{\zeta_{n,m}:m\in\cM_n}=\cbr{n^{-k/\log n}:k\in\bN, k\le \frac{\log n}{2}},
    \end{align}
which has the $\scO(\log n)$ cardinality. Then for an integer $k_n^*\le \log n/2$ such that $n^{-k_n^*/\log n}\le n^{-\kappa}\le n^{-(k_n^*-1)/\log n}=n^{1/\log n}n^{k_n^*/\log n}$, we have $n^{-k_n^*/\log n}\asymp n^{-\kappa}$ since $n^{1/\log n}=\e$. Thus if the model $m_n^*\in\cM_n$ with $\zeta_{n,m_n^*}=n^{-k_n^*/\log n}$ has the approximation error not larger than the estimation error, i.e., $\eta_{n,m_n^*}\lesssim\zeta_{n,m_n^*}$, then we get the optimal contraction rate $\epsilon_n(\cM_n)\asymp n^{-\kappa}$. For example, for deep neural network models, we will study in \cref{sec:dnn}, the model space is constructed similar to \labelcref{eq:modelspace:nested}, see \cref{assume:dnn}.
\end{remark}

\begin{remark}
As seen in \cref{remark:model_construction}, for adaptation, we need to include very large models with nearly $\scO(n^{1/2})$ complexities in our model space, when we cannot rule out the possibility of a very complex true model, for which the optimal rate is typically quite slow. In that case, the proposed adaptation method can be somewhat time-consuming, although we employ a scalable variational Bayes algorithm. This is a limitation of the proposed adaptive variational Bayes. One may circumvent this issue by using an alternative variational Bayes algorithm that efficiently finds a suitably complex model, for example, the spike-and-slab deep learning method by \cite{cherief2020convergence,bai2020efficient}. It is an interesting direction for further work to develop a computational strategy in the adaptive variational Bayes framework to avoid computation for unnecessarily complex models.
\end{remark}

\subsection{Model selection consistency}
\label{sec:theory:selection}

In this subsection, we study the concentration of the adaptive variational posterior over the model space. We write $\hvQ_n(\cM'):=\hvQ_n(\cup_{m\in\cM'}\Theta_{n,m})$ for any subset $\cM'\subset\cM_n$ for brevity.

We first show that the adaptive variational posterior puts a small mass on too complex models to estimate the true distribution. We define 
    \begin{align}
        \cM_{n}^{\textup{over}}(H):=\cbr{m\in\cM_n:\zeta_{n,m}\ge H\epsilon_n}
    \end{align}
for $H>1$, which is the set of models whose estimation errors are somewhat larger than the oracle rate, that is, the models that are overly complex. The next theorem shows that the expected variational posterior probability of $\cM_{n}^{\textup{over}}(H)$ tends to zero, provided that $H$ is sufficiently large.

\begin{theorem}[Selection, no severe overestimation]
\label{thm:over}
Assume that $n\epsilon_n^2\to \infty$. Then under Assumptions \labelcref{assume:individual:variational} and \labelcref{assume:aggregation}, there exists an absolute constant $H_1>0$ such that
    \begin{equation}
        \label{eq:over}
        \sup_{\blambda^\star\in\Lambda_n^\star}\P_{\blambda^\star}^{(n)}
        \sbr{\hvQ_n\del{\cM_{n}^{\textup{over}}(H_1)}}
       =\sco(1).
    \end{equation}
\end{theorem}

We now consider an underestimation problem. Define  
    \begin{align}
         \cM_{n}^{\textup{under}}(\eta^*; \blambda^\star)
         :=\cbr[2]{m\in\cM_n:\inf_{\btheta\in\Theta_{n,m}}\scd_n(\sT(\btheta),\blambda^\star)\ge\eta^*}
    \end{align}
for $\eta^*>0$ and $\blambda^\star\in\Lambda_n^\star$, which is the set of less expressive models that cannot approximate the true parameter $\blambda^\star$ accurately. The expected variational posterior probability of the models in $\cM_{n}^{\textup{under}}(\eta^*, \blambda^\star)$ vanishes asymptotically if the \textit{expressibility gap} $\eta^*$ is large enough to detect. In order to do that, one may introduce a ``beta-min'' condition, which assumes all the elements of the true parameter are sufficiently large to detect. See, for example, \citet{castillo2015bayesian} for sparse linear regression and \citet{ohn2023optimal} for mixture models. Then less complex models with smaller dimensions than the true parameter cannot mimic the true distribution. For example, in the sparse factor model example in \cref{sec:sparse:factor}, the beta-min condition on the nonzero rows of the true loading matrix provided in \labelcref{eq:sparse:factor:true:detect} creates an expressibility gap between sparser models and the true loading matrix since the loading matrices in the sparser models cannot recover the true loading matrix exactly. For technical details, see the proof of \cref{thm:sparse:factor:nfac} provided in \cref{appen:sparse:factor:proof:nfac} in the Supplementary Material.

\begin{theorem}[Selection, no underestimation]
\label{thm:under}
Under \cref{assume:individual,assume:aggregation}, we have
        \begin{equation}
        \sup_{\blambda^\star\in\Lambda_n^\star}\P_{\blambda^\star}^{(n)}
        \sbr[3]{\hvQ_n\del{\cM_{n}^{\textup{under}}(\uA_n\epsilon_n; \blambda^\star)}}
       =\sco(1)
    \end{equation}
for any diverging sequence $(\uA_n)_{n\in\bN}$ such that $\uA_n\to\infty$.
\end{theorem}

\section{Application: Adaptive variational deep learning}
\label{sec:dnn}

In this section, we propose a novel Bayesian deep learning procedure, called \textit{adaptive variational deep learning}, based on the proposed adaptive variational Bayes. Applying our general result in \cref{sec:theory}, we show that the adaptive variational deep learning adaptively attains minimax optimality for a number of statistical problems including regression function estimation, conditional class probability function estimation for binary classification and intensity estimation for Poisson point processes. We give the first example in this section but defer the last two examples to \cref{appen:dnn:additional:logistic} and \cref{appen:dnn:additional:ppp}, respectively, in the Supplementary Material. 

Deep learning is a statistical inference method that uses (deep) neural networks to model a ``target'' function, such as a regression function, which we desire to estimate. In this section, we assume that a target function is a real-valued function supported on the $d$-dimensional unit cube $[0,1]^d$. But our analysis can be easily generalized to any compact input domain and multivariate output. 

We introduce some notation used throughout this section. Let $\cF^d$ be the set of all measurable real-valued functions supported on $[0,1]^d$. For $f\in\cF^d$ and $q\in\bN$, let $\|f\|_q:=(\int_{[0,1]^d}|f(\x)|^q\d\x)^{1/q}$ denote the usual $\scL^q$ norm and $\|f\|_\infty:=\sup_{\x\in[0,1]^d}|f(\x)|$ the $\scL^\infty$ norm.

\subsection{Adaptive variational posterior over neural networks}
\label{sec:dnn:dnn}


We introduce neural network models for function estimation. For a positive integer $K\in\bN_{\ge2}$ larger than 1 and a $(K+1)$-dimensional vector of positive integers $\M_{1:(K+1)}:=(M_1,\dots, M_{K+1})\in\bN^{K+1}$, we denote $\hTheta_{\M_{1:(K+1)}}:=\bigotimes_{k=1}^{K}(\R^{M_{k+1}\times M_k}\times \R^{M_{k+1}})$. For a \textit{network parameter} $\btheta=((\W_k, \b_k))_{k\in[K]}\in\hTheta_{\M_{1:(K+1)}}$, we define the \textit{neural network (function)} $\net(\btheta):\R^{M_1}\mapsto \R^{M_{K+1}}$ induced by the network parameter $\btheta$ as
    \begin{align*}
         \net&(\btheta):\x\mapsto[\W_{K},\b_{K}]\circ\text{ReLU}\circ[\W_{K-1},\b_{K-1}]\circ\cdots\circ\text{ReLU}\circ[\W_{1},\b_{1}]\x,
    \end{align*}
where $[\W_{k},\b_{k}]:\x'\mapsto \W_k\x'+\b_k$ denotes the affine transformation represented as a multiplication by the \textit{weight matrix} $\W_k$ and an addition of the \textit{bias vector} $\b_k$ and $\text{ReLU}:\x'\mapsto \x'\vee\zero, $ does the elementwise ReLU (rectified linear unit) activation function.  Since we focus on the estimation of real-valued functions supported on $[0,1]^d$, we only consider network parameters in $\Theta^d:=\cup_{K=2}^\infty\cup_{\M_{1:(K+1)}\in\bN^{K+1}:M_1=d, M_{K+1}=1}\hTheta_{\M_{1:(K+1)}}$, the set of network parameters with input and output dimensions being $d$ and $1$, respectively.

In the application of adaptive variational Bayes to deep learning, we consider multiple disjoint parameter spaces indexed by the number of hidden layers $K\in\bN_{\ge2}$ and the number of hidden nodes of each layer $M\in\bN$ such as
    \begin{equation}
    \label{eq:nn_parameter_space}
        \Theta_{(K,M)}:=\Theta^d_{(K,M)}:=\hTheta_{(d,M_2,\dots,M_{K},1)} \mbox{ with } M_2=\cdots=M_{K}=M.
    \end{equation}
We refer to $K$ and $M$ as (network) \textit{depth} and \textit{width}, respectively, and further, a pair $(K,M)$ as a \textit{network architecture}. We denote by $J_{(K,M)}$ the dimension of the space $\Theta_{(K,M)}$, i.e.,
    \begin{align*}
        J_{(K,M)}:=(d+1)M+(K-2)(M^2+M)+(M+1).
    \end{align*}
Since $\Theta_{(K,M)}\cong \R^{J_{(K,M)}}$, we can deal with each member in the parameter space $\Theta_{(K,M)}$ as a $J_{(K,M)}$-dimensional real vector. For a technical reason (see \cref{remark:dnn:magnitude} below), we restrict network parameters to be bounded. For a given \textit{magnitude bound} $B>0$, define
    \begin{align}
    \label{eq:dnn:bounded_space}
       \Theta^{\le B}_{(K,M)}:=\Theta^{d,\le B}_{(K,M)}:=\cbr{\btheta\in\Theta_{(K,M)}^d:|\btheta|_\infty\le B}.
    \end{align}
In our collection of neural network models, we let each parameter space be of the form $\Theta^{\le B_n}_{(K,M)}$ for each $n\in\bN$, where the sequence of magnitude bounds $(B_n)_{n\in\bN}\subset\R_{\ge0}$ is a positive sequence that diverges at a suitable rate.

We describe our prior and variational family for neural networks. Let $\cM_n\subset \bN_{\ge2}\times \bN$ be a set of some network architectures to be considered in consequent inference tasks. We impose the prior distribution 
    \begin{equations}
    \label{eq:dnn_prior}
        \Pi_n&=\sum_{(K,M)\in\cM_n}\alpha_{n,(K,M)}\Pi_{n,(K,M)},
        \mbox{ where }\\
        &\Pi_{n,(K,M)}:=\Unif(-B_n, B_n)^{\otimes J_{(K,M)}}
        \mbox{ and }\alpha_{n,(K,M)}:= \frac{\e^{-\fa_0(KM)^2\log n}}{\sum_{(K',M')\in\cM_n}\e^{-\fa_0(K'M')^2\log n}}
    \end{equations}
for $\fa_0>0$. Note that the choice of the prior model probabilities $\balpha_n$ is fully adaptive because it only depends on the network depth and width, not on anything about the true regression function. For each network architecture $(K,M)\in\cM_n$, we consider the variational family given by
    \begin{equation}
    \label{eq:dnn_variational_family}
        \cQ_{n,(K,M)}:=\cbr{\bigotimes_{j=1}^{J_{(K,M)}}\Unif(-\psi_{1,j}, \psi_{2,j}):-B_n\le \psi_{1,j}<\psi_{2,j}\le B_n}.
    \end{equation}
Since the parameter spaces $\cbr[0]{\Theta^{\le B_n}_{(K,M)}}_{(K,M)\in\cM_n}$ are disjoint, by \cref{thm:compute}, the adaptive variational posterior is given by $\hvQ_n=\sum_{(K,M)\in\cM_n}\hgamma_{n,(K,M)}\hvQ_{n,(K,M)}$ with $\hvQ_{n,(K,M)}\in\argmin_{\vQ\in\cQ_{n,(K,M)}} \scE_n\del[1]{\vQ, \Pi_{n,(K,M)},\sp_n} $ and $\hgamma_{n,(K,M)}\propto\alpha_{n,(K,M)} \exp(-\scE_n\del[0]{\hvQ_{n,(K,M)}, \Pi_{n,(K,M)},\sp_n})$.

\begin{remark}
\label{remark:dnn:magnitude}
In order to facilitate theoretical analysis, we consider the uniform prior and variational families, which are rarely used in practice. It is necessary for the prior distribution to give sufficient mass near a network parameter used to approximate a target function, and our approximation analysis requires this network parameter to contain some elements diverging at a certain rate as the sample size grows. Therefore, to satisfy this requirement, the prior distribution should be sufficiently diffused, and our uniform prior with a diverging magnitude bound $B_n$ is one such distribution. For the same reason, a recent study \citep{kong2023masked} used a heavy-tailed prior with polynomially decreasing tails, such as Cauchy and Student-t distributions, and derived the optimal contraction rate of the original posterior over non-sparse neural networks. We expect that heavy-tailed priors can be shown to be optimal for variational deep learning also, but leave the detailed derivation for future work. To extend our theory to more practical prior distributions, including a product of standard Gaussian distributions, which is the most commonly used despite its inaccurate uncertainty quantification \citep{foong2020expressiveness}, we need a refined approximation analysis with a ``small'' magnitude bound, such as a constant bound, on network parameters. This is an interesting direction for future work. We refer to \citet{fortuin2022priors} for a comprehensive survey of practically used prior distributions for Bayesian deep learning.
\end{remark}

\subsection{Nonparametric regression with Gaussian errors}
\label{sec:dnn:reg}

In this subsection, we consider a nonparametric regression experiment with standard Gaussian errors. In this case, we observe $n$ independent outputs $\Y^{(n)}:=(Y_1,\dots, Y_n)\in \bY_n:=\R^n$, where each $Y_i$ is associated with a fixed $d$-dimensional input $\x_i$ that has been rescaled so that $\x_i\in[0,1]^d$.  We model the sample $\Y^{(n)}$ using a probability model $\cP(\R^n;\sp_n^{\text{Ga}},\cF^d)$  with a likelihood function $\sp_n^{\text{Ga}}:\cF^{d}\times \R^n\mapsto\R_{\ge0}$ defined as
    \begin{equation}
        \sp_n^{\text{Ga}}(f, \Y^{(n)})=\sp_n^{\text{Ga}}(f, \Y^{(n)};(\x_i)_{i\in[n]})
        =\prod_{i=1}^ng_{\N(0,1)}(Y_i-f(\x_i))
        \mbox{ for $f\in\cF^{d},$}
    \end{equation}
where $g_{\N(0,1)}$ stands for the density function of the standard Gaussian distribution. Here, we assume that the variance of the output distribution is fixed to be 1 for technical simplicity, however, we can easily extend this to an unknown variance case. We denote by $\P_f^{(n)}$ the distribution that has the density function $\sp_n(f, \cdot)$.  Assuming that $\Y^{(n)}\sim\P_{f^\star}^{(n)}$ for some \textit{true regression function} $f^\star\in\cF^d$, our aim is to accurately estimate $f^\star$ in terms of the empirical $\scL^2$ distance  defined as 
    \begin{equation*}
        \nnorm{f_0-f_1}=\del[3]{\frac{1}{n}\sum_{i=1}^n(f_0(\x_i)-f_1(\x_i))^2}^{1/2} \mbox{ for $f_0,f_1\in\cF^d$}.
    \end{equation*}

We provide an oracle contraction result, with which we can obtain adaptive optimal rates for a number of regression problems without difficulty.

\begin{theorem}[Oracle contraction rate, regression function]
\label{thm:dnn:oracle}
Let $\cF^\star\subset\cF^d$. For each $n\in\bN$, let $\cM_n\subset\bN_{\ge2}\times \bN$ be a set of some network architectures such that $\log|\cM_n|\lesssim \log n$ and $\max_{(K,M)\in\cM_n}(K\vee M)\lesssim n$. Assume that $1\le B_n\lesssim n^{\iota_0}$ for some absolute constant $\iota_0>0.$ Then 
        \begin{equation}
        \sup_{f^\star\in\cF^\star}\P_{f^\star}^{(n)}\sbr{\hvQ_n\del{\nnorm[1]{\net(\btheta)- f^\star}\ge A_n\epsilon_n(\cF^\star)}}=\sco(1)
    \end{equation} 
for any diverging sequence $(A_n)_{n\in\bN}\to\infty$, where 
    \begin{equation}
    \label{eq:dnn_oracle}
        \epsilon_n(\cF^\star):=\epsilon_n(\cF^\star;\cM_n):=\inf_{(K,M)\in\cM_n}\cbr{\sup_{f^*\in\cF^\star}\inf_{\btheta\in\Theta^{\le B_n}_{(K,M)}}\norm{\net(\btheta)-f^*}_\infty+KM\sqrt{\frac{\log n}{n}}}.
    \end{equation}
\end{theorem}

In the next subsection, we first focus on the case that the true regression function is assumed to be H\"older $\beta$-smooth.  This assumption is the most standard and classical assumption for nonparametric regression. We then consider a different class of regression functions, which exhibits a certain ``compositional structure'' to avoid the curse of dimensionality  \citep{schmidt2020nonparametric,kohler2021rate}, and show that the adaptive variational deep learning attains an optimal contraction rate adaptively. The detailed result is deferred to \cref{appen:dnn:additional:comp} in the Supplementary Material.

\subsubsection{Estimation of H\"older smooth regression functions}

Let $\cC^{p,d}$ denote the class of $p$-times differentiable functions supported on $[0,1]^d$ for $p\in\bN$ and $\cC^{0,d}$ the class of continuous functions on $[0,1]^d$. The \textit{H\"older space} of smoothness $\beta>0$ with domain $[0,1]^d$ and radius $F_0>0$ is defined as 
	\begin{equation*}
	\cH^{\beta, d, F_0}:=\cbr{f\in\cC^{\ceil{\beta-1},d}:\|f\|_{\cH^{\beta,d}}\le F_0},
	\end{equation*}
where $\|\cdot\|_{\cH^{\beta,d}}$ denotes the \textit{H\"older norm} defined by
    \begin{align*}
    \|f\|_{\cH^{\beta,d}}
        :=\max\cbr{\max_{\a\in\bN_0^d:|\a|_1<\beta}\|\partial^{\a}f\|_{\infty}
        ,\max_{\a\in\bN_0^d:|\a|_1= \ceil{\beta-1}}\sup_{\x_1,\x_2\in [0,1]^d, \x_1\neq\x_2 }\frac{|\partial^{\a}f(\x_1)-\partial^{\a}f(\x_2)|}{|\x_1-\x_2|_\infty^{\beta- \ceil{\beta-1}}}}.
    \end{align*}
Here, $\partial^{\a}f$ denotes the partial derivative of $f$ of order $\a:=(a_j)_{j\in[d]}$, that is, $\partial^{\a}f:=\frac{\partial^{|\a|_1}f}{\partial x_1^{a_1}\cdots \partial x_d^{a_d}}$.

We will show that the adaptive variational deep learning attains the minimax optimal contraction rate $n^{-\beta/(2\beta+d)}$  \citep{stone1982optimal}, up to a logarithmic factor, when we suitably select a set $\cM_n$ of network architectures and a magnitude bound $B_n$. We need to consider sufficiently many and diverse network architectures in order to make the oracle rate $\epsilon_n(\cF^\star;\cM_n)$ defined in \labelcref{eq:dnn_oracle} be the same as the optimal one. Also, the magnitude bound $B_n$ should be sufficiently large to control the approximation errors. For details, see our approximation results in \cref{appen:dnn:approx} in the Supplementary Material. In our several applications of adaptive variational deep learning, we assume the following conditions on $\cM_n$ and $B_n$. 

\begin{assumption}[Neural network model]
\label{assume:dnn}
For each $n\in\bN$, the set of network architectures is given by
    \begin{align}
    \label{eq:architecture_set}
        \cM_n=\cbr{\del{\floor{k_1\log n}\vee 2,\floor{n^{k_2/\log n}}}:k_1\in\sbr{\ceil{\log\log n}\vee1},k_2\in\sbr{0: \ceil{(\log n)/2}}}
    \end{align}
and the magnitude bound is given by $B_n\asymp n^{\iota_0}$ for an arbitrary constant $\iota_0\ge1/2$.
\end{assumption}

Under the above assumption, we establish adaptive optimal contraction.

\begin{corollary}[H\"older smooth regression function]
\label{col:dnn:holder}
Let $\beta>0$, $d\in\bN$ and $F_0>0$. Then under \cref{assume:dnn}, we have
    \begin{equation}
        \sup_{f^\star\in\cH^{\beta, d, F_0}}\P_{f^\star}^{(n)}\sbr{\hvQ_n\del{\nnorm[1]{\net(\btheta)- f^\star}\ge n^{-\frac{\beta}{2\beta+d}}\log^2n}}=\sco(1).
    \end{equation} 
\end{corollary}

We discuss the implication of the above results and the related work on estimating H\"older smooth regression functions with deep neural networks in \cref{appen:dnn:additional:review} in the Supplementary Material.

\section{Combinatorial model spaces}
\label{sec:sparse}

As we mentioned in \cref{sec:method}, the computational burden of the adaptive variational Bayes is proportional to the number of individual models. So far, we have focused on nested model spaces, where the individual models can be sorted by their complexity. The cardinalities of nested model spaces are usually not much large, for example, at most  $\scO(\log n)$ cardinality for the neural network model we considered (see \cref{assume:dnn}). Whereas, as we explained in \cref{subsec:modelspace}, combinatorial model spaces such as sparse models, have exponentially increasing cardinalities. Thus, a naive application of the adaptive variational Bayes to combinatorial model spaces is computationally intractable.

As a remedy, we propose employing ``sparsity-inducing'' priors and variational families in our framework to avoid dividing the entire parameter space by a combinatorial structure. By doing so, for a model space that has a combinatorial structure only, we do not need the adaptive variational Bayes algorithm that aggregates multiple individual posteriors. For example, this strategy leads to the same variational Bayes method for sparse linear regression as that proposed by \citet{ray2021variational}, while our general theory that will be developed in this section cannot be used to prove its optimality, see \cref{remark:sparse_regression} below. However, if a model space has an additional nested structure, multiple parameter spaces indexed by the nested structure continue to exist. For such cases, the adaptive variational Bayes can be implemented and shown to be adaptive to both the sparsity and complexity such as the smoothness of the model as shown in the next subsection. There are many interesting examples where the model spaces are both combinatorial and nested, for instance, sparse mixture models \citep{yao2022bayesian}, sparse row-rank matrix estimation \citep{babacan2012sparse}, sparse linear regression with an unknown smooth error distribution \citep{chae-semi, lee2021bayesian}, sparse factor models with the unknown factor dimensionality \citep{pati2014posterior,ohn2021posterior} and sparse nonparametric regression \citep{yang2015minimax,jiang2021variable}. We will investigate the last two examples to illustrate our approach and leave the other examples as future work.

\subsection{Adaptive contraction rates for combinatorial model spaces}
\label{sec:sparse:main}

In this subsection, we establish general results on adaptive contraction and model selection consistency of the adaptive variational posterior for model spaces with both nested and combinatorial structures. Let $\cM_n$ and $\cS_n$ be nested and combinatorial model spaces, respectively, and let $\Theta_{n,m,S}$ be an individual parameter space indexed by  $m\in\cM_n$  and $S\in\cS_n$. If computationally tractable, we can conduct variational optimization over a ``merged'' parameter space $\Theta_{n,m}:=\cup_{S\in\cS_n}\Theta_{n,m,S}$ for each $m\in\cM_n$, instead of every $\Theta_{n,m,S}$. By doing so, we have the adaptive variational posterior of the same form as in \labelcref{eq:vposterior}.   Recently, variational optimization algorithms have been developed for a number of sparse models, for example, sparse linear regression  \citep{huang2016variational,ray2021variational}, sparse factor models with ``fixed'' factor dimensionality \citep{ning2021spike} and sparse deep neural network models \citep{bai2020efficient}. We can adopt these existing algorithms for variational optimization over $\Theta_{n,m}$.

A potential problem of this strategy from a theoretical perspective is that the complexity of the merged parameter space $\Theta_{n,m}$ may be excessively large in high-dimensional settings, so we may fail to obtain a good contraction rate. We  overcome this by letting a prior $\Pi_{n,m}\in\cP(\Theta_{n,m})$  and a variational family $\cQ_{n,m}\subset\cP(\Theta_{n,m})$ control the complexity of $\Theta_{n,m}$ appropriately. We now provide a slightly modified assumption to formally describe the regularization effects of  $\{\Pi_{n,m}\}_{m\in\cM_n}$ and $\{\cQ_{n,m}\}_{m\in\cM_n}$. For ease of description, we write
\begin{align*}
    \Pi_{n,m}&=\sum_{S\in\cS_n}\alpha_{n,S|m}\Pi_{n,m,S} \\
    &\mbox{ with } \alpha_{n,S|m}:=\Pi_{n,m}(\Theta_{n,m,S}) 
    \mbox{ and } \Pi_{n,m,S}(\cdot):=\Pi_{n,m}(\cdot|\Theta_{n,m,S}). 
\end{align*}
for every $m\in\cM_n$. For example, if we impose a spike-and-slab prior  with slab probability $\omega\in[0,1]$ independently on each of $d$ parameters, then $\alpha_{n,S|m}=(\omega)^{|S|}(1-w)^{d-|S|}$. We expect that the conditional prior model probabilities $(\alpha_{n,S|m})_{m\in\cM_n, S\in \cS_n}$ penalize overly dense models with large complexities.

\begin{assumption}[Combinatorial model, sparsity-inducing prior and variational family] 
\label{assume:sparse}
There exist absolute constants  $J_0>0$, $H_0>1$ and $\fc_1,\dots,\fc_6>0$  such that the following hold for  any  $\blambda^\star\in\Lambda_n^\star$,  any $m\in\cM_n$, $S\in\cS_n$ and any sufficiently large $n\in\bN$.
\begin{enumerate}[label=\assumnum{assumption}\textsf{\arabic*}]
    \item \label{assume:sparse:testing} 
    (Testing) There exist a test function $\varphi_{n,m, S}:\bY_n\mapsto [0,1]$  such that
        \begin{align}
            \max\cbr{\P_{\blambda^\star}^{(n)}[\varphi_{n,m,S}], \sup_{\btheta\in\Theta_{n,m, S}:\scd_n(\sT(\btheta),\blambda^\star)\ge J_0\zeta}
            \P_{\sT(\btheta)}^{(n)}[1-\varphi_{n,m,S}]}
            \le \exp\del[1]{-\fc_1n\zeta^2}
         \end{align}
    for any $\zeta>\zeta_{n,m, S}\ge n^{-1/2}$.

    \item \label{assume:sparse:variational} 
    (Prior and variational family) There exists a distribution $\vQ_{n,m}^*\in\cQ_{n,m}$ such that
        \begin{equation}
      \kl\del[1]{\vQ_{n,m}^*, \Pi_{n,m}}  
       +  \vQ_{n,m}^*\sbr{\kl\del[1]{\P_{\blambda^\star}^{(n)}, \P_{\sT(\btheta)}^{(n)}}} 
                \le \fc_2n \inf_{S\in\cS_n}(\eta_{n,m, S}+\zeta_{n,m,S})^2.
    \end{equation}
  \end{enumerate}
Now define
        \begin{align}
             \label{eq:oracle2}
            \epsilon_n:=\epsilon_n(\cM_n\times \cS_n)
            :=\inf_{(m,S)\in\cM_n\times \cS_n}(\eta_{n,m,S}+\zeta_{n,m,S}).
        \end{align}

\begin{enumerate}[label=\assumnum{assumption}\textsf{\arabic*}, resume]
    \item \label{assume:sparse:aggregation} 
    (Model space and prior model probabilities) 
        The model space $\cM_n\times \cS_n$, prior model probabilities $(\alpha_{n,m})_{m\in\cM_n}$ and conditional prior model probabilities $(\alpha_{n,S|m})_{m\in\cM_n, S\in\cS_n}$ satisfy
        \begin{align}
            \abs{\cbr{(m,S)\in\cM_n\times \cS_n:\zeta_{n,m,S}\le H\epsilon_n}}&\le  \exp\del[1]{ \fc_3n(H\epsilon_n)^2} 
            \label{eq:assume:sparse:cardinality},\\
            \sum_{(m,S)\in\cM_n\times \cS_n:\zeta_{n,m,S}\ge H\epsilon_n}\alpha_{n,m}\alpha_{n,S|m}
            &\le \exp\del[1]{-\fc_4n(H\epsilon_n)^2} \label{eq:assume:sparse:largemodels}
        \end{align}
    for any $H>H_0$. 
    Moreover, there exists a model $(m_n^*, S_n^*)\in\cM_n\times \cS_n$ such that $\eta_{n,m_n^*, S_n^*}+\zeta_{n,m_n^*, S_n^*}\le (1+\fc_5)\epsilon_n$ and
        \begin{align}
         \label{eq:assume:sparse:best}
            \alpha_{n,m_n^*}\alpha_{n,S_n^*|m_n^*} &\ge\exp\del[1]{-\fc_6n\epsilon_n^2}.
        \end{align}
    \end{enumerate}
\end{assumption}

In the above, we assume the condition \labelcref{eq:assume:sparse:cardinality} which is a weaker version of \labelcref{eq:model_cardinality}, because  \labelcref{eq:model_cardinality} may fail for combinatorial model spaces, for example, we have that $|\cM_n\times \cS_n|\gtrsim \e^{d_n}\gg \e^{n\epsilon_n^2}$ when $d_n\gg n.$ But even in that case, the number of not much complex models can be bounded as in \labelcref{eq:assume:sparse:cardinality}.

The next theorem provides adaptive contraction and model selection consistency of the adaptive variational posterior applied to a combinatorial model space.

\begin{theorem}[Adaptive contraction rate and model selection consistency, combinatorial model spaces]
\label{thm:sparse:conv}
Under \cref{assume:sparse}, we have
    \begin{align}
    \label{eq:sparse:conv}        \sup_{\blambda^\star\in\Lambda_n^\star}\P_{\blambda^\star}^{(n)}\sbr{\hvQ_n\del{\scd_n(\sT(\btheta),\blambda^\star)\ge A_n\epsilon_{n}}}
        =\sco(1) 
    \end{align}
for any diverging sequence $(A_n)_{n\in\bN}\to\infty$, where $\epsilon_n$ is defined in \labelcref{eq:oracle2}, and hence
    \begin{align}
    \label{eq:sparse:under}
        \sup_{\blambda^\star\in\Lambda_n^\star}\P_{\blambda^\star}^{(n)}
        \sbr[3]{\hvQ_n\del[2]{\cbr[2]{(m,S)\in\cM_n\times \cS_n:\inf_{\btheta\in\Theta_{n,m,S}}\scd_n(\sT(\btheta),\blambda^\star)\ge A_n\epsilon_n}}}
       =\sco(1).
    \end{align}
On the other hand, under \cref{assume:sparse:variational,assume:sparse:aggregation}, there exists an absolute constant  $H_1>0$ such that
    \begin{align}
    \label{eq:sparse:over}
        \sup_{\blambda^\star\in\Lambda_n^\star}\P_{\blambda^\star}^{(n)}
        \sbr{\hvQ_n\del{\cbr{(m,S)\in\cM_n\times \cS_n:\zeta_{n,m,S}\ge H_1\epsilon_n }}}
         =\sco(1).
    \end{align}
\end{theorem}

\begin{remark}
\label{remark:sparse_regression}
\cref{thm:sparse:conv} cannot be directly applied to the sparse linear regression model to conclude the same results of \citet{ray2020spike}. In that paper, the authors used specific theoretical techniques developed in \citet{castillo2015bayesian} for the original posterior distribution with a spike and slab prior, which our theoretical conditions do not embrace.
\end{remark}

We give a sparse factor model example in the next subsection and a high-dimensional regression example with deep neural networks in \cref{appen:sparse:dnn} in the Supplementary Material. The adaptive variational Bayes can be successfully applied to these examples even though there exists a combinatorial structure in the model space.

\subsection{Application to sparse factor models}
\label{sec:sparse:factor}

In this subsection, we illustrate the adaptive variational Bayes in the high-dimensional sparse factor model described in \cref{example:sparsefactor}.

\subsubsection{Prior and variational family}
To appropriately address the sparse structure of the loading matrix, we impose a spike-and-slab prior distribution on the loading matrix  $\L\in\R^{d_n\times m}$ conditional on $m\in\cM_n$.  Concretely, we assume
\begin{equations}
    &\Pi_{n}=\sum_{m\in\cM_n}\alpha_{n,m}\Pi_{n,m}, \quad \mbox{ where } \alpha_{n,m}:=1/|\cM_n|\\
    &\mbox{ and }\Pi_{n,m}:= \cbr{(1- \omega_{n,m})\delta(\cdot;\zero_m) + \omega_{n,m} \N(\zero_m, \tau_0\I_m )}^{\otimes d_n}
    \mbox{ with } \omega_{n,m}:=d_n^{-(1+\fa_0)m}\\
    \end{equations}
for given constants $\tau_0>0$ and $\fa_0>0$. Note that our choice of the slab probability  $\omega_{n,m}$ is sufficient to prevent overestimation of the factor dimensionality, so we can choose the uniform prior for $\cM_n$ or any other mildly distributed priors. Moreover, this choice is independent to the true distribution.

For each model index $m\in\cM_n$, we consider a variational family of spike-and-slab distributions such as
    \begin{equation*}
        \cQ_{n,m}:=\cbr{\bigotimes_{j=1}^{d_n}\cbr{(1- \nu_{j})\delta(\cdot;\zero_m) + \nu_{j}\N(\bpsi_{j}, \bPhi_{j})}:\bpsi_{j}\in\R^{m},\bPhi_{j}\in\bS_{++}^m, \nu_{j}\in[0,1] }.
    \end{equation*}
The computation of each individual variational posterior $\hvQ_{n,m}\in\argmin_{\vQ \in\cQ_{n,m}}\scE_n(\vQ, \Pi_{n,m},\sp_n)$ can be efficiently done by the optimization algorithm developed by \citet{ning2021spike}, which, for the sake of completeness, are provided in  \cref{appen:sparse:factor:algorithm} in the Supplementary Material. Consequently, the adaptive variational Bayes method is computationally tractable. We provide a simulation study to examine the performance of the proposed methodology in \cref{appen:sparse:factor:simulation} in the Supplementary Material.

\subsubsection{Covariance matrix estimation and model selection}

We first recall some notation for matrices. For a $d\times m$-dimensional matrix $\L$, we denote the operator norm of the matrix $\L$ by $\opnorm{\L}$, that is, $\opnorm{\L}:=\sup_{\x\in\R^m:|\x|_2=1}|\L\x|_2$. Let $\sigma_1(\L)\ge \cdots \ge \sigma_{d\wedge m}(\L)$ be the ordered singular values of $\L$. For a square matrix $\bSigma$, we denote by $|\bSigma|$ its determinant.

Let $r_n\in\bN$ be the true factor dimensionality and $s_n\in[d_n]$ be the row sparsity of the true loading matrix. We assume that $s_n\ge r_n$ throughout this section because, if not,  the rank of the true loading matrix is less than $r_n$ and the estimation of the factor dimensionality is not meaningful anymore. We assume that the true covariance matrix belongs to a set defined as
    \begin{align*}
        \Lambda_n^\star:=\Big\{\sT(\L)\in\bS_{++}^{d_n}:\L\in&\R^{d_n\times r_n}, |\supp(\L)|\le s_n,\sigma_{1}(\L\L^\top)\le \basigma \Big\}
    \end{align*}
with some fixed $\basigma>0$. The following theorem derives a contraction rate of the adaptive variational posterior for sparse and spiked covariance matrix estimation, which is the same as the rates in the existing literature \citep{xie2018bayesian, ning2021spike}. 

\begin{theorem}[Covariance matrix]
\label{thm:sparse:factor:cov}
Assume that $s_nr_n\log d_n=\sco(n)$, $s_n\ge r_n$ and $\log d_n\gtrsim \log n$. Then if $\cM_n=[m_{\max,n}]$ with $r_n\le m_{\max,n}\le n$,
    \begin{equation}
        \sup_{\bSigma^\star\in\Lambda_n^\star}
        \P_{\bSigma^\star}^{(n)}\sbr[3]{\hvQ_n\del[3]{\opnorm{\sT(\L)-\bSigma^\star}\ge A_n\sqrt{\frac{s_nr_n\log d_n}{n}}}}=\sco(1)
    \end{equation} 
for any diverging sequence $(A_n)_{n\in\bN}\to\infty$.
\end{theorem}

Under the following additional conditions on the true covariance matrix,
    \begin{align}
    \label{eq:sparse:factor:true:detect}
        \Lambda_n^\star(\eta^*):=\cbr{\sT(\L)\in\Lambda_n^\star:\min\cbr[2]{\sigma_{r_n}(\L\L^\top),\min_{j\in\supp(\L)}\abs{\L_{j,:}}_2^2}\ge \eta^*},
    \end{align}
for $\eta^*>0$, the adaptive variational posterior can nearly consistently estimate the true factor dimensionality $r_n$ and sparsity $s_n$. The lower bound of the $r_n$-th eigenvalue of the row rank matrix $\L\L^\top$, which can be viewed as the ``eigengap'' between the spike and noise eigenvalues, is introduced to avoid underestimation of the factor dimensionality. Similarly,  the nonzero rows of the loading matrix are assumed to be large enough to not underestimate the sparsity. 

\begin{theorem}[Factor dimensionality and sparsity]
\label{thm:sparse:factor:nfac}
Suppose that the same assumptions as in \cref{thm:sparse:factor:cov} hold. Furthermore assume that $\eta_n^*\ge \uA_n\sqrt{s_nr_n\log d_n/n}$ for some diverging sequence  $(\uA_n)_{n\in\bN}\to\infty$. Then there exist absolute constants $H_1>1$ and $H_2>1$ such that
    \begin{align}
        \inf_{\bSigma^\star\in \Lambda_n^\star(\eta_n^*)}\P_{\bSigma^\star}^{(n)}
            \sbr{\hvQ_n\del{\cbr{m\in\cM_n:r_n\le m \le H_1 r_n}}}&\to 1,\label{eq:sparse:factor:nfac}\\
        \inf_{\bSigma^\star\in \Lambda_n^\star(\eta_n^*)}\P_{\bSigma^\star}^{(n)}
            \sbr{\hvQ_n\del{s_n\le |\supp(\L)| \le H_2 s_n}}&\to 1. \label{eq:sparse:factor:sparsity}
    \end{align}
\end{theorem}

Unfortunately, the lower bound $\sqrt{s_nr_n\log d_n/n}$ of the expressibility gap in the above theorem is in general $\sqrt{r_n}$ times larger than the optimal bound $\sqrt{s_n\log d_n/n}$ \citep{cai2015optimal}. When $r_n$ is bounded, they are the same.

\section{Regularization via variational approximation}
\label{sec:regularization}

The theoretical approach in \cref{sec:theory} utilizes the adaptive contraction property of the original posterior, which heavily relies on the regularization effect of the prior model probabilities to overly complex models. In this section, we provide a general situation in which adaptive inference can be made due to regularization from variational families, not prior model probabilities. To put it concretely, we show that the adaptive variational  Bayes procedure regularizes a model  $m\in\cM_n$ for which the quantity $\Psi_{n,m}$ defined below is large,
    \begin{align}
        \Psi_{n,m}:= \sup_{\tPi\in \cP(\Theta_{n,m})} \inf_{\vQ\in \cQ_{n,m}}\cbr{\kl(\vQ, \Pi_{n,m})-\kl(\vQ, \tPi)}.
     \end{align}   
This,  we call the \textit{implicit variational Bayes penalty (ivB penalty)} for a model $m\in\cM_n$, can be viewed as a penalty that is ``implicitly'' imposed by the variational Bayes procedure to that model, in the sense that this is not explicitly specified by an user unlike the prior penalty $-\log \del[0]{\alpha_{n,m}}$ used in the previous sections. Clearly, the ivB penalty depends on the prior and variational family but not on the likelihood, and we found that this is usually proportional to the complexity of a model, e.g., the number of parameters as illustrated in \cref{example:regularization:dnn} below.

The choice of a variational family is crucial for ivB regularization.  By definition, if the variational family $\cQ_{n,m}$ becomes ``broader'', then the ivB penalty $\Psi_{n,m}$ will be smaller and vice versa. This means that using ``narrow'' variational families is required to obtain a sufficiently powerful ivB penalty. In the opposite extreme case of using the broadest variational family $\cQ_{n,m}=\cP(\Theta_{n,m})$, we have $\Psi_{n,m}=0$, since the supremum is attained at $\tPi=\Pi_{n,m}$. Therefore, ivB regularization does not occur for original posteriors. 

The aim of this section is to show that the adaptive variational posterior can attain a near-optimal contraction rate adaptively through ivB regularization without tuning the prior model probabilities. Toward this aim, we first show that the adaptive variational posterior gives a negligible mass to the set defined as
    \begin{align}
        \cM_{n}^{\textup{ivB,over}}(A):=\cbr{m\in\cM_n:\Psi_{n,m}\ge An\epsilon_n^2},
    \end{align}
which is a set of models with substantially large ivB penalties.

\begin{theorem}[ivB regularization]
\label{thm:regularization:over}
Under Assumptions \labelcref{assume:individual:variational}, \labelcref{assume:aggregation:model_set} and \labelcref{assume:aggregation:model_prior_mass}, we have
    \begin{equation}
        \label{eq:regularization:over}
        \sup_{\blambda^\star\in\Lambda_n^\star}\P_{\blambda^\star}^{(n)}
        \sbr{\hvQ_n\del{    \cM_{n}^{\textup{ivB,over}}(A_n)}}
       =\sco(1).
    \end{equation}
for any diverging sequence $(A_n)_{n\in\bN}\to\infty$. 
\end{theorem}

\begin{remark}
The proof of the above result does not use the change-of-measure lemma (\cref{lemma:variational_ineq} in the Supplementary Material), unlike the proofs of the concentration results given in \cref{sec:theory,sec:sparse}. Instead, we directly upper bound the variational posterior model probabilities, of which the closed forms are provided in \labelcref{eq:vpost_mprob}.
\end{remark}

\cref{thm:regularization:over} enables us to focus on a ``sieve'' $\cM_{n}^{\textup{ivB,regular}}(A_n):=\cM_n\setminus\cM_{n}^{\textup{ivB,over}}(A_n)$ of appropriate complexity when we analyze the contraction behavior. The contraction rate is then determined by the maximal complexity $\zeta_n^\ddag$ of the models in the sieve $\cM_{n}^{\textup{ivB,regular}}(A_n)$, which is given as
        \begin{align}
        \label{eq:complexity_not_ivb}
          \zeta_n^\ddag:=  \max_{m\in \cM_{n}^{\textup{ivB,regular}}(A_n)}\zeta_{n,m}=\max_{m\in  \cM_n:\Psi_{n,m}<A_nn\epsilon_n^2 }\zeta_{n,m},
        \end{align}
as shown in the next theorem.

\begin{theorem}[Adaptive contraction rate through ivB regularization]
\label{thm:regularization:rate}
Under Assumptions \labelcref{assume:individual}, \labelcref{assume:aggregation:model_set}  and \labelcref{assume:aggregation:model_prior_mass}, we have
    \begin{align}
    \label{eq:regularization:rate}    \sup_{\blambda^\star\in\Lambda_n^\star}\P_{\blambda^\star}^{(n)}\sbr{\hvQ_n\del{\scd_n(\sT(\btheta),\blambda^\star)\ge (A_n\epsilon_{n})\vee \zeta_n^\ddag }}=\sco(1) 
    \end{align}
for any diverging sequence $(A_n)_{n\in\bN}\to\infty$, where $\zeta_n^\ddag$ is defined in \labelcref{eq:complexity_not_ivb}.
\end{theorem}

We say that the collection of the ivB penalties $\{\Psi_{n,m}:m\in \cM_n\}$ is \textit{ideal} if $\Psi_{n,m}\asymp n\zeta_{n,m}^2$ for any $m\in\cM_n$, because in that case, ivB regularization exactly recovers the oracle rate since $ \zeta_n^\ddag\lesssim \sqrt{A_n}\epsilon_n$, that is, the ivB penalties induce sufficient regularization. But in our examples given below and in the Supplementary Material (\cref{appen:example:gaussian_sequence}), we are only able to get non-ideal ivB penalties, which yield an extra logarithmic term in the contraction rate. That is, ivB regularization is not as effective as prior regularization in the current theoretical status.

\begin{example}[Neural networks]
\label{example:regularization:dnn}
Consider the adaptive variational deep learning procedure considered in \cref{sec:dnn}. To get a non-ignorable ivB penalty $\Psi_{n,(K,M)}$, we expand our parameter space as $\Theta^{\le B_n^2}_{(K,M)}$ by replacing the previous magnitude bound $B_n$ with $B_n^2$, which is endowed with the uniform prior $\Pi_{n,(K,M)}^*=\Unif(-B_n^2, B_n^2)^{\otimes J_{(K,M)}}$, while the variational family remains the same as   \labelcref{eq:dnn_variational_family}. Despite this expansion, it is easy to see that \cref{assume:individual} is still satisfied. Now, we let $\Pi_{n,(K,M)}^*=\Unif(-B_n, B_n)^{\otimes J_{(K,M)}}$. Then the ivB penalty for a model $(K,M)$ is lower bounded by
    \begin{align*}
        \Psi_{n,(K,M)}
        &\ge 
        \inf_{Q\in \cQ_{n,(K,M)}}\cbr{\kl(\vQ, \Pi_{n,(K,M)})- \kl(\vQ, \Pi_{n,(K,M)}^*)}\\       
        &=J_{(K,M)} \log\del{\frac{B_n^2}{B_n}}     
        \gtrsim KM^2\log n.
    \end{align*}
Note that the ivB penalty is not ideal, since $n\zeta_{n,(K,M)}^2=K^2M^2\log n>KM^2\log n$. So under \cref{assume:dnn} on the model space $\cM_n$, for any diverging sequence $(A_n)_{n\in\bN}$, the contraction rate $(A_n\epsilon_{n})\vee \zeta_n^\ddag$ is dominated by
    \begin{align*}
       \zeta_n^\ddag   
         &=\max_{(K,M)\in  \cM_n: KM^2\log n <A_nn\epsilon_n^2}KM\sqrt{\frac{\log n}{n}}\\
         &\le \max_{(K,M)\in  \cM_n}\sqrt{K}\sqrt{A_n}\epsilon_n
         \le \sqrt{A_n}\epsilon_n \log n \sqrt{\log \log n},
    \end{align*}
which is $\log n \sqrt{\log \log n}$ times slower than the oracle rate $\epsilon_n:=\epsilon_n(\cF^\star)$ in \labelcref{eq:dnn_oracle}.
\end{example}

\begin{remark}
\label{remark:reverse_prior}
In fact, even though penalizing prior model probabilities are not used, the original posterior can be adaptively optimal when the ``reverse'' prior mass condition is satisfied \citep{ghosal2008nonparametric, yang2017bayesian}. Adopting the notation of this paper, this condition is written as follows:   the individual prior $\Pi_{n,m}$  satisfies $\Pi_{n,m}(\scd_n(\sT(\btheta), \blambda^\star)\le \fc_1' \zeta_{n,m})\le \exp\del[0]{-\fc_2'n\zeta_{n,m}^2}$ for any overly complex model $m\in \cM_n$ with $\zeta_{n,m}\ge \fc_3'\epsilon_n$  for some sufficiently large positive constants $\fc_1'$, $\fc_2'$ and $\fc_3'$. In \citet{rousseau2017asymptotic}, a similar condition was assumed to derive an adaptive contraction rate of their empirical Bayes posterior with the maximum marginal likelihood estimator.
But the reverse prior mass condition does not hold or at least is hard to verify for ``non-identifiable'' models where characterizing the inverse of the natural parameterization map is almost impossible. For example, in nonparametric regression with neural networks, it is hard to identify ``every'' network parameter yielding a neural network close to a true regression function $f^\star$. Therefore, a tractable upper bound of the prior mass $\Pi_{n,(K,M)}(\nnorm{\net(\btheta)-f^\star}\le \fc_1' \zeta_{n,(K,M)})$ is not easy to be established, while ivB regularization is feasible as shown in \cref{example:regularization:dnn}. That is, we still have a case that the original posterior may not behave well, but its variational approximation by the adaptive variational Bayes contracts at a near-optimal rate. 
\end{remark}

\section{Adaptive variational quasi-posteriors}
\label{sec:quasi}

In this section, we analyze a variational approximation of a \textit{quasi-posterior} that is updated from a prior distribution through an alternative quasi-likelihood rather than of a standard likelihood. Quasi-posteriors have been used in a number of applications for various purposes, and a selective survey of related work is provided in \cref{appen:quasi:related} in the Supplementary Material.

\subsection{Adaptive variational Bayes with quasi-likelihood}

For a given \textit{quasi-likelihood function} $\sp_n^\natural:\Lambda_n\times \bY_n\mapsto\R_{\ge0}$, the quasi-posterior given the sample $\Y^{(n)}\in\bY_n$ is defined as
    \begin{equation}
        \d\Pi_n^\natural(\btheta|\Y^{(n)}):=\frac{ \sp_n^\natural(\sT(\btheta), \Y^{(n)})\d\Pi_n(\btheta)}{\int \sp_n^\natural(\sT(\btheta), \Y^{(n)})\d\Pi_n(\btheta) }.
    \end{equation}
Given a variation family $\cQ_n\subset\cP(\Theta_{n,\cM_n})$, we define an \textit{adaptive variational quasi-posterior} $\hvQ_n^\natural$ by
    \begin{equations}
    \label{eq:kl_min_quasi}
    \hvQ_n^\natural&\in\argmin_{\vQ\in\cQ_n}\kl(\vQ,\Pi_n^\natural(\cdot|\Y^{(n)}))\\
        &=\argmin_{\vQ\in\cQ_n}\cbr{ \scE_n\del[1]{\vQ, \Pi_{n,m}, \sp_n^{\natural}}:= -\int \log \sp_n^{\natural}(\sT(\btheta), \Y^{(n)}) \d\vQ(\btheta) + \kl(\vQ, \Pi_n)}.
    \end{equations}
By a similar argument to \cref{thm:compute}, the adaptive variational quasi-posterior can be equivalently written as 
    \begin{equations}
    \label{eq:quasi_variational_posterior}
        \hvQ_n^{\natural}&=\sum_{m\in\cM_n}\hgamma_{n,m}^{\natural}\hvQ_{n,m}^{\natural},
        \end{equations}
      where 
      \begin{equations}
     \hvQ_{n,m}^{\natural}\in\argmin_{\vQ\in\cQ_{n,m}}  \scE_n\del[1]{\vQ, \Pi_{n,m}, \sp_n^{\natural}}
        &\mbox { and } \hgamma_{n,m}^{\natural}:=\frac{ \alpha_{n,m} \e^{-\scE_n\del[1]{\hvQ_{n,m}^{\natural}, \Pi_{n,m},\sp_n^{\natural}}}}
        {\sum_{m'\in\cM_n}\alpha_{n,m'} \e^{-\scE_n\del[1]{\hvQ_{n,m'}^{\natural},\Pi_{n,m'},\sp_n^{\natural}}}},
    \end{equations}
which allows us to decompose the minimization in \labelcref{eq:kl_min_quasi} into computationally tractable minimization problems over individual models.

\subsection{Adaptive contraction rates of variational quasi-posteriors}
\label{sec:quasi:contraction}

In this subsection, we derive adaptive optimal contraction rates of the adaptive variational quasi-posterior. We take a similar approach as in \cref{sec:theory}, that is, we first prove that the original quasi-posterior contracts at a rate $\epsilon_n$ and bound the variational approximation gap as
    \begin{align}
    \label{eq:quasi:vgap}
        \P_{\blambda^\star}^{(n)}\sbr[2]{\kl(\hvQ^\natural_n, \Pi_n^\natural(\cdot|\Y^{(n)}))}
            \lesssim n\epsilon_n^2
    \end{align}
in order to derive the same contraction rate $\epsilon_n$ of the adaptive variational quasi-posterior $\hvQ^\natural_n$.

\begin{assumption}[Variational quasi-posterior]
\label{assume:quasi}
There exist absolute constants $\fc_1,\dots,\fc_5>0$ and $\rho>0$ such that the followings hold for any  $\blambda^\star\in\Lambda_n^\star$,  any $m\in\cM_n$ and any sufficiently large $n\in\bN$.
     \begin{enumerate}[label=\assumnum{assumption}\textsf{\arabic*}]
          \item  \label{assume:quasi:quasi}
        (Quasi-likelihood)   For any $\blambda\in\Lambda_n$, the quasi-likelihood $\sp_n^\natural:\Lambda_n\times \bY_n\mapsto\R_{\ge0}$ satisfies
            \begin{align}
            \label{eq:quasi:learning}
                 \P_{\blambda^\star}^{(n)}\sbr{\frac{\sp_n^\natural(\blambda,\Y^{(n)})}{\sp_n^\natural(\blambda^\star,\Y^{(n)})}}
                 &\le \e^{-\fc_1n\scd_n^2(\blambda,\blambda^\star)},\\
             \label{eq:quasi:bound}
                 \P_{\blambda^\star}^{(n)}\sbr[4]{\del{\frac{\sp_n^\natural(\blambda^\star,\Y^{(n)})}{\sp_n^\natural(\blambda,\Y^{(n)})}}^{\rho}}
                  &\le \e^{\fc_2n\scd_n^2(\blambda,\blambda^\star)}.
            \end{align}

    \item \label{assume:quasi:variational} 
    (Prior and variational family)  There exists a distribution $\vQ_{n,m}^*\in\cQ_{n,m}$ such that
        \begin{equation}
          \kl\del[1]{\vQ_{n,m}^*, \Pi_{n,m}}  
            + \vQ_{n,m}^*\sbr{n\scd_n^2(\sT(\btheta), \blambda^\star)}
        \le \fc_3n (\eta_{n,m}+\zeta_{n,m})^2.
        \end{equation}
    
    \item   \label{assume:quasi:model_prior}
    (Prior model probabilities) 
    There exists a model $m_n^*\in\cM_n$ such that $\eta_{n,m_n^*}+\zeta_{n,m_n^*}\le (1+\fc_4)\epsilon_n$  and $  \alpha_{n,m_n^*} \ge\exp\del[1]{-\fc_5n\epsilon_n^2},$  where $\epsilon_n$ is the oracle rate defined in \labelcref{eq:oracle_rate}.
    \end{enumerate}
\end{assumption}

For the adaptive variational quasi-posterior, it is unnecessary to penalize overly complex individual models through prior model probabilities, unlike  \cref{assume:aggregation:model_prior_pen} for the adaptive variational posterior. Thus we have more flexibility in choosing prior model probabilities. For example, we can consider the uniform distribution on the model space unless it is too large. Moreover, the cardinality of the model space is not restricted. Technically, both are consequences of  \labelcref{eq:quasi:learning} in \cref{assume:quasi:quasi}, by which we do not need to construct a suitable test function that requires controlling the complexity of the entire parameter space $\Theta_{n,\cM_n}$.

\cref{assume:quasi:quasi} is  similar to the ``sub-exponential loss'' assumption of \cite{syring2020gibbs} and the ``Bernstein'' assumption of \citet{alquier2016properties}, which investigated theoretical properties of quasi-posteriors. Basically, all these conditions require that the ratio of the log quasi-likelihoods of a parameter $\blambda$ and a  true one $\blambda^\star$ is a sub-exponential random variable and its variance is bounded by the expectation in a certain way. A number of interesting statistical problems can be dealt with under such assumptions, see  \cite{syring2020gibbs,alquier2016properties} and the examples therein. In addition, fractional likelihoods are encompassed in our quasi-likelihood framework, see \cref{lemma:quasi:fractional} in the Supplementary Material. However, any likelihood function $\sp_n$ does not satisfy \cref{assume:quasi:quasi} since $\P_{\blambda^\star}^{(n)}[\sp_n(\blambda,\Y^{(n)})/\sp_n(\blambda^\star,\Y^{(n)})]=1$ for any $\blambda^\star,\blambda\in\Lambda_n$, thus we needed to separately deal with the variational ``non-quasi'' posteriors.

Since our adaptive variational quasi-posterior is on the very complex parameter space $\Theta_{n,\cM_n}$, to verify \labelcref{eq:quasi:vgap} is not trivial. But the next theorem, the variational quasi-posterior counterpart of \cref{thm:vgap}, is constructive.

\begin{theorem}[Variational approximation gap to quasi-posterior]
\label{thm:quasi:vgap}
Suppose that \cref{assume:quasi:quasi} holds. 
Then for any $\blambda^\star\in\Lambda_n$, 
    \begin{equations}
    \label{eq:quasi:vgap_ineq}
        \P_{\blambda^\star}^{(n)}&\sbr{ \kl(\hvQ_n^\natural, \Pi_n^\natural(\cdot|\Y^{(n)}))}      \\
        &\le\inf_{m\in\cM_n}\inf_{\vQ_m\in\cQ_{n,m}}\cbr{-\log(\alpha_{n,m})+\kl(\vQ_m,\Pi_{n,m}) +\frac{\fc_2}{\rho}\vQ_m\sbr{n\scd_n^2(\sT(\btheta),\blambda^\star)}}.
    \end{equations}
Further suppose that \cref{assume:quasi:variational,assume:quasi:model_prior} holds additionally. Then \labelcref{eq:quasi:vgap} holds with $\epsilon_n=\epsilon_n(\cM_n)$.
\end{theorem}

The adaptive variational quasi-posterior enjoys the oracle contraction rate.

\begin{theorem}[Adaptive contraction rate of variational quasi-posterior]
\label{thm:quasi:conv}
Under \cref{assume:quasi}, we have
        \begin{equation}
        \label{eq:conv_quasi}
        \sup_{\blambda^\star\in\Lambda_n^\star}\P_{\blambda^\star}^{(n)}\sbr{\hvQ_n^\natural\del{\scd_n(\sT(\btheta),\blambda^\star)\ge A_n\epsilon_{n}}}=\sco(1)
    \end{equation} 
for any diverging sequence $(A_n)_{n\in\bN}\to\infty$.
\end{theorem}

\begin{remark}
\citet{alquier2016properties} showed that under a similar condition on the quasi-likelihood to  \cref{assume:quasi:quasi}, the variational quasi-posterior can attain optimal theoretical properties and illustrated their results in several interesting applications. But unlike ours, a general recipe for adaptation of (variational) quasi-posteriors is missing in their paper. 
\end{remark}

We illustrate our general result on adaptive contraction rates of variational quasi-posteriors in two specific examples, the stochastic block model in \cref{appen:quasi:sbm} and nonparametric regression with sub-Gaussian error in \cref{appen:quasi:subgauss} in the Supplementary Material.

\begin{acks}[Acknowledgments]
We are very grateful to the Editor, the Associate Editor and three reviewers for their valuable comments which have led to substantial improvement in our paper. We would like to thank Minwoo Chae, Kyoungjae Lee, Cheng Li and Ryan Martin for their helpful comments and suggestions.
\end{acks}

\begin{funding}
We acknowledge the generous support of NSF grants DMS CAREER 1654579 and DMS 2113642. The first author was supported by the National Research Foundation of Korea(NRF) grant funded by the Korea government(MSIT) (NRF-2022R1F1A1069695) and INHA UNIVERSITY Research Grant.
\end{funding}

\begin{supplement}
\stitle{Supplement to "Adaptive variational Bayes: Optimality, computation and applications''}
\sdescription{
In the Supplementary Material, we provide proofs for all the results presented in the main text, as well as some additional materials.
}
\end{supplement}

\bibliographystyle{plainnat}
\bibliography{_references}

\begin{thebibliography}{67}
\providecommand{\natexlab}[1]{#1}
\providecommand{\url}[1]{\texttt{#1}}
\expandafter\ifx\csname urlstyle\endcsname\relax
  \providecommand{\doi}[1]{doi: #1}\else
  \providecommand{\doi}{doi: \begingroup \urlstyle{rm}\Url}\fi

\bibitem[Alquier and Ridgway(2020)]{alquier2020concentration}
Pierre Alquier and James Ridgway.
\newblock Concentration of tempered posteriors and of their variational approximations.
\newblock \emph{The Annals of Statistics}, 48\penalty0 (3):\penalty0 1475--1497, 2020.

\bibitem[Alquier et~al.(2016)Alquier, Ridgway, and Chopin]{alquier2016properties}
Pierre Alquier, James Ridgway, and Nicolas Chopin.
\newblock On the properties of variational approximations of {Gibbs} posteriors.
\newblock \emph{The Journal of Machine Learning Research}, 17\penalty0 (1):\penalty0 8374--8414, 2016.

\bibitem[Atchad{\'e}(2017)]{atchade2017contraction}
Yves~A Atchad{\'e}.
\newblock On the contraction properties of some high-dimensional quasi-posterior distributions.
\newblock \emph{The Annals of Statistics}, 45\penalty0 (5):\penalty0 2248--2273, 2017.

\bibitem[Atchad{\'e} and Bhattacharyya(2018)]{atchade2018approach}
Yves~A Atchad{\'e} and Anwesha Bhattacharyya.
\newblock An approach to large-scale quasi-{Bayesian} inference with spike-and-slab priors.
\newblock \emph{arXiv preprint arXiv:1803.10282}, 2018.

\bibitem[Baraud and Birg{\'e}(2020)]{baraud2020robust}
Yannick Baraud and Lucien Birg{\'e}.
\newblock Robust {Bayes}-like estimation: Rho-{Bayes} estimation.
\newblock \emph{The Annals of Statistics}, 48\penalty0 (6):\penalty0 3699--3720, 2020.

\bibitem[Bhattacharya et~al.(2019)Bhattacharya, Pati, and Yang]{bhattacharya2019bayesian}
Anirban Bhattacharya, Debdeep Pati, and Yun Yang.
\newblock Bayesian fractional posteriors.
\newblock \emph{The Annals of Statistics}, 47\penalty0 (1):\penalty0 39--66, 2019.

\bibitem[Bhattacharya and Martin(2020)]{bhattacharya2020gibbs}
Indrabati Bhattacharya and Ryan Martin.
\newblock Gibbs posterior inference on multivariate quantiles.
\newblock \emph{arXiv preprint arXiv:2002.01052}, 2020.

\bibitem[Bishop and Nasrabadi(2006)]{bishop2006pattern}
Christopher~M Bishop and Nasser~M Nasrabadi.
\newblock \emph{Pattern recognition and machine learning}, volume~4.
\newblock Springer, 2006.

\bibitem[Buja et~al.(1989)Buja, Hastie, and Tibshirani]{buja1989linear}
Andreas Buja, Trevor Hastie, and Robert Tibshirani.
\newblock Linear smoothers and additive models.
\newblock \emph{The Annals of Statistics}, 17\penalty0 (2):\penalty0 453--510, 1989.

\bibitem[Catoni(2004)]{catoni2004statistical}
Olivier Catoni.
\newblock \emph{Statistical learning theory and stochastic optimization: Ecole d'Et{\'e} de Probabilit{\'e}s de Saint-Flour, XXXI-2001}, volume 1851.
\newblock Springer Science \& Business Media, 2004.

\bibitem[Ch{\'e}rief-Abdellatif(2019)]{cherief2019consistency}
Badr-Eddine Ch{\'e}rief-Abdellatif.
\newblock Consistency of {ELBO} maximization for model selection.
\newblock In \emph{Symposium on Advances in Approximate Bayesian Inference}, pages 11--31. PMLR, 2019.

\bibitem[Ch{\'e}rief-Abdellatif(2020)]{cherief2020convergence}
Badr-Eddine Ch{\'e}rief-Abdellatif.
\newblock Convergence rates of variational inference in sparse deep learning.
\newblock In \emph{Proceedings of the 37th International Conference on Machine Learning}, pages 1831--1842. PMLR, 2020.

\bibitem[Ch{\'e}rief-Abdellatif and Alquier(2020)]{cherief2020mmd}
Badr-Eddine Ch{\'e}rief-Abdellatif and Pierre Alquier.
\newblock {MMD-Bayes}: Robust {Bayesian} estimation via maximum mean discrepancy.
\newblock In \emph{Symposium on Advances in Approximate Bayesian Inference}, pages 1--21. PMLR, 2020.

\bibitem[Fan et~al.(2013)Fan, Liao, and Mincheva]{fan2013large}
Jianqing Fan, Yuan Liao, and Martina Mincheva.
\newblock Large covariance estimation by thresholding principal orthogonal complements.
\newblock \emph{Journal of the Royal Statistical Society. Series B, Statistical methodology}, 75\penalty0 (4), 2013.

\bibitem[Finocchio and Schmidt-Hieber(2023)]{finocchio2023posterior}
Gianluca Finocchio and Johannes Schmidt-Hieber.
\newblock Posterior contraction for deep {Gaussian} process priors.
\newblock \emph{Journal of Machine Learning Research}, 24\penalty0 (66):\penalty0 1--49, 2023.

\bibitem[Ga{i}ffas and Lecu{\'e}(2007)]{gaiffas2007optimal}
St{\'e}phane Ga{i}ffas and Guillaume Lecu{\'e}.
\newblock Optimal rates and adaptation in the single-index model using aggregation.
\newblock \emph{Electronic Journal of Statistics}, 1:\penalty0 538--573, 2007.

\bibitem[Gao and Zhou(2015)]{gao2015pca}
Chao Gao and Harrison~H Zhou.
\newblock Rate-optimal posterior contraction for sparse {PCA}.
\newblock \emph{The Annals of Statistics}, 43\penalty0 (2):\penalty0 785--818, 2015.

\bibitem[Gao and Zhou(2016)]{gao2016rate}
Chao Gao and Harrison~H Zhou.
\newblock Rate exact {Bayesian} adaptation with modified block priors.
\newblock \emph{The Annals of Statistics}, 44\penalty0 (1):\penalty0 318--345, 2016.

\bibitem[Gao et~al.(2015)Gao, Lu, and Zhou]{gao2015rate}
Chao Gao, Yu~Lu, and Harrison~H Zhou.
\newblock Rate-optimal graphon estimation.
\newblock \emph{The Annals of Statistics}, 43\penalty0 (6):\penalty0 2624--2652, 2015.

\bibitem[Gao et~al.(2020)Gao, van~der Vaart, and Zhou]{gao2020general}
Chao Gao, Aad van~der Vaart, and Harrison~H Zhou.
\newblock A general framework for {Bayes} structured linear models.
\newblock \emph{The Annals of Statistics}, 48\penalty0 (5):\penalty0 2848--2878, 2020.

\bibitem[Ghosal and Van~der Vaart(2017)]{ghosal2017fundamentals}
Subhashis Ghosal and Aad Van~der Vaart.
\newblock \emph{Fundamentals of nonparametric {Bayesian} inference}, volume~44.
\newblock Cambridge University Press, 2017.

\bibitem[Ghosal et~al.(2008)Ghosal, Lember, and van~der Vaart]{ghosal2008nonparametric}
Subhashis Ghosal, J{\"u}ri Lember, and Aad van~der Vaart.
\newblock Nonparametric {Bayesian} model selection and averaging.
\newblock \emph{Electronic Journal of Statistics}, 2:\penalty0 63--89, 2008.

\bibitem[Ghosh et~al.(2020)Ghosh, Pati, and Bhattacharya]{ghosh2020posterior}
Prasenjit Ghosh, Debdeep Pati, and Anirban Bhattacharya.
\newblock Posterior contraction rates for stochastic block models.
\newblock \emph{Sankhya A: The Indian Journal of Statistics}, 82\penalty0 (2):\penalty0 448--476, 2020.

\bibitem[Giordano et~al.(2022)Giordano, Ray, and Schmidt-Hieber]{giordano2022inability}
Matteo Giordano, Kolyan Ray, and Johannes Schmidt-Hieber.
\newblock On the inability of {Gaussian} process regression to optimally learn compositional functions.
\newblock In \emph{Proceedings of the 36th International Conference on Neural Information Processing Systems}, volume~35, pages 22341--22353, 2022.

\bibitem[Gr{\"u}nwald(2011)]{grunwald2011safe}
Peter Gr{\"u}nwald.
\newblock Safe learning: Bridging the gap between {Bayes}, {MDL} and statistical learning theory via empirical convexity.
\newblock In \emph{Proceedings of the 24th Annual Conference on Learning Theory}, pages 397--420. JMLR Workshop and Conference Proceedings, 2011.

\bibitem[Gr{\"u}nwald and Van~Ommen(2017)]{grunwald2017inconsistency}
Peter Gr{\"u}nwald and Thijs Van~Ommen.
\newblock Inconsistency of {Bayesian} inference for misspecified linear models, and a proposal for repairing it.
\newblock \emph{Bayesian Analysis}, 12\penalty0 (4):\penalty0 1069--1103, 2017.

\bibitem[Hamm and Steinwart(2021)]{hamm2021adaptive}
Thomas Hamm and Ingo Steinwart.
\newblock Adaptive learning rates for support vector machines working on data with low intrinsic dimension.
\newblock \emph{The Annals of Statistics}, 49\penalty0 (6):\penalty0 3153--3180, 2021.

\bibitem[Han(2021)]{han2021oracle}
Qiyang Han.
\newblock Oracle posterior contraction rates under hierarchical priors.
\newblock \emph{Electronic Journal of Statistics}, 15\penalty0 (1):\penalty0 1085--1153, 2021.

\bibitem[Hoffmann et~al.(2015)Hoffmann, Rousseau, and Schmidt-Hieber]{hoffmann2015adaptive}
Marc Hoffmann, Judith Rousseau, and Johannes Schmidt-Hieber.
\newblock On adaptive posterior concentration rates.
\newblock \emph{The Annals of Statistics}, 43\penalty0 (5):\penalty0 2259--2295, 2015.

\bibitem[Horowitz and Mammen(2007)]{horowitz2007rate}
Joel~L Horowitz and Enno Mammen.
\newblock Rate-optimal estimation for a general class of nonparametric regression models with unknown link functions.
\newblock \emph{The Annals of Statistics}, 35\penalty0 (6):\penalty0 2589--2619, 2007.

\bibitem[Imaizumi and Fukumizu(2022)]{imaizumi2020advantage}
Masaaki Imaizumi and Kenji Fukumizu.
\newblock Advantage of deep neural networks for estimating functions with singularity on hypersurfaces.
\newblock \emph{Journal of Machine Learning Research}, 23\penalty0 (111):\penalty0 1--54, 2022.

\bibitem[Jiang and Tokdar(2021)]{jiang2021consistent}
Sheng Jiang and Surya~T Tokdar.
\newblock Consistent {Bayesian} community detection.
\newblock \emph{arXiv preprint arXiv:2101.06531}, 2021.

\bibitem[Jiao et~al.(2023)Jiao, Shen, Lin, and Huang]{jiao2023deep}
Yuling Jiao, Guohao Shen, Yuanyuan Lin, and Jian Huang.
\newblock Deep nonparametric regression on approximate manifolds: Nonasymptotic error bounds with polynomial prefactors.
\newblock \emph{The Annals of Statistics}, 51\penalty0 (2):\penalty0 691--716, 2023.

\bibitem[Kim et~al.(2021)Kim, Ohn, and Kim]{kim2021fast}
Yongdai Kim, Ilsang Ohn, and Dongha Kim.
\newblock Fast convergence rates of deep neural networks for classification.
\newblock \emph{Neural Networks}, 138:\penalty0 179--197, 2021.

\bibitem[Kingma and Ba(2014)]{kingma2014adam}
Diederik~P Kingma and Jimmy Ba.
\newblock Adam: A method for stochastic optimization.
\newblock In \emph{International Conference on Learning Representations}, 2014.

\bibitem[Knoblauch et~al.(2022)Knoblauch, Jewson, and Damoulas]{knoblauch2022optimization}
Jeremias Knoblauch, Jack Jewson, and Theodoros Damoulas.
\newblock An optimization-centric view on {Bayes}' rule: Reviewing and generalizing variational inference.
\newblock \emph{The Journal of Machine Learning Research}, 23\penalty0 (1):\penalty0 5789--5897, 2022.

\bibitem[Kohler and Langer(2021)]{kohler2021rate}
Michael Kohler and Sophie Langer.
\newblock On the rate of convergence of fully connected deep neural network regression estimates.
\newblock \emph{The Annals of Statistics}, 49\penalty0 (4):\penalty0 2231--2249, 2021.

\bibitem[Kutoyants(2012)]{kutoyants2012statistical}
Yu~A Kutoyants.
\newblock \emph{Statistical inference for spatial Poisson processes}, volume 134.
\newblock Springer Science \& Business Media, 2012.

\bibitem[Le~Cam(1986)]{le1986asymptotic}
Lucien Le~Cam.
\newblock Asymptotic methods in statistical decision theory.
\newblock \emph{Springer Series in Statistics}, 1986.

\bibitem[Lee et~al.(2019)Lee, Lee, and Lin]{lee2019minimax}
Kyoungjae Lee, Jaeyong Lee, and Lizhen Lin.
\newblock Minimax posterior convergence rates and model selection consistency in high-dimensional {DAG} models based on sparse {C}holesky factors.
\newblock \emph{The Annals of Statistics}, 47\penalty0 (6):\penalty0 3413--3437, 2019.

\bibitem[Lember and van~der Vaart(2007)]{lember2007universal}
J{\"u}ri Lember and Aad van~der Vaart.
\newblock On universal {Bayesian} adaptation.
\newblock \emph{Statistics and Decisions-International Journal Stochastic Methods and Models}, 25\penalty0 (2):\penalty0 127--152, 2007.

\bibitem[L'Huillier et~al.(2023)L'Huillier, Travis, Castillo, and Ray]{l2023semiparametric}
Alice L'Huillier, Luke Travis, Isma{\"e}l Castillo, and Kolyan Ray.
\newblock Semiparametric inference using fractional posteriors.
\newblock \emph{arXiv preprint arXiv:2301.08158}, 2023.

\bibitem[Liu and Wang(2016)]{liu2016stein}
Qiang Liu and Dilin Wang.
\newblock Stein variational gradient descent: A general purpose {Bayesian} inference algorithm.
\newblock \emph{Proceedings of the 30th International Conference on Neural Information Processing Systems}, 29, 2016.

\bibitem[Lu et~al.(2021)Lu, Shen, Yang, and Zhang]{lu2020deep}
Jianfeng Lu, Zuowei Shen, Haizhao Yang, and Shijun Zhang.
\newblock Deep network approximation for smooth functions.
\newblock \emph{SIAM Journal on Mathematical Analysis}, 53\penalty0 (5):\penalty0 5465--5506, 2021.

\bibitem[Martin et~al.(2017)Martin, Mess, and Walker]{martin2017empirical}
Ryan Martin, Raymond Mess, and Stephen~G Walker.
\newblock Empirical {Bayes} posterior concentration in sparse high-dimensional linear models.
\newblock \emph{Bernoulli}, 23\penalty0 (3):\penalty0 1822--1847, 2017.

\bibitem[Matsubara et~al.(2022)Matsubara, Knoblauch, Briol, and Oates]{matsubara2021robust}
Takuo Matsubara, Jeremias Knoblauch, Fran{\c{c}}ois-Xavier Briol, and Chris~J Oates.
\newblock Robust generalised {Bayesian} inference for intractable likelihoods.
\newblock \emph{Journal of the Royal Statistical Society Series B: Statistical Methodology}, 84\penalty0 (3):\penalty0 997--1022, 2022.

\bibitem[Medina et~al.(2022)Medina, Olea, Rush, and Velez]{medina2021robustness}
Marco~Avella Medina, Jos{\'e} Luis~Montiel Olea, Cynthia Rush, and Amilcar Velez.
\newblock On the robustness to misspecification of $\alpha$-posteriors and their variational approximations.
\newblock \emph{Journal of Machine Learning Research}, 23\penalty0 (147):\penalty0 1--51, 2022.

\bibitem[Miller and Dunson(2019)]{miller2018robust}
Jeffrey~W Miller and David~B Dunson.
\newblock Robust {Bayesian} inference via coarsening.
\newblock \emph{Journal of the American Statistical Association}, 114\penalty0 (527):\penalty0 1113--1125, 2019.

\bibitem[Ning(2021)]{ning2021spike}
Bo~Ning.
\newblock Spike and slab {Bayesian} sparse principal component analysis.
\newblock \emph{arXiv preprint arXiv:2102.00305}, 2021.

\bibitem[Ohn and Kim(2022)]{ohn2022nonconvex}
Ilsang Ohn and Yongdai Kim.
\newblock Nonconvex sparse regularization for deep neural networks and its optimality.
\newblock \emph{Neural Computation}, 34\penalty0 (2):\penalty0 476--517, 2022.

\bibitem[Polson and Ro{\v{c}}kov{\'a}(2018)]{polson2018posterior}
Nicholas~G Polson and Veronika Ro{\v{c}}kov{\'a}.
\newblock Posterior concentration for sparse deep learning.
\newblock In \emph{Proceedings of the 32nd International Conference on Neural Information Processing Systems}, pages 938--949, 2018.

\bibitem[Ray and Szab{\'o}(2022)]{ray2021variational}
Kolyan Ray and Botond Szab{\'o}.
\newblock Variational {Bayes} for high-dimensional linear regression with sparse priors.
\newblock \emph{Journal of the American Statistical Association}, 117\penalty0 (539):\penalty0 1270--1281, 2022.

\bibitem[Ro{\v{c}}kov{\'a} and George(2016)]{rovckova2016fast}
Veronika Ro{\v{c}}kov{\'a} and Edward~I George.
\newblock Fast {Bayesian} factor analysis via automatic rotations to sparsity.
\newblock \emph{Journal of the American Statistical Association}, 111\penalty0 (516):\penalty0 1608--1622, 2016.

\bibitem[Schmidt-Hieber(2020)]{schmidt2020nonparametric}
Johannes Schmidt-Hieber.
\newblock Nonparametric regression using deep neural networks with {ReLU} activation function.
\newblock \emph{The Annals of Statistics}, 48\penalty0 (4):\penalty0 1875--1897, 2020.

\bibitem[Srivastava et~al.(2017)Srivastava, Engelhardt, and Dunson]{srivastava2017expandable}
Sanvesh Srivastava, Barbara~E Engelhardt, and David~B Dunson.
\newblock Expandable factor analysis.
\newblock \emph{Biometrika}, 104\penalty0 (3):\penalty0 649--663, 2017.

\bibitem[Stone(1985)]{stone1985additive}
Charles~J Stone.
\newblock Additive regression and other nonparametric models.
\newblock \emph{The Annals of Statistics}, 13\penalty0 (2):\penalty0 689--705, 1985.

\bibitem[Suzuki(2019)]{suzuki2018adaptivity}
Taiji Suzuki.
\newblock Adaptivity of deep {ReLU} network for learning in {Besov} and mixed smooth {Besov} spaces: Optimal rate and curse of dimensionality.
\newblock In \emph{International Conference on Learning Representations}, 2019.

\bibitem[Syring and Martin(2023)]{syring2020gibbs}
Nicholas Syring and Ryan Martin.
\newblock Gibbs posterior concentration rates under sub-exponential type losses.
\newblock \emph{Bernoulli}, 29\penalty0 (2):\penalty0 1080--1108, 2023.

\bibitem[van~de Geer et~al.(2014)van~de Geer, B{\"u}hlmann, Ritov, and Dezeure]{van2014asymptotically}
Sara van~de Geer, Peter B{\"u}hlmann, Ya’acov Ritov, and Ruben Dezeure.
\newblock On asymptotically optimal confidence regions and tests for high-dimensional models.
\newblock \emph{The Annals of Statistics}, 42\penalty0 (3):\penalty0 1166--1202, 2014.

\bibitem[Walker and Hjort(2001)]{walker2001bayesian}
Stephen Walker and Nils~Lid Hjort.
\newblock On {Bayesian} consistency.
\newblock \emph{Journal of the Royal Statistical Society Series B: Statistical Methodology}, 63\penalty0 (4):\penalty0 811--821, 2001.

\bibitem[Yang et~al.(2022)Yang, Li, and Wang]{yang2020approximation}
Yunfei Yang, Zhen Li, and Yang Wang.
\newblock Approximation in shift-invariant spaces with deep {ReLU} neural networks.
\newblock \emph{Neural Networks}, 153:\penalty0 269--281, 2022.

\bibitem[Yang and He(2012)]{yang2012bayesian}
Yunwen Yang and Xuming He.
\newblock Bayesian empirical likelihood for quantile regression.
\newblock \emph{The Annals of Statistics}, 40\penalty0 (2):\penalty0 1102--1131, 2012.

\bibitem[Yang et~al.(2016)Yang, Wang, and He]{yang2016posterior}
Yunwen Yang, Huixia~Judy Wang, and Xuming He.
\newblock Posterior inference in {Bayesian} quantile regression with asymmetric laplace likelihood.
\newblock \emph{International Statistical Review}, 84\penalty0 (3):\penalty0 327--344, 2016.

\bibitem[Zhang and Gao(2020)]{zhang2020convergence}
Fengshuo Zhang and Chao Gao.
\newblock Convergence rates of variational posterior distributions.
\newblock \emph{The Annals of Statistics}, 48\penalty0 (4):\penalty0 2180 -- 2207, 2020.

\bibitem[Zhang(2006{\natexlab{a}})]{zhang2006epsilon}
Tong Zhang.
\newblock From $\epsilon$-entropy to {KL}-entropy: Analysis of minimum information complexity density estimation.
\newblock \emph{The Annals of Statistics}, 34\penalty0 (5):\penalty0 2180--2210, 2006{\natexlab{a}}.

\bibitem[Zhang(2006{\natexlab{b}})]{zhang2006information}
Tong Zhang.
\newblock Information-theoretic upper and lower bounds for statistical estimation.
\newblock \emph{IEEE Transactions on Information Theory}, 52\penalty0 (4):\penalty0 1307--1321, 2006{\natexlab{b}}.

\bibitem[Zhao(2000)]{zhao2000bayesian}
Linda~H Zhao.
\newblock Bayesian aspects of some nonparametric problems.
\newblock \emph{The Annals of Statistics}, 28\penalty0 (2):\penalty0 532--552, 2000.

\end{thebibliography}

\newpage

\setcounter{page}{1}
\renewcommand{\thepage}{S-\arabic{page}}

\begin{appendices}
\crefalias{section}{appendix}
\crefalias{subsection}{appendix}
\crefalias{subsubsection}{appendix}

	\begin{center}
		{\large \textsc{Supplement to "Adaptive variational Bayes: Optimality, computation and applications''}} \\
		\medskip 
		{Ilsang Ohn and Lizhen Lin}
		\medskip
	\end{center}
	
In this supplement, we provide a table of contents and proofs for all the results presented in the main text, as well as some additional materials.
		 
\tableofcontents
\bigskip

\section{Numerical studies for comparison with model selection variational Bayes}
\label{appen:numeric_msvb}

In this section, we provide some numerical examples to illustrate the superiority of the adaptive variational Bayes (AVB) over the model selection variational Bayes (MSVB). 

\subsection{Gaussian mixtures}

Consider the Gaussian mixture model for a bivariate sample, where $\Y_1,\dots, \Y_n\iidsim \sum_{k=1}^m\varpi_k\N(\bvartheta_k, \bSigma_k)$ with $(\varpi_1,\dots, \varpi_m)\in\Delta_m$, $\bvartheta_k\in\R^2$ and $\bSigma_k\in\bS_{++}^2$ for $k\in[m]$. We  generate a sample of size $n=200$ from the true model
    \begin{align*}
        \sum_{k=1}^4&\varpi_k^\star \N(\bvartheta_k^\star, \bSigma_k^\star),
            \mbox{ where } (\varpi_1^\star,\dots, \varpi_4^\star)=\frac{1}{10}(3,3,2,2), \\
             & \vartheta_1^\star=(0,0), \vartheta_2^\star=(-4,-4), \vartheta_3^\star=(4,4), \vartheta_4^\star=(0,4) 
             \mbox{ and }  \bSigma_1^\star=\cdots=\bSigma_4^\star=\I_2.
    \end{align*}
We apply the two variational Bayes methods to estimate the Gaussian location-scale mixture model based on the generated sample. We consider mixture models with the number of components $m$ being at most 6, that is, we set $\cM_n=[6]$. For the prior model probabilities, we use $\alpha_{n,m}\propto \exp(-m\log m)$ for  $m\in[6]$.  Given the number of components $m,$ we impose the prior as $(\varpi_k)_{k\in[m]}\sim\Dir(m^{-1}\one_m)$,  and for each $k\in[m]$, $\bvartheta_k|\bSigma_k\sim  \N((0,0)^\top, \bSigma_k)$ and $\bSigma_k^{-1}\sim\texttt{Wishart}(10, 0.1\I)$ independently. For each $m$, we consider the mean-field variational family given by
    \begin{align*}
    \cQ_{n,m}:=\cbr{\vQ_{\text{weight}}\times\bigotimes_{k=1}^m\vQ_{\text{comp},k}:\vQ_{\text{weight}}\in \cP(\Delta_m), \vQ_{\text{comp},k}\in \cP(\R^2\times \bS_{++}^2)}.
    \end{align*}
We then use the coordinate ascent algorithm given in Section 10.2 of \citetS{bishop2006pattern} for solving each variational optimization problem.

The result is presented in \cref{fig:mixture}. The model selection variation Bayes chooses the mixture model with 3 components, which cannot capture the true model. But the adaptive variational Bayes smooths the two mixture models with 3 and 4 components, so this provides a better estimate.

\begin{figure}[t]
    \centering
    \includegraphics[scale=0.10]{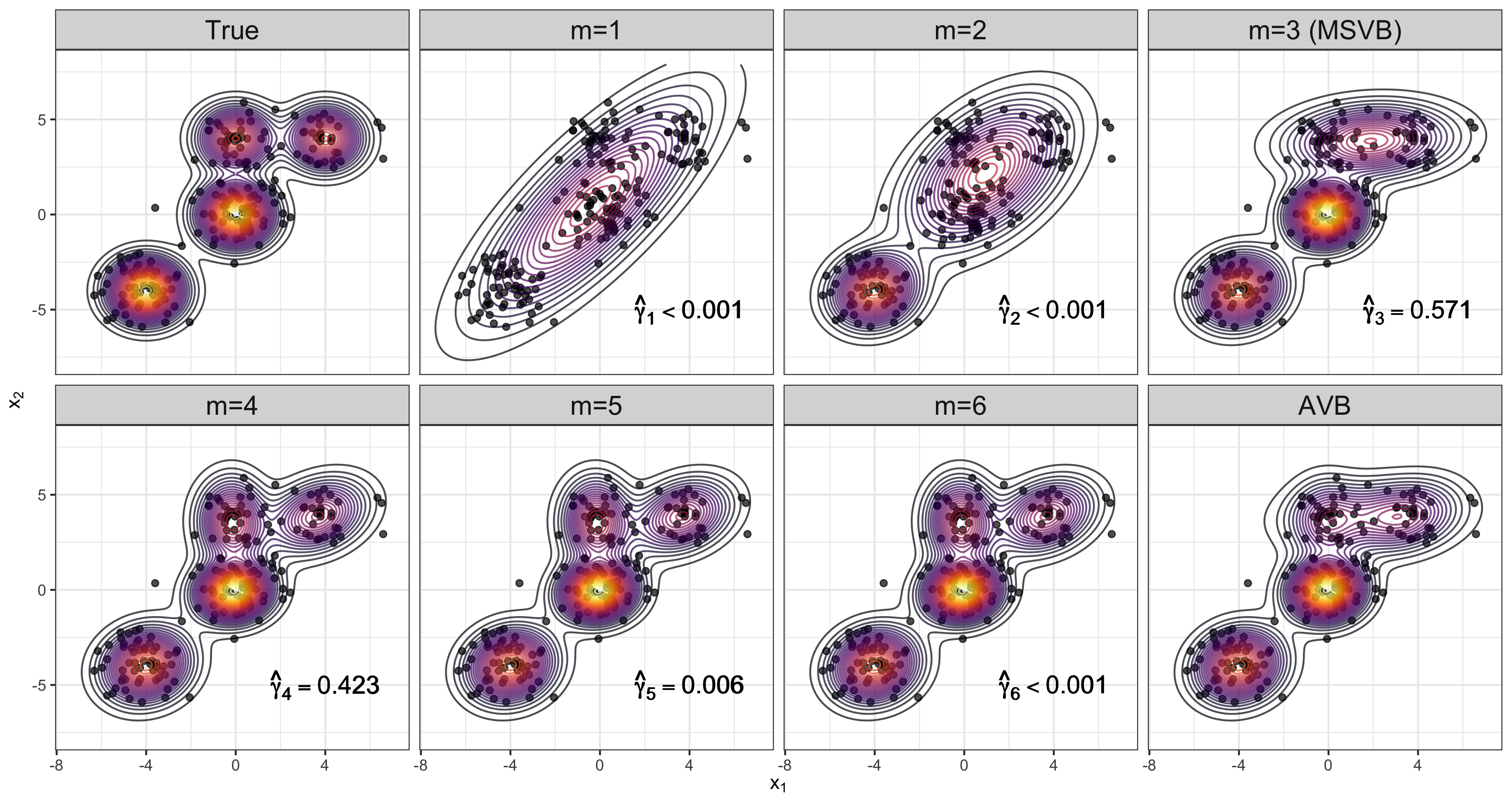}
    \caption{ Contour plots for the true and predictive densities.}
    \label{fig:mixture}
\end{figure}

\subsection{Neural network regression} 

We here compare the predictive abilities of the two variational posteriors over neural networks, computed by the adaptive variational Bayes and the model selection variational Bayes, respectively. We consider the set of network architectures given as $\cM_n=\{2,3\}\times\{25, 50, 100\}$. In this experiment, we consider the shallow architectures of depth 2 or 3 considering the following facts.  First, the required order for the depth in our theory is $\scO(\log n)$, which is far smaller than the order of width $\scO(n^{d/(4\beta+2d)})$. Moreover, several previous studies \citepS{liu2016stein, knoblauch2022optimization} chose shallow networks for their numerical experiments with the same data sets we use.

We let the prior probability of the model with depth $K$ and width $M$ be proportional to $\e^{-\fa_0K^2M^2\log n}$ with $\fa_0=10^{-4}$.  We train each individual variational posteriors for 2,000 epochs, using the ADAM optimizer \citepS{kingma2014adam} with a learning rate of $10^{-2}$. 

The performances of the two variational posteriors are evaluated on eight UCI data sets. For each data set, we split the data set into training (90\%) and test (10\%) sets randomly 50 times. We compute the posterior mean of the neural network based on 100 network parameters generated from the variational posterior and predictions are made based on the approximately computed posterior mean. We calculate the root mean square error (RMSE) on the test set for each of the 50 splits. \cref{tab:deep_regression} presents these RMSE values of the two variational posteriors and we see that the proposed adaptive variational posterior outperforms the model selection variational posterior.

\begin{table}[]
\caption{The average of the RMSEs, with its standard error in the parenthesis, over 50 runs. The best performance is indicated with a bold font for each of the eight benchmark UCI data sets.}
\label{tab:deep_regression}
\begin{tabular}{lcccc}
\hline
Data  & $n $    & $d$  & AVB                 & MSVB                \\\hline
Boston   & 506   & 13 & \textbf{3.807 $\pm$ 1.146}       & 4.008 $\pm$ 1.171      \\
Concrete & 1030  & 8  & \textbf{6.06 $\pm$ 0.512}        & 6.325 $\pm$ 0.562        \\
Energy   & 768   & 9  & \textbf{2.039 $\pm$ 0.188}       & 2.139 $\pm$ 0.251        \\
Kin8mn   & 8192  & 8  & \textbf{0.093 $\pm$ 0.004}       & \textbf{0.093 $\pm$ 0.004}        \\
Naval    & 11934 & 17 & \textbf{0.009 $\pm$ 0.003}       & \textbf{0.009 $\pm$ 0.003}       \\
Power    & 9568  & 4  & \textbf{4.308 $\pm$ 0.233}       & 4.314 $\pm$ 0.228        \\
Wine     & 1599  & 11 & \textbf{0.636 $\pm$ 0.035}       & 0.641 $\pm$ 0.036        \\
Yacht    & 308   & 6  & \textbf{2.424 $\pm$ 0.856}      & 2.853 $\pm$ 1.418       \\\hline
\end{tabular}
\end{table}

\section{A general approach for testing construction}
\label{appen:theory:testing}

In this section, we develop a new method to check the testing condition of \cref{assume:individual:testing}, which might be of independent interest. We will frequently use the next lemma in our applications, where we denote by $\scN(\zeta, \cB,\scd)$ the $\zeta$-covering number of a set $\cB$ with respect to a metric $\scd$.

\begin{lemma}[Sufficient condition of the testing condition]
\label{lemma:testing:renyi}
Suppose that there exist absolute constants $\rho_\circ\in(0,1)$, $\rho_\blacklozenge>1$,  and $\fc_2\ge\fc_1>0$ such that
        \begin{align}
            \fc_1n\scd_n^2(\blambda_0, \blambda_1)
                &\le \sD_{\rho_\circ}\del[1]{\P^{(n)}_{\blambda_0}, \P^{(n)}_{\blambda_1}}, \label{eq:tail1}\\
            \sD_{\rho_\blacklozenge}\del[1]{\P^{(n)}_{\blambda_0}, \P^{(n)}_{\blambda_1}}
                &\le \fc_2n\scd_n^2(\blambda_0, \blambda_1)\label{eq:tail2}
        \end{align}
for  any $\blambda_0, \blambda_1\in\Lambda_n$ and any sufficiently large $n\in\bN$. Moreover, assume that for any $u>0$, there exists an absolute constant $\fc_3(u)>0$ depending on $u$ such that
            \begin{align}
            \label{eq:covering}
                \sup_{\zeta>\zeta_{n,m}} \log\scN\del[1]{u\zeta,\cbr{\btheta\in\Theta_{n,m}:\scd_n(\sT(\btheta),\blambda^\star)\le 2\zeta}, \scd_n}
                \le \fc_3(u) n\zeta_{n,m}^2
            \end{align}
for any $\blambda^\star\in\Lambda_n^\star$, any $m\in\cM_n$ and any sufficiently large $n\in\bN$. 
Then \cref{assume:individual:testing}   holds.
\end{lemma}

\begin{proof}
The proof is deferred to \cref{appen:proof:testing:renyi}.
\end{proof}

In words, the assumptions \labelcref{eq:tail1,eq:tail2} require that the two divergences $\sD_{\rho_\circ}(\cdot, \cdot)$ and $\sD_{\rho_\blacklozenge}(\cdot, \cdot)$ are equivalent, while in  general, $\sD_{\rho_\circ}(\cdot, \cdot)\le \kl(\cdot, \cdot)\le\sD_{\rho_\blacklozenge}(\cdot, \cdot)$ always. This equivalence can hold for a statistical experiment consisting of light-tailed distributions. An useful sufficient condition of \labelcref{eq:tail1,eq:tail2} is that the log likelihood ratio is a sub-Gamma random variable, which was assumed by \citetS{han2021oracle} under the name of a ``local Gaussianity condition''. 

\begin{lemma}[Local Gaussian likelihood ratio]
\label{lemma:testing:local_gauss}
Suppose that there exist absolute constants $\fc_1'>1$, $\fc_2'>0$ and $\fc_3'>0$ such that
    \begin{align*}
          \frac{1}{\fc_1'}n\scd_n^2(\blambda_0,\blambda_1)
            \le \kl\del[1]{\P^{(n)}_{\blambda_0}, \P^{(n)}_{\blambda_1}}\le \fc_1'n\scd_n^2(\blambda_0,\blambda_1)
     \end{align*}
and that
    \begin{align*}
        \P_{\blambda_0}^{(n)}\sbr[3]{\e^{t\cbr[2]{\log\del{\frac{\sp_n(\blambda_0,\Y^{(n)})}{\sp_n(\blambda_1,\Y^{(n)})}} - \kl\del[1]{\P^{(n)}_{\blambda_0}, \P^{(n)}_{\blambda_1}}}} }
        \le \e^{\fc_2't^2n\scd_n^2(\blambda_0,\blambda_1)}
     \end{align*}
for any $t\in[-\fc_3',\fc_3'],$ any $\blambda_0, \blambda_1\in\Lambda_n$ and any sufficiently large $n\in\bN$. Then \labelcref{eq:tail1,eq:tail2}  hold.
\end{lemma}

\begin{proof}
The proof is deferred to \cref{appen:proof:local_gauss}.
\end{proof}
    
\begin{remark}
With \cref{lemma:testing:renyi}, we are able to derive a contraction rate of the adaptive variational posterior with respect to the metric $\scd_n$ that is equivalent to the KL divergence. Though the condition of the lemma is satisfied by a number of examples, such as experiments in \cref{sec:dnn} and \cref{appen:dnn:additional}, and many other examples in \citetS{han2021oracle}, this is limited for dealing with contraction with respect to weaker metrics. For example, the operator norm $\opnorm[0]{\bSigma}:=\sup_{\x\in\R^d:|\x|_2=1}|\bSigma\x|_2$ which is frequently used for evaluating estimation of a covariance matrix $\bSigma^\star\in\bS_{++}^{d}$ does not satisfy \labelcref{eq:tail2}. Manual construction of a test function is required in such a case. Meanwhile, for the Hellinger or total variation distances that are weaker than the KL divergence, we can directly check the testing condition of \cref{assume:individual:testing} by using some well-known results on testing construction. We refer to Appendix D of \citetS{ghosal2017fundamentals} for details.
\end{remark}

\section{Remarks on simplifications of the theoretical conditions of \citet{zhang2020convergence}}
\label{appen:theory:simplifications}

In this section, we give a detailed comparison of our \cref{assume:individual} with the assumptions of \citetS{zhang2020convergence}. Recall that the original posterior $\Pi_{n,m}(\cdot|\Y^{(n)})$ is written as
    \begin{align}
    \label{eq:theory:indv_post}
        \d\Pi_{n,m}(\btheta|\Y^{(n)}):= \frac{\sr_n(\sT(\btheta), \blambda^\star)}{D_{n,m}}\d\Pi_{n,m}(\btheta)
        \mbox{ with } D_{n,m}:=\int \sr_n(\sT(\btheta), \blambda^\star)\d\Pi_{n,m}(\btheta),
    \end{align}
where we use the shorthand  $\sr_n(\sT(\btheta),\blambda^\star):=\sp_n(\sT(\btheta),\Y^{(n)})/\sp_n(\blambda^\star,\Y^{(n)})$.

First, we remove the ``prior mass'' condition typically assumed in posterior contraction analysis. For a detailed explanation, we recall a ``traditional'' high probability lower bound of the denominator  $D_{n,m}$  of the original posterior in \labelcref{eq:theory:indv_post} such that
    \begin{align}
    \label{eq:denom:bound:traditional}
        \P_{\blambda^\star}^{(n)}&\del{D_{n,m}\ge \e^{-T} \Pi_{n,m}\del[1]{\cB_{\kl}(\blambda^\star, n\epsilon^2) }}\ge 1-\frac{2n\epsilon^2+1}{T}\\
        &\mbox{ with }\cB_{\kl}(\blambda^\star, n\epsilon^2)        :=\cbr{\btheta\in\Theta_{n,m}:\kl\del[1]{\P^{(n)}_{\blambda^\star},\P^{(n)}_{\sT(\btheta)}}\le n\epsilon^2} \nonumber
    \end{align}
for any $T>0$ and $\epsilon>0$ \citep[Lemma 6.26 of][]{ghosal2017fundamentals}. In view of this bound, the prior mass on the KL neighborhood of $\blambda^\star$ should be sufficiently large to yield a vanishing bound of the posterior probability that we want to control. In this paper, we develop a novel lower bound of $D_{n,m}$, which relies not on the prior mass but on the two quantities in \cref{assume:individual:variational}.  Indeed, in \cref{lemma:denominator} in  \cref{appen:proof:denom}, we show that for any $\vQ\in\cP(\Theta_{n,m})$ and $T>0$, 
    \begin{align}
    \label{eq:denom:bound:new}
       \P_{\blambda^\star}^{(n)}\del{ D_{n,m}\ge \e^{-T-\kl(\vQ,\Pi_{n,m})}}
       \ge  1- \frac{1}{T}\del[2]{2\vQ\sbr[1]{\kl\del[1]{\P_{\blambda^\star}^{(n)}, \P_{\sT(\btheta)}^{(n)}}} +1}.
    \end{align}
Thus, under \cref{assume:individual:variational}, we get a high probability lower bound of the denominator $D_{n,m}$ with $\vQ=\vQ_{n,m}^*$ and large $T>0$ such that $n (\eta_{n,m}+\zeta_{n,m})^2/T\to0$. It is easy to see that the traditional bound in \labelcref{eq:denom:bound:traditional} is recovered when we let $\vQ(\cdot)=\Pi_{n,m}(\cdot|\cB_{\kl}(\blambda^\star, n\epsilon^2))$. 

Second, we weaken the unnecessarily strong condition on the prior distribution. \citet{zhang2020convergence} employed an ``overwhelming'' probability lower bound of $D_{n,m}$ such that
    \begin{align*}
       \P_{\blambda^\star}^{(n)}&\del{ D_{n,m}\ge \e^{-T}\Pi_{n,m}\del[1]{\cB_{\rho}(\blambda^\star, n\epsilon^2)  }}
       \ge  1- \exp\del[0]{-\rho(T-n\epsilon^2)},\\
       &\mbox{ with }\cB_{\rho}(\blambda^\star, n\epsilon^2)        :=\cbr{\btheta\in\Theta_{n,m}:\sD_\rho\del[1]{\P^{(n)}_{\blambda^\star},\P^{(n)}_{\sT(\btheta)}}\le n\epsilon^2} 
    \end{align*}
for some $\rho>1$, which converges to 1 exponentially in $T>0$, in order to ensure that their PAC-Bayes oracle bound vanishes asymptotically. For this purpose, they assumed that the  prior mass on the ``smaller'' neighborhood $\cB_{\rho}(\blambda^\star, n\epsilon^2)\subset\cB_{\kl}(\blambda^\star, n\epsilon^2)$ is sufficiently large. We avoid requiring this stronger prior mass condition by slightly modifying the proof so that the polynomially converging probability bound in \labelcref{eq:denom:bound:new} can be used.

\section{Additional remarks and results on adaptive variational deep learning}
\label{appen:dnn:additional}

\subsection{Existing results on nonparametric regression with deep neural networks}
\label{appen:dnn:additional:review}

In this subsection, we discuss the previous results on the estimation of H\"older smooth regression function with deep neural networks and compare them with our result. Our \cref{col:dnn:holder} is the first adaptive optimal one for the non-sparse neural network model. This directly follows from \cref{thm:dnn:oracle} and the function approximation analysis by non-sparse neural networks given in \cref{appen:dnn:approx} which is largely borrowed from \citetS{kohler2021rate} and \citetS{lu2020deep}. Therefore, applying the Bayesian adaptation theory developed by \citetS{lember2007universal,ghosal2008nonparametric,han2021oracle}, one can expect adaptive optimality of an original posterior over non-sparse neural networks with varying depth and width, although this has not been done before. However, posterior computation is not easy since depth and width should be dealt with as random variables. On the frequentist side, \citetS{kohler2021rate} and \citetS{jiao2023deep} showed the optimality of an empirical risk minimizer over non-sparse neural networks. Even though their results are basically non-adaptive since network width and/or depth should be chosen based on the smoothness $\beta$, using a theoretically guaranteed model selection technique such as the training-validation approach \citepS[e.g.,][]{hamm2021adaptive}, the frequentist neural network estimator can be made adaptive without much difficulty.

For sparse neural network models, many studies have constructed optimal estimators, but most of them are non-adaptive in the sense that the network sparsity depends on the smoothness of the true regression function \citepS{schmidt2020nonparametric, suzuki2018adaptivity, imaizumi2020advantage, kim2021fast}. Certain penalization \citepS{ohn2022nonconvex} and hierarchical Bayes \citepS{polson2018posterior} methods can be used to attain sparse neural network estimators that are adaptive to the unknown smoothness. \citetS{cherief2020convergence} proposed an adaptive sparse deep learning method based on the model selection variational Bayes with fractional likelihoods of \citeS{cherief2019consistency}.

\subsection{Estimation of composition structured regression functions}
\label{appen:dnn:additional:comp}

To avoid the curse of dimensionality that exists in the convergence rate of estimation of H\"older smooth  regression function, one can make a certain structural assumption on the regression function. \citetS{schmidt2020nonparametric} and \citetS{kohler2021rate} considered so-called composition structured regression functions which include a single-index model \citepS{gaiffas2007optimal}, an additive model \citepS{stone1985additive, buja1989linear} and a generalized additive model with an unknown link function \citepS{horowitz2007rate} as particular instances. Under this structure assumption, they derived the improved convergence rates of sparse and non-sparse neural networks respectively, but note again, that their estimators are nonadaptive.

In this subsection, we show that the adaptive variational deep learning enjoys the optimal contraction rate under the composition structure assumption. Recall the definition of the class of composition structured functions considered in \citetS{schmidt2020nonparametric}
    \begin{equations}
    \label{eq:compose}
        &\cF^{\textsc{comp}}\del[1]{r, (\beta_\ell)_{\ell\in[r]}, (s_\ell)_{\ell\in[r]}, (d_\ell)_{\ell\in[r]}, F_0}\\
        &:=\cbr{f_r\circ(f_{r-1,j})_{j\in[d_{r}]}\cdots\circ (f_{1,j})_{j\in[d_{2}]}\in\cF^{d_1}: f_{\ell,j}\in\cH^{\beta_\ell,s_\ell,F_0}, 0\le f_{\ell,j}\le 1}
    \end{equations}
for the number of compositions $r\in\bN$ and the number of intermediate functions $d_\ell\in\bN$ (except $d_1$ being the input dimension), smoothness $\beta_\ell\in\R_+$, and \textit{intrinsic dimension} $s_\ell\in[d_\ell]$ at the $\ell$-th stage for $\ell\in[r]$.
We here assume that the outputs of the  functions $\{f_{\ell,j}:j\in[d_{\ell+1}],\ell\in[r-1])\}$ can be easily generalized to an arbitrary compact subset of the real line, but we do not consider this generalization for technical simplicity.

\begin{corollary}[Composition structured regression function]
\label{col:dnn:composite}
Let $r\in\bN$, $\mathbf{d}=(d_\ell)_{\ell\in[r]}$, $\bbeta:=(\beta_\ell)_{\ell\in[r]}\in\R_+^r$,  $\s:=(s_\ell)_{\ell\in[r]}\in\otimes_{\ell=1}^r[d_\ell]$ and $F_0>0$. Then under \cref{assume:dnn}, we have
    \begin{equation}
        \sup_{f^\star\in\cF^{\textsc{comp}}\del{r,\mathbf{d}, \bbeta,\s, F_0}} \P_{f^\star}^{(n)}\sbr{\hvQ_n\del{\nnorm{\net(\btheta)- f^\star}\ge \max_{\ell\in[r]}n^{-\frac{\beta_{\ge \ell}}{2\beta_{\ge \ell}+s_\ell}}\log^2n}}=\sco(1),
    \end{equation} 
where we let $\beta_{\ge \ell}:=\beta_\ell\prod_{h=\ell+1}^r(\beta_h\wedge 1)$ for $\ell\in[r-1]$ and $\beta_{\ge r}:=\beta_r$.
\end{corollary}

\begin{proof}
The proof is deferred to \cref{appen:dnn:proof:colloraries}.
\end{proof}

If the intrinsic dimensions $s_1, \dots, s_r$ are much smaller than the input dimension $d_1$, the convergence rate improves substantially compared to the rate for the H\"older smooth functions. Also, the contraction rate of the above corollary is minimax optimal \citepS[Theorem 3 of][]{schmidt2020nonparametric}.

Very recently, it was proven that no standard Gaussian process (GP) can attain the minimax optimal contraction rate \citepS{giordano2022inability}, while a deep GP that ``stacks'' multiple GPs can do \citepS{finocchio2023posterior} for estimating composition structured regression functions. Thus, theoretically, the proposed variational deep learning performs better than any GP and is comparable to the deep GP in this setup.  Also, any wavelet estimator only attains a sub-optimal convergence rate \citepS{schmidt2020nonparametric}.

\subsection{Binary classification}
\label{appen:dnn:additional:logistic}

We consider a statistical experiment of binary classification. Let $\Y:=(\X,Y)\in\bY:=[0,1]^d\times \{0,1\}$ be a pair of input-label variables such that $\X\sim \Unif([0,1]^d)$ and $Y|\X=\x\sim \Ber(p(\x))$ for some \textit{conditional class probability function} $p\in\cF^{d}_{[\varkappa,1-\varkappa]}:=\cbr[0]{p\in\cF^d:\varkappa\le p\le 1-\varkappa}$ for $\varkappa\in(0,1/2)$. Here, we assume that the conditional class probability is bounded away from 0 and 1. This is in order to avoid the diverging behavior of the likelihood near 0 or 1. Such an assumption has been commonly made in analyzing logistic models \citepS[e.g.,][]{van2014asymptotically,ohn2022nonconvex}. For each $n\in\bN$, we observe $n$ iid input-label pairs $\Y^{(n)}:=((\X_i,Y_i))_{i\in[n]}$ which are independently generated from the above probability model, i.e., $\Y^{(n)}\sim \P_p^{(n)}$ for $p\in\cF^d_{[\varkappa,1-\varkappa]}$, where $\P_p^{(n)}$ is a distribution with the likelihood $\sp_n^{\textup{Ber}}(p,\cdot)$ defined as
    \begin{equation}
        \sp_n^{\textup{Ber}}(p,\Y^{(n)})
        =\prod_{i=1}^n\cbr{p(\X_i)^{Y_i}\del{1-p(\X_i)}^{1-Y_i}}.
    \end{equation}

\begin{lemma}
\label{lemma:model:logistic}
Let $\varkappa\in(0,1/2)$ be an arbitrary small positive number. A sequence of the binary classification experiments  $\del[1]{\del[0]{\bY^{\otimes n},\cP(\bY^{\otimes n};\sp_n^{\textup{Ber}},\cF^d_{[\varkappa,1-\varkappa]})}}_{n\in\bN}$ satisfies  \labelcref{eq:tail1,eq:tail2}  with the metric $\scd_n:\cF^d_{[\varkappa,1-\varkappa]}\times\cF^d_{[\varkappa,1-\varkappa]}\mapsto \R_{\ge0}$ defined as $\scd_n(p_0,p_1)=\|p_0-p_1\|_2$ for $p_0,p_1\in\cF^d_{[\varkappa,1-\varkappa]}$.
\end{lemma}

\begin{proof}
The proof is deferred to \cref{appen:dnn:proof:logistic}.
\end{proof}

We model $p$ itself by a neural network, i.e., the likelihood of a network parameter $\btheta$ is given by $\prod_{i=1}^n(\net(\btheta)(\X_i))^{Y_i}(1-\net(\btheta)(\X_i))^{1-Y_i}$. To ensure that neural network estimates are bounded away from 0 and 1, we use a truncation operator as $\net_{[\varkappa,1-\varkappa]}(\btheta):=(\net(\btheta)\vee\varkappa)\wedge(1-\varkappa)$
for a network parameter $\btheta$.

The following corollary shows that if the true conditional probability function is H\"older $\beta$-smooth, our adaptive variational deep learning is able to estimate it with the optimal rate. The proof is almost similar to the proof of \cref{col:dnn:holder}, so we omit it.

\begin{corollary}[H\"older smooth conditional class probability]
\label{col:dnn:logistic}
Let $\beta>0$, $d\in\bN$,  $F_0>0$ and $\varkappa\in(0,1/2)$.
Then under \cref{assume:dnn}, we have
    \begin{equation}
        \sup_{p^\star\in\cH^{\beta, d, F_0}_{[\varkappa,1-\varkappa]}} \P_{p^\star}^{(n)}\sbr{\hvQ_n\del{\Norm{\net_{[\varkappa,1-\varkappa]}(\btheta)- p^\star}_2\ge n^{-\frac{\beta}{2\beta+d}}\log^2n}}=\sco(1),
    \end{equation} 
where  $\cH^{\beta, d, F_0}_{[\varkappa,1-\varkappa]}:=\cbr[1]{p\in\cH^{\beta, d, F_0}:\varkappa\le p\le 1-\varkappa}.$
\end{corollary}

\subsection{Intensity estimation for Poisson point processes}
\label{appen:dnn:additional:ppp}

Let  $\cX$ be the Borel $\sigma$-field of $[0,1]^d$. For each $n\in\bN$, we consider a Poisson point process $Y$ on $[0,1]^d$ with intensity $\lambda:[0,1]^d\mapsto\R_{\ge0}$, that is, $Y$ is an integer-valued random measure such that 
    \begin{enumerate}
        \item  for any $k\in\bN$ and any disjoint $X_1,\dots, X_k\in\cX $, the random variables $Y(X_1),\dots ,Y(X_k)$ are independent;
        \item for any $X\in\cX$, $Y(X)\sim \Pois\del{\int_X\lambda(\x)\d\x}$.
    \end{enumerate}
Let $\bar\bY$ be a set of realizations of the Poisson point process $Y$, i.e.,
    \begin{equation*}
        \bar\bY:=\cbr{\sum_{j=1}^{N}\delta(\cdot;\x_{j}'):N\in\bN,\x_{j}'\in[0,1]^d}.
    \end{equation*}
The aim is to recover the intensity function $\lambda$ of the Poisson point process $Y$ based on $n$ independent realizations $\Y^{(n)}:=(Y_i)_{i\in[n]}\in\bar\bY^{\otimes n}$ of $Y$. With a reference measure being the Poisson point process with constant intensity $\lambda\equiv1$, we can consider a likelihood function $\sp_n^{\textup{PPP}}:\cF_{\ge0}^d\times \bar\bY^{\otimes n}\mapsto\R_{\ge0}$ given by
    \begin{equation}
        \sp_n^{\textup{PPP}}(\lambda, \Y^{(n)})=\prod_{j=1}^n\exp\del{\int\log \lambda(\x)\d Y_i(\x)-\int_{[0,1]^d}(\lambda(\x)-1)\d\x}
        \mbox{ for }\lambda\in\cF_{\ge0}^d,
    \end{equation}
where $\cF_{\ge0}^d$ denotes the set of measurable nonnegative functions supported on $[0,1]^d$. Let $\P_{\lambda}^{(n)}$ be a distribution of the Poisson point process $Y$ with the likelihood function $\sp_n(\lambda,\cdot)$.

We first check the conditions in \cref{lemma:testing:renyi}. For this, we assume that every intensity function is bounded away from zero and infinity for technical simplicity. That is, we assume that every intensity belongs to a class $\cF^d_{[\varkappa_{\min},\varkappa_{\max}]}:=\cbr[0]{\lambda\in\cF^d:\varkappa_{\min}\le \lambda\le \varkappa_{\max}}$
for given $\varkappa_{\max}> \varkappa_{\min}>0$.

\begin{lemma}
\label{lemma:model:ppp}
For any $\varkappa_{\max}> \varkappa_{\min}>0$, a sequence of the Poisson point process experiments $\del[1]{\del[0]{\bar\bY^{\otimes n},\cP(\bar\bY^{\otimes n};\sp_n^{\textup{PPP}},\cF^d_{[\varkappa_{\min},\varkappa_{\max}]})}}_{n\in\bN}$ satisfies  \labelcref{eq:tail1,eq:tail2}   with the metric $\scd_n:\cF^d_{[\varkappa_{\min},\varkappa_{\max}]}\times\cF^d_{[\varkappa_{\min},\varkappa_{\max}]}\mapsto\R_{\ge0}$ defined as $\scd_n(\lambda_0,\lambda_1)=\|\lambda_0-\lambda_1\|_2$ for $\lambda_0,\lambda_1\in\cF^d_{[\varkappa_{\min},\varkappa_{\max}]}$.
\end{lemma}

\begin{proof}
The proof is deferred to \cref{appen:dnn:proof:ppp}.
\end{proof}

We directly model $\lambda$ by a neural network in a sense that the likelihood of a network parameter $\btheta$ is given by $\sp_n^{\textup{PPP}}(\net(\btheta),\Y^{(n)})$. To ensure boundedness of neural network intensity estimates, we apply a truncation operator as 
$\net_{[\varkappa_{\min},\varkappa_{\max}]}(\btheta):=(\net(\btheta)\vee\varkappa_{\min})\wedge\varkappa_{\max}$
for a network parameter $\btheta$. In the next corollary, we show that the adaptive variational deep learning attains the optimal rate for estimating H\"older $\beta$-smooth intensity functions \citepS{kutoyants2012statistical}. The proof is almost similar to the proof of \cref{col:dnn:holder}, thus we omit it.

\begin{corollary}[H\"older smooth intensity]
\label{col:dnn:ppp}
Let $\beta>0$,  $d\in\bN$, $F_0>0$ and $\varkappa_{\max}>\varkappa_{\min}>0$.
Then under \cref{assume:dnn}, we have
    \begin{equation}
        \sup_{\lambda^\star\in\cH^{\beta, d, F_0}_{[\varkappa_{\min},\varkappa_{\max}]}} \P_{\lambda^\star}^{(n)}\sbr[3]{\hvQ_n\del{\Norm{\net_{[\varkappa_{\min},\varkappa_{\max}]}(\btheta)- \lambda^\star}_2\ge n^{-\frac{\beta}{2\beta+d}}\log^2n}}=\sco(1),
    \end{equation} 
where $\cH^{\beta, d, F_0}_{[\varkappa_{\min},\varkappa_{\max}]}:=\cbr[1]{\lambda\in\cH^{\beta, d, F_0}:\varkappa_{\min}\le \lambda\le \varkappa_{\max}}.$
\end{corollary}

\section{Computation algorithm and simulation study for sparse factor models}

In this section, we carry out a simulation study to evaluate the performance of the adaptive variational Bayes method on sparse factor models, which was considered in \cref{sec:sparse:factor}.

\subsection{The PX-CAVI algorithm}
\label{appen:sparse:factor:algorithm}

We present the parameter expansion coordinate-ascent variational inference (PX-CAVI) algorithm proposed by \citetS{ning2021spike} for computing the variational posterior over each individual factor model with fixed factor dimensionality, say $m$, based on the spike-and-slab prior distribution
    \begin{equation*}
    \Pi_{n,m}:= \cbr{(1- \omega_{m})\delta(\cdot;\zero_m) + \omega_{m} \N(\zero_m, \tau_0\I_m )}^{\otimes d} 
    \text{ with } \omega_{m}:=d^{-(1+\fa_0)m}
    \end{equation*}
and the spike-and-slab variational family
    \begin{equation*}
        \cQ_{n,m}:=\cbr{\bigotimes_{j=1}^{d}\cbr{(1- \nu_{j})\delta(\cdot;\zero_m) + \nu_{j}\N(\bpsi_{j}, \bPhi_{j})}:\bpsi_{j}\in\R^{m},\bPhi_{j}\in\bS_{++}^m, \nu_{j}\in[0,1] }.
    \end{equation*}
Here and in the rest of the section, we write $d=d_n$ and $\omega_{m}=\omega_{n,m}$ for convenience.

Following \citetS{ning2021spike}, we introduce latent variables $(\x_i)_{i\in[n]}\in(\R^{m})^{\otimes n}$ and employ the parameter expansion technique. Then we can write the sparse factor model as
    \begin{align*}
        \Y_i|\x_i\indsim \N(\L \K\x_i, \I_{d}),
        \quad \x_i\iidsim \N\del[1]{\zero_m, (\K\K^{\top})^{-1}}
    \end{align*}
for $i\in[n]$, where $\K\in\R^{m\times m}$ is a lower triangular matrix that will be updated in the algorithm. We consider a  normal variational distribution $\bigotimes_{i=1}^n\N(\b_i, \V)$ for the latent variables $(\x_i)_{i\in[n]}$, and we estimate $(\b_i)_{i\in[n]}\in(\R^{m})^{\otimes n}$ and $\V\in\bS_{++}^m$. At every iteration of the algorithm, the following component-wise updates are conducted.

\begin{itemize}
    
    \item (Update $\hbpsi_j$ and $\hbPhi_j$) Calculate
    \begin{align*}
        \bar\H^{[t+1]}:=\frac{1}{n}\sum_{i=1}^n \H_i^{[t+1]} \text{ with }\H_i^{[t+1]}:=\K^{[t]}\cbr[1]{\b_i^{[t]}(\b_i^{[t]})^\top + \V^{[t]}}(\K^{[t]})^{\top}
    \end{align*}
    and then calculate
    \begin{align*}
        \bPhi_j^{[t+1]} &:=\bPhi^{[t+1]}:= \del[1]{n \bar\H^{[t+1]} + \tau_0\I_m}^{-1},\\
        \bpsi_j^{[t+1]} &:=\bPhi^{[t+1]}\sum_{i=1}^n Y_{ij}\K^{[t]}\b_i^{[t]}
    \end{align*}
    for each $j\in[d]$, where $Y_{ij}$ denotes the $j$-th element of $\Y_i$ for $i\in[n]$.
    
    \item (Update $\nu_j$) Let $\text{lgst}(z)=\e^z/(1+\e^z)$ for $z\in\R$. For each $j\in[d]$, calculate
        \begin{align*}
            \nu_j^{[t+1]}:= \text{lgst}\del[3]{\log\frac{\omega_{m}}{1-\omega_{m}}-B_{j}^{[t+1]}-\frac{1}{2}\sum_{i=1}^n E_{ij}^{[t+1]}  },
        \end{align*}
    where we define
    \begin{align*}
        E_{ij}^{[t+1]} :=-2Y_{ij} (\K^{[t]}\b_i^{[t]})^\top\bpsi_j^{[t+1]} + (\bpsi_j^{[t+1]})^\top\H_i^{[t+1]}\bpsi_j^{[t+1]} + \Tr\del[1]{\bPhi^{[t+1]}\H_i^{[t+1]}}
    \end{align*}
    and we denote by $B_{j}^{[t+1]}$ the KL divergence from the prior $\N(\zero_m, \tau_0\I_m)$ to the variational distribution $\N(\bpsi_j^{[t+1]},\bPhi_j^{[t+1]})$ for the parameter $\L_{j,:}$, that is,
    \begin{align*}
         B_{j}^{[t+1]}:= -\frac{m}{2} +\frac{1}{2}\log\del[3]{\frac{|\tau_0\I_m|}{\abs[0]{\bPhi^{[t+1]}}}} -\frac{1}{2\tau_0}\cbr[2]{(\bpsi_j^{[t+1]})^\top\bpsi_j^{[t+1]} +\Tr\del[1]{\bPhi^{[t+1]}}}.
    \end{align*}
    
    \item (Update $\b_i$ and $\V$) Calculate
        \begin{align*}
            \V^{[t+1]} &:=\del[3]{\sum_{j=1}^d\nu_j^{[t+1]}(\K^{[t]})^\top \cbr[1]{\bpsi_j^{[t+1]}(\bpsi_j^{[t+1]})^\top + \bPhi^{[t+1]}} \K^{[t]} + \I_m}^{-1},\\
            \b_i^{[t+1]}&:=\V^{[t+1]}\sum_{j=1}^d Y_{ij} \nu_j^{[t]}(\bpsi_j^{[t+1]})^\top\K^{[t]}
        \end{align*}
    for each $i\in[n]$.
    
    \item (Update $\K$) Let $\K^{[t+1]}$ be a lower triangular matrix satisfying
        \begin{align*}
            \K^{[t+1]}(\K^{[t+1]})^\top=\del[3]{\frac{1}{n}\sum_{i=1}^n\b_i^{[t+1]}(\b_i^{[t+1]})^\top + \V^{[t+1]} }^{-1},
        \end{align*}
    i.e., $\K^{[t+1]}$ is  the lower Cholesky factor of the matrix in the right-hand side of the preceding display.
    
\end{itemize}

\subsection{Simulation study}  
\label{appen:sparse:factor:simulation}

We empirically compare the performance of the adaptive variational Bayes (AVB) and other existing methods for covariance matrix estimation. We fix the hyperparameters of the prior as $\fa_0=0.01$ and $\tau_0=100$ and consider the model space $\cM=[10]$ for the factor dimensionality $m$. We use the variational posterior mean of the covariance matrix as a point estimator produced by the adaptive variational Bayes. For competitors, we consider the principal orthogonal complement thresholding method (POET, \citepS{fan2013large}), and two maximum a posteriori estimators that employ the multi-scale generalized double Pareto prior (MDP, \citepS{srivastava2017expandable}) and the spike-and-slab lasso with Indian buffet process prior (SSL-IBP, \citepS{rovckova2016fast}), respectively. For the POET method, the factor dimensionality must be chosen before its estimation, and we use the true factor dimensionality for this.  

A synthetic sample of size $n=200$ is generated from the $d=1000$-dimensional normal distribution with mean $\zero$ and covariance $\bSigma^\star:=\L^\star(\L^\star)^\top+\I$, where the true loading matrix $\L^\star\in\R^{d\times r}$ has $s$ many nonzero rows. We generate the true loading matrix $\L^\star$ as follows: we first select positions of $s$ nonzero rows and sample the elements in each nonzero row from the normal distribution with mean $\zero_r$ and variance $(5/\sqrt{s})^2\I_r$. We consider various values of the row sparsity $s\in\{20, 50, 100\}$ and various values of the factor dimensionality $r\in\{1,3,6\}$. We generate  100 synthetic data sets and we report the averages of the scaled operator norm losses $\opnorm{\hat\bSigma-\bSigma^\star}/\opnorm{\bSigma^\star}$ for an estimate $\hat\bSigma$ and the true covariance matrix $\bSigma^\star$ computed on the generated synthetic data sets. The result is presented in \cref{tab:factor:cov} and we can see that the proposed method performs well compared to the other methods.

\begin{table}[t]
\centering
\caption{The average of the scaled operator norm losses, and its standard error in the parenthesis, computed on 100 synthetic data sets with $n=200$ and $d=1000$. The best performance is indicated with a bold font for each of the nine considered settings.}
\label{tab:factor:cov}
\begin{tabular}{cccccc}
    \hline
 $s$ & $r$ & POET & MDP & SSL-IBP & AVB \\ 
  \hline
\multirow{3}{*}{20} & 1 & 0.563 (0.173) & 0.534 (0.163) & 0.296 (0.088) & \textbf{0.144 (0.062)} \\ 
   & 3 & 0.492 (0.151) & 0.473 (0.145) & 0.285 (0.083) & \textbf{0.161 (0.072)} \\ 
   & 6 & 0.409 (0.119) & 0.397 (0.115) & 0.281 (0.077) & \textbf{0.155 (0.061)} \\ 
   \hline
\multirow{3}{*}{50} & 1 & 0.549 (0.205) & 0.52 (0.193) & 0.286 (0.105) & \textbf{0.2 (0.077)} \\ 
   & 3 & 0.496 (0.186) & 0.473 (0.177) & 0.297 (0.119) & \textbf{0.206 (0.086)} \\ 
   & 6 & 0.468 (0.177) & 0.451 (0.171) & 0.323 (0.117) & \textbf{0.216 (0.091)} \\ 
   \hline
\multirow{3}{*}{100} & 1 & 0.558 (0.203) & 0.526 (0.192) & 0.32 (0.117) & \textbf{0.301 (0.112)} \\ 
   & 3 & 0.533 (0.195) & 0.508 (0.186) & \textbf{0.319 (0.123)} & 0.323 (0.124) \\ 
   & 6 & 0.523 (0.196) & 0.501 (0.188) & 0.362 (0.132) & \textbf{0.36 (0.144)} \\ 
   \hline
\end{tabular}
\end{table}

\section{Application to high-dimensional nonparametric regression using neural networks}
\label{appen:sparse:dnn}

In this section, we consider the nonparametric regression experiment considered in \cref{sec:dnn:reg}, but we allow the input dimension to diverge, which we denoted by $d_n$. We assume that the true regression function is a H\"older smooth function that only depends on $s_0$ elements in the $d_n$-dimensional input, that is, for any $\x:=(x_j)_{j\in[d_n]}\in[0,1]^{d_n}$, $f^\star(\x)=f_0^\star((x_j)_{j\in S^\star})$  for some $S^\star\subset[d_n]$ with $|S^\star|=s_0$ and  $f_0^\star\in\cH^{\beta, s_0, F_0}$. We let $\cF^{\textup{sparse}}\del{d_n,\beta,s_0, F_0}$ be a set of these $s_0$-sparse H\"older $\beta$-smooth functions on $[0,1]^{d_n}$. 

Our aim is to show that the adaptive deep learning achieves the optimal contraction rate to the true regression function $f^\star$, and this rate is adaptive simultaneously to the smoothness $\beta$ and the sparsity $s_0$. As we did in \cref{sec:dnn:reg}, we consider multiple network parameter spaces $\{\Theta_{(K,M)}^{d_n, \le B_n}\}_{(K,M)\in\cM_n}$ with various network architectures to adapt the smoothness, where $\Theta_{(K,M)}^{d_n, \le B_n}$ is the set of network parameters with input dimension $d_n$, depth $K$, width $M$ and a magnitude bound $B_n$ as defined in \labelcref{eq:dnn:bounded_space}. Let $J_{n,(K,M)}:=(d_n+1)M+(K-2)(M^2+M)+(M+1)$ be the dimension of the parameter space $\Theta_{(K,M)}^{d_n, \le B_n}$.

To appropriately address the sparse structure of the true regression function, we impose a spike-and-slab prior on the weight matrix at the first layer. For a network parameter $\btheta:=((\W_k,\b_k))_{k\in[K]}\in\Theta_{(K,M)}^{d_n, \le B_n}$, let $\W_{1,j,:}\in\R^{M}$  be the $j$-th row of the weight matrix $\W_1\in\R^{M\times d_n}$ at the first layer. We impose the prior distribution given as
    \begin{equations}
        \label{eq:sparse_prior_dnn}
        \SS_{n,(K,M)}:= (1-\omega_{n,(K,M)}&)\delta(\cdot;\zero_{M})+ \omega_{n,(K,M)}\Unif(-B_n, B_n)^{\otimes M}\in\cP(\R^M)\\
        &\mbox{ with }  \omega_{n,(K,M)}:=\exp\del{-\fa_1KM\log n-(1+\fa_2)\log d_n}
    \end{equations}
on each $\W_{1,j,:}$, for some arbitrary constants $\fa_1>0$ and $\fa_2>0$. In words, under the above prior, the event $\{\W_{1,j,:}=\zero\}$ happens with probability $1- \omega_{n,(K,M)}$, and given this event, $x_j$ does not affect the output of the neural network. A novel feature of our prior is that the degree of sparse regularization depends on the neural network's depth and width, enabling us to attain optimality adaptive to both the unknown smoothness and sparsity. We impose an uniform prior distribution on the network parameter's other elements. To sum up, the prior distribution we assume is 
    \begin{equations}
        \Pi_n&=\sum_{(K,M)\in\cM_n}\alpha_{n,(K,M)}\Pi_{n,(K,M)},\\
        &\mbox{ where }\Pi_{n,(K,M)}:=\del[1]{\SS_{n,(K,M)}}^{\otimes d_n}\times\Unif(-B_n, B_n)^{\otimes (J_{n,(K,M)}-d_nM)}\\
        &\mbox{ and }\alpha_{n,(K,M)}:= \frac{\e^{-\fa_3(KM)^2\log n}}{\sum_{(K',M')\in\cM_n}\e^{-\fa_3(K'M')^2\log n}}
    \end{equations}
for some arbitrary constant $\fa_3>0$ such that $\fa_3>\fa_1\wedge\fa_2$, where $\fa_1>0$ and $\fa_2>0$ are the hyperparameters of the spike-and-slab prior in \labelcref{eq:sparse_prior_dnn}. For each network architecture $(K,M)\in\cM_n$, we consider a variational family such as
    \begin{align*}
        \cQ_{n,(K,M)}:=\bigg\{\bigotimes_{h=1}^{d_n}&\cbr[2]{(1-\nu_h)\delta(\cdot;\zero_{M}) + 
        \nu_h\mathlarger{\mathlarger{\otimes}}_{j=(h-1)M+1}^{hM}\Unif(-\psi_{1,j}, \psi_{2,j})} \\
        &\times\bigotimes_{j=d_nM+1}^{J_{n,(K,M)}}\Unif(-\psi_{1,j}, \psi_{2,j})
        :-B_n\le \psi_{1,j}<\psi_{2,j}\le B_n, \nu_h\in[0,1]\bigg\}.
    \end{align*}
    
The adaptive variational posterior with the above prior and variational families achieves the optimal rate.

\begin{theorem}[Sparse H\"older smooth regression function] 
\label{thm:sparse:dnn}
Let $(d_n)_{n\in\bN}\subset\bN$, $\beta>0$, $s_0\in\bN$ and $F_0>0$. Assume $s_0\le d_n$ eventually. Then under \cref{assume:dnn}, we have
    \begin{equation}
        \sup_{f^\star\in\cF^{\textup{sparse}}\del{d_n,\beta,s_0, F_0}} 
        \P_{f^\star}^{(n)}\sbr{\hvQ_n\del{\nnorm{\net(\btheta)- f^\star}\ge n^{-\frac{\beta}{2\beta+s_0}}\log^2n+ A_n\sqrt{\frac{\log d_n}{n}}}}=\sco(1),
    \end{equation}
for any diverging sequence $(A_n)_{n\in\bN}\to\infty$.
\end{theorem}

\begin{proof}
The proof is deferred to \cref{appen:sparse:dnn:proof}.
\end{proof}

The contraction rate in the above theorem is optimal up to a logarithmic factor \citep{yang2015minimax}. \citet{yang2015minimax} and \citet{jiang2021variable} showed that Gaussian processes combined with sparsity-inducing priors attain the optimal contraction rate. Another optimal Bayesian approach is Bayesian trees (and their ensembles) with the ``spike-and-tree'' prior developed by \citet{rovckova2020posterior}, but their theoretical analysis is limited to the cases that $0<\beta\le 1$. 

If $\log d_n/n=\sco(1)$, that is, the dimension diverges slower than the exponential of the sample size, then the contraction rate is the same as the minimax rate $n^{-\beta/({2\beta+s_0})}$ for ``low''-dimensional nonparametric regression with $s_0$-dimensional inputs.

\section{Application to a Gaussian sequence model}
\label{appen:example:gaussian_sequence}

As a fully nonparametric example, we consider the Gaussian sequence model
    \begin{align*}
        Y_i = \theta_i + \frac{1}{\sqrt{n}}Z_i, \quad Z_i\iidsim \N(0,1), \quad \text{ for }i\in\bN,
    \end{align*}
which is equivalent to the prototypical white noise model $\d Y(t) = \theta(t) + n^{-1/2} \d W(t)$ for $t\in[0,1]$ with $\theta:[0,1]\mapsto \R $ being $\scL^2$-integrable and $W$ a standard Brownian motion \citepS{le1986asymptotic}.
We let $\P_{\btheta}^{(n)}:=\otimes_{i=1}^\infty \N(\theta_i,n^{-1})$ denote the distribution above. We assume that the true parameter $\btheta^\star\in \R^{\infty}$ belongs to the $\beta$-Sobolev ball
    \begin{align*}
        \Theta^\star = \Theta_\beta(B):=\cbr{\btheta=(\theta_i)_{i\in\bN}:\sum_{i=1}^\infty i^{2\beta}\theta_i^2\le B^2}
     \end{align*}
for given smoothness $\beta>0$ and radius $B>0$. We will show that the adaptive variational Bayes can attain the minimax optimal rate $n^{-\beta/(2\beta+1)}$ (up to a logarithmic factor) for estimating the true parameter $\btheta^\star\in  \Theta_\beta(B)$ adaptively in terms of the $\scL^2$ distance $\scd_n(\btheta,\btheta^\star) = \abs[0]{\btheta-\btheta^\star}_2$. We consider the set of models $\{\Theta_{n,m}:=\R^m:m\in \cM_n\}$, where the model space is given by 
    \begin{align*}
        \cM_n := \cbr{\ceil{\del{\frac{n}{\log n}}^{k/\log n}}:k\in\sbr{\ceil{\log n}}}.
    \end{align*}
On each model $\Theta_{n,m}$ we impose a standard Gaussian prior $\Pi_{n,m}:=\otimes_{i=1}^m \N(0,1)$. Let $\sT$ be a natural parametrization map such that $\sT(\btheta)=(\btheta^\top,\zero_\infty^\top)^\top\in \R^{\infty}$, which maps any finite sequence to an infinite one by padding zeros.

In the next two examples, we see how prior regularization, which is a standard method, and ivB regularization, which is proposed in \cref{sec:regularization}, are applied in this setup, respectively.

\subsection{Prior regularization}

Here we consider a variational family
    \begin{align*}
        \cQ_{n,m}:=\cbr{\bigotimes_{i=1}^m \N(\psi_i,\Phi_i):\psi_i\in \R, \Phi_i>0}
    \end{align*}
for each $m\in \cM_n$.  \cref{assume:individual:variational} can be verified with $\zeta_{n,m}=m\sqrt{\log n/n}$ and $\eta_{n,m}=Bm^{-\beta}$ as follows: letting $Q_{n,m}^*:=\otimes_{i=1}^m \N(\theta_i^\star,n^{-1})\in \cQ_{n,m}$, we have
    \begin{align*}
        \int \kl(\P_{\btheta^\star}^{(n)}, \P_{\sT(\btheta)}^{(n)})\d Q_{n,m}^*(\btheta)
        &=\frac{n}{2}\int\abs{\sT(\btheta)-\btheta^\star}_2^2\d Q_{n,m}^*(\btheta)\\
        &= \frac{m}{2} + \frac{n}{2}\sum_{i=m+1}^\infty (\theta^\star_i)^2\le \frac{1}{2}\del{m+ n m^{-2\beta}B^2}
    \end{align*}
and
    \begin{align*}
         \kl(Q_{n,m}^*, \Pi_{n,m})
        &= \frac{1}{2}\cbr{m\log n+\frac{m}{n} + \sum_{i=1}^m  (\theta^\star_i)^2 -m}\lesssim m \log n.
    \end{align*}
Moreover, we set the prior model probabilities as $\alpha_{n,m}\propto\exp(- \fa_0m \log n)$ for $\fa_0>0$ for any $m\in \cM_n$ which satisfies \cref{assume:aggregation:model_prior_pen,assume:aggregation:model_prior_mass} by \cref{lemma:model_prob_choice}. Since \cref{assume:individual:testing} is met by Lemma B.7 of \citetS{zhang2020convergence}, we get a contraction rate $\epsilon_n$ given by
    \begin{align*}
        \epsilon_n=\min_{m\in \cM_n}\cbr{\sqrt{\frac{m\log n}{n}}+ m^{-\beta}}\asymp \del{\frac{\log n}{n}}^{\beta/(2\beta+1)}.
    \end{align*}

\subsection{ivB regularization}
In this subsection,  we do not use the complexity prior $\alpha_{n,m}\propto\exp(- \fa_0m \log n)$ considered above for regularizing overly complex models. Instead, we consider a ``constrained'' variational family given below to employ the regularization effect of variational approximation. Concretely, we consider
    \begin{align}
    \label{eq:gaussian_sequence:rvf}
        \cQ_{n,m}^\dag:=\cbr{\bigotimes_{i=1}^m \N(\psi_i,\Phi_i)\in \cQ_{n,m}: \sum_{i=1}^m(\psi_i^2+\Phi_i)\le B^2+1}.
    \end{align}
But the variational distribution $Q_{n,m}^*:=\otimes_{i=1}^m \N(\theta_i^\star,n^{-1})$ we used before for checking  \cref{assume:individual:variational} is still in $\cQ_{n,m}^\dag$ since $\sum_{i=1}^m\{(\theta_i^\star)^2 + n^{-1}\}\le \sum_{i=1}^m i^{2\beta}(\theta_i^\star)^2 +1\le B^2+1 $. This implies that  \cref{assume:individual:variational} holds also for the constrained variational family $ \cQ_{n,m}^\dag$. On the other hand, letting $\Pi_{n,m}^*=\N(0,1/2)^{\otimes m}$, we have that each ivB penalty is lower bounded as
    \begin{align*}
        \Psi_{n,m}   
        &\ge \inf_{\vQ\in \cQ_{n,m}^\dag}\cbr{\kl(\vQ, \Pi_{n,m})-\kl(\vQ, \Pi_{n,m}^*)}\\
        &=\inf_{(\psi_i,\Phi_i)_{i\in[m]}\in(\R\times \R_+)^{\otimes m}:\sum_{i=1}^m(\psi_i^2+\Phi_i)\le B^2+1}\cbr{\frac{m}{2}\log 2- \frac{1}{2}\sum_{i=1}^m(\psi_i^2+\Phi_i)}\\
        &\ge \frac{m}{2}\log 2-\frac{1}{2} (B^2+1).
  \end{align*}
Therefore, for any diverging sequence $(A_n)_{n\in\bN}$, a contraction rate is given by
    \begin{align*}
        \zeta_n^\ddag = \max_{m\in\cM_n: \Psi_{n,m}    < A_nn\epsilon_n^2} \zeta_{n,m}
        &\le  \max_{m\in\cM_n: m  <2 n\epsilon_n^2 + (B^2+1)/2 } \zeta_{n,m}\\
        &\lesssim \sqrt{(A_nn\epsilon_n^2 +B^2)\frac{\log n}{n}}\\
        &\lesssim \sqrt{A_n}\epsilon_n \sqrt{\log n}=\sqrt{A_n}n^{-\beta/(2\beta+1)}(\log n)^{(2\beta+1/2)/(2\beta+1)},
    \end{align*}
which is $\sqrt{\log n}$ times larger than the rate $A_n\epsilon_n$ we obtained using prior regularization.

\subsection{Simulation study}

We conduct a simulation study to illustrate the performance of the adaptive variational Bayes for the mean estimation of the Gaussian sequence model. For given smoothness $\beta>0$, we generate the true sequence $\btheta^\star$ as $\theta_i^\star=5 s_i i^{-\beta-0.6}$ with each $s_i$ randomly chosen as $-1$ or 1 with equal probabilities. Then the generated sequence belongs to $\Theta_\beta(B)$ with $B=5(\sum_{i=1}^\infty i^{-1.2})^{1/2}\le12.$ In the simulation, we consider three values 0.5, 1, and 1.5 of $\beta$. For the sample size $n$, we consider $10^2$, $10^3$, and $10^4$.

We consider two adaptive variational Bayes methods, each of which was considered in the previous two sections, respectively. For the prior regularized method (AVBp), we use the prior model probabilities such that $\alpha_{n,m}\propto\exp(- \fa_0m \log n)$ with $\fa_0=0.1$. For the ivB regularized method (AVBi), we use the uniform prior model probabilities, $\alpha_{n,m}=1/|\cM_n|$ but restricted variational families $ \cQ_{n,m}^\dag$ given in \labelcref{eq:gaussian_sequence:rvf} with $B= 12$ for every $m\in\cM_n$. For competitors, we first consider the Bayesian method using the rescaled Gaussian prior (RGP) $\theta_i\indsim \N(0, i^{-2\beta_0-1})$ with the oracle choice $\beta_0=\beta$. This is a non-adaptive but minimax optimal estimator \citepS{zhao2000bayesian}. Another competitor we consider is the mean-field variational Bayes (MFVB) estimator with prior on the truncation level $m$, which was developed and shown to be adaptively optimal by \citetS{zhang2020convergence}. For the MFVB method, we impose the prior distribution $\Pi(m)\propto \exp(-0.1 m \log n)$ and $\theta_1,\dots, \theta_m|m\iidsim \N(0,1)$, which is the same as the prior we use for the adaptive variational Bayes method with prior regularization. 

We run the simulation 100 times, and, for each method, we compute the square of the $\scL^2$ distance between the (variational) posterior mean of $\btheta$ and the true value $\btheta^\star$. The result is presented in \cref{tab:gaussian_sequence:simul} and we can see that two adaptive variational Bayes methods are comparable to the oracle method RGP and superior to the MFVB. Interestingly, ivB regularization works well in this setting, which is in contrast to the $\log n$ sub-optimality in the theory.

\begin{table}[t]
\centering
\caption{The average of the square of the $\scL^2$ distances to the true parameter, and its standard error in the parenthesis, computed on 100 synthetic data sets. The best performance is indicated with a bold font for each of the nine considered settings.}
\label{tab:gaussian_sequence:simul}
\begin{tabular}{cccccc}\hline
 $\beta$ & $n$ & RGP & MFVB & AVBp & AVBi \\ 
  \hline
\multirow{3}{*}{0.5} & 100 & 1.032 (0.088) & 1.238 (0.194) & 0.81 (0.1) & \textbf{0.742 (0.152)} \\ 
   & 1000 & \textbf{0.282 (0.014)} & 0.451 (0.048) & 0.3 (0.008) & 0.3 (0.009) \\ 
   & 10000 & \textbf{0.073 (0.002)} & 0.15 (0.011) & 0.097 (0.001) & 0.097 (0.001) \\ 
   \hline
\multirow{3}{*}{1} & 100 & 0.349 (0.062) & 0.36 (0.096) & 0.214 (0.087) & \textbf{0.195 (0.07)} \\ 
   & 1000 & 0.068 (0.008) & 0.098 (0.021) & 0.063 (0.025) & \textbf{0.051 (0.021)} \\ 
   & 10000 & \textbf{0.013 (0.001)} & 0.024 (0.004) & 0.015 (0.001) & 0.015 (0.001) \\ 
   \hline
\multirow{3}{*}{1.5} & 100 & 0.19 (0.038) & 0.189 (0.086) & 0.105 (0.046) & \textbf{0.101 (0.042)} \\ 
   & 1000 & 0.031 (0.005) & 0.041 (0.013) & 0.028 (0.01) & \textbf{0.023 (0.01)} \\ 
   & 10000 & \textbf{0.005 (0.001)} & 0.008 (0.002) & 0.005 (0) & 0.005 (0) \\ 
   \hline
\end{tabular}
\end{table}

\bigskip

\section{Related work on quasi-posteriors}
\label{appen:quasi:related}

In this section, we discuss some examples of the quasi-posterior and the related literature.

    \begin{itemize}
        \item \textit{Fractional posterior}.  A fractional posterior is a quasi-posterior where a fractional likelihood $(\sp_n)^{\kappa}$ for some exponent $\kappa\in(0,1)$ is used as a quasi-likelihood. This replacement gives a theoretical advantage that its concentration properties can be established without a testing construction \citepS{walker2001bayesian, zhang2006epsilon, bhattacharya2019bayesian, martin2017empirical,lee2019minimax}.  Moreover, there is empirical and theoretical evidence that the fractional posterior improves robustness to model misspecification \citepS{grunwald2011safe,grunwald2017inconsistency,miller2018robust,medina2021robustness}. A recent paper \citepS{l2023semiparametric} extensively studied the asymptotic properties of the fractional posterior with a sample-size dependent choice of the exponent $\kappa=\kappa_n$, and characterized the explicit dependence of the contraction rate on $\kappa_n$.
        
        \item \textit{Gibbs posteriors}. A loss function based quasi-likelihood, which measures the ``difference'' between the parameter of interest and the sample, is used. This is particularly useful when the specification of a likelihood function is not straightforward and/or is too restrictive for defining the true model. Quantile regression is a representative example. Several quasi-likelihood functions, including the check loss based quasi-likelihood (also known as the asymmetric Laplace likelihood) \citepS{yang2016posterior} and an empirical likelihood \citepS{yang2012bayesian} have been proposed. General theories for the contraction properties of Gibbs posteriors were established by \citetS{atchade2017contraction, atchade2018approach} and \citetS{syring2020gibbs}. From a PAC-Bayes perspective, a Gibbs posterior is theoretically appealing since it minimizes a PAC-Bayes upper bound of risk over all probability distributions \citepS{catoni2004statistical,zhang2006information}. Based on this fact, \citetS{alquier2016properties} provided theoretical conditions under which variational approximations of Gibbs posteriors have optimal contraction properties. They also studied interesting applications such as classification with the hinge loss and ranking with the pairwise ranking loss.
        
        \item \textit{Robust posteriors}. It is known that original posteriors are not robust to model misspecification. Some robust surrogates for original posteriors have been proposed.   \citetS{baraud2020robust} proposed a robust quasi-posterior where the logarithmic function in the log-likelihood ratios is replaced by some bounded function. \citetS{cherief2020mmd} and \citetS{matsubara2021robust} considered a quasi-likelihood function based on the maximum mean discrepancy and Stein discrepancy measures, respectively, which are more robust to model misspecification.
    \end{itemize}

Our theoretical framework for variational quasi-posteriors can be applied to fractional posteriors as we showed in \cref{lemma:quasi:fractional}, as well as Gibbs posteriors with ``sub-exponential'' loss functions satisfying \cref{assume:quasi:quasi}. But for robust posteriors, there is a technical difficulty in relating a robust quasi-likelihood function and the measure of contraction $\scd_n$. So \cref{assume:quasi:quasi,assume:quasi:variational} are not directly established and even violated. Developing an unified framework for robust posteriors is an interesting avenue for future work.

\section{Applications of adaptive variational quasi-posterior}
\label{appen:quasi:application}

\subsection{Application to stochastic block models}
\label{appen:quasi:sbm}

In this subsection, we consider the stochastic block model described in \cref{example:sbm}, to illustrate a situation where the use of a quasi-likelihood provides some theoretical advantages.

\subsubsection{Random graph experiment}

We first introduce some notation. For $n\in\bN$, let $\bL_n:=\{(i,j)\in [n]^2:i>j\}$ be the set of 2-dimensional indices with the first larger than the second. For a set $I$, let $I^{\bL_n}:=\cbr{(b_{i,j})_{(i,j)\in\bL_n}:b_{i,j}\in I}$  be the set of lower triangular arrays of $I$-valued elements. Here and afterward, we index all the notation by the number of connections $\ban:=|\bL_n|=n(n-1)/2$ instead of the number of nodes $n$, to make the notation consistently. 

Suppose that we observe a triangular array $\Y^{(\ban)}:=(Y_{i,j})_{(i,j)\in\bL_n}\in \bY_{\ban}:=\{0,1\}^{\bL_n}$ which represents the existence of connections between $n$-nodes. The $(i,j)$-th element $Y_{i,j}$ is encoded to be 1 if the $i$-th and $j$-th nodes are connected and to be 0 otherwise. We model the sample using a set of distributions
    \begin{equation*}
        \cP_{\ban}:=\cbr{\P^{(\ban)}_{\bOmega}:=\bigotimes_{(i,j)\in\bL_n}\Ber(\Omega_{i,j}):\bOmega\in[0,1]^{\bL_n}},
    \end{equation*}
i.e., $\cP_{\ban}:=\cP(\{0,1\}^{\bL_n};\sp_{\ban},[0,1]^{\bL_n})$ where $\sp_{\ban}:\{0,1\}^{\bL_n}\times[0,1]^{\bL_n}\mapsto \R_{\ge0}$ is a Bernoulli likelihood function  defined as  $\sp_{\ban}(\bOmega, \Y^{(n)})=\prod_{(i,j)\in\bL_n}\Omega_{i,j}^{Y_{i,j}}(1-\Omega_{i,j})^{1-Y_{i,j}}$. We call the parameter $\bOmega\in[0,1]^{\bL_n}$ the \textit{connectivity probability array}. A general goal is to provide a good estimate of the true connectivity probability array in terms of the (scaled) $\scL^2$ error defined as
    \begin{equation*}
        \scd_{\ban,2}(\bOmega_0,\bOmega_1)=\del[3]{\frac{1}{\ban}\sum_{(i,j)\in\bL_n}(\Omega_{0,i,j}-\Omega_{1,i,j})^2}^{1/2}
        \mbox{ for $\bOmega_0,\bOmega_1\in[0,1]^{\bL_n}$.}
    \end{equation*}

\subsubsection{Gaussian quasi-likelihood}

Here we propose to use a \textit{Gaussian quasi-likelihood}
    \begin{equation}
        \sp_{\ban}^{\natural}(\bOmega, \Y^{(\ban)})=\prod_{(i,j)\in\bL_n}\e^{-(Y_{i,j}-\Omega_{i,j})^2}
    \end{equation}
instead of the Bernoulli likelihood $\sp_{\ban}$. This is because we want to remove the unpleasant feature of the Bernoulli likelihood that the KL divergence $\kl(\P^{(\ban)}_{\bOmega_0},\P^{(\ban)}_{\bOmega_1})$ can not be bounded by the $\scL^2$ error between $\bOmega_1$ and $\bOmega_0$ due to its boundary behavior. To be specific, suppose that there are only $2$ nodes, i.e., 1 connectivity. Then the KL divergence
    \begin{align*}
        \kl(\P^{(1)}_{\Omega_0},\P^{(1)}_{\Omega_1})
        &=\Omega_0\log\del{\frac{\Omega_0}{\Omega_1}}+(1-\Omega_0)\log\del{\frac{1-\Omega_0}{1-\Omega_1}}
    \end{align*}
is  larger than $C(\Omega_0-\Omega_1)^2$ for any large $C>0$ when $\Omega_1$ is very close to 0 or 1. Thus related Bayesian literature restricts the parameter space to be $[\varkappa,1-\varkappa]^{\bL_n}$ for some sufficiently small $\varkappa\in(0,1/2)$ to avoid such a problem \citepS{ghosh2020posterior, jiang2021consistent}. Here, we show that the use of the Gaussian  quasi-likelihood detours this issue without the parameter restriction. Similarly to ours, \citetS{gao2020general} showed theoretical optimality of a quasi-posterior distribution with the Gaussian quasi-likelihood but by using a completely different proof technique. Also, they did not study its variational approximation.

In the next lemma, we show that the Gaussian quasi-likelihood satisfies \cref{assume:quasi:quasi}.

\begin{lemma}
\label{lemma:sbm:quasi}
Let $\rho>0$. Then for any $\bOmega_0\in[0,1]^{\bL_n}$ and  $\bOmega_1\in[0,1]^{\bL_n}$,
    \begin{align}
        \P_{\bOmega_0}^{(\ban)}\sbr{\frac{\sp_{\ban}^{\natural}(\bOmega_1,\Y^{(\ban)})}{\sp_{\ban}^{\natural}(\bOmega_0,\Y^{(\ban)})}}
        & \le \e^{-\frac{1}{2}\ban\scd_{\ban,2}^2(\bOmega_0,\bOmega_1)}, 
    \label{eq:sbm_quasi_learning}\\
        \P_{\bOmega_0}^{(\ban)}\sbr[4]{\del{\frac{\sp_{\ban}^{\natural}(\bOmega_0,\Y^{(\ban)})}{\sp_{\ban}^{\natural}(\bOmega_1,\Y^{(\ban)})}}^{\rho}}
        &\le \e^{\frac{\rho(\rho+2)}{2}\ban\scd_{\ban,2}^2(\bOmega_0,\bOmega_1)}.      
    \label{eq:sbm_quasi_bound}
    \end{align}
\end{lemma}

\begin{proof}
    The proof is deferred to \cref{appen:quasi:proof:sbm:cond}.
\end{proof}

\subsubsection{Model, prior and variational posterior}

The stochastic block model assumes that the nodes are grouped into several \textit{communities} and the connectivity probabilities between nodes depend exclusively on their community membership. Recall the definitions $\cU_m:=\cbr[0]{\U\in [0,1]^{m\times m}:\U=\U^\top}$ and $\cZ_{n,m}:=\cbr[0]{\Z=(\z_1,\dots,\z_n)^\top\in\{0,1\}^{n\times m}: |\z_i|_1=1}$ given in \Cref{eq:def_sm,eq:def_znm}, respectively, and let $\Theta_{\ban,m}:=\cU_m\times \cZ_{n,m}$ for $m\in[n]$. The connectivity probability  array $\bOmega$ is assumed to be an output of a map $\sT:=\sT_{\ban}:\cup_{m=1}^n\Theta_{\ban,m}\to [0,1]^{\bL_n}$ defined by
    \begin{equation*}
        \sT(\U, \Z)=(\z_i^\top\U\z_j)_{(i,j)\in\bL_n}\in[0,1]^{\bL_n}
    \end{equation*}
for $\U\in\cU_m$ and $\Z=(\z_1,\dots,\z_n)^\top\in\cZ_{n,m}$ for every $m\in[n]$.

We consider multiple disjoint parameter spaces $\cbr[0]{\Theta_{\ban,m}}_{m\in\cM_{\ban}}$ with a model space $\cM_{\ban}\subset [n]$, a set of some numbers of communities. We impose a prior distribution given by
    \begin{equations}
    \Pi_{\ban}=&\sum_{m\in\cM_{\ban}}\alpha_{\ban,m}\Pi_{\ban,m},  
    \mbox{ with }
    \Pi_{\ban,m}:=\bigotimes_{(k,h)\in[m]^2:k\le h}\Unif(0,1)\times\bigotimes_{i=1}^{n}\Cat(m^{-1}\one_m),
    \end{equations}
where the uniform distributions are imposed on the lower triangular entries of $\U\in\cU_m$, and the entries above the diagonal are determined to be symmetric. For prior model probabilities  $(\alpha_{\ban,m})_{m\in\cM_{\ban}}$, we can choose, for example, a distribution with  exponential tail such that $\alpha_{\ban,m}\propto \e^{-\fa_0m}$ for $\fa_0>0$ or the uniform distribution such that $(\alpha_{\ban,m})_{m\in\cM_{\ban}}$ or any other mild one. Note that both the choices are independent to the true distribution and satisfy \cref{assume:quasi:model_prior}. 
For each model $m\in\cM_n$, we consider a variational family given by
    \begin{equation*}
        \cQ_{\ban,m}:=\cbr{\bigotimes_{(k,h)\in[m]^2:k\le h}\Unif(\psi_{1,k,h},\psi_{2,k,h})
        \times\bigotimes_{i=1}^{n}\Cat(\bnu_i):0\le\psi_{1,k,h}\le\psi_{2,k,h}\le1 ,\bnu_i\in\Delta_m}.
    \end{equation*}
Then due to \labelcref{eq:quasi_variational_posterior}, the adaptive variational quasi-posterior is given by $\hvQ_{\ban}^\natural=\sum_{m\in\cM_{\ban}}\hgamma_{\ban,m}^\natural\hvQ_{\ban,m}^\natural$  where  $\hvQ_{\ban,m}^\natural\in\argmin_{\vQ\in\cQ_{\ban,m}} \scE_n\del[1]{\vQ, \Pi_{\ban, m}, \sp_{\ban}^{\natural}}$ and $\hgamma_{\ban,m}^\natural\propto\alpha_{\ban,m}\exp\del[0]{-\scE_n\del[0]{\hvQ_{\ban,m}^\natural, \Pi_{\ban,m}, \sp_{\ban}^{\natural}}}$.

\subsubsection{Estimation of connectivity probability array}

The following oracle result provides a convenient theoretical tool to derive adaptive optimal contraction rates in related problems such as stochastic block model and graphon estimations, which we provide in \cref{appen:quasi:sbm}.

\begin{theorem}[Oracle contraction rate, connectivity probability array]
\label{thm:sbm:oracle}
Let $(\Lambda_{\ban}^\star)_{n\in\bN}$ be a sequence of sets such that $\Lambda_{\ban}^\star\subset[0,1]^{\bL_n}$.
Then if the prior model probabilities  $(\alpha_{\ban,m})_{m\in\cM_{\ban}}$ satisfy \cref{assume:quasi:model_prior},
        \begin{equation}
        \sup_{\bOmega^\star\in\Lambda_{\ban}^\star}\P_{\bOmega^\star}^{(\ban)}\sbr{\hvQ_{\ban}^
        \natural\del{\scd_{\ban,2}(\sT(\U,\Z),\bOmega^\star)\ge A_n\epsilon_n(\Lambda_{\ban}^\star)}}=\sco(1)
    \end{equation} 
for any diverging sequence $(A_n)_{n\in\bN}\to\infty$,
where
    \begin{equation}            \epsilon_n(\Lambda_{\ban}^\star):=\inf_{m\in\cM_{\ban}}\cbr{\sup_{\bOmega^*\in\Lambda_{\ban}^\star}\inf_{(\U,\Z)\in\Theta_{\ban,m}}\scd_{\ban,2}(\sT(\U,\Z),\bOmega^*)+\sqrt{\frac{m^2}{n^2}\log n+\frac{\log m}{n}}}.
    \end{equation}
\end{theorem}

\begin{proof}
    The proof is deferred to \cref{appen:quasi:proof:sbm:orcale}.
\end{proof}


In the next two subsections, we derive contraction rates of the adaptive variational quasi-posterior over the stochastic block models for specific examples.

\subsubsection{Estimation of stochastic block models}

The first example is about a situation where the sample is generated by the stochastic block model itself. The next corollary shows that the adaptive variational quasi-posterior is near optimal in this situation. This is a trivial consequence of \cref{thm:sbm:oracle}, and so we omit the proof.

\begin{corollary}[Stochastic block model]
Let $(m^\star_n)_{n\in\bN}$ be an arbitrary sequence such that $m^\star_n\in[n]$. Then with the model space $\cM_{\ban}=[n]$ and  prior model probabilities  $(\alpha_{\ban,m})_{m\in\cM_{\ban}}$ satisfying \cref{assume:quasi:model_prior}, we have
    \begin{equation}
        \sup_{(\U^\star,\Z^\star)\in\Theta^\star_{\ban,m^\star_n}}\P_{\sT(\U^\star,\Z^\star)}^{(\ban)}\sbr{\hvQ_{\ban}\del{\scd_{\ban,2}(\sT(\U,\Z),\sT(\U^\star,\Z^\star))\ge A_n  \epsilon_n^{\textup{SBM}}(m_n^\star)}}=\sco(1)
    \end{equation} 
for any diverging sequence $(A_n)_{n\in\bN}\to\infty$, where
    \begin{equation*}
        \epsilon_n^{\textup{SBM}}(m_n^\star):=\sqrt{\frac{(m^\star_n)^2}{n^2}\log n+\frac{\log m^\star_n}{n}}.
    \end{equation*}
\end{corollary}

The contraction rate $\epsilon_n^{\textup{SBM}}(m_n^\star)$ is $\log^{1/2}n$ times slower than the minimax optimal rate $\sqrt{(m_n^\star/n)^2 + \log(m_n^\star)/n}$ \citepS{gao2015rate}. Such a  $\log n$ sub-optimality is common in the Bayesian nonparametric literature, and we refer to \citetS{gao2016rate,hoffmann2015adaptive} for the related discussion. Similarly, in \citetS{ghosh2020posterior}, the same additional $\log^{1/2}n$ term appeared in the contraction rate of their original posterior distribution over stochastic block models. If we assume that $m_n^\star\lesssim \sqrt{n}$, then since $(m_n^\star)^2\log n\lesssim n\log (m_n^\star)$, our rate becomes the same as the optimal rate. \citetS{jiang2021consistent} used such an assumption and could avoid the  $\log n$ sub-optimality problem.

\subsubsection{Estimation of smooth graphons}

The second example is on  the estimation of smooth \textit{graphton}s. For a function $f:[0,1]^2\mapsto [0,1]$ and $n$-dimensional $[0,1]$-valued vector $\x:=(x_i)_{i\in[n]}\in[0,1]^n$, we define the corresponding graphon as
    \begin{equation*}
        \textup{Graphon}_n(f,\x):=\del[1]{f(x_i,x_j)}_{(i,j)\in\bL_n}\in[0,1]^{\bL_n}.
    \end{equation*}
For a given class $\widetilde\cF$ of some $[0,1]$-valued functions supported on $[0,1]^2$, we denote
    \begin{equation}
        \Lambda_{\ban}^{\textsc{G}}(\widetilde\cF):=\cbr{\textup{Graphon}_n(f,\x)\in[0,1]^{\bL_n}:f\in\widetilde\cF, \x\in[0,1]^n}.
    \end{equation}
We denote by $\cH^{\beta,2, F_0}_{[0,1]}$ the set of $[0,1]$-valued H\"older $\beta$-smooth functions supported on $[0,1]^2$. The next corollary provides the optimality of our adaptive variational quasi-posterior.

\begin{corollary}[Smooth graphon]
\label{col:graphon}
Let $\beta>0$ and $F_0>0$. Then with the model space $\cM_{\ban}=\cbr[1]{\lfloor n^{k/\log n}\rfloor:k\in\sbr[1]{0:\ceil{\log n/2}} }$ and prior model probabilities $(\alpha_{\ban,m})_{m\in\cM_{\ban}}$ satisfying \cref{assume:quasi:model_prior}
Then
    \begin{equation}
        \sup_{\bOmega^\star\in  \Lambda_{\ban}^{\textup{\textsc{G}}}\del[1]{\cH^{\beta,2, F_0}_{[0,1]}}}\P_{\bOmega^\star}^{(\ban)}\sbr[3]{\hvQ_{\ban}\del[2]{\scd_{\ban,2}(\sT(\U,\Z),\bOmega^\star)\ge A_nn^{-\del{\frac{\beta}{\beta+1}}\wedge\frac{1}{2}}\log^{\frac{1}{2}} n}}=\sco(1)
    \end{equation}
for any diverging sequence $(A_n)_{n\in\bN}\to\infty$.
\end{corollary}

\begin{proof}
The proof is deferred to \cref{appen:quasi:proof:graphon}.
\end{proof}

The minimax optimal rate for the smooth graphon estimation is $n^{-\beta/(\beta+1)}+\sqrt{\log n/n}$ \citepS[Theorem 2.4 of][]{gao2015rate}. Our contraction rate in \cref{col:graphon} is the exactly same as the optimal rate for sufficiently smooth graphon with $\beta\ge1$ but is slightly slower  with the extra $\log^{1/2} n$ term than the optimal rate when $\beta\in(0,1)$.

\subsection{Application to nonparametric regression with sub-Gaussian errors}
\label{appen:quasi:subgauss}


A real-valued random variable $Y\sim \P$ or its distribution $\P$ is said to be \textit{sub-Gaussian} with mean $\vartheta\in\R$ and variance proxy $\varsigma^2>0$, and  denoted by $\P\in\text{subG}(\vartheta,\varsigma^2)$, if $\P Y=\vartheta$ and $\P\e^{t(Y-\vartheta)}\le \e^{\varsigma^2t^2/2}$ for any $t\in\R$. We here consider a nonparametric regression experiment with sub-Gaussian errors, where each real-valued output $Y_i$ is a sub-Gaussian random variable with mean $f(\x_i)$ and variance proxy $\varsigma^2>0$ for some regression function $f\in\cF^d$ and given inputs $(\x_i)_{i\in[n]}\in([0,1]^d)^{\otimes n}$. In other words, we assume that the distribution of $(Y_i)_{i\in[n]}$ belongs to a set 
    \begin{equation*}
        \text{subG}_n(f^\star,\varsigma^2):=\text{subG}_n(f^\star,\varsigma^2;(\x_i)_{i\in[n]})
        :=\cbr{\bigotimes_{i=1}^n\P_i:\P_i\in\text{subG}(f^\star(\x_i),\varsigma^2)}.
    \end{equation*}
This is a more general situation than the regression experiment with Gaussian errors considered in \cref{sec:dnn:reg}. Also, we do not require that the errors are identically distributed.

A difficulty of Bayesian inference for the regression experiment with sub-Gaussian errors is that we cannot construct a likelihood function of an explicit form. Fortunately, we can obtain optimal contraction of the adaptive variational posterior when we use a Gaussian quasi-likelihood, which is given as
    \begin{equation}
        \sp_n^{\natural,\kappa}(f, \Y^{(n)})=\prod_{i=1}^n\exp\del[2]{-\frac{\kappa}{2}(Y_i-f(\x_i))^2}
        \mbox{ for }f\in\cF^d,
    \end{equation}
where we call the constant $\kappa>0$ the \textit{learning rate} following the related literature \citepS{grunwald2011safe, syring2020gibbs, bhattacharya2020gibbs}, because it determines how much the quasi-posterior learns from the current sample. The next lemma  proves that if the learning rate is sufficiently small such that $\kappa\in(0,1/\varsigma^2)$, \cref{assume:quasi:quasi} is met.

\begin{lemma}
\label{lemma:subgauss}
Let $\varsigma>0$, $\kappa>0$ and $\rho>0$. Then for any $f,f^\star\in\cF^d$ and $\P_{\star}^{(n)}\in\textup{subG}_n(f^\star,\varsigma)$,
    \begin{align}
        \P_{\star}^{(n)}\sbr{\frac{\sp_n^{\natural,\kappa}(f, \Y^{(n)})}{\sp_n^{\natural,\kappa}(f^\star, \Y^{(n)})}}
        &\le \e^{-\frac{1}{2}\kappa(1-\kappa\varsigma^2)n\nnorm{f-f^\star}^2}, \label{eq:subg_quasi_learning}\\
        \P_{\star}^{(n)}\sbr[4]{\del{\frac{\sp_n^{\natural,\kappa}(f^\star,\Y^{(n)})}{\sp_n^{\natural,\kappa}(f,\Y^{(n)})}}^{\rho}}
        &\le \e^{\frac{1}{2}\rho\kappa(1+\rho\kappa\varsigma^2)n\nnorm{f-f^\star}^2}.  \label{eq:subg_quasi_bound}
    \end{align}
\end{lemma}

\begin{proof}
The proof is deferred to \cref{appen:quasi:proof:subgauss}.
\end{proof}

By \cref{thm:quasi:conv} together with \cref{lemma:subgauss}, an adaptive optimality result can be obtained. For example, it is possible for the adaptive variational deep learning. Consider the adaptive variational quasi-posterior over neural networks with the Gaussian quasi-likelihood, 
    \begin{align*}
        \hvQ_n^{\natural,\kappa}&=\sum_{(K,M)\in\cM_n}\hgamma_{n,(K,M)}^{\natural,\kappa}\hvQ_{n,(K,M)}^{\natural,\kappa}\\
        &\mbox{ with }\hvQ_{n,(K,M)}^{\natural,\kappa}\in\argmin_{\vQ\in\cQ_{n,(K,M)}}\scE_n\del[0]{\vQ, \Pi_{n,(K,M)}, \sp_n^{\natural,\kappa}}\\
        &\mbox{ and }\hgamma_{n,(K,M)}^{\natural,\kappa}\propto \alpha_{n,(K,M)} \exp\del[0]{-\scE_n\del[0]{\hvQ_{n,(K,M)}^{\natural,\kappa}, \Pi_{n,(K,M)},\sp_n^{\natural,\kappa}}},
    \end{align*}
where the prior and variational families are the same as \labelcref{eq:dnn_prior} and \labelcref{eq:dnn_variational_family}, respectively. Then it can achieve the same oracle contraction rate as that for the Gaussian regression.

\begin{corollary}[Oracle contraction rate, sub-Gaussian regression]
Suppose that the same assumptions of \cref{thm:dnn:oracle} hold. Then if we use the learning rate $\kappa\in(0,1/\varsigma^2)$, 
        \begin{equation}
        \sup_{f^\star\in\cF^\star}\sup_{\P_{\star}^{(n)}\in\textup{subG}_n(f^\star,\varsigma^2)}\P_{\star}^{(n)}\sbr{\hvQ_n^{\natural,\kappa}\del{\nnorm{\net(\btheta)- f^\star}\ge A_n\epsilon_n(\cF^\star)}}=\sco(1)
    \end{equation} 
for any diverging sequence $(A_n)_{n\in\bN}\to\infty$, where $\epsilon_n(\cF^\star)$ is defined in \labelcref{eq:dnn_oracle}.
\end{corollary}

\section{Proofs of general results in \cref{sec:method,sec:theory,sec:regularization} and \cref{appen:theory:testing}}

\subsection{Proof for \cref{sec:method}}

\subsubsection{Proof of \cref{thm:compute}}

\begin{proof}
For any $\vQ_n:=\sum_{m\in\cM_n}\gamma_{n,m}\vQ_{n,m}\in\cQ_n$ with $\vQ_{n,m}\in\cQ_{n,m}$, we have
    \begin{align*}
        \kl(\vQ_n, \Pi_n)
        &=\sum_{m\in\cM_n}\gamma_{n,m}\int_{\Theta_{n,m}} \log\del{\frac{\sum_{m\in\cM_n}\gamma_{n,m}\d\vQ_{n,m}(\btheta)}{\sum_{m\in\cM_n}\alpha_{n,m}\d\Pi_{n,m}(\btheta)}}\d\vQ_{n,m}(\btheta)\\
        &=\sum_{m\in\cM_n}\gamma_{n,m}\int_{\Theta_{n,m}} \log\del{\frac{\gamma_{n,m}\d\vQ_{n,m}(\btheta)}{\alpha_{n,m}\d\Pi_{n,m}(\btheta)}}\d\vQ_{n,m}(\btheta)\\
        &=\sum_{m\in\cM_n}\gamma_{n,m}\log\del{\frac{\gamma_{n,m}}{\alpha_{n,m}}} +
        \sum_{m\in\cM_n}\gamma_{n,m} \int_{\Theta_{n,m}} \log\del{\frac{\d\vQ_{n,m}(\btheta)}{\d\Pi_{n,m}(\btheta)}}\d\vQ_{n,m}(\btheta)\\
        &=\kl(\bgamma_n, \balpha_n)+ \sum_{m\in\cM_n}\gamma_{n,m} \kl(\vQ_{n,m}, \Pi_{n,m}),
    \end{align*}
where the second inequality holds because $\{\Theta_{n,m}\}_{m\in\cM_n}$ are disjoint. Here we denote $\bgamma_n:=(\gamma_{n,m})_{m\in\cM_n}$. Hence, we have the following identities 
    \begin{equations}
    \label{eq:elbo_decompose}
       \scE_n(\vQ_n,\Pi_n,\sp_n):=&-\int_{\Theta_{n,\cM_n}}\log\del[1]{\sp_n(\sT(\btheta), \Y^{(n)})}\d\vQ_n(\btheta)+\kl(\vQ_n, \Pi_n)\\
        =&-\sum_{m\in\cM_n}\gamma_{n,m}\int_{\Theta_{n,m}} \log\del[1]{\sp_n(\sT(\btheta), \Y^{(n)})}\d \vQ_{n,m}(\btheta)\\
        &\qquad\qquad+\kl(\bgamma_n, \balpha_n)+ \sum_{m\in\cM_n}\gamma_{n,m} \kl(\vQ_{n,m}, \Pi_{n,m})\\
        =&\kl(\bgamma_n, \balpha_n) +\sum_{m\in\cM_n}\gamma_{n,m}\scE_n(\vQ_{n,m},\Pi_{n,m},\sp_n)
    \end{equations}
for any $\vQ_n\in\cQ_n.$ This implies that minimizing the negative ELBO $\scE_n(\vQ_n,\Pi_n,\sp_n)$ with respect to $\vQ_{n,m}$ is independent to $\bgamma_n$ and the other variational distributions $\vQ_{n,m'}$, $m'\neq m$. This implies \labelcref{eq:vpost_comp}. \Cref{eq:vpost_mprob} follows from
    \begin{align*}
        \scE_n\del[3]{\sum_{m\in\cM_n}\gamma_{n,m}\hvQ_{n,m},\Pi,\sp_n}
        &=\sum_{m\in\cM_n}\gamma_{n,m}\log \del{\frac{\gamma_{n,m}}{\alpha_{n,m}\e^{-\scE_n(\hvQ_{n,m},\Pi_{n,m},\sp_n)}}}
        \end{align*}
is minimized at $\hgamma_{n,m}\propto \alpha_{n,m}\e^{-\scE_n(\hvQ_{n,m},\Pi_{n,m},\sp_n)}$.
\end{proof}

\subsection{Proofs for \cref{sec:theory:rate}}

\subsubsection{Proof of \cref{lemma:model_prob_choice}}

\begin{proof}
Note that
    \begin{align*}
        Z_{\alpha,n}=\sum_{m\in\cM_n}\e^{-\fa_0n\zeta_{n,m}^2}
        &\ge   \e^{-\fa_0n\inf_{m\in\cM_n}\zeta_{n,m}^2}
        \ge 
        \e^{-\fa_0n\epsilon_n^2}.
    \end{align*}
Thus, we have
    \begin{align*}
         \sum_{m\in\cM_n:\zeta_{n,m}\ge H\epsilon_n}\alpha_{n,m}
         &\le \frac{1}{Z_{\alpha,n}}\sum_{m\in\cM_n:\zeta_{n,m}\ge H\epsilon_n}\e^{-\fa_0n\zeta_{n,m}^2}\\
         &\le \e^{\fa_0n\epsilon_n^2}|\cM_n|\e^{-\fa_0H^2n\epsilon_n^2}\\
         &\le \e^{-(\fa_0H^2-(\fa_0+\fc_3))n\epsilon_n^2},
    \end{align*}
which implies \labelcref{eq:assume:aggregation:largemodels} with $H_0>((\fa_0+\fc_3)/\fa_0)^{1/2}$. 
Next, since $ Z_{\alpha,n}=\sum_{m\in\cM_n}\e^{-\fa_0n\zeta_{n,m}^2}\le |\cM_n|\le \e^{\fc_3n\epsilon_n^2},$ we have
    \begin{align*}
        \alpha_{n,m_n^*}\ge\e^{-\fc_3n\epsilon_n^2}\e^{-\fa_0n\zeta_{n,m_n^*}^2}
        \ge \e^{-(\fc_3+\fa_0(1+\fc_5)^2)n\epsilon_n^2},
    \end{align*}
which proves \labelcref{eq:assume:aggregation:best}.
\end{proof}

\subsubsection{Proof of \cref{thm:vgap}}

\begin{proof}
We first prove the assertion \labelcref{eq:vgap_ineq}. For any distribution $\vQ\in\cQ_n$, we have the following series of equalities
    \begin{align*}
         \P_{\blambda^\star}^{(n)}& \sbr{\kl(\vQ, \Pi_n(\cdot|\Y^{(n)})) } \\
         &= \int \P_{\blambda^\star}^{(n)}\sbr{\log\del{\frac{\sp_{n,\Pi_n}(\Y^{(n)})\d\vQ(\btheta)}{\sp_n(\sT(\btheta), \Y^{(n)})\d\Pi_n(\btheta)}}} \d\vQ(\btheta) \\
         &=\kl(\vQ, \Pi_n) + \int\P_{\blambda^\star}^{(n)}\sbr{\log\del{\frac{\sp_n(\blambda^\star, \Y^{(n)})}{\sp_n(\sT(\btheta), \Y^{(n)})}}}\d\vQ(\btheta)  +\P_{\blambda^\star}^{(n)}\sbr{\log\del{\frac{\sp_{n,\Pi_n}(\Y^{(n)})}{\sp_n(\blambda^\star, \Y^{(n)})}}}\\
         &= \kl(\vQ,\Pi_n)+ \vQ\sbr{\kl\del{\P_{\blambda^\star}^{(n)},\P_{\sT(\btheta)}^{(n)}}}
         +\P_{\blambda^\star}^{(n)}\sbr{\log\del{\frac{\sp_{n,\Pi_n}(\Y^{(n)})}{\sp_n(\blambda^\star, \Y^{(n)})}}}.
    \end{align*}
But by Jensen's inequality, we have
    \begin{align*}
     \P_{\blambda^\star}^{(n)}\sbr{\log\del{\frac{\sp_{n,\Pi_n}(\Y^{(n)})}{\sp_n(\blambda^\star, \Y^{(n)})}}}
     \le \log \del{ \P_{\blambda^\star}^{(n)}\sbr{\frac{\sp_{n,\Pi_n}(\Y^{(n)})}{\sp_n(\blambda^\star, \Y^{(n)})}}}=0.
    \end{align*}
Thus, by the definition of $\hvQ_n$,
    \begin{align*}
         \P_{\blambda^\star}^{(n)} \sbr{\kl(\hvQ_n, \Pi_n(\cdot|\Y^{(n)})) } 
         &= \P_{\blambda^\star}^{(n)}\sbr{\inf_{\vQ\in\cQ_n}\kl(\vQ, \Pi_n(\cdot|\Y^{(n)})) } \\
         &\le \inf_{\vQ\in\cQ_n}\P_{\blambda^\star}^{(n)}\sbr{\kl(\vQ, \Pi_n(\cdot|\Y^{(n)})) }\\
         &\le \inf_{\vQ\in\cQ_n}\cbr{\kl(\vQ,\Pi_n)+ \vQ\sbr{\kl\del{\P_{\blambda^\star}^{(n)},\P_{\sT(\btheta)}^{(n)}}}},
    \end{align*}
which proves the first inequality in \labelcref{eq:vgap_ineq}. The second inequality follows from
    \begin{align*}
        \kl(\vQ_m,\Pi_n)=-\log(\alpha_{n,m}) + \kl(\vQ_m,\Pi_{n,m})
    \end{align*}
for any $\vQ_m\in\cQ_{n,m}$ and $m\in\cM_n$, combining with the fact that $\cup_{m\in\cM_n}\cQ_{n,m}\subset\cQ_{n}$.

Now suppose that \cref{assume:individual:variational} holds. Then by  \labelcref{eq:vgap_ineq} we have proven above, we have for any $m\in\cM_n$,
\begin{align*}
        \P_{\blambda^\star}^{(n)}&\sbr{\kl(\hvQ_n, \Pi_n(\cdot|\Y^{(n)}) )}\\
        &\le -\log\alpha_{n,m}
         + \inf_{\vQ_{m}\in\cQ_{n,m}}\sbr{\kl\del[1]{\vQ_{m}, \Pi_{n,m}} 
            + \vQ_m\sbr{\kl\del[1]{\P_{\blambda^\star}^{(n)}, \P_{\sT(\btheta)}^{(n)}}} }\\
        &\le -\log\alpha_{n,m} + \fc_2n(\eta_{n,m}+\zeta_{n,m})^2.
    \end{align*}
Thus, for $m_n^*$ satisfying $\eta_{n,m_n^*}+\zeta_{n,m_n^*}\le (1+\fc_5)\epsilon_n$ and \labelcref{eq:assume:aggregation:best}, we finally have
\begin{align*}
        \P_{\blambda^\star}^{(n)}\sbr{\kl(\hvQ_n, \Pi_n(\cdot|\Y^{(n)}))}
        &\le \fc_6n\epsilon_n^2 + \fc_2n(\eta_{n,m_n^*}+\zeta_{n,m_n^*})^2\\
        &\le (\fc_6 + \fc_2(1+\fc_5)^2)n\epsilon_n^2
    \end{align*}
which proves the second assertion of \labelcref{eq:theory:vgap}.
\end{proof}

\subsection{Proofs of \cref{thm:conv,thm:over,thm:under}}

In this subsection, we give the proofs of the contraction properties of the adaptive variational posterior. We divide the proofs into several steps. 

\subsubsection{Change-of-measure lemma}
\label{appen:proof:variational_ineq}

The following lemma enables us to control the variational posterior probability of any measurable event using the corresponding probability under the original posterior. It has been used in the related literature, e.g., Lemma B.1 of \citetS{zhang2020convergence} and Theorem 5 of \citetS{ray2021variational}. We further employ this lemma to find a suitable lower bound of the denominator of the original posterior. For details, see \cref{appen:proof:denom}.

\begin{lemma}
\label{lemma:variational_ineq}
Let $\Theta$ be a measurable space. Then for any two distributions $\vQ_0,\Pi_0\in\cP(\Theta)$ and any measurable function $\sF:\Theta\mapsto \R$,
    \begin{equation}
    \label{eq:kl_ineq_fn}
        \vQ_0\sbr{\sF}\le \kl(\vQ_0, \Pi_0) +\log\del[1]{\Pi_0\sbr[1]{\e^{\sF}}}.
    \end{equation}
In particular, for any measurable subset $\Theta'\subset\Theta$ and positive constant $\upsilon>0$,
    \begin{equation}
    \label{eq:kl_ineq_prob}
        \vQ_0(\Theta')\le \frac{1}{\upsilon}\cbr{\kl(\vQ_0, \Pi_0)+ \e^\upsilon\Pi_0(\Theta')}.
    \end{equation}
\end{lemma}

\begin{proof}
If $\vQ_0$ is not absolutely continuous with respect to $\Pi_0$ then $\kl(\vQ_0,\Pi_0)=\infty$, so the result trivially holds. Now assume otherwise. Recall the following well-known duality formula (e.g., Lemma 2.2 of \citetS{alquier2020concentration}),
    \begin{align*}
        \log\del[1]{\Pi_0\sbr[1]{\e^{\sF}}}=\sup_{\vQ'\ll \Pi_0}\sbr{\vQ'\sbr{\sF}-\kl(\vQ', \Pi_0)},
    \end{align*}
from which \labelcref{eq:kl_ineq_fn} directly follows. For the proof of \labelcref{eq:kl_ineq_prob}, we let $\sF(\btheta):=\upsilon\ind(\btheta\in\Theta')$. Then we have
    \begin{align*}
        \e^\upsilon\Pi_0(\Theta')
        \ge \log(1+ \e^\upsilon\Pi_0(\Theta'))
        \ge\log\del[2]{\int\e^{\sF(\btheta)}\d\Pi_0(\btheta)},
    \end{align*}
which completes the proof.
\end{proof}

\subsubsection{Lower bound of the denominator of the original posterior}
\label{appen:proof:denom}

In this subsection, we provide a high probability lower bound of the the denominator $\int \frac{\sp_n(\sT(\btheta), \Y^{(n)})}{\sp_n(\blambda^\star, \Y^{(n)})}\d\Pi(\btheta)$ of the original posterior. Given a parameter $\blambda^\star\in\Lambda_n$, a positive value $T>0$ and distributions $\bXi,\Pi\in\cP(\Theta_{n,\cM_n})$, we define the event
    \begin{align}
    \label{eq:denom:event}
        \bA_n(T,\bXi,\Pi,\blambda^\star):=\cbr{\Y^{(n)}\in\bY_n:\int \frac{\sp_n(\sT(\btheta), \Y^{(n)})}{\sp_n(\blambda^\star, \Y^{(n)})}\d\Pi(\btheta) \ge\exp\del[1]{-T-\kl(\vQ, \Pi)}}.
    \end{align}
    
\begin{lemma}
\label{lemma:denominator}
For the event $\bA_n(T,\bXi,\Pi,\blambda^\star)$ defined in \labelcref{eq:denom:event}, we have
    \begin{equation}
    \label{eq:denom_bound}
       \P^{(n)}_{\blambda^\star}\del{ \bA_n(T,\bXi,\Pi,\blambda^\star)}
        \ge 1- \frac{1}{T}\del{2\vQ\sbr{\kl\del[1]{\P_{\blambda^\star}^{(n)}, \P_{\sT(\btheta)}^{(n)}}}+1}.
    \end{equation}
\end{lemma}

\begin{proof}
We start with applying \labelcref{eq:kl_ineq_fn} in \cref{lemma:variational_ineq}  with $\sF=\log\frac{\sp_n(\sT(\btheta),\Y^{(n)})}{\sp_n(\blambda^\star, \Y^{(n)})}$, $\vQ_0=\vQ$ and $\Pi_0=\Pi$ to obtain
    \begin{align*}
        \log \int\frac{\sp_n(\sT(\btheta),\Y^{(n)})}{\sp_n(\blambda^\star, \Y^{(n)})}\d\Pi(\btheta)
        &\ge \int \log\del{\frac{\sp_n(\sT(\btheta), \Y^{(n)})}{\sp_n(\blambda^\star, \Y^{(n)})}}\d\vQ(\d\btheta)-\kl(\vQ,\Pi).
    \end{align*}
Hence, we have
    \begin{align*}
        \P^{(n)}_{\blambda^\star}\del{ \bA_n^\complement(t,\bXi,\Pi,\blambda^\star)}
        &= \P^{(n)}_{\blambda^\star}\del{\log\int \frac{\sp_n(\sT(\btheta), \Y^{(n)})}{\sp_n(\blambda^\star, \Y^{(n)})}\d\Pi(\btheta)<-T-\kl(\vQ, \Pi)}\\
        &\le \P^{(n)}_{\blambda^\star}\del{-\int \log\del{\frac{\sp_n(\blambda^\star, \Y^{(n)})}{\sp_n(\sT(\btheta), \Y^{(n)})}}\d\vQ(\btheta)\le -T}\\
        &\le \P^{(n)}_{\blambda^\star}\del{\int 0\vee\log\del{\frac{\sp_n(\blambda^\star, \Y^{(n)})}{\sp_n(\sT(\btheta), \Y^{(n)})}}\d\vQ(\btheta)\ge T}\\
        &\le \frac{1}{T} \P^{(n)}_{\blambda^\star}\sbr[3]{\int0\vee\log\del{\frac{\sp_n(\blambda^\star, \Y^{(n)})}{\sp_n(\sT(\btheta), \Y^{(n)})}}\d\vQ(\btheta)}\\
        &\le\frac{1}{T} \vQ\sbr{\kl\del[1]{\P_{\blambda^\star}^{(n)}, \P_{\sT(\btheta)}^{(n)}}+\sqrt{\frac{1}{2}\kl\del[1]{\P_{\blambda^\star}^{(n)}, \P_{\sT(\btheta)}^{(n)}}}}\\
        &\le \frac{1}{T}\del[2]{2\vQ\sbr{\kl\del[1]{\P_{\blambda^\star}^{(n)}, \P_{\sT(\btheta)}^{(n)}}}+1}.
    \end{align*}
Here we use Markov's inequality for the fourth line and use Fubini's theorem and Lemma B.13 of \citetS{ghosal2017fundamentals} for the fifth line. For the last line, we use a simple inequality $z+\sqrt{z/2}\le z+ (z\vee1)\le 2z+1$ for $z\ge0$.  We complete the proof.
\end{proof}

\subsubsection{Contraction of the original posterior}
\label{appen:proof:post_conv}

In this subsection, we derive the adaptive contraction of the original posterior distribution as an intermediate result.

\begin{theorem}[Contraction of original posterior]
\label{thm:post_conv}
Under \cref{assume:individual,assume:aggregation}, there exist absolute constants $A_0>0$ and $\fc_0'>0$ such that
        \begin{equations}
        \label{eq:regular_conv}
        \sup_{\blambda^\star\in\Lambda_n^\star}\P_{\blambda^\star}^{(n)}\sbr{\Pi_n\del{\scd_n(\sT(\btheta),\blambda^\star)\ge A_0\baepsilon|\Y^{(n)}}\ind\del{\bA_n(n\baepsilon^2,\vQ,\Pi_n,\blambda^\star)}}
        \le 3\e^{\kl(\vQ,\Pi_n)-\fc_0'n\baepsilon^2}
    \end{equations} 
for any $\vQ\in\cP(\Theta_{n,\cM_n})$ and $\baepsilon\ge \epsilon_n$, where $\bA_n(n\baepsilon^2,\vQ,\Pi_n,\blambda^\star)$ is the event defined using \labelcref{eq:denom:event} with $T=n\baepsilon^2$.
\end{theorem}

\def\likratio{\frac{\sp_n(\sT(\btheta),\Y^{(n)})}{\sp_n(\blambda^\star,\Y^{(n)})}}

\begin{proof}
Fix $\blambda^\star\in\Lambda_n^\star$, $\vQ\in\cP(\Theta_{n,\cM_n})$ and $\baepsilon\ge \epsilon_n$. Let
    \begin{equation*}
        \cK_n(A_0\baepsilon):=\cK_n(A_0\baepsilon;\blambda^\star):=\cbr{\btheta\in\Theta_{n,\cM_n}:\scd_n(\sT(\btheta),\blambda^\star)\ge A_0\baepsilon}
    \end{equation*}
and $\bA_n:=\bA_n(n\baepsilon^2,\vQ,\Pi_n,\blambda^\star)$ for ease of notation. We note first that, on $\bA_n$, by definition, 
    \begin{align}
    \label{eq:denom:lowerbdd}
        D_n:=\int \frac{\sp_n(\sT(\btheta), \Y^{(n)})}{\sp_n(\blambda^\star, \Y^{(n)})}\d\Pi_n(\btheta)
        \ge \e^{-n\baepsilon^2-\kl(\vQ,\Pi_n)}.
    \end{align}
Next, we divide the model space $\cM_n$ as $\cM_n=\cM_n^{+}\cup\cM_n^{-}$ with
    \begin{align*}
        \cM_n^{+}&:=\cbr{m\in\cM_n:\zeta_{n,m}\ge H_1\baepsilon},\\
         \cM_n^{-}&:=\cbr{m\in\cM_n:\zeta_{n,m}<H_1\baepsilon},
    \end{align*}
where $H_1>H_0$ is the constant which will be specified later. We denote $\Theta_{n,\cM_n^{+}}:=\cup_{m\in\cM_n^{+}}\Theta_{n,m}$ and $\Theta_{n,\cM_n^{-}}:=\cup_{m\in\cM_n^{-}}\Theta_{n,m}$. We will bound each term of the second line of the next display separately:
    \begin{align}
        \P_{\blambda^\star}^{(n)}&\sbr{\Pi_n\del[1]{\cK_n(A_0\baepsilon)|\Y^{(n)}}\ind(\bA_n)}\nonumber\\
         &  \le  \P_{\blambda^\star}^{(n)}\sbr{\Pi_n\del[1]{\Theta_{n,\cM_n^{+}}|\Y^{(n)}}\ind(\bA_n)}
        +\P_{\blambda^\star}^{(n)}\sbr{\Pi_n\del[1]{\cK_n(A_0\baepsilon)\cap\Theta_{n,\cM_n^{-}}|\Y^{(n)}}\ind(\bA_n)}  
            \label{eq:pp_decomp}.
    \end{align}
For the first term of \labelcref{eq:pp_decomp}, by \cref{eq:denom:lowerbdd}, we have
    \begin{equations}
    \label{eq:post_bound_large_models}
        \P_{\blambda^\star}^{(n)}\sbr{\Pi_n\del[1]{\Theta_{n,\cM_n^{+}}|\Y^{(n)}}\ind(\bA_n)}
        &=\P_{\blambda^\star}^{(n)}\sbr{\frac{\int_{\Theta_{n,\cM_n^{+}}} \likratio\d\Pi_n(\btheta)}{\int \likratio \d\Pi_n(\btheta)}\ind(\bA_n)}\\
        &\le \e^{n\baepsilon^2+\kl(\vQ,\Pi_n)}\Pi_n(\Theta_{n,\cM_n^{+}})\\
        &=  \e^{n\baepsilon^2+\kl(\vQ,\Pi_n)}\sum_{m\in\cM_n^{+}}\alpha_{n,m}\\
        &\le  \e^{n\baepsilon^2+\kl(\vQ,\Pi_n)}\e^{-\fc_4n(H_1\baepsilon)^2},
    \end{equations}
where the last line follows from \cref{assume:aggregation:model_prior_pen} since $H_1\baepsilon>H_0\epsilon_n$ by assumption.  Thus, taking $H_1$ such that $H_1^2>1/\fc_4$, the preceding display is further bounded by $\e^{\kl(\vQ,\Pi_n)-\fc_1'n\baepsilon^2}$ for some absolute constant $\fc_1'>0$. We next investigate the second term of  \labelcref{eq:pp_decomp}. For any measurable function $\varphi:\bY_n\mapsto[0,1]$, we have 
    \begin{align*}
         &\P_{\blambda^\star}^{(n)}\sbr{\Pi_n\del[1]{\cK_n(A_0\baepsilon)\cap\Theta_{n,\cM_n^{-}}|\Y^{(n)}}\ind(\bA_n)}\\
         &\le\P_{\blambda^\star}^{(n)}[\varphi \ind(\bA_n)] + \P_{\blambda^\star}^{(n)}\sbr{\Pi_n\del[1]{\cK_n(A_0\baepsilon)\cap\Theta_{n,\cM_n^{-}}|\Y^{(n)}}(1-\varphi)\ind(\bA_n)}.
    \end{align*}
To bound the two terms in the preceding display exponentially, we now aggregate the tests $\varphi_{n,m}$ over $\cM_n^{-}$ as $\varphi_n:=\max_{m\in\cM_n^{-}}\varphi_{n,m}$. If we assume $A_0\ge J_0H_1$ so that $(A_0/J_0)\baepsilon\ge H_1\epsilon_n>\max_{m\in\cM_n^{-}}\zeta_{n,m}$, then from \cref{assume:individual:testing}, we get exponentially decaying type-\rom{1} and -\rom{2} bounds as
    \begin{equations}
    \label{eq:aggtest_type1}
       \P_{\blambda^\star}^{(n)}[\varphi_n\ind(\bA_n)] 
       \le\P_{\blambda^\star}^{(n)}[\varphi_n] 
        &\le \sum_{m\in\cM_n^{-}}\P_{\blambda^\star}^{(n)}[\varphi_{n,m}]\\
        &\le |\cM_n|\e^{-\fc_1n(A_0/J_0)^2\baepsilon^2}\\
        &\le \e^{\fc_3n\epsilon_n^2-\fc_1n(A_0/J_0)^2\baepsilon^2}
    \end{equations}
and 
    \begin{equations}
     \label{eq:aggtest_type2}
       \sup_{\btheta\in\cK_n(A_0\baepsilon)\cap\Theta_{n,\cM_n^{-}}}\P_{\blambda}^{(n)}[1-\varphi_n]
        &\le \max_{m\in\cM_n^{-}} \sup_{\btheta\in\Theta_{n,m}:\scd_n(\sT(\btheta), \blambda^\star)>J_0(A_0/J_0)\baepsilon}\P_{\blambda^\star}^{(n)}[1-\varphi_{n,m}]\\
        &\le \e^{-\fc_1n(A_0/J_0)^2\baepsilon^2}.
    \end{equations}
Here the last inequality of \labelcref{eq:aggtest_type1} follows from \cref{assume:aggregation:model_set}. Moreover, by \labelcref{eq:denom:lowerbdd,eq:aggtest_type2}, we have
    \begin{equations}
    \label{eq:posterior_type2_error}
          \P_{\blambda^\star}^{(n)}&\sbr{\Pi_n\del[1]{\cK_n(A_0\baepsilon)\cap\Theta_{n,\cM_n^{-}}|\Y^{(n)}}(1-\varphi_n)\ind(\bA_n)}\\
          &=\int_{\cK_n(A_0\baepsilon)\cap\Theta_{n,\cM_n^-}}\P_{\blambda^\star}^{(n)}\sbr{\frac{\sp_n(\sT(\btheta),\Y^{(n)})}{\sp_n(\blambda^\star,\Y^{(n)})}(1-\varphi_n)\frac{1}{D_n}\ind(\bA_n)}\d\Pi_n(\btheta)  \\
        &\le  \e^{n\baepsilon^2+\kl(\vQ,\Pi_n)}
         \sup_{\btheta\in\Theta_{n,m}:\scd_n(\sT(\btheta),\blambda^\star)\ge A_0\baepsilon} \P_{\sT(\btheta)}^{(n)}[1-\varphi_n]\\
        &\le \e^{\kl(\vQ,\Pi_n)-(\fc_1(A_0/J_0)^2-1)n\baepsilon^2},
    \end{equations}
where we use Fubini's theorem in the equality. Thus if we choose $A_0$ such that $A_0^2>J_0^2(H_1^2\vee \fc_3^2\vee \fc_1^{-1})$,  the second term of  \labelcref{eq:pp_decomp} is further bounded by $\e^{-\fc_{2}'n\baepsilon^2}$ for some absolute constant $\fc_{2}'>0$, which completes the proof.
\end{proof}

\subsubsection{Proof of \cref{thm:conv}}

\begin{proof}
Let $\cK_n:=\cbr{\btheta\in\Theta_{n,\cM_n}:\scd_n(\sT(\btheta),\blambda^\star)\ge A_n\epsilon_{n}}$ be the event of interest.
Moreover, let $m_n^*\in\cM_n$ be a model index  $\eta_{n,m_n^*}+\zeta_{n,m_n^*}\le (1+\fc_5)\epsilon_n$ and $\alpha_{n,m_n^*} \ge\exp\del[1]{-\fc_6n\epsilon_n^2}$ and let $\vQ_{n,m_n^*}^*\in\cQ_{n,m_n^*}$ be a distribution that satisfies \cref{assume:individual:variational}. We set
    \begin{align*}
        \bA_n:=\bA_n\del{(A_n/A_0)^2n\epsilon_n^2,\vQ_{n,m_n^*}^*,\Pi_n,\blambda^\star},
    \end{align*}
that is, $\bA_n$ is the event defined using \labelcref{eq:denom:event} with $T=(A_n/A_0)^2n\epsilon_n^2$ and $\vQ=\vQ_{n,m_n^*}^*$, where $A_0$ is the constant appearing in \cref{thm:post_conv}. We start with 
    \begin{align*}
        \P_{\blambda^\star}^{(n)}[\hvQ_n(\cK_n)]
        \le \P_{\blambda^\star}^{(n)}[\hvQ_n(\cK_n)\ind(\bA_n)] 
        +  \P_{\blambda^\star}^{(n)}(\bA_n^\complement).
    \end{align*}
By \cref{lemma:denominator}, we have
    \begin{align*}
         \P_{\blambda^\star}^{(n)}(\bA_n^\complement)\le \frac{1}{(A_n/A_0)^2n\epsilon_n^2}(2n((1+\fc_5)\epsilon_n)^2+1)=\sco(1)
    \end{align*}
since $A_n\to\infty$. We now focus on the  term $\P_{\blambda^\star}^{(n)}[\hvQ_n(\cK_n)\ind(\bA_n)].$ For this, we employ \cref{eq:kl_ineq_prob} in  \cref{lemma:variational_ineq} with $\vQ_0=\hvQ_n$, $\Pi_0=\Pi_n(\cdot|\Y^{(n)})$ and $\upsilon=\upsilon_n:=A_n n\epsilon_{n}^2$ to obtain
    \begin{equation*}
        \P_{\blambda^\star}^{(n)}[\hvQ_n(\cK_n)\ind(\bA_n)]
        \le \frac{1}{\upsilon_n}\P_{\blambda^\star}^{(n)}[\kl(\hvQ_n, \Pi_n(\cdot|\Y^{(n)}) )]+ \frac{1}{\upsilon_n}\e^{\upsilon_n}\P_{\blambda^\star}^{(n)}[\Pi_n(\cK_n|\Y^{(n)})\ind(\bA_n)].
    \end{equation*}
The first term of the right-hand side goes to zero since the variational approximation gap is bounded above by $\fc_1'n\epsilon_n^2$ for some $\fc_1'>0$ by \cref{thm:vgap}. On the other hand, by \cref{thm:post_conv} together with  \cref{assume:individual:variational} and the assumption \labelcref{eq:assume:aggregation:best},  we have
    \begin{align*}
        \e^{\upsilon_n}\P_{\blambda^\star}^{(n)}[\Pi_n(\cK_n|\Y^{(n)})\ind(\bA_n)]
        &\le 3 \exp\del{v_n + \kl(\vQ_{n,m_n^*}^*,\Pi_n)-\fc_0'A_n^2n\epsilon_n^2}\\
        &= 3 \exp\del{v_n -\log \alpha_{n,m_n^*}+ \kl(\vQ_{n,m_n^*}^*,\Pi_{n,m_n^*})-\fc_0'A_n^2n\epsilon_n^2}\\
        &\le3\exp\del{(\fc_6+\fc_2(1+\fc_5)^2)n\epsilon_n^2+ A_nn\epsilon_n^2-\fc_0'A_n^2n\epsilon_n^2},
    \end{align*}
which goes to zero since $A_n\to\infty$  and $n\epsilon_n^2\ge1$ by assumption. This completes the proof.
\end{proof}

\subsubsection{Proof of \cref{thm:over}}

\begin{proof}
By using a similar argument as  the proof of  \cref{thm:conv}, which is based on \cref{thm:vgap} and \cref{lemma:variational_ineq}, it is enough to prove that the expected original posterior probability of the event of interest is exponentially bounded. This can be easily done by using the derivation in \labelcref{eq:post_bound_large_models} and choosing a sufficiently large $H_1>0$.
\end{proof}

\subsubsection{Proof of \cref{thm:under}}

\begin{proof}
This is a direct consequence of \cref{thm:conv} because it follows that  $\scd_n(\sT(\btheta),\blambda^\star)>\uA_n\epsilon_n$ for any $\btheta\in\Theta_{n,m}$ for $m\in \cM_{n}^{\textup{under}}(\uA_n\epsilon_n; \blambda^\star)$.
\end{proof}

\subsection{Proofs for \cref{sec:regularization}}

\subsubsection{Proof of \cref{thm:regularization:over}}

\begin{proof}
Let $m_n^*\in\cM_n$ be a model index  $\eta_{n,m_n^*}+\zeta_{n,m_n^*}\le (1+\fc_5)\epsilon_n$ and $\alpha_{n,m_n^*} \ge\exp\del[1]{-\fc_6n\epsilon_n^2}$, which is assumed to be exist in \cref{assume:aggregation:model_prior_mass}, and $\vQ_{n,m_n^*}^*\in\cQ_{n,m_n^*}$ be a distribution that satisfies \cref{assume:individual:variational}. We start with observing that, for any $m\in \cM_n$,
\begin{align*}
    \hgamma_{n,m}
    &=\frac{\alpha_{n,m}\exp\del[1]{-\scE_n\del[0]{\hvQ_{n,m},\Pi_{n,m},\sp_n}}}{\sum_{m'\in \cM_n}\alpha_{n,m'}\exp\del[1]{-\scE_n\del[0]{\hvQ_{n,m'},\Pi_{n,m'},\sp_n}}}\\
    &\le \frac{\exp\del[1]{-\scE_n\del[0]{\hvQ_{n,m},\Pi_{n,m},\sp_n}}}{\alpha_{n,m_n^*}\exp\del[1]{-\scE_n\del[0]{\hvQ_{n,m_n^*},\Pi_{n,m_n^*},\sp_n}}}\\
    &\le \frac{\exp\del[1]{-\scE_n\del[0]{\hvQ_{n,m},\Pi_{n,m},\sp_n}}}{\alpha_{n,m_n^*}\exp\del[1]{-\scE_n\del[0]{\vQ_{n,m_n^*}^*,\Pi_{n,m_n^*},\sp_n}}}\\
     &\le \frac{\exp\del[1]{-\scE_n\del[0]{\hvQ_{n,m},\Pi_{n,m},\sp_n}}}{\exp\del[1]{-\fc_6n\epsilon_n^2-\scE_n\del[0]{\vQ_{n,m_n^*}^*,\Pi_{n,m_n^*},\sp_n}}},
    \end{align*}
where we use the fact that $\alpha_{n,m}\le 1$ for the first inequality,  and the optimality of $\hvQ_{n,m_n^*}^*$ for the second inequality. By dividing both the numerator and denominator of the last line of the above display by the likelihood $\sp_n(\blambda^\star,\Y^{(n)})$ of the true $\blambda^\star$, we have
    \begin{align}
    \label{eq:vb_post_model_prob_bound}
    \hgamma_{n,m}
    \le \frac{\exp\del[1]{\int \log \sr_n(\sT(\btheta), \blambda^\star) \d \hvQ_{n,m}(\btheta) -\kl( \hvQ_{n,m}, \Pi_{n,m})} }{\exp\del[1]{-\fc_6n\epsilon_n^2+\int \log \sr_n(\sT(\btheta), \blambda^\star) \d \vQ_{n,m_n^*}(\btheta) -\kl( \vQ_{n,m_n^*}^*, \Pi_{n,m_n^*})}}
    \end{align}
where we let  $\sr_n(\sT(\btheta), \blambda^\star):=\sp_n(\sT(\btheta),\Y^{(n)})/\sp_n(\blambda^\star,\Y^{(n)})$. Let
    \begin{align*}
        \Pi_{n,m}^*\in\argmax_{\tPi\in  \cP(\Theta_{n,m})}\inf_{\vQ\in \cQ_{n,m}}\cbr{\kl(\vQ, \Pi_{n,m})-\kl(\vQ, \tPi)}.
    \end{align*}    
For $m\in\cM_n$ such that $\Psi_{n,m}>0$, it should be that $\kl(\vQ, \Pi_{n,m}^*)<\infty$ for any $\vQ\in\cQ_{n,m}$. Then for such a model $m\in\cM_n$, we have 
    \begin{align*}
        \int& \log \sr_n(\sT(\btheta), \blambda^\star) \d \hvQ_{n,m}(\btheta) -\kl( \hvQ_{n,m}, \Pi_{n,m})\\
        &= \int \log \sr_n(\sT(\btheta), \blambda^\star) \d \hvQ_{n,m}(\btheta) -\kl( \hvQ_{n,m}, \Pi_{n,m}^*) + \kl( \hvQ_{n,m}, \Pi_{n,m}^*)-\kl( \hvQ_{n,m}, \Pi_{n,m})\\
        &\le\log \del{ \int \sr_n(\sT(\btheta), \blambda^\star) \d \Pi_{n,m}^*(\btheta) }-\cbr{\kl( \hvQ_{n,m}, \Pi_{n,m})-\kl(\hvQ_{n,m}, \Pi_{n,m}^*)}\\
        &\le \log \del{ \int \sr_n(\sT(\btheta), \blambda^\star) \d \Pi_{n,m}^*(\btheta) }-\inf_{\vQ\in \cQ_{n,m}}\cbr{\kl( \vQ, \Pi_{n,m})-\kl(\vQ, \Pi_{n,m}^*)}\\
        &= \log \del{ \int \sr_n(\sT(\btheta), \blambda^\star) \d \Pi_{n,m}^*(\btheta) }-        \Psi_{n,m},
    \end{align*}
where the first inequality is obtained by applying \cref{lemma:variational_ineq} with  $\sF=\log \sr_n(\sT(\btheta), \blambda^\star)$, $\vQ_0=\hvQ_{n,m}$ and $\Pi_0=\Pi_{n,m}^*$. On the other hand, for lower bounding the denominator of \labelcref{eq:vb_post_model_prob_bound}, we define the set
    \begin{align*}
        \bA_n  :=\cbr{\Y^{(n)}\in\bY_n:\int \log  \sr_n(\sT(\btheta), \blambda^\star)\d \vQ_{n,m_n^*}^*(\btheta)\ge -\sqrt{A_n} n\epsilon_n^2}.
    \end{align*}
Then on the event $\bA_n$, the denominator is bounded below as
    \begin{align*}
    \exp &\del{-\fc_6n\epsilon_n^2+\int \log \sr_n(\sT(\btheta), \blambda^\star) \d \vQ_{n,m_n^*}(\btheta) -\kl( \vQ_{n,m_n^*}^*, \Pi_{n,m_n^*})}\\
        &\ge \exp\del{-\fc_6n\epsilon_n^2 - \sqrt{A_n}n\epsilon_n^2-\kl(\vQ_{n,m_n^*}^*, \Pi_{n,m_n^*})}\ge \exp\del[1]{-\fc_1'\sqrt{A_n}n\epsilon_n^2}
    \end{align*}
for some absolute constant $\fc_1'>0$, where the last inequality holds due to \cref{assume:individual:variational}.  Furthermore, we have 
    \begin{align*}
        \P_{\blambda^\star}^{(n)}( \bA_n^\complement)
        &= \P^{(n)}_{\blambda^\star}\del{\int \log\frac{\sp_n(\sT(\btheta), \Y^{(n)})}{\sp_n(\blambda^\star, \Y^{(n)})}\d \vQ_{n,m_n^*}^*(\btheta)<-\sqrt{A_n} n\epsilon_n^2}\\
        &\le \P^{(n)}_{\blambda^\star}\del{\int 0\vee\log\del{\frac{\sp_n(\blambda^\star, \Y^{(n)})}{\sp_n(\sT(\btheta), \Y^{(n)})}}\d\vQ_{n,m_n^*}^*(\btheta)\ge \sqrt{A_n} n\epsilon_n^2}\\
        &\le \frac{1}{\sqrt{A_n} n\epsilon_n^2} \P^{(n)}_{\blambda^\star}\sbr[3]{\int0\vee\log\del{\frac{\sp_n(\blambda^\star, \Y^{(n)})}{\sp_n(\sT(\btheta), \Y^{(n)})}}\d\vQ(\btheta)}\\
        &\le \frac{1}{\sqrt{A_n} n\epsilon_n^2}\del[2]{2\vQ\sbr{\kl\del[1]{\P_{\blambda^\star}^{(n)}, \P_{\sT(\btheta)}^{(n)}}}+1}
        \le \frac{\fc_2(1+\fc_5)^2}{\sqrt{A_n}} 
    \end{align*}
where we use Markov's inequality for the third line and use Fubini's theorem and Lemma B.13 of \citetS{ghosal2017fundamentals} for the last line. Hence,
    \begin{align*}
        &\P_{\blambda^\star}^{(n)}\sbr{\hvQ_n\del[1]{ \cM_{n}^{\textup{ivB,over}}(A_n)}}\\
         &\le  \P_{\blambda^\star}^{(n)}\sbr[3]{\sum_{m\in\cM_n:\Psi_{n,m}\ge A_nn\epsilon_n^2}\hgamma_{n,m}\ind(\bA_n)} +  \P_{\blambda^\star}^{(n)}( \bA_n ^\complement)\\
        &\le\sum_{m\in\cM_n:\Psi_{n,m}\ge A_nn\epsilon_n^2} \e^{\fc_1'\sqrt{A_n} n\epsilon_n^2 -\Psi_{n,m}} \P_{\blambda^\star}^{(n)}\sbr[2]{\int \sr_n(\sT(\btheta), \blambda^\star) \d \Pi_{n,m}^*(\btheta) } +\frac{\fc_2(1+\fc_5)^2}{\sqrt{A_n}} \\
         &\le |\cM_n|\e^{\fc_1'\sqrt{A_n} n\epsilon_n^2-A_nn\epsilon_n^2 } \int \P_{\blambda^\star}^{(n)}[\sr_n(\sT(\btheta), \blambda^\star)] \d \Pi_{n,m}^*(\btheta)  +\frac{\fc_2(1+\fc_5)^2}{\sqrt{A_n}} \\
        &\le \e^{\fc_3n\epsilon_n^2+ \fc_1'\sqrt{A_n} n\epsilon_n^2-A_nn\epsilon_n^2 }  +\frac{\fc_2(1+\fc_5)^2}{\sqrt{A_n}}   =\sco(1)
    \end{align*}
as $A_n\to\infty$, where we use Fubini's theorem in the fourth line, and the last inequality follows from \cref{assume:aggregation:model_set}.
\end{proof}

\subsubsection{Proof of \cref{thm:regularization:rate}}

\begin{proof}

Let $\cK_n:=\cbr[0]{\btheta\in\Theta_{n,\cM_n}:\scd_n(\sT(\btheta),\blambda^\star)\ge (A_n\epsilon_{n})\vee \zeta_n^\ddag}$ be the event of interest. Moreover, we define 
    \begin{align*}
       \Theta_{n}^{\textup{ivB,over}}&:=\bigcup_{m\in  \cM_{n}^{\textup{ivB,over}}(A_n)}\Theta_{n,m}\\
        \Theta_{n}^{\textup{ivB,regular}}&:=\Theta_{n,\cM_n}\setminus     \Theta_{n}^{\textup{ivB,over}}=\bigcup_{m\in  \cM_{n}^{\textup{ivB,regular}}(A_n)}\Theta_{n,m}.
    \end{align*}
As we did in the proof of \cref{thm:conv}, we denote by
    \begin{align*}
        \bA_n:=\bA_n\del{A_nn\epsilon_n^2,\vQ_{n,m_n^*}^*,\Pi_n,\blambda^\star},
    \end{align*}
the event \labelcref{eq:denom:event} with $T=A_nn\epsilon_n^2$ and $\vQ=\vQ_{n,m_n^*}^*$. Then we decompose the variational probability of $\cK_n$ as
    \begin{align*}
        \P_{\blambda^\star}^{(n)}[\hvQ_n(\cK_n)]
        &\le \P_{\blambda^\star}^{(n)}[\hvQ_n(\Theta_{n}^{\textup{ivB,over}})] 
        +  \P_{\blambda^\star}^{(n)}[\hvQ_n(\cK_n\cap\Theta_{n}^{\textup{ivB,regular}})] \\
        &\le  \P_{\blambda^\star}^{(n)}[\hvQ_n(\Theta_{n}^{\textup{ivB,over}})] 
        +  \P_{\blambda^\star}^{(n)}[\hvQ_n(\cK_n\cap\Theta_{n}^{\textup{ivB,regular}})\ind(\bA_n)] +  \P_{\blambda^\star}^{(n)}(\bA_n^\complement).
    \end{align*}
By \cref{thm:regularization:over} and \cref{lemma:denominator} respectively, the first and third terms in the above display converge to zero as $n\to\infty$. Thus it suffices to bound the second term. For convenience let
    \begin{align*}
     \cK_n^{\textup{regular}} := \cK_n\cap\Theta_{n}^{\textup{ivB,regular}}
     =\cbr{\btheta\in\Theta_{n}^{\textup{ivB,regular}}:\scd_n(\sT(\btheta),\blambda^\star)\ge (A_n\epsilon_{n})\vee \zeta_n^\ddag}.
    \end{align*}
We use \cref{eq:kl_ineq_prob} in  \cref{lemma:variational_ineq} with $\vQ_0=\hvQ_n$, $\Pi_0=\Pi_n(\cdot|\Y^{(n)})$ and $\upsilon=\upsilon_n:=A_n n\epsilon_{n}^2$ to obtain
    \begin{align*}
        \P_{\blambda^\star}^{(n)}[\hvQ_n(\cK_n^{\textup{regular}})\ind(\bA_n)]
        \le \frac{1}{\upsilon_n}\P_{\blambda^\star}^{(n)}[\kl(\hvQ_n, \Pi_n(\cdot|\Y^{(n)}) )]+ \frac{1}{\upsilon_n}\e^{\upsilon_n}\P_{\blambda^\star}^{(n)}[\Pi_n(\cK_n^{\textup{regular}}|\Y^{(n)})\ind(\bA_n)],
    \end{align*}
where the first term  goes to zero by \cref{thm:vgap}. Thus, it remains to show that the second term goes to zero. Note that for any measurable function $\varphi:\bY_n\mapsto[0,1]$, we have 
    \begin{align*}
         \P_{\blambda^\star}^{(n)}\sbr{\Pi_n\del[1]{ \cK_n^{\textup{regular}}|\Y^{(n)}}\ind(\bA_n)}
         \le\P_{\blambda^\star}^{(n)}[\varphi \ind(\bA_n)] + \P_{\blambda^\star}^{(n)}\sbr{\Pi_n\del[1]{ \cK_n^{\textup{regular}}|\Y^{(n)}}(1-\varphi)\ind(\bA_n)}.
    \end{align*}
Let $\varphi_n:=\max_{m\in\cM_{n}^{\textup{ivB,regular}}(A_n)}\varphi_{n,m}$. Then by definition, $ A_n\epsilon_n\vee \zeta_n^\ddag\ge  \zeta_n^\ddag\ge \zeta_{n,m}$ for any $m\in \cM_{n}^{\textup{ivB,regular}}(A_n)$. Hence by \cref{assume:individual:testing}, we have
    \begin{align*}
       \P_{\blambda^\star}^{(n)}[\varphi_n\ind(\bA_n)] 
       \le\P_{\blambda^\star}^{(n)}[\varphi_n] 
        &\le \sum_{m\in\cM_{n}^{\textup{ivB,regular}}(A_n)}\P_{\blambda^\star}^{(n)}[\varphi_{n,m}]\\
        &\le |\cM_n|\e^{-\fc_1n((A_n\epsilon_n)\vee \zeta_n^\ddag)^2}\\
        &\le \e^{\fc_3n\epsilon_n^2-\fc_1n((A_n\epsilon_n)\vee \zeta_n^\ddag)^2}.
    \end{align*}
Moreover, by \cref{assume:individual:testing} again,
    \begin{align*}
       \sup_{\btheta\in\cK_n^{\textup{regular}}}\P_{\blambda}^{(n)}[1-\varphi_n]       
        &\le \e^{-\fc_1n((A_n\epsilon_n)\vee \zeta_n^\ddag)^2}.
    \end{align*}
Using a similar argument to that in \labelcref{eq:posterior_type2_error} in the proof of \cref{thm:post_conv}, we have
    \begin{align*}
        \P_{\blambda^\star}^{(n)}\sbr{\Pi_n\del[1]{ \cK_n^{\textup{regular}}|\Y^{(n)}}(1-\varphi_n)\ind(\bA_n)}
        &\le \e^{A_nn\epsilon_n^2 -\log\del[0]{\alpha_{n,m_n^*}} +  \kl(\vQ_{n,m_n^*}^*, \Pi_{n,m_n^*})-\fc_1n((A_n\epsilon_n)\vee \zeta_n^\ddag)^2}\\
        &\le\e^{(A_n+\fc_6+\fc_2(1+\fc_5)^2)n\epsilon_n^2-\fc_1n((A_n\epsilon_n)\vee \zeta_n^\ddag)^2}\\
        &\lesssim\e^{-\fc_1'n((A_n\epsilon_n)\vee \zeta_n^\ddag)^2}
    \end{align*}
for some absolute constant $\fc_1'>0$, which completes the proof.
\end{proof}

\subsection{Proofs for \cref{appen:theory:testing}}
\label{appen:main_proof:testing}

\subsubsection{Proof of \cref{lemma:testing:renyi}}
\label{appen:proof:testing:renyi}

We divide the proof into two steps. The first step is to show that when the two conditions \labelcref{eq:tail1,eq:tail2} are met, there is a good test for a convex alternative.

\begin{lemma}[Testing,  convex alternatives]
\label{lemma:testing:convex}
Suppose that \labelcref{eq:tail1,eq:tail2} holds. Then there exist absolute constants $\fc_1'>0$ and $\fc_2'\in(0,1)$ such that for any sufficiently large $n\in\bN$ and any $\blambda_0, \blambda_1\in\Lambda_n$, there exist a test function $\tvarphi_n:\bY_n\mapsto [0,1]$ satisfying
    \begin{equation}
        \max\cbr{\P^{(n)}_{\blambda_0}[\tvarphi_n],
        \sup_{\blambda\in\cB_{\Lambda_n,\scd_n}(\blambda_1,\zeta)}\P^{(n)}_{\blambda}[1-\tvarphi_n]}
        \le 2\e^{-\fc_1'n\zeta^2},
    \end{equation}
whenever $\zeta< \fc_2'\scd_n(\blambda_0, \blambda_1)$, where $\cB_{\Lambda_n,\scd_n}(\blambda_1,\zeta) :=\cbr{\blambda\in\Lambda_n:\scd_n(\blambda,\blambda_1)\le \zeta}.$
\end{lemma}

\begin{proof}
Let $\zeta>0$ be  such that $\zeta< \sqrt{\fc_1''}\scd_n(\blambda_0, \blambda_1)$, where $\fc_1''>0$ will be specified later. For notational convenience, we denote the likelihood ratio function by
    \begin{equation*}
        \sr_n(\blambda, \blambda')
        :=\sr_n(\blambda, \blambda';\Y^{(n)}):=\frac{\sp_n(\blambda,\Y^{(n)})}{\sp_n(\blambda',\Y^{(n)})}
    \end{equation*}
for $\blambda, \blambda'\in\Lambda$. We will show that the test function defined by
    \begin{equation*}
        \tvarphi_n:=\tvarphi_n(\Y^{(n)}):=\ind\del{\sr_n(\blambda_0, \blambda_1)\le 1}
    \end{equation*}
satisfies the desired error bound.  For the type \rom{1} error,  by using  Markov's inequality and  \labelcref{eq:tail1}, we get
    \begin{align*}
         \P^{(n)}_{\blambda_0}[\tvarphi_n]
         &= \P^{(n)}_{\blambda_0}\del{(\sr_n(\blambda_0, \blambda_1))^{\rho_\circ-1}\ge 1 }\\
         &\le  \exp\del{(\rho_\circ-1) \sD_{\rho_\circ}(\P^{(n)}_{\blambda_0}, \P^{(n)}_{\blambda_1})}\\
         &\le  \exp\del{-(1-\rho_\circ)\fc_1 n\scd_n^2(\blambda_0,\blambda_1)}\\
         &\le \exp\del[2]{-\frac{(1-\rho_\circ)\fc_1}{\fc_1''}n\zeta^2}.
    \end{align*}
For the type \rom{2} error, we consider the event
    \begin{equation*}
        \bB_{n,\blambda}:=\cbr{\Y^{(n)}\in\bY_n:\log(\sr_n(\blambda, \blambda_1))<\fc_2''n\zeta^2}
    \end{equation*}
for each $\blambda\in\cB_{\Lambda_n,\scd_n}(\blambda_1,\zeta)$, where $\fc_2''>0$ will be specified later. Then, since $\rho_\blacklozenge>1$ and $\scd_n(\blambda, \blambda_1)\le\zeta$, by Markov's inequality and  \labelcref{eq:tail2}.
    \begin{align*}
         \P^{(n)}_{\blambda}\del[1]{\bB_{n,\blambda}^{\complement}}
         &\le \P^{(n)}_{\blambda}\del{(\sr_n(\blambda, \blambda_1))^{\rho_\blacklozenge-1}>\e^{\fc_1''(\rho_\blacklozenge-1)n\zeta^2}}\\
         &\le\e^{-\fc_2''(\rho_\blacklozenge-1)n\zeta^2}\P^{(n)}_{\blambda}\sbr{(\sr_n(\blambda, \blambda_1))^{\rho_\blacklozenge-1}}\\
         &= \e^{-\fc_2''(\rho_\blacklozenge-1)n\zeta^2} \exp\del[1]{(\rho_\blacklozenge-1) \sD_{\rho_\blacklozenge}(\P^{(n)}_{\blambda}, \P^{(n)}_{\blambda_1})}\\
        &= \e^{-\fc_2''(\rho_\blacklozenge-1)n\zeta^2+(\rho_\blacklozenge-1)\fc_2n\scd_n^2(\blambda, \blambda_1)}\\
        & \le \exp\del[1]{-\del{\fc_2''-\fc_2}(\rho_\blacklozenge-1)n\zeta^2}.
    \end{align*}
By the definition of the set $ \bB_{n,\blambda}$, we have
    \begin{align*}
        \P^{(n)}_{\blambda}\sbr{(1-\tvarphi_n)\ind(\bB_{n,\blambda})}
        &=\P^{(n)}_{\blambda_1}\sbr{(1-\tvarphi_n) \sr_n(\blambda,\blambda_1)\ind(\bB_{\blambda})}
        \le \e^{\fc_2''n\zeta^2}\P^{(n)}_{\blambda_1}\sbr[1]{1-\tvarphi_n},
    \end{align*}
where, by using Markov's inequality and  \labelcref{eq:tail1} again, the expectation is further bounded as
    \begin{align*}
\P^{(n)}_{\blambda_1}\sbr{1-\tvarphi_n}
        &=\P^{(n)}_{\blambda_1}\del{(\sr_n(\blambda_1, \blambda_0))^{\rho_\circ-1}\ge 1}\\
        &\le \exp\del{(\rho_\circ-1) \sD_{\rho_\circ}(\P^{(n)}_{\blambda_1}, \P^{(n)}_{\blambda_0})}\\
        &\le  \exp\del{-(1-\rho_\circ)\fc_1n\scd_n^2(\blambda_0,\blambda_1)}\\
        &\le \exp\del[2]{-\frac{(1-\rho_\circ)\fc_1}{\fc_1''}n\zeta^2}.
    \end{align*}
Hence if we choose $\fc_2''=2\fc_2$ and $\fc_1''=\fc_1(1-\rho_\circ)/(4\fc_2)\in(0,1)$, we have 
    \begin{align*}
          \max&\cbr{\P^{(n)}_{\blambda_0}[\tvarphi_n], \sup_{\blambda\in\cB_{\Lambda_n,\scd_n}(\blambda_1,\zeta)}\P^{(n)}_{\blambda}[1-\tvarphi_n]}\\
        &\le\max\cbr{\e^{-2\fc_2''n\zeta^2}, \sup_{\blambda\in\cB_{\Lambda_n,\scd_n}(\blambda_1,\zeta)}
        \cbr{\P^{(n)}_{\blambda}\sbr{(1-\tvarphi_n)\ind(\bB_{\blambda})} +\P^{(n)}_{\blambda}(\bB_{\blambda}^{\complement})}}\\
        &\le 2\e^{-((\rho_\blacklozenge-1)\wedge 2)\fc_2n\zeta^2},
    \end{align*}
 which completes the proof.
\end{proof}

Using the above lemma, we can prove \cref{lemma:testing:renyi} under the additional assumption \labelcref{eq:covering}.

\begin{proof}[Proof of \cref{lemma:testing:renyi}]
Fix  $\blambda^\star\in\Lambda_n^\star$ and  $\zeta>\zeta_{n,m}$.
Define
    \begin{equation*}
        \Theta_{n,m,j}:=\cbr{\btheta\in\Theta_{n,m}:j\zeta<\scd_n(\sT(\btheta),\blambda^\star)\le (j+1)\zeta}.
    \end{equation*}
For a small $u>0$ which will be specified later, there exists a subset $\cbr{\blambda_{n,m,j,\ell}}_{\ell\in[N_{n,m,j}]}\subset\sT(\Theta_{n,m,j})$ with $N_{n,m,j}:=\scN(uj\zeta, \Theta_{n,m,j},\scd_n)$ such that
    \begin{equation*}
        \Theta_{n,m,j}\subset\bigcup_{\ell=1}^{N_{n,m,j}} \Theta_{n,m,j,\ell}
        \mbox{ where } \Theta_{n,m,j,\ell}:=\cbr{\btheta\in\Theta_{n,m}:\scd_n(\sT(\btheta),\blambda_{n,m,j,\ell})\le uj\zeta}.
    \end{equation*}
Then by \cref{lemma:testing:convex}, there exist  absolute constants $\fc_1''>0$, and $\fc_2''>0$ such that for any $n\in\bN$, $j>J_0$ and $\ell\in[N_{n,m,j}]$,  we have a test function $\varphi_{n,m,j,\ell}:\bY_n\mapsto[0,1]$ that satisfies
    \begin{equation*}
        \max\cbr{\P^{(n)}_{\blambda^\star}[\varphi_{n,m,j,\ell}] , \sup_{\btheta\in\Theta_{n,m,j,\ell}}\P^{(n)}_{\sT(\btheta)}[1-\varphi_{n,m,j,\ell}]}
        \le 2\e^{-\fc_2''nj^2\zeta^2},
    \end{equation*}
whenever $u<\fc_2''$, because $\scd_n(\blambda^\star,\blambda_{n,m,j,\ell} )\ge j\zeta$. Then by \labelcref{eq:covering} with a suitably chosen $u$, e.g.,  $u=\fc_2''/2$, and the assumption $\zeta>\zeta_{n,m}$, there exists an absolute constant $\fc_3''>0$ such that
    \begin{equation*}
        N_{n,m,j}
        \le \scN(u(j\epsilon),\cbr{\btheta\in\Theta_{n,m}:\scd_n(\sT(\btheta),\blambda^\star)\le 2j\epsilon},\scd_n)\le \e^{\fc_3''n\zeta_{n,m}^2}.
    \end{equation*}
We now define the combined test function $\varphi_{n,m}:=\sup_{j\in\bN:j>J_0}\sup_{\ell\in[N_{n,m,j}]}\varphi_{n,m,j,\ell} $. Then the type-\rom{1} error of the test $\varphi_n$ is bounded as
    \begin{align*}
        \P_{\blambda^\star}^{(n)}[\varphi_{n,m}]
        &\le \sum_{j\in\bN:j>J_0}\sum_{\ell=1}^{N_{n,m,j}} \P^{(n)}_{\blambda^\star}[\varphi_{n,m,j,\ell}]\\
        &\le \sum_{j\in\bN:j>J_0}N_{n,m,j}\del{2\e^{-\fc_1''nj^2\zeta^2}}\\
        &\le \frac{2}{1-\e^{-\fc_1''n\zeta^2}}\e^{\fc_3''n\zeta_{n,m}^2}\e^{-\fc_1''nJ_0^2\zeta^2}
    \end{align*}
which can be further bounded by $\e^{-\fc_4''n\zeta^2}$ up to a constant factor for some $\fc_4''>0$ if $J_0>\sqrt{\fc_1''/\fc_3''}$. The type \rom{2} error is bounded as
    \begin{align*}
        \sup_{\btheta\in\Theta_{n,m}:\scd_n(\sT(\btheta),\blambda^\star)\ge \epsilon}\P_{\sT(\btheta)}^{(n)}[1-\varphi_{n,m}]
        &\le\sup_{j\in\bN:j>J_0}\sup_{\ell\in[N_{n,m,j}]}\sup_{\btheta\in\Theta_{n,m,j,\ell}} \P^{(n)}_{\sT(\btheta)}[1-\varphi_{n,m,j,\ell}]\\
        &\le 2\e^{-\fc_2''J_0^2n\zeta^2}.
    \end{align*}
Since $n\zeta^2\ge1$ by assumption, the constant multiplicative factors $2/(1-\e^{-\fc_1''n\zeta^2})$ and $2$ of the testing error bounds can be absorbed into the exponents when $J_0$ is large. The proof is done.
\end{proof}

\subsubsection{Proof of \cref{lemma:testing:local_gauss}}
\label{appen:proof:local_gauss}

\begin{proof}
By the second assumption of the lemma, if $\rho-1\in[-\fc_3',\fc_3']$, we have that
  \begin{align*}
        \e^{(\rho-1)  \sD_{\rho}(\P^{(n)}_{\blambda_0}, \P^{(n)}_{\blambda_1})}
        &=\P_{\blambda_0}^{(n)}\del{\frac{\sp_n(\blambda_0,\Y^{(n)})}{\sp_n(\blambda_1,\Y^{(n)})}}^{\rho-1}\\
        &=\P_{\blambda_0}^{(n)}\e^{(\rho-1)\del{\log\del{\frac{\sp_n(\blambda_0,\Y^{(n)})}{\sp_n(\blambda_1,\Y^{(n)})}}}}\\
        &\le  \e^{(\rho-1)\kl\del[0]{\P^{(n)}_{\blambda_0}, \P^{(n)}_{\blambda_1}}+\fc_2'(\rho-1)^2n\scd_n^2(\blambda_0,\blambda_1)}.
     \end{align*}
From the preceding display and the assumption that $(\fc_1')^{-1}n\scd_n^2(\lambda_0,\lambda_1)\le \kl\del[0]{\P^{(n)}_{\blambda_0}, \P^{(n)}_{\blambda_1}}\le \fc_1'n\scd_n^2(\lambda_0,\lambda_1)$, for $\rho_\blacklozenge>1$ such that $\rho_\blacklozenge-1<\fc_3'$, we have
    \begin{align*}
    \sD_{\rho_\blacklozenge}(\P^{(n)}_{\blambda_0}, \P^{(n)}_{\blambda_1})
    \le 
    \cbr{\fc_1' + \fc_2'(\rho_\blacklozenge-1)}n\scd_n^2(\blambda_0,\blambda_1)
    \end{align*}
and for $\rho_\circ\in(0,1)$ such that $\rho_\circ>(1-\fc_3')\vee0$, we have
    \begin{align*}
        \sD_{\rho_\circ}(\P^{(n)}_{\blambda_0}, \P^{(n)}_{\blambda_1})
        \ge 
        \cbr{\frac{1}{\fc_1'}+ \fc_2'(1-\rho_\circ)}n\scd_n^2(\blambda_0,\blambda_1).
    \end{align*}
This completes the proof.
\end{proof}

\section{Proofs for deep learning applications in \cref{sec:dnn} and \cref{appen:dnn:additional}}

Throughout this section, for a set of some network parameters $\tilde\Theta\subset\Theta$, we let $\net(\tilde\Theta):=\cbr{\net(\btheta):\btheta\in\tilde\Theta}$, which is a set of neural networks induced by network parameters in $\tilde\Theta$.

\subsection{Preliminaries}

\subsubsection{Function approximation by neural networks}
\label{appen:dnn:approx}

In this subsection, we give results on the approximation ability of neural networks for some smooth function classes. These are important for the proofs of \cref{col:dnn:holder,col:dnn:composite}. 

We first state the lemma on neural network approximation of H\"older smooth functions.

\begin{theorem}[Approximation of H\"older smooth functions]
\label{lemma:dnn:approx:holder}
Let $\beta>0$, $d\in\bN$, and $F_0>0$. Then there exist absolute constants $M_0\in\bN$, $K_0\in\bN$,  $\fc_1>0$ and $\fc_2>0$ such that for any $K\in\bN_{\ge K_0}$, $M\in\bN_{\ge M_0}$, $B>\fc_1M^{\max\{1,2(\beta-1)/d\}}$ and $f^\star\in\cH^{\beta, d, F_0}$, there exists a network parameter $\btheta^*\in \Theta^{\le B}_{(K,M)}$ such that
    \begin{equation}
       \norm{\net(\btheta^*)-f^\star}_\infty\le \fc_2\del{2^{-K}+M^{-2\beta/d}}.
    \end{equation}
\end{theorem}

\begin{proof}
The proof is deferred to \cref{appen:proof_dnn_approx:holder}.
\end{proof}

The conclusion of \cref{lemma:dnn:approx:holder} is almost similar to Theorem 2 a) of \citetS{kohler2021rate}, but the authors of the paper did not explicitly state the magnitude bound $B$ of network parameters, which is important in our application of variational deep learning. Due to this difference, we provide detailed proof for the sake of completeness.

For \cref{col:dnn:composite}, we need the next lemma on neural network approximation of the composition structured functions.

\begin{theorem}[Approximation of composition structured functions]
\label{lemma:dnn:approx:composite}
Let $r\in\bN$, $\mathbf{d}=(d_\ell)_{\ell\in[r]}$, $\bbeta:=(\beta_\ell)_{\ell\in[r]}\in\R_+^r$,  $\s:=(s_\ell)_{\ell\in[r]}\in\otimes_{\ell=1}^r[d_\ell]$ and $F_0>0$. Then there exists absolute constants $M_0\in\bN$, $K_0\in\bN$, $\fc_1>0$, $\fc_2>0$ and $\fc_3>0$ such that for any $K\in\bN_{\ge K_0}$, $M\in\bN_{\ge M_0}$, $B>\fc_1\max_{\ell\in[r]}M^{\max\{1,2(\beta_\ell-1)/s_\ell\}}$  and  $f^\star\in\cF^{\textsc{comp}}\del{r,\mathbf{d}, \bbeta,\s, F_0}$, there exists a network parameter $\btheta^*\in \Theta^{\le B}_{(K,M)}$ such that
    \begin{equation}
       \norm{\net(\btheta^*)-f^\star}_\infty\le \fc_2\del{\e^{-\fc_3K}+M^{-2\max_{\ell\in[r]}(\beta_{\ge \ell}/s_\ell)}},
    \end{equation}
where we let $\beta_{\ge \ell}:=\beta_\ell\prod_{h=\ell+1}^r(\beta_h\wedge 1)$ for $\ell\in[r-1]$ and $\beta_{\ge r}:=\beta_r$.
\end{theorem}

\begin{proof}
The proof is deferred to \cref{appen:proof_dnn_approx:composition}.
\end{proof}

\subsubsection{Auxiliary results}

In this subsection, we state the following two properties of the neural network model, which are used in the proofs.

\begin{lemma}[Lemma 9 of \citetS{ohn2022nonconvex}]
\label{lemma:dnn_lip}
Let $K\in\bN_{\ge2}$, $M\in\bN$ and $B\ge 1$. Then for any $\btheta_1,\btheta_2\in \Theta^{\le B}_{(K,M)}$,
    \begin{equation*}
       \norm{\net(\btheta_1)-\net(\btheta_2)}_{\infty} \le K(B(M+1))^K\abs{\btheta_1-\btheta_2}_\infty.
    \end{equation*}
\end{lemma}

\begin{lemma}[Covering numbers of neural network classes]
\label{lemma:dnn_complexity}
For any $\zeta>0$, $K\in\bN_{\ge2}$, $M\in\bN$ and $B\ge1$,
    \begin{align*}
        \log \scN\del{\zeta, \net\del[1]{\Theta^{\le B}_{(K,M)}},\|\cdot\|_\infty}
        &\le 2KJ_{(K, M)}\log\del{\frac{BK(M+1)}{\zeta}}\\
        &\le 2(d+1)K^2M^2\log\del{\frac{BK(M+1)}{\zeta}}.
    \end{align*}
\end{lemma}

\begin{proof}
Since the dimension of the parameter space $\Theta^{\le B}_{(K,M)}$ is denoted by $J_{(K,M)}$, and is bounded as $J_{(K,M)}\le (d+1)KM^2$, the desired result follows from the very well known result for the covering number bound of neural network classes,  e.g., Lemma 3 of  \citetS{suzuki2018adaptivity} and Proposition A.1 of \citetS{kim2021fast}.
\end{proof}

\subsection{Proofs of \cref{thm:dnn:oracle,col:dnn:holder,col:dnn:composite}}

\subsubsection{Proof of \cref{thm:dnn:oracle}}

We first show that a suitable test function exists for the regression experiment in \cref{sec:dnn:reg}.

\begin{lemma}
\label{lemma:model:gaussreg}
A sequence of the Gaussian regression experiments $\{(\R^n,\cP(\R^n;\sp_n^{\textup{Ga}},\cF^d))\}_{n\in\bN}$ satisfies \labelcref{eq:tail1,eq:tail2} with the metric $\scd_n:\cF^d\times \cF^d\mapsto\R_{\ge0}$ defined as $\scd_n(f_0,f_1)=\nnorm{f_0-f_1}$ for $f_0,f_1\in\cF^{d}$.
\end{lemma}

\begin{proof}
The proof is deferred to \cref{appen:dnn:proof:model:gaussreg}.
\end{proof}

\begin{proof}[Proof of \cref{thm:dnn:oracle}]
We use \cref{lemma:testing:renyi} to check \cref{assume:individual:testing}. First, \cref{lemma:model:gaussreg} verifies \cref{eq:tail1,eq:tail2}. For \cref{eq:covering}, we define
    \begin{equation*}
        \zeta_{n,(K,M)}:=KM\sqrt{\log n/n}.
    \end{equation*}
Then by \cref{lemma:dnn_complexity} and the assumptions that $\max_{(K,M)\in\cM_n}(K\vee M)\lesssim n$ and $1\le B_n\lesssim n^{\iota_0}$, there is an absolute constant $\fc_1>0$ such that
    \begin{align*}
        \sup_{\zeta>\zeta_{n,(K,M)}} &\log\scN\del{u\zeta,\cbr{\btheta\in\Theta^{\le B_n}_{(K,M)}:\nnorm{\net(\btheta)-f^\star}\le 2\zeta}, \|\cdot\|_{n,2}}\\
        &\le \log\scN\del{u\zeta_{n,m},\net\del{\Theta^{\le B_n}_{(K,M)}}, \|\cdot\|_{n,2}}
        \le \log\scN\del{u\zeta_{n,m},\net\del{\Theta^{\le B_n}_{(K,M)}}, \|\cdot\|_{\infty}}\\
        &\le 2(d+1) K^2M^2\log\del{\frac{\fc_1n^{\nu+1/2}}{u}}.
    \end{align*}
By using the inequality $(x+1)y\ge x+y$ for any $x>0$ and $y\ge 1$, the preceding display is further bounded by $c(u)K^2M^2\log n$ with  $c(u):=2(d+1)(\nu+\frac{1}{2})\log(\e\fc_1/u)$ for any $n\ge \e^2$, which verifies \labelcref{eq:covering}. 

We proceed to check \cref{assume:individual:variational}. Let
    \begin{equation*}
        \eta_{n,(K,M)}:=\sup_{\tif\in\cF^\star}\inf_{\btheta\in\Theta^{\le B_n}_{(K,M)}}\norm[1]{\net(\btheta)-\tif}_\infty+n^{-1}.
    \end{equation*}
Then by definition, for any $f^\star\in\cF^\star$, there is a network parameter $\btheta_{n,(K,M)}^*( f^\star)\in\Theta^{\le B_n}_{(K,M)}$ such that 
    \begin{equation*}
        \nnorm[1]{\net(\btheta_{n,(K,M)}^*( f^\star))-f^\star}\le\norm[1]{\net(\btheta_{n,(K,M)}^*( f^\star))-f^\star}_\infty \le \eta_{n,(K,M)}.
    \end{equation*} 
To ease description, for each $(K, M)\in\cM_n$, let  $\btheta^*:=\btheta_{n,(K,M)}^*( f^\star)\in\Theta^{\le B_n}_{(K,M)}$ and let $\btheta^*_j$ be the $j$-th element of $\btheta^*$.  Furthermore, let $\vQ^*\in\cP(\Theta^{\le B_n}_{(K,M)})$ be a distribution such that
    \begin{equation*}
        \vQ^*:=\bigotimes_{j=1}^{J_{(K,M)}}\Unif\del{(\theta_j^*-t_n)\vee(-B_n), (\theta_j^*+t_n)\wedge B_n}
    \end{equation*}
with $t_n:=\zeta_{n,(K, M)}\del[0]{2K(B_n(M+1))^{K}}^{-1}$. Then, we have that
       \begin{equations}
    \label{eq:dnn:variational_kl}
      \vQ^*\sbr{\kl\del{\P_{f^\star}^{(n)},\P_{\net(\btheta)}^{(n)}}}
      &=\frac{n}{2}\int \nnorm{\net(\btheta)-f^\star}^2\d\vQ^*(\btheta)\\
       &\le n \int \nnorm{\net(\btheta)-\net(\btheta^*)}^2\d\vQ^*(\btheta) + n\eta^2_{n,(K,M)} \\
        &\le n\int K^2(B_n(M+1))^{2K}\abs{\btheta-\btheta^*}_\infty^2\d\vQ^*(\btheta) +n\eta^2_{n,(K,M)}\\
        &\le n\zeta^2_{n,(K, M)} +n\eta^2_{n,(K,M)}.
    \end{equations}
Furthermore,  by the assumption that $(K\vee M \vee B_n)\lesssim n^{\iota_0 \vee 1}$ and the inequality $J_{(K,M)}\le (d+1)KM^2$, 
    \begin{align*}
      \kl(\vQ^*, \Pi_n)
      &\le\sum_{j=1}^{J_{(K,M)}}\log\del{\frac{2B_n}{t_n}}
       \lesssim (K(M)^2)K \log n =2 n\zeta_{n,(K,M)}^2,
    \end{align*}
which verifies \cref{assume:individual:variational}.

Lastly, \cref{assume:aggregation} follows from $|\cM_n|\lesssim  \exp(\fc_2\log n)$ for some $\fc_2>0$ and \cref{lemma:model_prob_choice}, so the proof is done.
\end{proof}

\subsubsection{Proof of \cref{col:dnn:holder,col:dnn:composite}}
\label{appen:dnn:proof:colloraries}

\begin{proof}[Proof of \cref{col:dnn:holder}]
By \cref{lemma:dnn:approx:holder}, for any sufficiently large $n$, the convergence rate is given by
    \begin{equation*}
    \epsilon_n\del[1]{\cH^{\beta, d, F_0}}\asymp \min_{(K,M)\in\cM_n}\del{2^{-K}+M^{-2\beta/d}+KM\sqrt{\frac{\log n}{n}} }.
    \end{equation*}
Let $k_n^*\in \sbr[0]{0:\lceil(\log n)/2\rceil}$ be an integer such that
    \begin{equation*}
       \floor{n^{k_n^*/\log n}}\le n^{\frac{d}{4\beta+2d}}\le \floor{n^{(k_n^*+1)/\log n}}\le n^{(k_n^*+1)/\log n},
    \end{equation*}
and choose $K=\floor{\log n\times \log\log n}$ and $M=\lfloor n^{k_n^*/\log n} \rfloor$. Then since $\lfloor n^{k_n^*/\log n}\rfloor\ge \frac{1}{2}n^{k_n^*/\log n}$ for any $n\ge \e^2$, and $n^{-1/\log n}=1/\e$, we have
    \begin{align*}
        \epsilon_n\del[1]{\cH^{\beta, d, F_0}}
        &\lesssim (n^{k_n^*/\log n})^{-2\beta/d} + n^{\frac{d}{4\beta+2d}}\log n(\log\log n)\sqrt{\frac{\log n}{n}}\\
        &\le \del{\frac{1}{2}n^{\frac{d}{4\beta+2d}-1/\log n}}^{-2\beta/d} + n^{-\frac{\beta}{2\beta+d}}\log^{7/4}n\\
         &\lesssim n^{-\frac{\beta}{2\beta+d}}\log^{7/4}n.
    \end{align*}
provided that
    \begin{equation*}
        B_n\ge \sqrt{n}\ge n^{\max\cbr{\frac{d}{4\beta+2d}, \frac{\beta-1}{2\beta+d}}}  \ge M^{\max\{1,2(\beta-1)/d\}}.
    \end{equation*}
Taking $A_n=\log ^{1/4}n$, we complete the proof.
\end{proof}

\begin{proof}[Proof of \cref{col:dnn:composite}]
For simplicity, let $t^*:=\max_{\ell\in[r]}(\beta_{\ge \ell}/s_\ell)$.
Let $k_n^*\in \sbr[0]{0:\lceil(\log n)/2 \rceil}$ be an integer such that
    \begin{equation*}
       \floor{n^{k_n^*/\log n}}\le n^{\frac{1}{4t^*+2}}\le \floor{n^{(k_n^*+1)/\log n}}\le n^{(k_n^*+1)/\log n},
    \end{equation*}
and choose $K=\floor{\log\log n}$ and $M=\lfloor n^{k_n^*/\log n} \rfloor$. Then since $\lfloor n^{k_n^*/\log n}\rfloor\ge \frac{1}{2}n^{k_n^*/\log n}$ for any $n\ge \e^2$, and $n^{-1/\log n}=1/\e$, we have, by \cref{lemma:dnn:approx:composite},
    \begin{align*}
        \epsilon_n\del[1]{\cF^{\textup{DAG}}\del{r,\bbeta,\s, F_0}}
        &\lesssim (n^{k_n^*/\log n})^{-2t^*} + n^{\frac{1}{4t^*+2}}\log n(\log\log n)\sqrt{\frac{\log n}{n}}\\
        &\le \del{\frac{1}{2}n^{\frac{1}{4t^*+2}-1/\log n}}^{-2t^*} + n^{-\frac{t^*}{2t^*+1}}\log^{7/4}n\\
         &\lesssim n^{-\frac{t^*}{2t^*+1}}\log^{7/4}n
         =n^{-\max_{\ell\in[r]}\frac{\beta_{\ge \ell}}{2\beta_{\ge \ell}+s_\ell}}\log^{7/4}n,
    \end{align*}
provided that
    \begin{equation*}
        B_n\ge \sqrt{n}\ge \max_{\ell\in[r]}n^{\max\cbr{\frac{s_\ell}{4\beta_\ell+2s_\ell}, \frac{\beta_\ell-1}{2\beta_\ell+s_\ell}}}  \ge \max_{\ell\in[r]}M^{\max\{1,2(\beta-1)/d\}}.
    \end{equation*}
Taking $A_n=\log ^{1/4}n$, we complete the proof.
\end{proof}

\subsection{Proofs of \cref{lemma:model:gaussreg,lemma:model:ppp,lemma:model:logistic}: Checking \labelcref{eq:tail1,eq:tail2} in \cref{lemma:testing:renyi}}

\subsubsection{Proof of \cref{lemma:model:gaussreg}}
\label{appen:dnn:proof:model:gaussreg}

\begin{proof}
For any $\rho\in(0,1)\cup(1,\infty)$ and any $f_0,f_1\in\cF^d$,
    \begin{align*}
        \sD_\rho\del{\P_{f_0}^{(n)}, \P_{f_1}^{(n)}}
        &=\frac{1}{\rho-1}\log\del{\frac{1}{(2\pi)^{n/2}}\prod_{i=1}^n\int\e^{-\frac{\rho-1}{2}\del{(y_i-f_0(\x_i))^2-(y_i-f_1(\x_i))^2}}\e^{-\frac{1}{2}(y_i-f_0(\x_i))^2}\d y_i}\\
         &=\frac{1}{\rho-1}\log\del{\frac{1}{(2\pi)^{n/2}}\prod_{i=1}^n\int\e^{-\frac{1}{2}\del{\rho(y_i-f_0(\x_i))^2+(1-\rho)(y_i-f_1(\x_i))^2}}\d y_i}\\
         &=\frac{1}{\rho-1}\log\del{\prod_{i=1}^n\e^{\frac{\rho(\rho-1)}{2}(f_0(\x_i)-f_1(\x_i))^2}\frac{1}{2\pi}\int\e^{-\frac{1}{2}\{y_i-(\rho f_0(\x_i)+(1-\rho)f_1(\x_i))\}^2}\d y_i}\\
        &=\frac{1}{\rho-1}\sum_{i=1}^n\log\del{\e^{\frac{\rho(\rho-1)}{2}(f_0(\x_i)-f_1(\x_i))^2}}\\
        &=\frac{\rho n}{2}\nnorm{f_0-f_1}^2=\frac{\rho}{2}  n\scd_n^2(f_0,f_1),
    \end{align*}
which completes the proof
\end{proof}

\subsubsection{Proof of \cref{lemma:model:logistic}}
\label{appen:dnn:proof:logistic}

\begin{proof}
We prove the desired result by verifying the two conditions of \cref{lemma:testing:local_gauss}. For notational convenience, we denote the likelihood ratio by 
    \begin{align*}
        \sr(p_0,p_1)&:= \sr_n(p_0,p_1;\Y^{(n)})
        :=\frac{\sp_n^{\text{Ber}}(p_0,\Y^{(n)})}{\sp_n^{\text{Ber}}(p_1,\Y^{(n)})}\\
        &=\exp\del{\sum_{i=1}^n\cbr{Y_i\log\del{\frac{p_0(\X_i)}{p_1(\X_i)}}+(1-Y_i)\log\del{\frac{1-p_0(\X_i)}{1-p_1(\X_i)}} }}.
    \end{align*}
The KL divergence from $\P_{p_0}^{(n)}$ to $\P_{p_1}^{(n)}$ is given by
    \begin{align*}
        \kl(\P_{p_0}^{(n)},\P_{p_1}^{(n)})
        &=\P_{p_0}^{(n)}\sbr{\log \sr(p_0,p_1)}\\
        &=n\int\cbr{p_0(\x)\log\del{\frac{p_0(\x)}{p_1(\x)}} +(1-p_0(\x))\log\del{\frac{1-p_0(\x)}{1-p_1(\x)}} } \d\x.
    \end{align*}
Let $D_{w_0}:[\varkappa,1-\varkappa]\mapsto\R$ for fixed $w_0\in[\varkappa,1-\varkappa]$ be a function such that $D_{w_0}(w)=w_0 \log(w_0/w)-(1-w_0)\log((1-w_0)/(1-w))$. Then by the Taylor expansion at $w_0$, we have
    \begin{equation*}
        D_{w_0}(w)=\frac{1}{2}D''_{w_0}(\tiw)(w-w_0)^2=\frac{1}{2}\del{\frac{w_0}{\tiw^2}+\frac{1-w_0}{(1-\tiw)^2}}(w-w_0)^2
    \end{equation*}
for some $\tiw$ that lies between $w$ and $w_0$. Since $\varkappa/(1-\varkappa)^2\le D''_{w_0}\le (1-\varkappa)/\varkappa^2$ on $[\varkappa,1-\varkappa]$, the first condition of \cref{lemma:testing:local_gauss}, which assumes the equivalence between the KL divergence and the metric $\scd_n$, holds. For the second condition of \cref{lemma:testing:local_gauss}, we first use Hoeffding's lemma to obtain
    \begin{equation}
    \label{eq:logit_gauss}
        \P_{p_0}^{(n)}\sbr{\e^{t\cbr{\log\sr(p_0,p_1) -\kl(\P_{p_0}^{(n)},\P_{p_1}^{(n)})}}}
        \le \cbr{\int_{[0,1]^d}\e^{(t^2/8)(\text{logit}(p_0(\x))-\text{logit}(p_1(\x)))^2}\d\x}^n,
    \end{equation}
where we write $\text{logit}(w):=\log(w/(1-w))$ for $w\in[0,1]$. Note that for any $p_0,p_1\in[\varkappa,1-\varkappa]$
    \begin{align*}
        (\text{logit}(p_0)-\text{logit}(p_1))^2
        &=\cbr{\log\del{1+\frac{p_0-p_1}{(1-p_0)p_1}}}^2\le \frac{1}{\varkappa^4}(p_0-p_1)^2.
    \end{align*}
Thus, for any $|t|\le \sqrt{ \fc_0(8\varkappa^4)}$ with $\fc_0>0$ arbitrarily chosen, the right-hand side of \labelcref{eq:logit_gauss} is bounded by 
    \begin{align*}
       \cbr{ \int_{[0,1]^d}\e^{t^2/(8\varkappa^4)(p_0(\x)-p_1(\x))^2}\d\x}^n
        &\le  \cbr{ \int_{[0,1]^d}\cbr{1+(\e^{t^2/(8\varkappa^4)}-1)(p_0(\x)-p_1(\x))^2}\d\x}^n\\
        &=\cbr{1+(\e^{t^2/(8\varkappa^4)}-1)\|p_0-p_1\|_2^2}^n\\
        &\le\cbr{1+(\e^{\fc_0}-1)t^2\|p_0-p_1\|_2^2}^n\\
        &\le \e^{(\e^{\fc_0}-1)t^2n\|p_0-p_1\|_2^2},
    \end{align*}
which verifies the second condition of \cref{lemma:testing:local_gauss}.
\end{proof}

\subsubsection{Proof of \cref{lemma:model:ppp}}
\label{appen:dnn:proof:ppp}

\begin{proof}
Let $\scI(\lambda):=\int_{[0,1]^d}\lambda(\x)\d\x$. Note first that
    \begin{align*}
        \sD_{2}\del{\P^{(n)}_{\lambda_0}, \P^{(n)}_{\lambda_1}}
         &=\log\cbr{\e^{\scI(\lambda_1)-\scI(\lambda_0)}\P_{\lambda_0}\sbr{\e^{\int\log(\lambda_0/\lambda_1)\d Y}}}^n,
    \end{align*}
where $Y\sim \P_{\lambda_0}$. But since
    \begin{align*}
         \P_{\lambda_0}\sbr{\e^{\int\log(\lambda_0/\lambda_1)\d Y}}
         &=\sum_{N=0}^\infty \e^{-\scI(\lambda_0)}\frac{1}{N!}\cbr{\int\frac{\lambda_0(\x)}{\lambda_1(\x)}\lambda_0(\x)\d\x}^N\\
         &=\exp\del{-\scI(\lambda_0)+\int\frac{\lambda_0^2(\x)}{\lambda_1(\x)}\d\x},
    \end{align*}
we have
\begin{align*}
         \sD_{2}\del{\P^{(n)}_{\lambda_0}, \P^{(n)}_{\lambda_1}}
         &=n\int\cbr{\frac{\lambda_0^2(\x)}{\lambda_1(\x)}-2\lambda_0(\x)+\lambda_1(\x)}\d\x\\
          &=n\int\cbr{\frac{1}{\lambda_1(\x)}(\lambda_0(\x)-\lambda_1(\x))^2\d\x}\\
          &\le \frac{n}{\varkappa_{\min}}\|\lambda_0-\lambda_1\|_2^2,
    \end{align*}
which proves \labelcref{eq:tail2}. For \labelcref{eq:tail1}, we have that by the similar calculation as before
    \begin{align*}
        \sD_{1/2}\del{\P^{(n)}_{\lambda_0}, \P^{(n)}_{\lambda_1}}
        &=-2\log \P^{(n)}_{\lambda_0}\sbr{\del{\frac{\sp_n^{\text{PPP}}(\lambda_1,\Y^{(n)})}{\sp_n^{\text{PPP}}(\lambda_0,\Y^{(n)})}}^{1/2}}\\
         &=-2\log\cbr{\e^{-\frac{1}{2}\scI(\lambda_1)+\frac{1}{2}\scI(\lambda_0)}\P_{\lambda_0}\sbr{\e^{\int\log(\lambda_1/\lambda_0)^{1/2}\d Y}}}^n\\
         &=n\scI(\lambda_1)-n\scI(\lambda_0)+2n\scI(\lambda_0)+2n\int\lambda_1(\x)^{1/2}\lambda_0(\x)^{1/2}\d\x\\
         &=n\int\del{\lambda_1(\x)^{1/2}-\lambda_0(\x)^{1/2}}^2\d\x\\
         &=n\int\frac{1}{(\lambda_1(\x)^{1/2}+\lambda_0(\x)^{1/2})^2}\del{\lambda_1(\x)-\lambda_0(\x)}^2\d\x\\
         &\ge\frac{n}{4\varkappa_{\max}}\|\lambda_0-\lambda_1\|_2^2,
    \end{align*}
which completes the proof.
\end{proof}

\bigskip

\section{Proofs of results for combinatorial model spaces in \cref{sec:sparse} and \cref{appen:sparse:dnn}}

\subsection{Proof for \cref{sec:sparse:main}}

\subsubsection{Proof of \cref{thm:sparse:conv}}

\begin{proof}
The proof is basically the same as those  of \cref{thm:conv,thm:over,thm:under}.  The only difference, which lies in \cref{eq:sparse:conv} and arises from the weakened assumption \labelcref{eq:assume:sparse:cardinality}, is to bound the type-\rom{1} error of the test function 
    \begin{align*}
        \varphi_n:=\max_{(m,S)\in \cV_n}\varphi_{n,m,S}
    \end{align*}
where $(\varphi_{n,m,S})_{m\in\cM_n, S\in\cS_n}$ are the test functions given in \cref{assume:sparse:testing} and we define, for sufficiently large $H_1>1$ and given $\baepsilon\ge\epsilon_n$,
    \begin{align*}
        \cV_n:=\cbr{(m,S)\in\cM_n\times \cS_n:\zeta_{n,m,S}\le H_1\baepsilon}.
    \end{align*}
Then, by the assumption \labelcref{eq:assume:sparse:cardinality},  \labelcref{eq:aggtest_type1} in the proof of \cref{thm:post_conv} used in the proof of \cref{thm:conv} is replaced with
    \begin{equations}
       \P_{\blambda^\star}^{(n)}[\varphi_n]
        &\le \sum_{(m,S)\in\cV_n}\P_{\blambda^\star}^{(n)}[\varphi_{n,m,S}]\\
        &\le \abs{\cV_n}\e^{-\fc_1n(A_0/J_0)^2\baepsilon^2}\\
        &\le \e^{\fc_3n(H_1\baepsilon)^2-\fc_1n(A_0/J_0)^2\baepsilon^2}.
    \end{equations}
The rest of the proof of \labelcref{eq:sparse:conv} as well as the proofs of \labelcref{eq:sparse:over,eq:sparse:under} are almost similar, so we omit them.
\end{proof}

\subsection{Proofs for \cref{sec:sparse:factor}}

We give a technical lemma used in the proofs. For a matrix $\L$, we let $\fnorm{\L}:=\sqrt{\Tr(\L^\top\L)}$ denote its Frobenius norm.

\begin{lemma}
\label{lemma:sparse:factor:div}
Let $\bSigma^\star:=\sT(\L^\star):=\L^\star(\L^\star)^\top+\I\in\Lambda_n^\star$. Then for any $\bSigma:=\sT(\L):=\L\L^\top+\I$,
    \begin{align}
        \label{eq:sparse:factor:kl}
        \kl(\P_{\bSigma^\star}^{(n)}, \P_{\bSigma}^{(n)})\le \frac{n}{4}\fnorm{\bSigma^\star-\bSigma}^2.
    \end{align}
\end{lemma}

\begin{proof}
Let $\tsigma_1,\dots,\dots \tsigma_{d_n}$ be the eigenvalues of $(\bSigma^\star)^{1/2}\bSigma^{-1} (\bSigma^\star)^{1/2}$. Then
    \begin{align*}
	\kl(\P_{\bSigma^\star}^{(n)}, \P_{\bSigma}^{(n)})
	&=\frac{n}{2}\cbr{\Tr(\bSigma^\star\bSigma^{-1})+\log|\bSigma^\star\bSigma^{-1}|-1}\\
	&=\frac{n}{2}\cbr{\Tr((\bSigma^\star)^{1/2}\bSigma^{-1} (\bSigma^\star)^{1/2})+\log|(\bSigma^\star)^{1/2}\bSigma^{-1} (\bSigma^\star)^{1/2}|-1}\\
	&=\frac{n}{2}\sum_{j=1}^{d_n}\cbr{\tsigma_j-1+\log(\tsigma_j)}	
	\le \frac{n}{4}\sum_{j=1}^{d_n}(\tsigma_j-1)^2,
\end{align*}
where the last inequality follows from the inequality $z-\log(1+z)\le z^2/2$ for any $z\ge0$. We complete the proof by noting that
    \begin{equations}
    \label{eq:sparse:factor:eigensumeq}
    \sum_{j=1}^{d_n}(\tsigma_j-1)^2
        &=\fnorm[1]{(\bSigma^\star)^{1/2}\bSigma^{-1} (\bSigma^\star)^{1/2}-\I}^2
        =\fnorm[1]{\bSigma^\star\bSigma^{-1} -\I}^2\\
        &\le \opnorm[1]{\bSigma^{-1}}^2\fnorm{\bSigma^\star-\bSigma}^2
        \le \fnorm{\bSigma^\star-\bSigma}^2,
    \end{equations}
where the second equality holds due to the fact that the two matrices $(\bSigma^\star)^{1/2}\bSigma^{-1} (\bSigma^\star)^{1/2}$ and $\bSigma^\star\bSigma^{-1}$ have the same set of non-zero eigenvalues by similarity and the last inequality follows from that $\norm[0]{\bSigma^{-1}}_{\text{op}}\le 1$.
\end{proof}

\subsubsection{Proof of \cref{thm:sparse:factor:cov}}

\begin{proof}
Let $\cS_n:=\bP([d_n]):=\{S:S\subset[d_n]\}$. 
Fix $\bSigma^\star=\sT(\L^\star)\in\Lambda_n^\star$. We first verify the testing condition of \cref{assume:sparse:testing} with
    \begin{align*}
        \zeta_{n,m,S}:=\sqrt{(|S|\vee s_n)m\log d_n/n} 
        \mbox{ for } (m,S)\in\cM_n\times \cS_n.
    \end{align*}
Here, we use the information of the true sparsity $s_n$, because the test we shall use depends on $s_n$. Namely, by Lemma 5.7 of \citetS{gao2015pca}, there exists a test function $\varphi_{n,m,S}$ such that
    \begin{align*}
    \max\cbr{\P_{\bSigma^\star}^{(n)}[\varphi_{n,m,S}], \sup_{\bSigma\in\Theta_{n,m, S}:\opnorm{\bSigma-\bSigma^\star}>J_0\zeta}\P_{\bSigma}^{(n)}[1-\varphi_{n,m,S}]}
    &\le \exp\del{\fc_1(|S|+s_n-n\zeta^2)}
    \end{align*}
for any $\zeta>0$ for some absolute constants $\fc_1>0$ and $J_0>0$. Hence, \cref{assume:sparse:testing} is satisfied since the testing error of $\varphi_{n,m,S}$ is exponentially bounded whenever $\zeta>\zeta_{n,m,S}\gtrsim \sqrt{(|S|+s_n)/n}$.

We proceed to verify \cref{assume:sparse:variational}. Although \cref{assume:sparse:variational} requires that every individual model and the prior satisfy the given condition, one can notice that it suffices to prove it only for $m=r_n$ and $S=S^\star:=\supp(\L^\star)$ because the optimal rate can be attained by the model with $m=r_n$ and $S=S^\star$ and we can make \cref{assume:sparse:variational} hold for any other individual models by letting $\eta_{n,m,S}=\infty.$ For $r_n\in\cM_n$ and $S^\star\in\cS_n$, we can set $\eta_{n,r_n, S^\star}=0$ since $\L^\star\in\Theta_{n,r_n, S^\star}$. We then have $\epsilon_n=\eta_{n,r_n, S^\star}+\zeta_{n,r_n, S^\star}=\sqrt{s_nr_n\log d_n/n}$. Let
    \begin{align*}
        \vQ^*:= \bigotimes_{j\notin S^\star}\delta(\cdot;\zero_{r_n})\times 
        \bigotimes_{j\in S^\star}\bigotimes_{k=1}^{r_n}\N(L^\star_{j,k}, \tau_0) \in\cP(\Theta_{n,r_n}).
    \end{align*}
Note that for any $\L\in\Theta_{n,r_n, S^\star}$,
	\begin{equations}
	\label{eq:sparse:factor:fnorm_bound}
	    \fnorm{\bSigma^\star-\sT(\L)}&=
	    \fnorm{\L\L^\top-\L^\star(\L^\star)^\top} \\
	    &\le \fnorm{\L-\L^\star}^2 + 2\fnorm{\L^\star(\L-\L^\star)^\top}\\
	    &\le \fnorm{\L-\L^\star}^2 + 2\opnorm{\L^\star}\fnorm{\L-\L^\star}\\
	     &\le \fnorm{\L-\L^\star}^2 + 2\sqrt{\basigma}\fnorm{\L-\L^\star}.
	\end{equations}
Therefore, by  \cref{lemma:sparse:factor:div} and Jensen's inequality, we have
    \begin{align*}
      \vQ^*\sbr{\kl\del{\P_{\bSigma^\star}^{(n)},\P_{\sT(\L)}^{(n)} }}
        &\le \frac{n}{4} \del[2]{\vQ^*\fnorm{\L-\L^\star}^2 + 2\sqrt{\basigma}\vQ^*\sbr{\fnorm{\L-\L^\star}}}\\
        &\le \frac{n}{4} \del[3]{\vQ^*\sbr{\fnorm{\L-\L^\star}^2} + 2\sqrt{\basigma}\sqrt{\vQ^*\sbr{\fnorm{\L-\L^\star}^2}}}.
    \end{align*}
Since $\vQ^*(L_{j,k})=\N(L_{j,k}^\star,\tau_0)$ for $j\in S^\star$, the above display is further bounded as
        \begin{align*}
      \vQ^*\sbr{\kl\del{\P_{\bSigma^\star}^{(n)},\P_{\sT(\L)}^{(n)} }}
         &\le \frac{n}{4} \del{s_nr_n\tau_0 + 2\sqrt{\basigma}\sqrt{s_nr_n\tau_0}}\\
         &\lesssim ns_nr_n\le n\zeta_{n,r_n, S^\star}^2.
    \end{align*}
We bound the KL divergence from the prior $\Pi_{n,r_n}$  to $\vQ^*$. We use the following equations for the KL divergence between two spike-and-slab distributions,
    \begin{equations}
	\label{eq:sparse:kl_spikeslab}
        \kl((1-\omega_1)&\delta(\cdot;\zero) +\omega_1\vQ_1,(1-\omega_0)\delta(\cdot;\zero)+\omega_0\vQ_0)\\
        &=(1-\omega_1)\log\del{\frac{1-\omega_1}{1-\omega_0}}
        + \omega_1\log\del{\frac{\omega_1}{\omega_0}}
        + \omega_1 \kl(\vQ_1,\vQ_0)\\
        &=\begin{cases}
            \log\del{\frac{1}{1-\omega_0}} & \mbox{ if } \omega_1=0\\
            \log\del{\frac{1}{\omega_0}} + \kl(\vQ_1,\vQ_0)& \mbox{ if } \omega_1=1
        \end{cases}
    \end{equations}
which holds for any $\omega_0\in(0,1)$. By this, we have
    \begin{align*}
      \kl(\vQ^*, \Pi_{n,r_n})
      &\le (d_n-|S^\star|)\log\del{\frac{1}{1-\omega_{n,r_n}}} + |S^\star|\log\del{\frac{1}{\omega_{n,r_n}}}\\
      &\quad\quad + \sum_{j\in S^\star}\sum_{k=1}^{r_n}\kl\del{\N(L_{j,k}^\star,\tau_0), \N(0,\tau_0)}.
    \end{align*}
We separately bound the three terms in the above display.
The first term vanishes since
    \begin{align*}
        (1-\omega_{n,r_n})^{d_n-|S^\star|}&\ge\e^{-\omega_{n,r_n}(d_n-|S^\star|)}\\
        &\ge \exp\del[1]{-\e^{\log d_n-(1+\fa_0)r_n\log d_n}}\to 1,
      \end{align*}
as $n\to \infty$. The second term is bounded as
    \begin{align*}
        |S^\star|\log\del{\frac{1}{\omega_{n,r_n}}}
        \le (1+\fa_0)s_nr_n \log d_n
    \end{align*}
and the third term is as
    \begin{align*}
        \sum_{j\in S^\star}\sum_{k=1}^{r_n}\kl\del{\N(L_{j,k}^\star,\tau_0), \N(0,\tau_0)}
        &\le \sum_{j\in S^\star}\sum_{k=1}^{r_n}\frac{1}{2\tau_0}(L_{j,k}^\star)^2\\
        &=\frac{1}{2\tau_0}\fnorm{\L^\star}^2
        \le\frac{r_n}{2\tau_0}\opnorm{\L^\star}^2\lesssim r_n.
    \end{align*}
Combining all the derived results, we can see that \cref{assume:sparse:variational} is satisfied.

It remains to check \cref{assume:sparse:aggregation}. First, since $\alpha_{n,S|m}= \omega_{n,m}^{|S|}(1-\omega_{n,m})^{d_n-|S|}\le  \omega_{n,m}^{|S|}$, we have
    \begin{equations}
    	\label{eq:sparse:priorbound_largemodels}
        \sum_{(m,S)\in\cM_n\times \cS_n: \zeta_{n,m,S}\ge H\epsilon_n}\alpha_{n,m}\alpha_{n,S|m}
        &\le \sum_{m\in\cM_n}\alpha_{n,m}\sum_{S\in\cS_n: \zeta_{n,m,S}\ge H\epsilon_n}\alpha_{n,S|m}\\
        &\le \sum_{m\in\cM_n:m\ge r_n}\alpha_{n,m}\sum_{S\in\cS_n:|S|\ge H^2s_n}\alpha_{n,S|m}\\
        &\le \sum_{m\in\cM_n:m\ge r_n}\alpha_{n,m}\sum_{s\in[d_n]:s\ge H^2s_n}\binom{d_n}{s}\e^{-(1+\fa_0)sm\log d_n}\\
        &\le \sum_{m\in\cM_n:m\ge r_n}\alpha_{n,m}\e^{-\fa_0H^2s_nm\log d_n}\\
        &\le \e^{\log n}\e^{-\fa_0H^2s_nr_n\log d_n} 
        \le \e^{-\fc_3H^2s_nr_n\log d_n}
    \end{equations}
for some absolute constant $\fc_3>0$, which verifies  \labelcref{eq:assume:sparse:largemodels}. Also, \labelcref{eq:assume:sparse:cardinality} is met because
    \begin{align*}
        \abs{\cbr{(m,S)\in\cM_n\times \cS_n: \zeta_{n,m,S}< H\epsilon_n}}    
        &\le |\cM_n|\times\abs{\cbr{S\in \cS_n: |S|< H^2s_n}} \\
        &\le \e^{\log n} d_n^{H^2 s_n + 1}
        \le \e^{\fc_4H^2s_n\log d_n}
    \end{align*}
for some  absolute constant $\fc_4>0$.
Lastly \labelcref{eq:assume:sparse:best} is verified as
        \begin{align*}
       \alpha_{n,r_n}\alpha_{n,S^\star|r_n} 
       &\ge \frac{1}{n}\omega_{n,r_n}^{|S^\star|}(1-\omega_{n,r_n})^{d_n-|S^\star|}\\
       &\gtrsim \e^{-\log n}\omega_{n,r_n}^{|S^\star|}
       \ge \e^{-\fc_5 s_nr_n\log d_n}
    \end{align*}
for some  absolute constant $\fc_5>0$, which completes the proof.
\end{proof}

\subsubsection{Proof of \cref{thm:sparse:factor:nfac}}
\label{appen:sparse:factor:proof:nfac}

In the following proof, we only prove the first result \labelcref{eq:sparse:factor:nfac} on the factor dimensionality. The second result \labelcref{eq:sparse:factor:sparsity}  can be proved by the exactly same argument.

\begin{proof}
We first consider  underestimation of the true factor dimensionality $r_n$.  For any $\L\in\Theta_{n,m}$ with $m<r_n$,  since $\L^\star$ is of full rank, there is a vector $\x_1\in \R^{d_n}$ with $|\x_1|_2=1$ such that $(\L^\star)^\top\x_1\neq 0$ but $\L^\top\x_1=0$. This implies that
    \begin{align*}
        \inf_{\L\in\Theta_{n,m}}\opnorm{\sT(\L)-\bSigma^\star}\ge \sigma_{r_n}\del{\L^\star(\L^\star)^\top}\ge \eta_n^*.
    \end{align*}
Thus, \labelcref{eq:sparse:under}  in \cref{thm:sparse:conv} leads to the desired.

For the overestimation problem, we start with the bound such as
    \begin{align*}
        \P_{\bSigma^\star}^{(n)}\hvQ_n\del{\cbr{m\in\cM_n:m> H_1r_n}}
        &\le     \P_{\bSigma^\star}^{(n)}\hvQ_n\del{\cbr{(m, S)\in\cM_n\times\cS_n:m> H_1r_n, |S|<s_n}} \\
        &\quad +    \P_{\bSigma^\star}^{(n)}\hvQ_n\del{\cbr{(m, S)\in\cM_n\times\cS_n:m> H_1r_n, |S|\ge s_n}},
    \end{align*}
where we let $\cS_n:=\bP([d_n])$. 
We first bound the first term in the above display. For any $\L\in\Theta_{n,m, S}$ with $|S|<s_n$,  since there exists an index $j^*\in S^\star\setminus S$, we have that for a vector $\x_2$ of 0's except the $j^*$-th element equal to 1, 
    \begin{align*}
        \opnorm{\L\L^\top-\L^\star(\L^\star)^\top }&\ge \abs{\L\L^\top\x_2-\L^\star(\L^\star)^\top \x_2}_2\\
        &\ge  \abs[1]{\L^\star(\L^\star_{j,:})^\top }\ge \abs[1]{\L^\star_{j,:}}_2^2,
    \end{align*}
which is larger than $\eta_n^*$ by assumption, so the first term vanishes. For the second term also converges to zero, since we have
    \begin{align*}
        \Pi_n&\del{\cbr{(m, S)\in\cM_n\in:m> H_1r_n, |S|\ge s_n}}\\
        &= \Pi_n\del{\cbr{(m, S)\in\cM_n\times\cS_n:m|S|> H_1s_nr_n, |S|\ge s_n}}\\
        &\le \Pi_n\del{\cbr{(m, S)\in\cM_n\times\cS_n:\zeta_{n,m,S}> H_1\epsilon_n}}
        \le \e^{-\fc_3 H_1^2 n\epsilon_n^2},
    \end{align*}
where the last inequality was established already in \labelcref{eq:sparse:priorbound_largemodels} in the proof of \cref{thm:sparse:factor:cov}.
\end{proof}

\subsection{Proof for \cref{appen:sparse:dnn}}

\subsubsection{Proof of \cref{thm:sparse:dnn}}
\label{appen:sparse:dnn:proof}

\begin{proof}
Before giving the proof, we introduce additional notation. For $\btheta:=((\W_k,\b_k))_{k\in[K]}\in\Theta_{(K,M)}^{\le B_n}$, we define its ``input'' support
    \begin{align*}
        \supp_1(\btheta):=\cbr{j\in[d_n]: \abs{\W_{1,j,:}}_0>0}.
    \end{align*}
Then the cardinality of $\supp_1(\btheta)$ is equivalent to the number of input variables that affect the output of the neural network $\net(\btheta)$.  We consider a subset of each individual network parameter space with an input sparsity constraint, which is defined as
    \begin{align*}
        \Theta_{(K, M), S}^{d_n,\le B_n}:=\cbr{\btheta\in\Theta_{(K, M)}^{d_n,\le B_n}: \supp_{1}(\btheta)=S }
    \end{align*}
and 
    \begin{align*}
         J_{(K, M),S}:=\abs{\Theta_{(K, M), S}^{d_n,\le B_n}}
         &=(|S|+1)M+(K-2)(M^2+M)+(M+1)\\
         &\le |S|M + 2KM^2
    \end{align*}
for each $S\in\cS_n:=\bP([d_n])$.
    
We first show the conditions in \cref{lemma:testing:renyi} are satisfied to check \cref{assume:sparse:testing}. First \labelcref{eq:tail1,eq:tail2} are met by \cref{lemma:model:gaussreg}. Moreover, by \cref{lemma:dnn_complexity}, we have
    \begin{align*}
        \log \scN\del{n^{-1}, \net\del[1]{\Theta^{\d_n,le B}_{(K,M), S}},\|\cdot\|_\infty}
        &\le 2( |S|KM + 2K^2M^2)\log(nB_nK(M+1))\\
        &\lesssim ( |S|KM + K^2M^2) \log n
    \end{align*}
where the last line follows from the assumption that $\max_{(K,M)\in\cM_n}(K\vee M)\lesssim n$ and $1\le B_n\lesssim n^{\iota_0}$. Therefore, the condition \labelcref{eq:covering} in \cref{lemma:testing:renyi} is met with
     \begin{align*}
        \zeta_{n,(K,M), S}:=\del{\frac{1}{n} \cbr{(|S|KM +K^2M^2)\log n+|S|\log d_n}}^{1/2},
    \end{align*}
where the term $|S|\log d_n$ is additionally introduced to control the complexity arising from the variable selection.

Now we proceed to check  \cref{assume:sparse:variational}. Let  
    \begin{align*}
         \eta_{n,(K,M),S^\star}:=\sup_{\tif\in\cF^{\textup{sparse}}\del{d_n,\beta,s_0, F_0}}\inf_{\btheta\in\Theta^{d_n,\le B_n}_{(K,M),S^\star}}\norm[1]{\net(\btheta)-\tif}_\infty+n^{-1}
    \end{align*}
and let $\eta_{n,(K,M),S}:=\infty$ for $S\neq S^\star$. By this trick, we can simply check  \cref{assume:sparse:variational} with the upper bound $n(\eta_{n,(K,M),S^\star}+\zeta_{n,(K,M),S^\star})^2$. For each $(K, M)\in\cM_n$, by slightly abusing a notation, let $\btheta^*:=\btheta^*_{(K,M), S^\star}\in\Theta^{d_n,\le B_n}_{(K,M), S^\star}$ be a network parameter such that $\norm{\net(\btheta^*)-f^\star}_\infty \le \eta_{n,(K,M),S^\star}$. Also, let $\vQ^*:=\vQ^*_{(K,M), S^\star}\in\cQ_{n,(K,M)}$ be a distribution such that
    \begin{align*}
        \vQ^*:= \bigotimes_{h=1}^{d_n}&\cbr[2]{(1-\nu_h^*)\delta(\cdot;\zero_{M+1})+\nu_h^*\bigotimes_{j=(h-1)M+1}^{hM}\Unif(-(\theta_j^*-t_n)\vee(-B_n), (\theta_j^*+t_n)\wedge B_n)}\\
        &\times\bigotimes_{j=d_nM+1}^{J_{n,(K,M)}}\Unif(-(\theta_j^*-t_n)\vee(-B_n), (\theta_j^*+t_n)\wedge B_n),
    \end{align*}
where $\nu_h^*=1$ if $h\in  S^\star$ and $\nu_h^*=0$ otherwise, $\theta_j^*$ denotes the $j$-th element of $\btheta^*$ and we let $t_n:=\zeta_{n,(K, M), S^\star}\del[0]{2K(B_n(M+1))^{K}}^{-1}$. Then by a similar argument to \labelcref{eq:dnn:variational_kl} in the proof of \cref{thm:dnn:oracle}, we have
    \begin{align*}
      \vQ^*\sbr{\kl\del[1]{\P_{f^\star}^{(n)},\P_{\net(\btheta)}^{(n)} }}
        &\le n\zeta^2_{n,(K, M),S^\star} +n\eta^2_{n,(K,M),S^\star}.
    \end{align*}
But bounding the KL divergence between $\vQ^*$ and the prior $\Pi_{n,(K,M)}$ is more involved. We use \labelcref{eq:sparse:kl_spikeslab} that was given in the proof of \cref{thm:sparse:factor:cov} to obtain
    \begin{align*}
      \kl(\vQ^*, \Pi_{n, (K, M)})
      &\le (d_n-| S^\star|)\log\del{\frac{1}{1-\omega_{n,(K, M)}}} + |S^\star|\log\del{\frac{1}{\omega_{n,(K, M)}}}+ J_{(K,M),S^\star}\log\del{\frac{2B_n}{t_n}}\\
      &\lesssim  -\log\del[1]{(1-\omega_{n,(K, M)})^{d_n-|S^\star|}\omega_{n,(K, M)}^{| S^\star|}}+J_{(K,M),S^\star}\log n\\
      &\lesssim  -|S^\star|\log\del[1]{\omega_{n,(K, M)}}+n\zeta^2_{n,(K, M), S^\star}\\
      &\lesssim |S^\star|(KM\log n + \log d_n) + n\zeta^2_{n,(K, M), S^\star}\\
      &\lesssim n\zeta^2_{n,(K, M), S^\star},
    \end{align*}
where the third inequality follows from
    \begin{align*}
        (1-\omega_{n,(K^*, M^*)})^{d_n-|S^\star|}&\ge\e^{-\omega_{n,(K^*, M^*)}(d_n-|S^\star|)}\\
        &\ge \exp\del[1]{-\e^{\log d_n-\fa_1K^*M^*\log n-(1+\fa_2)\log d_n}}
        \to1.
      \end{align*}
Hence,  \cref{assume:sparse:variational} is satisfied.
    
It remains to check \cref{assume:sparse:aggregation}. Let
    \begin{align*}
        \epsilon_n:= n^{-\frac{\beta}{2\beta+s_0}}\log^{7/4} n +  \sqrt{\frac{\log d_n}{n}}.
    \end{align*}
For \labelcref{eq:assume:sparse:largemodels}, since
    \begin{align*}
        \alpha_{n,S|(K,M)}&=w_{n,(K,M)}^{|S|}(1-w_{n,(K,M)})^{d_n-|S|}\\
        &\le \e^{-\fa_1|S|KM\log n-(1+\fa_2)|S|\log d_n},
    \end{align*}
we have
    \begin{align*}
    &\sum_{((K,M),S)\in\cM_n\times \cS_n: \zeta_{n,(K,M), S}\ge H\epsilon_n}\alpha_{n,(K,M)} \alpha_{n,S|(K,M)}\\
    &\quad \le \sum_{(K,M)\in\cM_n: K^2M^2\log n\ge H^2n\epsilon_n^2}\alpha_{n,(K,M)} + \sum_{(K,M)\in\cM_n: K^2M^2\log n< H^2n\epsilon_n^2}\alpha_{n,(K,M)}\sum_{S\in\cS_n:|S|\ge \us_{n,(K,M)}} \alpha_{n,S|(K,M)}\\
     &\quad \le |\cM_n|\e^{- \fa_3 H^2n\epsilon_n^2}+ \sum_{(K,M)\in\cM_n: K^2M^2\log n< H^2n\epsilon_n^2}\alpha_{n,(K,M)}\sum_{S\in\cS_n:|S|\ge \us_{n,(K,M)}} \alpha_{n,S|(K,M)}.
    \end{align*}
where we let
    \begin{align*}
        \us_{n,(K,M)}:= \frac{H^2n\epsilon_n^2 -K^2M^2\log n}{KM\log n+\log d_n}.
    \end{align*}
We proceed to 
    \begin{align*}
    &\sum_{(K,M)\in\cM_n: K^2M^2\log n< H^2n\epsilon_n^2}\alpha_{n,(K,M)}\sum_{S\in\cS_n:|S|\ge \us_{n,(K,M)}} \alpha_{n,S|(K,M)}\\
    &\le  \sum_{(K,M)\in\cM_n: K^2M^2\log n< H^2n\epsilon_n^2}\alpha_{n,(K,M)}\sum_{s\in[d_n]:s\ge \us_{n,(K,M)}}\binom{d_n}{s} \e^{-\fa_1sKM\log n-(1+\fa_2)s\log d_n}\\
    &\le  \sum_{(K,M)\in\cM_n: K^2M^2\log n< H^2n\epsilon_n^2}\alpha_{n,(K,M)}\sum_{s\in[d_n]:s\ge \us_{n,(K,M)}} \e^{-\fa_1sKM\log n-\fa_2s\log d_n}\\
    &\le  \sum_{(K,M)\in\cM_n: K^2M^2\log n< H^2n\epsilon_n^2}\alpha_{n,(K,M)} \e^{-(\fa_1\wedge\fa_2)(H^2n\epsilon_n^2 -K^2M^2\log n)}\\
    &= \sum_{(K,M)\in\cM_n: K^2M^2\log n< H^2n\epsilon_n^2}\frac{1}{Z_{\alpha,n}}\e^{-\fa_1\wedge\fa_2H^2n\epsilon_n^2- (\fa_3-\fa_1\wedge\fa_2)K^2M^2\log n}\\
    &\lesssim \e^{-\fc_2H^2n\epsilon_n^2}
    \end{align*}
for some constant $\fc_2>0$, where the last line follows from
    \begin{align*}
       Z_{\alpha,n}&:=\sum_{(K,M)\in\cM_n}\e^{-\fa_3K^2M^2\log n}\\
       &\ge \exp\del[1]{-\fa_3 \log n\min_{(K,M)\in\cM_n} (KM)^2}\ge \exp(-2\fa_3\log n)
    \end{align*}
and the assumption $\fa_3>\fa_1\wedge \fa_2$. Next, \labelcref{eq:assume:sparse:cardinality} can be verified as
    \begin{align*}
        &\abs{\cbr{((K,M),S)\in\cM_n\times \cS_n: \zeta_{n,(K,M), S}< H\epsilon_n}}\\
        &\le  \abs{\cbr{((K,M),S)\in\cM_n\times \cS_n: |S|\log d_n<H^2n\epsilon_n^2}} \\
        &\le |\cM_n|d_n^{H^2n\epsilon_n^2/\log d_n+1} \\
        &\lesssim \e^{H^2n\epsilon_n^2 + \log d_n + \log n}
        \le \e^{2H^2n\epsilon_n^2}.
    \end{align*}
Lastly, for \labelcref{eq:assume:sparse:best}, we prove the existence of a ``best'' network architecture $(K^*,M^*)$ that satisfies $\eta_{n,(K^*,M^*),S^\star}+\zeta_{n,(K^*,M^*),S^\star}\lesssim \epsilon_n$. Let $k_n^*\in \sbr[0]{0:\lceil(\log n)/2 \rceil}$ be an integer such that
    \begin{equation*}
       \floor{n^{k_n^*/\log n}}\le n^{\frac{s_0}{4\beta+2s_0}}\le \floor{n^{(k_n^*+1)/\log n}}\le n^{(k_n^*+1)/\log n},
    \end{equation*}
and choose $K^*=\floor{\log n\times \log\log n}$ and $M^*=\lfloor n^{k_n^*/\log n} \rfloor$ so that
    \begin{align*}
        \zeta_{n,(K^*,M^*),S^\star}&=\del{\frac{1}{n} \cbr{(s_0K^*M^* +(K^*M^*)^2)\log n+s_0\log d_n}}^{1/2}\\
        & \lesssim  K^*M^*\sqrt{\frac{\log n}{n}} + \sqrt{\frac{\log d_n}{n}}
        \lesssim \epsilon_n
    \end{align*} 
and that by \cref{lemma:dnn:approx:holder},
    \begin{align*}
        \eta_{n,(K^*,M^*),S^\star}\lesssim n^{-\frac{\beta}{2\beta+s_0}}\le \epsilon_n.
    \end{align*}
For such $K^*$ and $M^*$, we have
    \begin{align*}
        \alpha_{n,(K^*,M^*)} \alpha_{n,S^\star|(K^*,M^*)} 
        &=Z_{\alpha,n}^{-1}\e^{-\fa_3(K^*M^*)^2\log n} (1-\omega_{n,(K^*, M^*)})^{d_n-|S^\star|}\omega_{n,(K^*, M^*)}^{|S^\star|}\\
        &\gtrsim \e^{-\fa_3(K^*M^*)^2\log n-\fa_1s_0K^*M^*\log n-(1+\fa_2)s_0\log d_n}
        \gtrsim \e^{-\fc_3n\epsilon_n^2}
    \end{align*}
for some constant $\fc_3>0$, where the second inequality holds due to $Z_{\alpha,n}\le |\cM_n|\lesssim (\log n)\log\log n$. This completes the proof.
\end{proof}

\section{Additional results and proofs for adaptive variational quasi-posteriors in \cref{sec:quasi} and \cref{appen:quasi:application}}

\subsection{Factional likelihoods in our framework}
\label{appen:quasi:fractional}

In this subsection, we show that fractional likelihoods satisfy \cref{assume:quasi:quasi}.
 
\begin{proposition}
\label{lemma:quasi:fractional}
When the conditions \labelcref{eq:tail1,eq:tail2} in \cref{lemma:testing:renyi} hold for a likelihood function $\sp_n:\Lambda_n\times \bY_n\mapsto\R_{\ge0}$, then its fractional version $(\sp_n)^{\kappa}$ satisfies \cref{assume:quasi:quasi} for any $\kappa\in(0,1)$.
\end{proposition}


\begin{proof}
We first verify \labelcref{eq:quasi:learning}.
Since either $\kappa>\rho_\circ$ or $(1-\kappa)\ge 1-\rho_\circ$ holds,  by \labelcref{eq:tail1}, we have either
    \begin{align*}
         \P_{\blambda^\star}^{(n)}\sbr{\frac{(\sp_n(\blambda,\Y^{(n)}))^{\kappa}}{(\sp_n(\blambda^\star,\Y^{(n)}))^{\kappa}}}
         &=\P_{\blambda}^{(n)}\sbr{\del{\frac{\sp_n(\blambda,\Y^{(n)})}{\sp_n(\blambda^\star,\Y^{(n)})}}^{\kappa-1}}\\
         &=\exp\del{-(1-\kappa)\sD_{\kappa}(\P^{(n)}_{\blambda}, \P^{(n)}_{\blambda^\star})}\\
         &\le \exp\del{-(1-\kappa)\sD_{\rho_\circ}(\P^{(n)}_{\blambda}, \P^{(n)}_{\blambda^\star})}\\
         &\le \exp\del{-(1-\kappa)\fc_1n\scd_n^2(\blambda,\blambda^\star)}
    \end{align*}
or 
    \begin{align*}
         \P_{\blambda^\star}^{(n)}\sbr{\frac{(\sp_n(\blambda,\Y^{(n)}))^{\kappa}}{(\sp_n(\blambda^\star,\Y^{(n)}))^{\kappa}}}
         &=\P_{\blambda^\star}^{(n)}\sbr{\del{\frac{\sp_n(\blambda^\star,\Y^{(n)})}{\sp_n(\blambda,\Y^{(n)})}}^{-\kappa}}\\
         &=\exp\del{-\kappa\sD_{1-\kappa}(\P^{(n)}_{\blambda^\star}, \P^{(n)}_{\blambda})}\\
         &\le \exp\del{-\kappa\sD_{1-\rho_\circ}(\P^{(n)}_{\blambda^\star}, \P^{(n)}_{\blambda})}\\
         &\le \exp\del{-\kappa\frac{1-\rho_\circ}{\rho_\circ}\sD_{\rho_\circ}(\P^{(n)}_{\blambda}, \P^{(n)}_{\blambda^\star})}\\
         &\le\exp\del{-\kappa\frac{1-\rho_\circ}{\rho_\circ}\fc_1n\scd_n^2(\blambda,\blambda^\star)}.
         \end{align*}
This implies \labelcref{eq:quasi:learning}. Moreover, \labelcref{eq:quasi:bound} is met with $\rho=(\rho_\blacklozenge-1)/\kappa>0$, as
    \begin{align*}
         \P_{\blambda^\star}^{(n)}\sbr{\del{\frac{(\sp_n(\blambda^\star,\Y^{(n)}))^\kappa}{(\sp_n(\blambda,\Y^{(n)}))^\kappa}}^{\rho}}
         &= \P_{\blambda^\star}^{(n)}\sbr{\del{\frac{\sp_n(\blambda^\star,\Y^{(n)})}{\sp_n(\blambda,\Y^{(n)})}}^{\rho_\blacklozenge-1}}\\
         &=\exp\del{(\rho_\blacklozenge-1)\sD_{\rho_\blacklozenge}(\P^{(n)}_{\blambda^\star}, \P^{(n)}_{\blambda})}\\
         &\le\exp\del{(\rho_\blacklozenge-1)\fc_2n\scd_n^2(\blambda^\star,\blambda)}
    \end{align*}
where the inequality follows from  \labelcref{eq:tail2}.
\end{proof}

\subsection{Proofs for \cref{sec:quasi:contraction} }

\subsubsection{Proof of \cref{thm:quasi:vgap}}

\begin{proof}
We first prove \labelcref{eq:quasi:vgap_ineq}.
Let $\sp_{n,\Pi_n}^\natural(\Y^{(n)}):=\Pi_n[\sp_n^\natural(\sT(\btheta),\Y^{(n)})]$. For any $\vQ\in\cQ_n$, by Fubini's theorem, we get
    \begin{equations}
    \label{eq:quasi:vgap:klexpansion}
        \P_{\blambda^\star}^{(n)}&\sbr{\kl(\vQ,\Pi_{n}^\natural(\cdot|\Y^{(n)}))}\\
        &=\vQ\sbr[4]{\P_{\blambda^\star}^{(n)}\sbr[4]{\log\del[3]{\frac{\sp_{n,\Pi_n}^\natural(\Y^{(n)})\d\vQ}{\sp_n^\natural(\sT(\btheta),\Y^{(n)})\d\Pi_n}}}}\\
        &= \kl(\vQ,\Pi_{n})+\vQ\sbr[4]{\P_{\blambda^\star}^{(n)}\sbr[4]{\log\del[3]{\frac{\sp_n^\natural(\blambda^\star,\Y^{(n)})}{\sp_n^\natural(\sT(\btheta),\Y^{(n)})}}}}
         +\P_{\blambda^\star}^{(n)}\sbr[4]{\log\del[3]{\frac{\sp_{n,\Pi_n}^\natural(\Y^{(n)})}{\sp_n^\natural(\blambda^\star,\Y^{(n)})}}}.
    \end{equations}   
By Jensen's inequality, Fubini's theorem and the assumption \labelcref{eq:quasi:learning}
    \begin{equations}
    \label{eq:quasi:vgap:third}
        \P_{\blambda^\star}^{(n)}\sbr{\log\del[3]{\frac{\sp_{n,\Pi_n}^\natural(\Y^{(n)})}{\sp_n^\natural(\blambda^\star,\Y^{(n)})}}}
        &\le \log\del[4]{\P_{\blambda^\star}^{(n)}\sbr[3]{\frac{\sp_{n,\Pi_n}^\natural(\Y^{(n)})}{\sp_n^\natural(\blambda^\star,\Y^{(n)})}}}\\
        &=\log\del[4]{\Pi_{n}\sbr[4]{\P_{\blambda^\star}^{(n)}\sbr[3]{\frac{\sp_n^\natural(\sT(\btheta),\Y^{(n)})}{\sp_n^\natural(\blambda^\star,\Y^{(n)})}}}}\\
        &\le \log\del[2]{\Pi_{n}\sbr[1]{\e^{-\fc_1n\scd_n^2(\sT(\btheta),\blambda^\star)}}}
        \le0
    \end{equations}
Moreover, by using Jensen's inequality and the assumption \labelcref{eq:quasi:bound},
    \begin{equations}
    \label{eq:quasi:vgap:second}
        \vQ\sbr[4]{\P_{\blambda^\star}^{(n)}\log\del[3]{\frac{\sp_n^\natural(\blambda^\star,\Y^{(n)})}{\sp_n^\natural(\sT(\btheta),\Y^{(n)})}}}
        &= \frac{1}{\rho} \vQ\sbr[4]{\P_{\blambda^\star}^{(n)}\log\del[3]{\frac{\sp_n^\natural(\blambda^\star,\Y^{(n)})}{\sp_n^\natural(\sT(\btheta),\Y^{(n)})}}^{\rho}}\\
        &\le \frac{1}{\rho}\vQ\sbr[4]{\log \P_{\blambda^\star}^{(n)}\sbr[3]{\del[3]{\frac{\sp_n^\natural(\blambda^\star,\Y^{(n)})}{\sp_n^\natural(\sT(\btheta),\Y^{(n)})}}^{\rho}}}\\
        &\le  \frac{1}{\rho}\vQ\sbr{\fc_2n\scd_n^2(\sT(\btheta),\blambda^\star)}.
    \end{equations}
Substituting   \labelcref{eq:quasi:vgap:third,eq:quasi:vgap:second} in the last line of \labelcref{eq:quasi:vgap:klexpansion}, we have
    \begin{align*}
        \P_{\blambda^\star}^{(n)}\sbr{\kl(\hvQ_n,\Pi_{n}^\natural(\cdot|\Y^{(n)}))}
        &=\P_{\blambda^\star}^{(n)}\sbr{\inf_{\vQ\in\cQ_n}\kl(\vQ,\Pi_{n}^\natural(\cdot|\Y^{(n)}))}\\
        &\le \inf_{\vQ\in\cQ_n}\P_{\blambda^\star}^{(n)}\sbr{\kl(\vQ,\Pi_{n}^\natural(\cdot|\Y^{(n)}))}\\
        &\le \inf_{\vQ\in\cQ_n}\cbr{\kl(\vQ,\Pi_{n})+\frac{1}{\rho}\vQ\sbr{\fc_2n\scd_n^2(\sT(\btheta),\blambda^\star)}}. \end{align*}
Then the proof is complete by the fact that $ \kl(\vQ_m,\Pi_n)=-\log(\alpha_{n,m}) + \kl(\vQ_m,\Pi_{n,m})$ for any $\vQ_m\in\cQ_{n,m}$ and $m\in\cM_n$.

We proceed to prove the second assertion on \labelcref{eq:quasi:vgap}. By \labelcref{eq:quasi:vgap_ineq} we have proven above and \cref{assume:quasi:variational},  for any $m\in\cM_n$, if follows that
\begin{align*}
        \P_{\blambda^\star}^{(n)}&\sbr{\kl(\hvQ_n, \Pi_n(\cdot|\Y^{(n)}) )}\\
        &\le -\log\alpha_{n,m}
         + \inf_{\vQ_{m}\in\cQ_{n,m}}\sbr{\kl\del[1]{\vQ_{m}, \Pi_{n,m}}  
            + \vQ_m\sbr{n\scd_n^2(\sT(\btheta),\blambda^\star)} }\\
        &\le -\log\alpha_{n,m} + \fc_3 n(\eta_{n,m}+\zeta_{n,m})^2.
    \end{align*}
Thus, by \cref{assume:quasi:model_prior}, for $m_n^*\in\cM_n$ therein, it follows that
    \begin{align*}
        \P_{\blambda^\star}^{(n)}\sbr{\kl(\hvQ_n, \Pi_n(\cdot|\Y^{(n)}))}
        \lesssim n\epsilon_n^2 + n(\eta_{n,m_n^*}+\zeta_{n,m_n^*})^2
        \lesssim n\epsilon_n^2,
    \end{align*}
which completes the proof.
\end{proof}

\subsubsection{Proof of \cref{thm:quasi:conv}}

We first provide a technical lemma that plays a similar role to \cref{lemma:denominator}. Similarly to  \labelcref{eq:denom:event},  we define the event
    \begin{align}
    \label{eq:quasi:denom:event}
        \bA_n^\natural(T,\bXi,\Pi,\blambda^\star):=\cbr{\Y^{(n)}\in\bY_n:\int \frac{\sp_n^\natural(\sT(\btheta), \Y^{(n)})}{\sp_n^\natural(\blambda^\star, \Y^{(n)})}\d\Pi(\btheta) \ge\exp\del[1]{-T-\kl(\vQ, \Pi)}}
    \end{align}
for  a parameter $\blambda^\star\in\Lambda_n$, a positive value $T>0$ and distributions $\bXi,\Pi\in\cP(\Theta_{n,\cM_n})$.

\begin{lemma}
\label{lemma:quasi:denominator}
Assume that \labelcref{eq:quasi:bound} holds. Then
for the event $\bA_n(T,\bXi,\Pi,\blambda^\star)$ defined in \labelcref{eq:denom:event}, we have
    \begin{equation}
    \label{eq:denom_bound:quasi}
       \P^{(n)}_{\blambda^\star}\del{ \bA_n^\natural(T,\bXi,\Pi,\blambda^\star)}
        \ge 1-  \frac{1}{\rho T}\del{\fc_2\vQ\sbr{n\scd_n^2(\sT(\btheta),\blambda^\star)}+1}
    \end{equation}
\end{lemma}

\begin{proof}
For notational simplicity,  we denote the ratio of quasi-likelihoods by
    \begin{align*}
       \sr_n^\natural(\blambda,\blambda')
       := \sr_n^\natural(\blambda,\blambda', \Y^{(n)})
       :=\frac{\sp_n^\natural(\blambda, \Y^{(n)})}{\sp_n^\natural(\blambda', \Y^{(n)})}.
    \end{align*}
We start with applying \labelcref{eq:kl_ineq_fn} in \cref{lemma:variational_ineq} with $\sF=\log \sr_n^\natural(\sT(\btheta),\blambda^\star)$, $\vQ_0=\vQ$ and $\Pi_0=\Pi$ to obtain
    \begin{align*}
        \log \int \sr_n^\natural(\sT(\btheta),\blambda^\star)\d\Pi(\btheta)
        &\ge\int \log\del[1]{\sr_n^\natural(\sT(\btheta),\blambda^\star)}\d\vQ(\d\btheta)-\kl(\vQ,\Pi).
    \end{align*}
Therefore, by applying  Markov's inequality, Fubini's theorem and Jensen's inequality in sequence, we have
    \begin{align*}
        \P^{(n)}_{\blambda^\star}\del{ (\bA_n^\natural(t,\bXi,\Pi,\blambda^\star))^{\complement}}
        &= \P^{(n)}_{\blambda^\star}\del{\log\int\sr_n^\natural(\sT(\btheta),\blambda^\star)\d\Pi(\btheta)\le -T-\kl(\vQ, \Pi)}\\
        &\le \P^{(n)}_{\blambda^\star}\del{\int \log\del[1]{\sr_n^\natural(\sT(\btheta),\blambda^\star)}\d\vQ(\btheta)\le -T}\\
        &= \P^{(n)}_{\blambda^\star}\del{\int \log\del[1]{(\sr_n^\natural(\blambda^\star,\sT(\btheta)))^\rho}\d\vQ(\btheta)\ge \rho T}\\
        &\le \P^{(n)}_{\blambda^\star}\del{\int \log\del[1]{1+\del[1]{\sr_n^\natural(\blambda^\star,\sT(\btheta))}^\rho}\d\vQ(\btheta)\ge \rho T}\\
        &\le \frac{1}{\rho T} \P^{(n)}_{\blambda^\star}\sbr{\int\log\del[1]{1+\del[1]{\sr_n^\natural(\blambda^\star,\sT(\btheta))}^\rho}\d\vQ(\btheta)}\\
         &= \frac{1}{\rho T}\int \P^{(n)}_{\blambda^\star}\sbr{\log\del{1+\del[1]{\sr_n^\natural(\blambda^\star,\sT(\btheta))}^\rho}}\d\vQ(\btheta)\\
        &\le \frac{1}{\rho T} \int\log\del[2]{1+\P^{(n)}_{\blambda^\star}\sbr{\del[1]{\sr_n^\natural(\blambda^\star,\sT(\btheta))}^\rho}}\d\vQ(\btheta)
    \end{align*}
By the assumption \labelcref{eq:quasi:bound} and the inequality $\log (1+\e^z)\le z+1$ for $z>0$, the integrand in the last line of the above display is bounded by
    \begin{align*}
        \log\del{1+\P^{(n)}_{\blambda^\star}\sbr{\del{\frac{\sp_n^\natural(\blambda^\star, \Y^{(n)})}{\sp_n^\natural(\sT(\btheta), \Y^{(n)})}}^\rho}}
        &\le \log\del{1+\e^{\fc_2n\scd_n^2(\sT(\btheta),\blambda^\star)}}\\
        &\le 1 + \fc_2n\scd_n^2(\sT(\btheta),\blambda^\star),
    \end{align*}
which completes the proof.
\end{proof}

With \cref{lemma:quasi:denominator} in hand, we complete the proof as follows.

\begin{proof}[Proof of \cref{thm:quasi:conv}]
Fix $\blambda_n^\star\in\Lambda_n^\star$ and let
    \begin{equation*}
        \cK_n:=\cbr{\btheta\in\Theta_{n,\cM_n}:\scd_n(\sT(\btheta),\blambda^\star)\ge A_n\epsilon_n}.
    \end{equation*}
Let  $m^*_n\in\cM_n$ be a model index that satisfies \cref{assume:quasi:model_prior} and $\vQ_{n,m_n^*}^*\in\cQ_{n,m_n^*}$ be a distribution satisfying \cref{assume:quasi:variational}. Now with the expression given in \labelcref{eq:quasi:denom:event}, we define
    \begin{align*}
       \bA_n^\natural:= \bA_n^\natural\del{A_nn\epsilon_n^2,\vQ_{n,m_n^*}^*,\Pi_n,\blambda^\star}.
    \end{align*}
Then we have by \cref{lemma:variational_ineq},
    \begin{align}
          \P^{(n)}_{\blambda^\star}&[\hvQ_n^\natural(\cK_n)] 
         \le\P^{(n)}_{\blambda^\star}((\bA_n^\natural)^\complement) +
         \P^{(n)}_{\blambda^\star}[\hvQ_n^\natural(\cK_n)\ind(\bA_n^\natural)]
         \nonumber\\
         & \le \P^{(n)}_{\blambda^\star}((\bA_n^\natural)^\complement) +\frac{1}{\upsilon_n}\P^{(n)}_{\blambda^\star}[\kl(\hvQ_n^\natural,\Pi_n^\natural(\cdot|\Y^{(n)})]
         +\frac{1}{\upsilon_n}\e^{\upsilon_n}\P^{(n)}_{\blambda^\star}[\Pi_n^\natural(\cK_n|\Y^{(n)})\ind(\bA_n^\natural)], \label{eq:quasi:post:decomp}
    \end{align}
where $\upsilon_n:=A_nn\epsilon_n^2$. We will show that each of the three terms in  \labelcref{eq:quasi:post:decomp} goes to zero as $n\to\infty$. By \cref{lemma:quasi:denominator} together with \cref{assume:quasi:variational}, we have
    \begin{align*}
        \P^{(n)}_{\blambda^\star}((\bA_n^\natural)^\complement)
        &\le \frac{1}{\rho A_nn\epsilon_n^2}\del{\fc_2\vQ_{n,m_n^*}\sbr{n\scd_n^2(\sT(\btheta),\blambda^\star)}+1}\\
        &\le \frac{1}{\rho A_nn\epsilon_n^2}\del{\fc_2 \fc_3 (1+\fc_4)^2n\epsilon_n^2+1}=\sco(1)
    \end{align*}
since $A_n\to\infty$. The second term in  \labelcref{eq:quasi:post:decomp} also converges to zero by employing \cref{thm:quasi:vgap} with \cref{assume:quasi:variational,assume:quasi:model_prior}. Lastly, for the third term in  \labelcref{eq:quasi:post:decomp}, we note that
    \begin{align}
        \P^{(n)}_{\blambda^\star}[\Pi_n^\natural(\cK_n|\Y^{(n)})\ind(\bA_n^\natural)]
        &=\P^{(n)}_{\blambda^\star}\sbr{\frac{\int_{\cK_n} \frac{\sp_n^\natural(\sT(\btheta), \Y^{(n)})}{\sp_n^\natural(\blambda^\star, \Y^{(n)})}\d\Pi_n(\btheta)}{\int \frac{\sp_n^\natural(\sT(\btheta), \Y^{(n)})}{\sp_n^\natural(\blambda^\star, \Y^{(n)})}\d\Pi_n(\btheta) }\ind(\bA_n^\natural)}\nonumber\\
        &\le\e^{A_nn\epsilon_n^2+\kl(\vQ_{n,m_n^*},\Pi_n)} \P_{\blambda^\star}^{(n)}\sbr{\int_{\cK_n} \frac{\sp_n^\natural(\sT(\btheta), \Y^{(n)})}{\sp_n^\natural(\blambda^\star, \Y^{(n)})}\d\Pi_n(\btheta)}. \label{eq:quasi:post:bound}
    \end{align}
By \labelcref{eq:quasi:learning} in \cref{assume:quasi:quasi}, the expectation term of \labelcref {eq:quasi:post:bound} is bounded by
    \begin{equations}
    \label{eq:quasi_numer_bound}
        \P_{\blambda^\star}^{(n)}\sbr{\int_{\cK_n} \frac{\sp_n^\natural(\sT(\btheta), \Y^{(n)})}{\sp_n^\natural(\blambda^\star, \Y^{(n)})}\d\Pi_n(\btheta)}
        &=\int_{\cK_n} \P_{\blambda^\star}^{(n)}\sbr{\frac{\sp_n^\natural(\sT(\btheta), \Y^{(n)})}{\sp_n^\natural(\blambda^\star, \Y^{(n)})}}\d\Pi_n(\btheta)\\
        &\le \int_{\cK_n}\e^{-\fc_1n\scd_n^2(\sT(\btheta),\blambda^\star)}\d\Pi_n(\btheta)\\
        &\le \e^{-\fc_1A_n^2n\epsilon_n^2}
    \end{equations}
for any sufficiently large $n$. In addition, by \cref{assume:quasi:variational,assume:quasi:model_prior},
    \begin{align*}
        \kl(\vQ_{n,m_n^*}^*,\Pi_n)
        = -\log \alpha_{n,m_n^*} +   \kl(\vQ_{n,m_n^*}^*,\Pi_{n,m_n^*})
        \le (\fc_5 + \fc_3(1+\fc_4)^2) n\epsilon_n^2.
    \end{align*}
Thus, since  $A_n\to\infty$ and $n\epsilon_n^2\ge1$ by assumption, the third term in  \labelcref{eq:quasi:post:decomp} goes to zero. We complete the proof.
\end{proof}

\subsection{Proofs for \cref{appen:quasi:sbm}}

\subsubsection{Proof of \cref{lemma:sbm:quasi}}
\label{appen:quasi:proof:sbm:cond}

\begin{proof}
Let $\P_{\Omega}$ denote the expectation operator under the Bernoulli distribution with mean $\Omega\in[0,1]$. Note first that
    \begin{align*}
        \P_{\bOmega_0}^{(\ban)}\sbr{\frac{\sp_{\ban}^\natural(\bOmega_1,\Y^{(\ban)})}{\sp_{\ban}^\natural(\bOmega_0,\Y^{(\ban)})}}
        &=\prod_{(i,j)\in\bL_n}\P_{\Omega_{0,i,j}}\sbr{\e^{-(Y_{i,j}-\Omega_{1,i,j})^2+(Y_{i,j}-\Omega_{0,i,j})^2}}\\
       &=\prod_{(i,j)\in\bL_n}\cbr{\e^{-(\Omega_{1,i,j}-\Omega_{0,i,j})^2}\P_{\Omega_{0,i,j}}\sbr{\e^{-2(\Omega_{0,i,j}-\Omega_{1,i,j})(Y_{i,j}-\Omega_{0,i,j})}}}.
    \end{align*}
By Hoeffding's lemma together with the fact that $\P_{\Ber(\Omega_{0,i,j})}[-2(\Omega_{0,i,j}-\Omega_{1,i,j})(Y_{i,j}-\Omega_{0,i,j})]=0$, we further have
    \begin{align*}
        \P_{\Omega_{0,i,j}}\sbr{\e^{-2(\Omega_{0,i,j}-\Omega_{1,i,j})(Y_{i,j}-\Omega_{0,i,j})}}
        \le \e^{\frac{1}{2}(\Omega_{0,i,j}-\Omega_{1,i,j})^2}.
    \end{align*}
Hence,
    \begin{align*}
        \P_{\bOmega_0}^{(\ban)}\sbr{\frac{\sp_{\ban}^\natural(\bOmega_1,\Y^{(\ban)})}{\sp_{\ban}^\natural(\bOmega_0,\Y^{(\ban)})}}
        &\le \prod_{(i,j)\in\bL_n}\e^{-\frac{1}{2}(\Omega_{0,i,j}-\Omega_{1,i,j})^2}=\e^{-\frac{1}{2}\ban\scd_{\ban,2}^2(\bOmega_0,\bOmega_1)},
    \end{align*}
which proves \labelcref{eq:sbm_quasi_learning}. On the other hand, by the similar calculation
    \begin{align*}
        \P_{\bOmega_0}^{(\ban)}\sbr{\del{\frac{\sp_{\ban}^\natural(\bOmega_0,\Y^{(\ban)})}{\sp_{\ban}^\natural(\bOmega_1,\Y^{(\ban)})}}^{\rho}}
        &=\prod_{(i,j)\in\bL_n}\P_{\Omega_{0,i,j}}\sbr{\e^{-\rho(Y_{i,j}-\Omega_{0,i,j})^2+\rho(Y_{i,j}-\Omega_{1,i,j})^2}}\\
        &=\prod_{(i,j)\in\bL_n}\cbr{\e^{\rho(\Omega_{1,i,j}-\Omega_{0,i,j})^2}\P_{\Omega_{0,i,j}}\sbr{\e^{2\rho(\Omega_{0,i,j}-\Omega_{1,i,j})(Y_{i,j}-\Omega_{0,i,j})}}}\\
        &\le \prod_{(i,j)\in\bL_n}\cbr{\e^{\rho(\Omega_{1,i,j}-\Omega_{0,i,j})^2}\e^{\frac{\rho^2}{2}(\Omega_{1,i,j}-\Omega_{0,i,j})^2}}\\
        &=\e^{\frac{\rho(\rho+2)}{2}\ban\scd_{\ban,2}^2(\bOmega_0,\bOmega_1)},
    \end{align*}
which proves \labelcref{eq:sbm_quasi_bound}.
\end{proof}

\subsubsection{Proof of \cref{thm:sbm:oracle}}
\label{appen:quasi:proof:sbm:orcale}

For the proof, the next lemma is useful.

\begin{lemma}
\label{lemma:sbm_max_bound}
Let $\U_0:=(U_{0,k,h})_{k\in[m],h\in[m]}\in\cU_m$ and $\U_1:=(U_{1,k,h})_{k\in[m],h\in[m]}\in\cU_m$. Then for any $\Z:=(\z_1,\dots,\z_n)^\top:=(z_{ik})_{i\in[n],k\in[m]}\in\cZ_{n,m}$, 
    \begin{align*}
        \scd_{\ban,2}(\sT(\U_0,\Z),\sT(\U_1,\Z))
        \le \max_{(k,h)\in[m]^{\otimes 2}}|U_{0,k,h}-U_{1,k,h}|.
    \end{align*}
\end{lemma}

\begin{proof}
Let $\ban_{k,h}:=\sum_{(i,j)\in\bL_n}\ind(z_{ik}=1,z_{jh}=1)$.  Then the conclusion straightforwardly follows from
    \begin{align*}
        \scd_{\ban,2}^2(\sT(\U_0,\Z),\sT(\U_1,\Z))
        &=\frac{2}{n(n-2)}\sum_{(i,j)\in\bL_n}\del{\z_i\U_0\z_j-\z_i\U_1\z_j}^2\\
        &=\frac{2}{n(n-2)}\sum_{(k,h)\in[m]^{\otimes 2}}\ban_{k,h}(U_{0,k,h}-U_{1,k,h})^2\\
        &\le\max_{(k,h)\in[m]^{\otimes 2}}|U_{0,k,h}-U_{1,k,h}|^2.
    \end{align*}
\end{proof}

The proof is done as follows.

\begin{proof}[Proof of \cref{thm:sbm:oracle}]
\cref{assume:quasi:quasi} is verified by \cref{lemma:sbm:quasi}. Thus it suffices to check \cref{assume:quasi:variational}. Let
    \begin{align*}
        \eta_{\ban,m}&:=\sup_{\bOmega^*\in\Lambda_{\ban}^\star}\inf_{(\U,\Z)\in\Theta_{\ban,m}}\scd_{\ban,2}(\sT(\U,\Z),\bOmega^*)+n^{-1}\\ \zeta_{\ban,m}&:=\sqrt{m^2\log n/n^2+\log m/n}.
    \end{align*}
By definition of infimum, for any $\bOmega^\star\in\Omega_{\ban}^\star$, there exists $(\U_{\ban,m}^*,\Z_{\ban,m}^*):=(\U_{\ban,m}^*(\bOmega^\star),\Z_{\ban,m}^*(\bOmega^\star))\in\Theta_{\ban,m}$ such that $\scd_{\ban,2}(\sT(\U_{\ban,m}^*,\Z_{\ban,m}^*),\bOmega^\star)\le \eta_{\ban,m}$. Moreover,  let
    \begin{equation*}
        \vQ_{\ban,m}^*:=\bigotimes_{(k,h)\in[m]^2:k\le h}\Unif\del{\psi_{\ban,1,k,h}^*, \psi_{\ban,2,k,h}^*}
        \times \bigotimes_{i=1}^{n}\Cat(\Z_{\ban,m,i}^*)\in\cQ_{\ban,m},
    \end{equation*}
where we denote $\psi_{\ban,1,k,h}^*:=(U_{\ban,m,k,h}^*-m/n)\vee0$ and $\psi_{\ban,2,k,h}^*:=(U_{\ban,m,k,h}^*+m/n)\wedge1$ and let $\Z_{\ban,m,i}^*\in\{0,1\}^n$ be the $i$-th row of $\Z_{\ban,m}$. Then we have 
    \begin{align*}
        \kl(\vQ_{\ban,m}^*,\Pi_{\ban,m})
        &=\sum_{(k,h)\in[m]^2:k\le h}\kl\del{\Unif\del{\psi_{\ban,1,k,h}^*, \psi_{\ban,2,k,h}^*},\Unif(0,1)} \\
        &\qquad\qquad + \sum_{i=1}^n\kl\del{\Cat(\Z_{\ban,m,i}^*), \Cat(m^{-1}\one_n)}\\
        &\le\sum_{(k,h)\in[m]^2:k\le h}\log\del{\frac{2 n}{m}} + \sum_{i=1}^n\log m\\
        &\le m^2\log n+n\log m 
        \lesssim \ban\zeta_{\ban,m}^2.
    \end{align*}
Moreover, by  \cref{lemma:sbm_max_bound}, 
    \begin{align*}
        \vQ_{\ban,m}^*&\sbr{\ban\scd_{\ban,2}^2(\sT(\U,\Z),\bOmega^\star)}\\
        &\le \vQ_{\ban,m}^*\sbr{\ban\scd_{\ban,2}^2\del[1]{\sT(\U,\Z),\sT(\U_{\ban,m^*}^*,\Z_{\ban,m}^*)}} +\ban\scd_{\ban,2}^2(\sT(\U_{\ban,m}^*,\Z_{\ban,m}^*), \bOmega^\star)\\
        &\le \ban\vQ_{\ban,m}^*\sbr{\del[2]{\max_{(k,h)\in[m]^{\otimes 2}}|U_{k,h}-U_{\ban,m,k,h}^*|}^2} +\ban\eta_{\ban,m^*}^2\\
        &\le \ban\frac{(m)^2}{n^2}+\ban\eta_{\ban,m^*}^2
      \le  \ban(\zeta_{\ban,m^*}^2+\eta_{\ban,m^*}^2).
    \end{align*}
The proof is done.
\end{proof}

\subsubsection{Proof of \cref{col:graphon}}
\label{appen:quasi:proof:graphon}

For the proof, we employ the following approximation result for smooth graphons.

\begin{lemma}[Lemma 2.1 of \citetS{gao2015rate}]
\label{lemma:graphon_approx}
There exists an absolute constant $\fc_1>0$ such that for any $\bOmega^\star\in\Lambda_{n}^{\textup{\textsc{G}}}\del[0]{\cH^{\beta,2, F_0}_{[0,1]}}$ and $m\in[n]$,
    \begin{equation*}
        \scd_{\ban,2}^2(\sT(\U^*,\Z^*),\bOmega^\star)\le  \fc_1F_0^2\del{\frac{1}{m^2}}^{\beta\wedge1}
    \end{equation*}
for some $\U^*\in\cU_m$ and $\Z^*\in\cZ_{n,m}$.
\end{lemma}

\begin{proof}[Proof of \cref{col:graphon}]
Let $\Lambda_{\ban}^\star:=\Lambda_{\ban}^{\textup{\textsc{G}}}\del[0]{\cH^{\beta,2, F_0}_{[0,1]}}$ for simplicity. Let $k_n^*\in \sbr[0]{0:\lceil\frac{1}{2}\log n\rceil}$ be an integer such that
    \begin{equation*}
       \floor{n^{k_n^*/\log n}}\le n^{\frac{1}{\beta\wedge 1+1}}\le \floor{n^{(k_n^*+1)/\log n}}\le n^{(k_n^*+1)/\log n}
    \end{equation*}
and let $m_n^*:= \lfloor n^{k_n^*/\log n}\rfloor$. Note that $n^{\frac{1}{\beta\wedge 1+1}}\ge m_n^*\gtrsim n^{-1/\log n}n^{\frac{1}{\beta\wedge 1+1}}=\e^{-1} n^{\frac{1}{\beta\wedge 1+1}}$. Therefore, by \cref{lemma:graphon_approx},
    \begin{align*}
        \epsilon_{n}^2(\Lambda_{\ban}^\star)
        &\lesssim \sup_{\bOmega^*\in\Lambda_{\ban}^\star}\inf_{(\U,\Z)\in\Theta_{\ban,m_n^*}}\scd_{\ban,2}^2(\sT(\U,\Z),\bOmega^*) 
        +\del{\frac{(m_n^*)^2}{n^2}\log n+\frac{\log m_n^*}{n}}\\
        &\lesssim  \del{\frac{1}{(m_n^*)^2}}^{\beta\wedge1} + n^{\frac{2}{\beta\wedge 1+1}-2}\log n + \frac{\log n}{n}\\
        &\lesssim  n^{-\frac{2(\beta\wedge 1)}{\beta\wedge 1+1}}\log n=n^{-\del{\frac{2\beta}{\beta+1}}\wedge1}\log n,
    \end{align*}
which completes the proof.
\end{proof}

\subsection{Proof for \cref{appen:quasi:subgauss}}

\subsubsection{Proof of \cref{lemma:subgauss}}
\label{appen:quasi:proof:subgauss}

\begin{proof}
Let $\P_{\star,i}$ be the marginal distribution of $Y_i$ under $\P_{\star}^{(n)}$ for $i\in[n]$. Since $Y_i$ is sub-Gaussian with mean $f^\star(\x_i)$ and variance proxy $\varsigma^2>0$, \labelcref{eq:subg_quasi_learning} follows from
    \begin{align*}
        \P_{\star}^{(n)}\sbr{\frac{\sp_n^{\natural,\kappa}(f, \Y^{(n)})}{\sp_n^{\natural,\kappa}(f^\star, \Y^{(n)})}}
        &=\prod_{i=1}^n\cbr{\e^{-\frac{\kappa}{2}(f(\x_i)-f^\star(\x_i))^2}\P_{\star,i}\sbr{\e^{-\kappa(f^\star(\x_i)-f(\x_i))(Y_i-f^\star(\x_i))}} }\\
        &\le \prod_{i=1}^n\cbr{\e^{-\frac{\kappa}{2}(f(\x_i)-f^\star(\x_i))^2}\e^{(\kappa^2\varsigma^2/2)(f^\star(\x_i)-f(\x_i))^2} }\\
        &= \prod_{i=1}^n\e^{-\frac{\kappa}{2}(1-\kappa\varsigma^2)(f(\x_i)-f^\star(\x_i))^2}.
    \end{align*}
Similarly,  \labelcref{eq:subg_quasi_bound} follows from
        \begin{align*}
         \P_{\star}^{(n)}\sbr{\frac{\sp_n^{\natural,\kappa}(f^\star,\Y^{(n)})}{\sp_n^{\natural,\kappa}(f,\Y^{(n)})}}^{\rho}
        &=\prod_{i=1}^n\cbr{\e^{\frac{\rho\kappa}{2}(f(\x_i)-f^\star(\x_i))^2}\P_{\star,i}\sbr{\e^{\rho\kappa(f^\star(\x_i)-f(\x_i))(Y_i-f^\star(\x_i))}} }\\
        &\le \prod_{i=1}^n\e^{\frac{1}{2}\rho\kappa(1+\rho\kappa\varsigma^2)(f(\x_i)-f^\star(\x_i))^2}.
    \end{align*}
The proof is done.
\end{proof}

\section{Proofs of neural network approximation results in \cref{appen:dnn:approx}}
\label{appen:proof_dnn_approx}

\subsection{Approximation of H\"older smooth functions}
\label{appen:proof_dnn_approx:holder}

In this subsection, we provide the proof of \cref{lemma:dnn:approx:holder} as well as additional lemmas used in the proof.

\subsubsection{Additional notation}

We first introduce some additional definitions and notation. For $T\in\bN$, $\delta\in(0,1/(3T)]$ and $d\in\bN$, we define
    \begin{align*}
        \icubes(T,\delta):=\icubes(T,\delta,d)&:=\cbr{(x_1,\dots, x_d)\in[0,1]^d:x_i\notin \bigcup_{t=1}^{T-1}\del{\frac{t}{T}-\delta, \frac{t}{T}}}\\
        &=\bigcup_{\t\in[0:T-1]^d}\bigotimes_{j=1}^d\sbr{\frac{t_j}{T},\frac{t_j+1}{T}-\delta\ind(t_j\neq T-1)},
    \end{align*}
which is an union of ``$\delta$-margined sub-hypercubes'' that are of the form $\bigotimes_{j=1}^d[t_j/T, (t_j+1)/T-\delta\ind(t_j\neq T-1)]\subset[0,1]^d$. Note that $\icubes(T,\delta) \subsetneq[0,1]^d$, and the region $[0,1]^d\setminus \icubes(T,\delta)$ will not be considered in our first  approximating neural network construction. Moreover, for $\x:=(x_j)_{j\in[d]}\in\icubes(T,\delta)$, let
    \begin{equation}
    \label{eq:grid_allocation}
        \fv(\x):=\fv_{T,\delta,d}(\x):=\del{\frac{\floor{Tx_j}}{T}-\frac{T-1}{T}\ind(x_j=1)}_{j\in[d]},
    \end{equation}
which is the leftmost vertex of the $\delta$-margined sub-hypercube, to which the input $\x$ belongs. For $\beta>0$ and $d\in\bN$, we denote
    \begin{equation*}
        \cA_{d,\beta}:=\cbr{\a\in\bN_0^d:|\a|_1<\beta},
    \end{equation*}
which can be viewed as a set of $d$-dimensional monomial exponents with degree less than $\beta>0$. Lastly, for a continuous function $f\in\cC^{0,d}$, its modulus of continuity $\beth_f$ is a function on $\R_+$ defined as
    \begin{equation*}
        \beth_f(\delta)=\sup_{\x_1,\x_2\in[0,1]^d:|\x_1-\x_2|_2\le \delta}|f(\x_1)-f(\x_2)|
    \end{equation*}
for $\delta>0$. Note that since $|\x_1-\x_2|_\infty\le |\x_1-\x_2|_2$ for any $\x_1,\x_2\in\R^d$, we have that $\beth_f(\delta)\le F_0\delta^{\beta\wedge 1}$  for any $f\in\cH^{\beta, d, F_0}$.

\subsubsection{Proof of \cref{lemma:dnn:approx:holder}}

We first provide two key lemmas and then give the proof of \cref{lemma:dnn:approx:holder}.

\begin{lemma}
\label{lemma:dnn:approx:extend}
Suppose that given $\varepsilon>0$, $T\in\bN$, $\delta\in(0,1/(3T)]$ and $f\in\cC^{0,d}$, there exists a network parameter $\btheta_0\in\Theta_{(K^\dag,M^\dag)}^{\le B^\dag}$ for $K^\dag\in\bN_{\ge2}$, $M^\dag\in\bN$ and $B^\dag>0$ such that
    \begin{equation*}
        \sup_{\x\in\icubes(T,\delta)}\abs{\net(\btheta_0)(\x)-f(\x)}\le \varepsilon.
    \end{equation*}
Then, there exists a network parameter $\btheta\in\Theta_{(K,M)}^{\le B}$ with $K:=K^\dag+4d$, $M:=3^d(M^\dag\vee 14)$ and $B:=(B^\dag\vee 1)+d\delta$ such that
    \begin{equation*}
        \sup_{\x\in[0,1]^d}\abs{\net(\btheta)(\x)-f(\x)}\le \varepsilon+d\beth_f(\delta).
    \end{equation*}
\end{lemma}

\begin{proof}
The proof is deferred to \cref{appen:proof_nn_inner_to_whole}.
\end{proof}

\begin{lemma}
\label{lemma:dnn:approx:inner}
Let $f\in\cH^{\beta, d, F_0}$, $K\in\bN_{\ge2}$, $M\in\bN$ and $\delta\in(0,1/(3T)]$ with $T:=\lfloor M^{1/d}\rfloor^2$. Then there exists a network parameter $\btheta\in\Theta_{(K^\dag, M^\dag)}^{\le B^\dag}$ with 
    \begin{align*}
        K^\dag&:=(K+12)+(K+5)\ceil{\log_2(\beta\vee1)},\\
        M^\dag&:=4dM+2\max\cbr{2dM^{1/d}, 3\ceil{\beta}(d+1)^{\lceil\beta\rceil}},\\
        B^\dag&:=\max\cbr{M+2, 2(\delta M^{2/d})^{-1}, F_0}
    \end{align*}
such that
    \begin{equation*}
       \sup_{\x\in\icubes(T,\delta)}\abs{\net(\btheta)(\x)-f(\x)}
       \le (3\beta^2+2)F_0(d+1)^{\lceil\beta\rceil}2^{-K}+ C_{\beta, d, F_0}M^{-\beta/d},
    \end{equation*}
where $C_{\beta, d, F_0}:=F_0\sum_{\a\in\bN_0^d:|\a|_1=\ceil{\beta-1}}(1/\a!)$.
\end{lemma}

\begin{proof}
The proof is deferred to \cref{appen:proof_nn_inner_approx}.
\end{proof}

\cref{lemma:dnn:approx:holder} is immediate from \cref{lemma:dnn:approx:extend,lemma:dnn:approx:inner}.

\begin{proof}[Proof of \cref{lemma:dnn:approx:holder}]
Recall that $\beth_f(\delta)\le F_0\delta^{\beta\wedge 1}$ for any $f\in\cH^{\beta,d, F_0}$. Thus taking $\delta=(6M^{2/d})^{-\beta/(\beta\wedge1)}\le (3T)^{-1}$, we obtain the approximation error in the conclusion by \cref{lemma:dnn:approx:extend,lemma:dnn:approx:inner}. To specify the magnitude bound $B$, we note that 
    \begin{align*}
         M\vee(\delta M^{2/d})^{-1}&\le 6^{\beta/(\beta\wedge1)}\cbr{M\vee(M^{2/d})^{\beta/(\beta\wedge1)-1}}\\
         &\le 6^{\beta/(\beta\wedge1)}\cbr{ M\vee(M^{2(\beta-1)/d})}.
        \end{align*}
We complete the proof.
\end{proof}

\subsubsection{Basic lemmas}

In this subsection, we give basic technical lemmas related to neural network construction, which are frequently used in the proofs of the lemmas in this section. 

\begin{lemma}
\label{lemma:nn_composite}
Let $K_1, K_2\in\bN_{\ge2}$ and let $\M_1:=(M_{1,1},\dots,M_{1,K_1+1})\in\bN^{K_1+1}$ and $\M_2:=(M_{2,1},\dots,M_{2,K_2+1})\in\bN^{K_2+1}$ with $M_{1,K_1+1}=M_{2,1}$. Then for any two network parameters $\btheta_1\in\hTheta_{\M_1}$ and $\btheta_2\in\hTheta_{\M_2}$ , there exists a network parameter $\btheta\in\Theta_{(K, M)}^{\le B}$ with $K:=K_1+K_2$, $M:=|\M_1|_\infty\vee |\M_2|_\infty\vee (2M_{1,K_1+1})$ and $B:=|\btheta_1|_\infty\vee|\btheta_2|_\infty$ such that $\net(\btheta)\equiv \net(\btheta_2)\circ\net(\btheta_1)$.
\end{lemma}

\begin{proof}
We write $\btheta_1:=((\W_{1,k},\b_{1,k}))_{k\in[K_1]}$ and $\btheta_2:=((\W_{2,k},\b_{2,k}))_{k\in[K_2]}$ for convenience. We construct $\btheta:=((\W_k,\b_k))_{k\in[K]}$ as follows:
    \begin{align*}
        \btheta:=\Big((\W_{1,1}&,\b_{1,1}),\dots,(\W_{1,K_1-1},\b_{1,K_1-1}),(\W^*_{1,K_1},\b_{1,K_1}^*),\\ &(\W^*_{2,1},\b_{2,1}),(\W_{2,2},\b_{2,2}),\dots,(\W_{2,K_2},\b_{2,K_2}) \Big),
    \end{align*}
where   
    \begin{align*}
        \W^*_{1,K_1}:=\del{\begin{array}{c} \W_{1,K_1}\\-\W_{1,K_1-1}\end{array}} \quad
        \b^*_{1,K_1}:=\del{\begin{array}{c} \b_{1,K_1}\\-\b_{1,K_1-1}\end{array}} \quad
        \W^*_{2,1}:=\del{\begin{array}{cc} \W_{2,1} & -\W_{2,1}\end{array}}.
    \end{align*}
Then the output of the $K_1$-th hidden layer of $\net(\btheta)$ is equal to $(\relu(\net(\btheta_1)),\relu(-\net(\btheta_1)))$. Thus by the identity $x=\relu(x)-\relu(-x)$, we obtain the desired result.
\end{proof}

\begin{lemma}
\label{lemma:nn_identity_ft}
For any $K\in\bN_{\ge2}$, there exists a network parameter $\btheta_{\textup{id},K}\in\Theta_{(K, 2)}^{\le 1}$ such that
    \begin{equation*}
         \net(\btheta_{\textup{id},K})(x)=x
    \end{equation*}
for any $x\in\R$.
\end{lemma}

\begin{proof}
By the identity  $x=\relu(x)-\relu(-x)$ and the projection property  $\relu\circ\relu\equiv\relu$ of the ReLU activation function, the desired result easily follows.
\end{proof}

\subsubsection{Proof of \cref{lemma:dnn:approx:extend}}
\label{appen:proof_nn_inner_to_whole}

We first state two key lemmas, which are taken from the existing literature.

\begin{lemma}[Lemma 3.1 of \citetS{lu2020deep}, Lemma 6.2 of \citetS{yang2020approximation}]
\label{lemma:nn_mid}
There exists a network parameter $\btheta_{\mid}\in\Theta_{(4,14)}^{\le 1}$ such that
    \begin{equation*}
        \net(\btheta_{\mid})(x_1,x_2,x_3)=\mid(x_1,x_2,x_3)
    \end{equation*}
for any $(x_1,x_2,x_3)\in\R^d$, where $\mid:\R^3\mapsto\R$ is a
function returning the middle value of three inputs.
\end{lemma}

\begin{proof}
Note that we have a neural network expression for the max function such as
    \begin{align*}
        \max(x_1,x_2)&=\frac{1}{2}(x_1+x_2)+\frac{1}{2}|x_1-x_2|\\
        &=\frac{1}{2}\relu(x_1+x_2)-\frac{1}{2}\relu(-x_1-x_2)+\frac{1}{2}\relu(x_1-x_2)+\frac{1}{2}\relu(-x_1+x_2)\\
        &=\net(\btheta_{\max,2})(x_1,x_2) \mbox { with }\btheta_{\max,2}\in\Theta_{(2,4)}^{\le 1}.
    \end{align*}
Thus, using \cref{lemma:nn_composite}, we get a network parameter $\btheta_{\max,3}\in\Theta_{(4,6)}^{\le 1}$ such that
    \begin{align*}
        \max(x_1,x_2,x_3)
        = \max\del{\max(x_1,x_2),\net(\btheta_{\mathrm{id}, 2})(x_3)}
        =\net(\btheta_{\max,3})(x_1,x_2,x_3)
    \end{align*}
for any $(x_1,x_2,x_3)\in\R^3$, where $\btheta_{\mathrm{id}, 2}\in\Theta_{(2,2)}^{\le 1}$ is the network parameter constructed in \cref{lemma:nn_identity_ft}, whose realization is the identity on $\R$. Similarly, we can find a network parameter $\btheta_{\min,3}\in\Theta_{(4,6)}^{\le 1}$ such that $\min(\cdot,\cdot,\cdot)\equiv\net(\btheta_{\min,3})(\cdot,\cdot,\cdot)$. By the above constructions, \cref{lemma:nn_identity_ft} and the fact that
 \begin{align*}
        \mid(x_1,x_2,x_3)
        &= x_1+x_2+x_3-\max(x_1,x_2,x_3)-\min(x_1,x_2,x_3),
    \end{align*}
we conclude the desired result.
\end{proof}

\begin{lemma}[Lemma 3.4 of \citetS{lu2020deep}]
\label{lemma:general_inner_to_whole}
Suppose that, given $\varepsilon>0$, $T\in\bN$ and $\delta\in(0,1/(3T)]$,  two functions $f\in\cC^{0,d}$ and $g_0:\R^d\mapsto \R$ satisfy
    \begin{equation*}
         \sup_{\x\in\icubes(T,\delta)}\abs{g_0(\x)-f(\x)}\le \varepsilon.
    \end{equation*}
Then if we recursively define functions $g_1,\dots, g_d$ as
    \begin{equation*}
        g_{j}(\cdot):=\mid\del{g_{j-1}(\cdot-\delta\mathbf{e}_j), g_{j-1}(\cdot), g_{j-1}(\cdot+\delta\mathbf{e}_j)}
    \end{equation*}
for $j=[d]$, it follows that
    \begin{equation*}
         \sup_{\x\in[0,1]^d}\abs{g_d(\x)-f(\x)}\le\varepsilon+d\beth_f(\delta),
    \end{equation*}
where $\{\mathbf{e}_j\}_{j\in[d]}$ is the  standard basis in $\R^d$.
\end{lemma}

We now provide the proof of \cref{lemma:dnn:approx:extend}.

\begin{proof}[Proof of \cref{lemma:dnn:approx:extend}]
We will prove the desired result by employing \cref{lemma:general_inner_to_whole} in with $g_0=\net(\btheta_0)$. By induction, it suffices to show that if $g_{j-1}=\net(\btheta_{j-1})$ with $\btheta_{j-1}\in\Theta_{(K_j,M_j)}^{\le B_j}$, 
    \begin{equation}
    \label{eq:mid_induction}
        g_{j}=\net(\btheta_j)
        \mbox{ for some }
         \btheta_j\in\Theta_{(K_j+4, 3M_j\vee 14)}^{\le (B_j+\delta)\vee 1}.
    \end{equation}
By just adding $-\delta\mathbf{e}_j$ or $\delta\mathbf{e}_j$ to the first bias vector of $\btheta_{j-1}$, we can easily construct two network parameters $\btheta_{j-1}^-\in\Theta_{(K_j,M_j)}^{\le B_j+\delta}$ and $\btheta_{j-1}^+ \in\Theta_{(K_j,M_j)}^{\le B_j+\delta}$ such that $\net(\btheta_{j-1}^-)(\cdot)\equiv g_{j-1}(\cdot-\delta\mathbf{e}_j)$ and $\net(\btheta_{j-1}^+)(\cdot)\equiv g_{j-1}(\cdot+\delta\mathbf{e}_j)$. On the other hand, by \cref{lemma:nn_mid}, there is a network parameter $\btheta_{\mid}\in\Theta_{(4,14)}^{\le 1}$ that exactly recovers the middle value function. Therefore, using \cref{lemma:nn_composite}, we prove \labelcref{eq:mid_induction}.
\end{proof}

\subsubsection{Proof of \cref{lemma:dnn:approx:inner}}
\label{appen:proof_nn_inner_approx}

Before the proof, we give five lemmas used in the proof. First, we borrow the following two lemmas from \citetS{schmidt2020nonparametric}.

\begin{lemma}[Lemma A.2 of \citetS{schmidt2020nonparametric}]
\label{lemma:nn_multiply}
For any $K\in\bN_{\ge2}$, there exists a network parameter
    \begin{equation*}
        \btheta_{\mathrm{mult},K}\in\hTheta_{\M} \mbox{ with } \M=\del{2,6,\dots,6,1}\in\bN^{6+K}
    \end{equation*}
such that
    \begin{equation*}
        \sup_{(\x_1,\x_2)\in[0,1]^2}\abs{\net(\btheta_{\mathrm{mult},K})(x_1,x_2)-x_1x_2}\le2^{-K}
    \end{equation*}
and $\net(\btheta_{\mathrm{mult},K})(x_1,0)=\net(\btheta_{\mathrm{mult},K})(0,x_2)=0$ for any $x_1,x_2\in[0,1]$.
\end{lemma}

\begin{lemma}[Lemma A.4 of \citetS{schmidt2020nonparametric}]
\label{lemma:nn_monomial}
For any $d\in\bN$, $\beta>0$ and $K\in\bN_{\ge2}$, there exists a network parameter
    \begin{equation*}
        \btheta_{\mathrm{mon}, d,\beta, K}\in\hTheta_{\M} \mbox{ with } \M=\del{d,6\ceil{\beta}|\cA_{d,\beta}|,\dots,6\ceil{\beta}|\cA_{d,\beta}|,|\cA_{d,\beta}|}\in\bN^{3+(K+5)\ceil{\log_2(\beta\vee1)}}
    \end{equation*}
such that
    \begin{equation*}
        \sup_{\x\in[0,1]^d}\abs{\net( \btheta_{\mathrm{mon}, d,\beta, K})(\x)-(\x^{\a})_{\a\in\cA_{d,\beta}}}_{\infty}\le \beta^22^{-K}
    \end{equation*}
and $\net( \btheta_{\mathrm{mon},d, \beta, K})\in[0,1]^{|\cA_{d,\beta}|}$.
\end{lemma}

The next lemma introduces a neural network that appropriately ``discretizes'' an one-dimensional input. This neural network is useful to approximate the function $\fv(\cdot)$ defined in \labelcref{eq:grid_allocation}.

\begin{lemma}
\label{lemma:nn_step}
Let $T_0\in\bN$, $T:=T_0^2$ and $\delta\in(0,1/(3T)]$. Then there exists a network parameter $\btheta_{\mathrm{step}, T_0,\delta}\in\hTheta_{(1,T_0, 2, T_0,1)}$ such that
    \begin{equation*}
        \net(\btheta_{\mathrm{step}, T_0, \delta})(x)=\frac{\floor{Tx}}{T}-\frac{T-1}{T}\ind(x=1)
    \end{equation*}
for any $x\in\icubes(T, \delta,1):=[0,1]\setminus(\cup_{t=1}^{T-1}(t/T-\delta, t/T))$ and $|\btheta|_\infty\le (\delta T)^{-1}$
\end{lemma}

\begin{proof}
Let $\mathrm{step}:[0,1]\mapsto [0,1]$ be a function defined as
    \begin{equation*}
        \mathrm{step}(x)=\sum_{t=1}^{T-1}\frac{1}{\delta T}\del{\relu\del{x-\frac{t}{T}+\delta}-\relu\del{x-\frac{t}{T}}}
    \end{equation*}
for $x\in[0,1]$. Then we have that $\mathrm{step}(x)=\floor{Tx}/T-\{(T-1)/T\}\ind(x=1)$ for any $x\in\icubes(T, \delta,1)$. Thus, it remains to construct a neural network that recovers the function $\mathrm{step}(\cdot)$. 

Let $\btheta_{1}\in\hTheta_{(1,2T_0-2,1)}$ and $\btheta_{2}\in\hTheta_{(1,2T_0-2,1)}$  be network parameters such that
    \begin{align*}
        \net(\btheta_{1})(x)&=\sum_{t=1}^{T_0-1}\frac{T_0}{\delta T}\del{\relu\del{x-t\frac{T_0}{T}+\delta}-\relu\del{x-t\frac{T_0}{T}}}, \\
        \net(\btheta_{2})(x)&=\sum_{t=1}^{T_0-1}\frac{1}{\delta T}\del{\relu\del{x-\frac{t}{T}+\delta}-\relu\del{x-\frac{t}{T}}}
    \end{align*}
for $x\in\R$. Then for any $x\in\icubes(T, \delta,1)$, we have $\net(\btheta_{1})(x)=\floor{T_0x}/T_0-\{(T_0-1)/T_0\}\ind(x=1)$. Since
    \begin{equation*}
        x-\frac{\floor{T_0x}}{T_0}\in\bigcup_{t=0}^{T_0-1}\sbr{\frac{t}{T},\frac{t+1}{T}-\delta}\subset\icubes(T, \delta,1)
    \end{equation*}
for any $x\in\icubes(T, \delta,1)\setminus\{1\}$, we have
    \begin{align*}
        \net(\btheta_1)+&\net(\btheta_{2})(x-\net(\btheta_{1})(x))\\
        &=\frac{\floor{T_0x}}{T_0}-\frac{T_0-1}{T_0}\ind(x=1)+\frac{\floor{T(x-\floor{T_0x}/T_0)}}{T} +\frac{T_0-1}{T}\ind(x=1)
         \\
        &=\frac{T_0\floor{T_0x}}{T}+\frac{\floor{T x-T_0\floor{T_0x}}}{T}+\frac{T-1}{T}\ind(x=1)\\
            &=\frac{\floor{T x}}{T}+\frac{T-1}{T}\ind(x=1).
    \end{align*}
Thus, using the identity $\relu(x)=x$ for $x\ge 0$, a network parameter $\btheta_{\mathrm{step}, T_0,\delta}$ such that $\net(\btheta_{\mathrm{step}, T_0,\delta})\equiv \relu(\relu(\net(\btheta_1))) +\net(\btheta_{2})(\relu(\relu(x)-\net(\btheta_{1})(x)))$ satisfies the desired properties in the lemma.
\end{proof}

The following lemma constructs a neural network enumerates the mutivariate indices $\t\in[0:(T-1)]^d$.

\begin{lemma}
\label{lemma:nn_enumeration}
Let $d\in\bN$, $T_0\in\bN$ and $T:=T_0^2$. Then, there is a network parameter $\btheta_{\mathrm{enum},T_0}\in\hTheta_{1,2d(T_0+1),2d,2}$ such that
    \begin{equation*}
        \net(\btheta_{\mathrm{enum},T_0})(\t)=\del{1+\sum_{j=1}^d\floor{\frac{t_j}{T_0}}T_0^{j-1},1+\sum_{j=1}^d\del{t_j-T_0\floor{\frac{t_j}{T_0}}}T_0^{j-1} }
    \end{equation*}
for any $\t\in[0:(T-1)]^d$ and $|\btheta_{\mathrm{enum},T_0}|_\infty\le T_0^{d-1} $.
\end{lemma}

\begin{proof}
Let $\btheta_1\in\hTheta_{1, 2T_0,1}$ be a network parameter such that
    \begin{align*}
        \net(\btheta_1)(x)=\sum_{t'=1}^{T_0}2\cbr{\relu\del{x-t'T_0+\frac{1}{2}}-\relu\del{x-t'T_0}},
    \end{align*}
for $x\in\R$. Then, for any $t\in[0:(T-1)]$,  we have $\net(\btheta_1)(t)=\floor{t/T_0}$, and thus 
    \begin{align*}
        \net(\btheta_{2,j})(t)
        &=\del{T_0^{j-1}\net(\btheta_1)(t), T_0^{j-1}(\relu(t)-T_0\net(\btheta_1)(t))}\\
        &=\del{\floor{\frac{t_j}{T_0}}T_0^{j-1},\del{t_j-T_0\floor{\frac{t_j}{T_0}}}T_0^{j-1} }
    \end{align*}
for each $j\in[d]$, where $\btheta_{2,j}\in\hTheta_{1,2T_0+1,2}$ is a neural network with $|\btheta_{2,j}|_\infty\le T_0^{(j-1)\vee 2}$. Since $\relu(\net(\btheta_{2,j})(t))=\net(\btheta_{2,j})(t)$ for any $t\in[0:(T-1)]$, the neural network $\btheta_{\mathrm{enum},T_0}\in\hTheta_{1,2d(T_0+1),2d,2}$ defined as
    \begin{equation*}
        \net(\btheta_{\mathrm{enum},T_0})\equiv\sum_{j=1}^d\relu(\net(\btheta_{2,j}))+\one_2
    \end{equation*}
satisfies the desired condition of the lemma.
\end{proof}

Lastly, we construct the following ``point-fitting'' neural network that approximates the local Taylor coefficients $\partial^{\a}f(\t/T)$ for $\a\in\cA_{d,\beta}$ and $\t\in[0:T-1]^d$.

\begin{lemma}
\label{lemma:nn_point_fit}
Let $T_0\in\bN$ and $\Y:=(y_{t_1,t_2})_{t_1\in[T_0], t_2\in[T_0]}\in[0,1]^{T_0\times T_0}$. Then there exists a network parameter $\btheta_{\mathrm{fit},\Y,T_0}\in\Theta_{(6,4T_0)}^{\le B}$ with $B:= \max\cbr{3, T_0+1/6}$ such that
    \begin{equation*}
        \net(\btheta_{\mathrm{fit},\Y,T_0})(t_1,t_2)=y_{t_1,t_2}
    \end{equation*}
for any $(t_1,t_2)\in[T_0]^{\otimes 2}$.
\end{lemma}

\begin{proof}
For $T\in\bN$, let $\xi_{\mathrm{id},T}:=\net(\btheta_{\mathrm{id},T})$, where $\btheta_{\textup{id},T}\in\Theta_{(T, 2)}^{\le 1}$ is the depth $T$ network parameter constructed in \cref{lemma:nn_identity_ft}, whose realization is the identity on $\R$. For $t\in\bN$, define a function $\xi_{\mathrm{hat},t}:\R\mapsto [0,1]$ as
    \begin{align*}
        \xi_{\mathrm{hat},t}(x)
        =\relu\del{1-3\cbr{\relu\del{x-t-\frac{1}{6}} + \relu\del{-x+t-\frac{1}{6}}} }
    \end{align*}
for $x\in\R.$ Note that $\xi_{\mathrm{hat},t}(x)=1$ if $x\in[t-1/6,t+1/6]$, $\xi_{\mathrm{hat},t}(x)=0$ if $x\in(-\infty,t-1/2]\cup[t+1/2,\infty)$ and $0\le \xi_{\mathrm{hat},t}\le 1$ on $\R$. For notational convenience, let $\xi_{\mathrm{hat},1:T_0}:=(\xi_{\mathrm{hat},t})_{t\in[T_0]}\in[0,1]^{T_0}$. We also define a function $\xi_{\mathrm{test},t}:\R\times[0,1]\mapsto [0,1]$ as
    \begin{align*}
        \xi_{\mathrm{test},t}(x,y)&=\relu\del{\xi_{\mathrm{id},2}(y)-3\cbr{\relu\del{x-t-\frac{1}{6}} + \relu\del{-x+t-\frac{1}{6}}} }\\
        &=\relu\del{\relu(y)-\relu(-y)-3\cbr{\relu\del{x-t-\frac{1}{6}} + \relu\del{-x+t-\frac{1}{6}}} }
    \end{align*}
for $(x,y)\in\R\times[0,1]$. Then, for any $y\in[0,1]$, we have $\xi_{\mathrm{test},t}(x,y)=y$ for $x\in[t-1/6,t+1/6]$ and  $\xi_{\mathrm{test},t}(x,y)=0$ for $x\in(-\infty,t-1/2]\cup[t+1/2,\infty)$.

Next, recall $\Y:=(y_{t_1,t_2})_{t_1\in[T_0], t_2\in[T_0]}\in[0,1]^{T_0\times T_0}$. Note that, by the assumption that every $y_{t_1,t_2}$ lies in $[0,1]$ and the fact that $\sum_{t=1}^{T_0}\xi_{\mathrm{hat},t}(x)\in [0,1]$ for any $x\in\R$, we have $\mathbf{e}_{T_0,t}^\top\relu(\Y\xi_{\mathrm{hat},1:T_0}(x))\in[0,1]$ for any $x\in\R$ and any $t\in[T_0]$, where $\mathbf{e}_{T_0,t}\in\R^{T_1}$ is the $T_0$-dimensional vector with a 1 in the $t$-th coordinate and 0's elsewhere. Then, since both $\xi_{\mathrm{hat},t}$ and $\xi_{\mathrm{test},t}$ can be exactly represented by neural networks with depth 3, there is a depth 6 neural network $\btheta_{\mathrm{fit},\Y,T_0}\in\hTheta_{(2,2T_0+2,T_0+2, T_0+1, 4T_0, T_0, 1) }$ such that
    \begin{align*}
       \net(\btheta_{\mathrm{fit},\Y,T_0})(x_1,x_2)= \sum_{t=1}^{T_0}\xi_{\mathrm{test},t}\del{\relu(\xi_{\mathrm{id},3}(x_2)),\mathbf{e}_{T_0,i}^\top\relu(\Y\xi_{\mathrm{hat},1:T_0}(x_1) )}
    \end{align*}
for $(x_1,x_2)\in\R^2.$ Then it is easy to see that $\net(\btheta_{\mathrm{fit},\Y,T_0})(t_1,t_2)=y_{t_1,t_2}$ for any $(t_1,t_2)\in[T_0]^{\otimes 2}$ and that $|\btheta_{\mathrm{fit},\Y,T_0}|_\infty\le \max\cbr{3, T_0+1/6}$. We complete the proof.
\end{proof}

We are ready to prove \cref{lemma:dnn:approx:inner}.

\begin{proof}[Proof of \cref{lemma:dnn:approx:inner}]
By Taylor's theorem, the assumption that $\|f\|_{\cH^{\beta,d}}\le F_0$ and the inequality that $|(\x-\fv(\x))^{\a}|\le |\x-\fv(\x)|_\infty^{|\a|_1}$, we have
    \begin{equations}
    \label{eq:app_err_taylor}
         \sup_{\x\in\icubes(T,\delta)}&\abs{f(\x)-\sum_{\a\in\cA_{d,\beta}}\frac{\partial^{\a}f(\fv(\x))}{\a!}(\x-\fv(\x))^{\a}}\\
         &\le \sup_{\x\in\icubes(T,\delta)}\abs{\sum_{\a\in\bN_0^d:|\a|_1=\ceil{\beta-1}}\sup_{u\in[0,1]} \sbr{\partial^{\a}f(\fv(\x)+u(\x-\fv(\x)))-\partial^{\a}f(\fv(\x))}\frac{(\x-\fv(\x))^{\a}}{\a!}}\\
         &\le \sup_{\x\in\icubes(T,\delta)}\sum_{\a\in\bN_0^d:|\a|_1=\ceil{\beta-1}}\del{F_0\sup_{u\in[0,1]} \abs{u(\x-\fv(\x))}_\infty^{\beta-\ceil{\beta-1}}\frac{|\x-\fv(\x)|_\infty^{|\a|_1}}{\a!}}\\
         &= \del{F_0\sum_{\a\in\bN_0^d:|\a|_1=\ceil{\beta-1}}\frac{1}{\a!}}\sup_{\x\in\icubes(T,\delta)}|\x-\fv(\x)|_\infty^{\beta}\\
         &\le 
          C_{\beta, d, F_0}T^{-\beta},
    \end{equations}
where we used multi-index notation $\x^{\a}:=\prod_{j=1}^dx_j^{a_j}$ and $\a!:=\prod_{j=1}^d(a_j!)$ and let $C_{\beta, d, F_0}:=F_0\sum_{\a\in\bN_0^d:|\a|_1=\ceil{\beta-1}}(1/\a!)$. Thus it suffices to find a neural network approximation of the local Taylor polynomial function
    \begin{equation}
        P_f(\x)=\sum_{\a\in\cA_{d,\beta}}\frac{\partial^{\a}f(\fv(\x))}{\a!}(\x-\fv(\x))^{\a}
    \end{equation}
supported on $\icubes(T,\delta)$. 

Let $T_0:=\lfloor M^{1/d}\rfloor=\sqrt{T}$. For $\btheta_{\mathrm{step}, T_0,\delta}\in\Theta_{(4, T_0\vee2)}^{\le (\delta T)^{-1}}$ defined in \cref{lemma:nn_step}, we let $\xi_{\mathrm{grid}}:\icubes(T,\delta)\mapsto [0,1]^d$ be a function defined as
    \begin{equation*}
        \xi_{\mathrm{grid}}(\x)=\del{\net(\btheta_{\mathrm{step}, T_0,\delta})(x_j)}_{j\in[d]}
    \end{equation*}
for $\x\in\icubes(T,\delta)$. Note that $\xi_{\mathrm{grid}}$ is a neural network with depth $4$ and width $d(T_0\vee2)\le 2dT_0$, which satisfies $\xi_{\mathrm{grid}}\equiv \fv$ on $\icubes(T,\delta)$. We also define a function $\xi_{\mathrm{mon}}:=(\xi_{\mathrm{mon}, \a})_{\a\in\cA_{d,\beta}}:\icubes(T,\delta)\mapsto \R^{|\cA_{d,\beta}|}$ such that
    \begin{equation*}
        \xi_{\mathrm{mon}}(\x)=\net(\btheta_{\mathrm{mon},d, \beta, K})(\x-  \xi_{\mathrm{grid}}(\x))
    \end{equation*}
for $\x\in\icubes(T,\delta)$, where  $\btheta_{\mathrm{mon}, d,\beta, K}$ is a network parameter with depth $2+(K+5)\lceil\log_2(\beta\vee1)\rceil$ and width $6\lceil\beta\rceil|\cA_{d,\beta}|\le6\lceil\beta\rceil(d+1)^{\lceil\beta\rceil}$, which is defined in \cref{lemma:nn_monomial}.  Then by \cref{lemma:nn_monomial}, we have  that
    \begin{equations}
    \label{eq:app_err_mom}
        \sup_{\x\in\in\icubes(T,\delta)}&\abs{\xi_{\mathrm{mon},\a}(\x)-(\x- \fv(\x))^{\a}}\\
        &=\sup_{\x\in\in\icubes(T,\delta)}\abs{\net(\btheta_{\mathrm{mon}, d,\beta, K})(\x- \fv(\x))-(\x- \fv(\x))^{\a}}\le\beta^22^{-K}
    \end{equations}
for any $\a\in\cA_{d,\beta}$.
On the other hand, since the identity map $\x\mapsto\x$ can be exactly represented by a neural network with width $2d$ by \cref{lemma:nn_identity_ft},   \cref{lemma:nn_composite} implies that  there is a network parameter $\btheta_1$ with depth $K_1:=6+(K+5)\ceil{\log_2(\beta\vee1)}$ and width $M_1:=(4dM^{1/d})\vee( 6\lceil\beta\rceil(d+1)^{\lceil\beta\rceil}) \ge (2d(T_0+1))\vee( 6\lceil\beta\rceil(d+1)^{\lceil\beta\rceil})$, which satisfies $\xi_{\mathrm{mon}}\equiv\net(\btheta_1)$ on $\icubes(T,\delta)$. 

We then for each $\a\in\cA_{d,\beta}$ construct a neural network that approximates the local Taylor coefficients. First let $\xi_{\mathrm{enum}}:[0:T-1]^d\mapsto [T_0^d]^{\otimes2}$ be a function such that
    \begin{equation*}
        \xi_{\mathrm{enum}}(\t)=\del{1+\sum_{j=1}^d\floor{\frac{t_j}{T_0}}T_0^{j-1},1+\sum_{j=1}^d\del{t_j-T_0\floor{\frac{t_j}{T_0}}}T_0^{j-1} }.
    \end{equation*}
Note that the map $\xi_{\mathrm{enum}}$ is a bijection and can be exactly represented by a neural network $\btheta_{\mathrm{enum},T_0}\in\hTheta_{1,2d(T_0+1),2d,1}$ with $|\btheta_{\mathrm{enum},T_0}|_\infty\le T_0^{d-1}$ by \cref{lemma:nn_enumeration}. On the other hand, for $\a\in\cA_{d,\beta}$ and $\t\in[0:T-1]^d$, let
    \begin{equation*}
        h_{\a, \t}:=\frac{1}{2}\del{F_0^{-1}\partial^{\a}f\del[0]{T^{-1}\t}+1}
    \end{equation*}
and $\h_{\a}:=(h_{\a,\t})_{\t\in[0:T-1]^d}$. With this notation, we can write
    \begin{equation*}
        P_f(\x)=\sum_{\a\in\cA_{d,\beta}}\frac{F_0}{\a!}\del{2\h_{\a,T\fv(\x)}-1}(\x-\fv(\x))^{\a}
    \end{equation*}
for any $\x\in\icubes(T,\delta)$.  We use \cref{lemma:nn_point_fit} to complete the construction. Since $f\in\cH^{\beta,d, F_0}$, we have $|\h_{\a}|_\infty\le 1$, and so we can apply the lemma. Let $\xi_{\mathrm{coef},\a}:\icubes(T,\delta)\mapsto\R$ be a function defined as
    \begin{equation*}
        \xi_{\mathrm{coef},\a}(\x)
        =\net(\btheta_{\mathrm{fit},\h_{\a},T_0^d})\circ\net(\btheta_{\mathrm{enum},T_0})
        \circ\del[1]{T\net(\btheta_{\mathrm{step}, T_0,\delta})(\x)}
    \end{equation*}
for $\x\in\icubes(T,\delta)$, where $\btheta_{\mathrm{fit}, \h_{\a},T_0^d}$ is a network parameter with depth $6$ and width $4T_0^d$, which satisfies 
    \begin{equation*}
        \net\del[1]{\btheta_{\mathrm{fit}, \h_{\a},T_0^d}}\circ \xi_{\mathrm{enum}}(\t)=h_{\a, \t}
    \end{equation*}
and $|\btheta_{\mathrm{fit}, \h_{\a},T_0^d}|_\infty\le T_0^d+2\le M+2$. The existence of the network parameter $\btheta_{\mathrm{fit}, \h_{\a},T_0^d}$ is guaranteed by  \cref{lemma:nn_point_fit}. Then since $\net(\btheta_{\mathrm{enum},T_0})\circ(T\net(\btheta_{\mathrm{step}, T_0,\delta}))\equiv \xi_{\mathrm{enum}}\circ\fv$ on $\icubes(T,\delta)$,  we have
    \begin{equations}
        \label{eq:app_err_coef}
          &\abs{\del{ 2F_0\xi_{\mathrm{coef},\a}(\x)-F_0}-\partial^{\a}f\del{\fv(\x)} }\\
          &= \abs{\del{ 2F_0\net(\btheta_{\mathrm{fit},\h_{\a},T_0^d})\circ\net(\btheta_{\mathrm{enum},T_0})(T\fv(\x))-F_0}-\partial^{\a}f\del{\fv(\x)} }\\
          &=     \abs{2F_0 h_{\a, \t}-F_0 -\partial^{\a}f\del{\fv(\x)} }=0
    \end{equations}
for any $\x\in\icubes(T,\delta)$. Due to \cref{lemma:nn_composite}, there is a neural network $\btheta_{2,\a}\in\Theta_{(K_2, M_2) }^{\le B_2}$ with $K_2:=13=6+4+3$, $ M_2:= 4dM\ge\max\{2d(T_0+1),4T_0^d\}$ and $B_2:=\max\{M+2, 2(\delta M^{2/d})^{-1}\}$ (since $\delta T =\delta\lfloor M^{1/d}\rfloor^{2}\ge \delta M^{2/d}/2$) such that $  \xi_{\mathrm{coef},\a}\equiv\net(\btheta_{2,\a})$
on $\icubes(T,\delta)$.

Lastly, let $\xi_{\mathrm{Taylor}}:\icubes(T,\delta)\mapsto \R$ be a function defined as
        \begin{equation*}
        \xi_{\mathrm{Taylor}}(\x)
        =\sum_{\a\in\cA_{d,\beta}}\frac{F_0}{\a!}\cbr{2\net(\btheta_{\mathrm{mult}, K})\del{\xi_{\mathrm{coef},\a}(\x), \xi_{\mathrm{mon}, \a}(\x)}-\xi_{\mathrm{mon}, \a}(\x)}
    \end{equation*}
for $\x\in\icubes(T,\delta)$, where $\btheta_{\mathrm{mult}, K}\in\Theta_{(5+K, 6)}^{\le 1}$ is the network parameter defined in \cref{lemma:nn_multiply}.  Since $\xi_{\mathrm{coef},\a}, \xi_{\mathrm{mon}, \a}\in[0,1]$, we have
        \begin{align*}
        &\abs{\net(\btheta_{\mathrm{mult}, K})\del{\xi_{\mathrm{coef},\a}(\x), \xi_{\mathrm{mon}, \a}(\x)}-h_{\a, T\fv(\x)}(\x-\fv(\x))^{\a}}\\
        &\le \abs{\net(\btheta_{\mathrm{mult}, K})\del{\xi_{\mathrm{coef},\a}(\x), \xi_{\mathrm{mon}, \a}(\x)} -h_{\a, T\fv(\x)}\xi_{\mathrm{mon}, \a}(\x)}\\
        & \hspace{15em}+\abs{h_{\a, T\fv(\x)}\xi_{\mathrm{mon}, \a}(\x)-h_{\a, T\fv(\x)}(\x-\fv(\x))^{\a} }\\
        &\le 2^{-K}+\abs{\xi_{\mathrm{coef},\a}(\x)\xi_{\mathrm{mon}, \a}(\x)-h_{\a, T\fv(\x)}\xi_{\mathrm{mon}, \a}(\x)}\\
        & \hspace{15em}+\abs{h_{\a, T\fv(\x)}\xi_{\mathrm{mon}, \a}(\x)-h_{\a, T\fv(\x)}(\x-\fv(\x))^{\a} }\\
        &\le 2^{-K}+ \abs{\xi_{\mathrm{mon}, \a}(\x)-(\x-\fv(\x))^{\a}}\\
        &\le (\beta^2+1)2^{-K}
    \end{align*}
for any  $\x\in\icubes(T,\delta)$ and any $\a\in\cA_{d,\beta}$, where the second inequality follows from \cref{lemma:nn_multiply}, the fourth one from \labelcref{eq:app_err_coef}  and the last one from \labelcref{eq:app_err_mom}. Thus, it follows that
      \begin{align*}
        \sup_{\x\in\icubes(T,\delta)}\abs{\xi_{\mathrm{Taylor}}(\x)-P_f(\x)}
        &\le \sum_{\a\in\cA_{d,\beta}}\frac{F_0}{\a!}(3\beta^2+2)2^{-K}\\
        &\le (3\beta^2+2)F_0(d+1)^{\lceil\beta\rceil}2^{-K}.
    \end{align*}
Combining the above display and \labelcref{eq:app_err_taylor}, we obtain the error bound in the conclusion. On the other hand, using \cref{lemma:nn_composite} again, we can show that there is a network parameter $\btheta\in\Theta_{(K^\dag, M^\dag)}^{B^\dag}$ with $K^\dag:=(K_1\vee K_2)+K+6=K_1+K+6$, $M^\dag:=M_1+M_2$ and $B^\dag= \max\{M+2, 2(\delta M^{2/d})^{-1}, F_0\}$ such that $\net(\btheta)\equiv  \xi_{\mathrm{Taylor}}$ on $\icubes(T,\delta)$, which completes the proof.
\end{proof}

\subsection{Approximation of composition structured functions}
\label{appen:proof_dnn_approx:composition}

\subsubsection{Proof of \cref{lemma:dnn:approx:composite}}

We need the following technical lemma.

\begin{lemma}[Lemma 3 of \citetS{schmidt2020nonparametric}]
\label{lemma:composite_err}
For any two functions $f^{[1]}=f_r^{[1]}\circ(f_{r-1,j}^{[1]})_{j\in[d_{r}]}\cdots\circ (f_{1,j}^{[1]})_{j\in[d_{2}]}$ and $f^{[2]}=f_r^{[2]}\circ(f_{r-1,j}^{[2]})_{j\in[d_{r}]}\cdots\circ (f_{1,j}^{[2]})_{j\in[d_{2}]}$ in $\cF^{\textsc{comp}}\del{r,\mathbf{d}, \bbeta,\s, F_0}$,  
    \begin{align*}
        &\norm[1]{f^{[1]}-f^{[2]}}_\infty
       \le F_0(2F_0)^{\sum_{\ell=1}^r\beta_{\ell}}\sum_{\ell=1}^r\max_{j\in [d_{\ell+1}]}\norm[1]{f_{\ell,j}^{[1]}-f_{\ell,j}^{[2]}}_\infty^{\prod_{h=\ell+1}^r(\beta_h\wedge 1)}
    \end{align*}
with the convention $\prod_{h=r+1}^r(\beta_h\wedge 1)=1$.
\end{lemma}

We prove \cref{lemma:dnn:approx:composite} using \cref{lemma:dnn:approx:holder,lemma:composite_err} as follows.

\begin{proof}[Proof of \cref{lemma:dnn:approx:composite}]
For $f^\star\in\cF^{\textsc{comp}}\del{r,\mathbf{d}, \bbeta,\s, F_0}$, we write $f^\star:=f_r^\star\circ(f_{r-1,j}^\star)_{j\in[d_{r}]}\circ\cdots\circ (f_{1,j}^\star)_{j\in[d_{2}]}$, where each $f_{\ell,j}\in\cH^{\beta_\ell,s_\ell, F_0}$. For notational convenience, we write $f_{r,1}^\star:=f_r^\star$ and $d_{r+1}:=1$ so that the set of all component functions of $f^\star$ can be written as $\{f_{\ell,j}:j\in[d_{\ell+1}], \ell\in[r]\}$. By \cref{lemma:dnn:approx:holder}, for any large enough $\tiK\in\bN$ and $\tiM\in\bN$ and any $B>\fc_1'\max_{\ell\in[r]}\tiM^{\max\{1,2(\beta_\ell-1)/s_\ell\}}$, there exists absolute constants $\fc_1'>0$ and $\fc_2'>0$ such that there is a network parameter $\btheta_{\ell,j}^*\in \Theta^{\le B}_{(\tiK,\tiM)}$ satisfying
    \begin{equation*}
        \norm[1]{\net(\btheta_{\ell,j}^*)-f_{\ell,j}^\star}_\infty\le \fc_2'(2^{-\tiK}+\tiM^{-2\beta_\ell/s_\ell})
    \end{equation*}
for each $j\in[d_{\ell+1}]$ and $\ell\in[r]$. Then the function defined as
    \begin{align*}
        f_{\text{net,comp}}:=\net(\btheta_{r,1}^*)\circ\del[1]{\net(\btheta_{r-1,j}^*)}_{j\in[d_{r}]}\circ\cdots
        \circ \del[1]{\net(\btheta_{1,j}^*)}_{j\in[d_{2}]}
    \end{align*}
satisfies, by \cref{lemma:composite_err},  
    \begin{align*}
      \norm{f^\star- f_{\text{net,comp}} }_\infty
      &\le \fc_3'\sum_{\ell=1}^{r}\del{2^{-\tiK}+\tiM^{-2(\beta_\ell/s_\ell)}}^{\prod_{h=\ell+1}^r(\beta_h\wedge 1)}\\
      &\le \fc_3'\sum_{\ell=1}^{r}\del{2^{-\tiK\prod_{h=\ell+1}^r(\beta_h\wedge 1)}+\tiM^{-2(\beta_\ell/s_\ell)\prod_{h=\ell+1}^r(\beta_h\wedge 1)}}\\
      &\le r\fc_3' \del{\e^{-\fc_4'\tiK}+\tiM^{-2\max_{\ell\in[r]}(\beta_{\ge \ell}/s_\ell)}}
    \end{align*}
for some absolute constants $\fc_3'>0$ and $\fc_4'>0$, where we use the inequality $(x+y)^{u} \le x^u+y^u$ for $x,y\in\R_+$ and $u\in(0,1]$ for deriving the second inequality. Furthermore, by parallelizing the neural networks $(\net(\btheta_{\ell,j}^*))_{j\in[d_{\ell+1}]}$ for $\ell\in[r]$ and staking these $r$ parallelized neural networks by  \cref{lemma:nn_composite}, we have the network parameter $\btheta^*$ with depth $K:=r\tiK$,  width $M:=2\tiM\max_{2\le \ell\le r-1}d_\ell $ and magnitude bound $B$ such that $\net(\btheta^*)\equiv f_{\text{net,comp}}$, which completes the proof.
\end{proof}

\bibliographystyleS{plainnat}
\bibliographyS{_references}

\end{appendices}

\end{document}